\documentclass[11pt,reqno]{amsart}
\usepackage{amsmath,amsthm,amsfonts,amssymb,mathrsfs,bm,graphicx,stmaryrd}

\usepackage{mathtools}
\usepackage{dsfont}
\usepackage{multicol}
\usepackage[colorlinks=true,linkcolor=blue]{hyperref}
\hypersetup{bookmarksdepth=2}
\usepackage{enumitem}
\usepackage{bbm}
\usepackage{mathrsfs}
\usepackage{versions}
\usepackage{subcaption}

\usepackage[letterpaper,hmargin=1.0in,vmargin=1.0in]{geometry}
\parskip 	\smallskipamount

\newtheorem{theorem}{Theorem}[section]
\newtheorem{lemma}[theorem]{Lemma}

\newtheorem{proposition}[theorem]{Proposition}

\newtheorem{remark}[theorem]{Remark}
\newtheorem{definition}[theorem]{Definition}
\numberwithin{equation}{section}

\newtheorem{question}{Question}
\newtheorem{maintheorem}{Theorem}

\newcommand{\E}{{\mathbb{E}}}

\newcommand{\Z}{\mathbb{Z}}

\newcommand{\R}{\mathbb{R}}

\newcommand{\e}{\varepsilon}

\newcommand{\C}{\mathbb{C}}
\newcommand{\D}{\mathbb{D}}

\newcommand{\N}{\mathbb{N}}

	\renewcommand{\P}{\mathbb{P}}
\newcommand{\Q}{\mathbb{Q}}

\newcommand{\T}{\mathbb{T}}

\newcommand{\cP}{\mathcal{P}}

\newcommand{\sF}{\mathscr{F}}

\newcommand{\mydot}{\,{\mathrel{\scalebox{0.7}{$\boldsymbol{\odot}$}}}\,}

\renewcommand{\d}{\ d}

\renewcommand{\emptyset}{\varnothing}
\renewcommand\d[1]{\ d #1}

\renewcommand{\setminus}{\backslash}

\DeclareMathOperator{\Var}{Var}

\newcommand{\emm}{\mathbf{Metric}}
\newcommand{\emf}{\mathbf{Field}}
\newcommand{\emv}{\mathbf{Measure}}

\def\ba{\begin{align}}
\def\ea{\end{align}}
\def\bs{\begin{split}}
\def\es{\end{split}}

\begin{document}
\title[Environment seen from  infinite geodesics in 
LQG]{Environment seen from  infinite geodesics in 
Liouville Quantum Gravity
}

\author{Riddhipratim Basu, Manan Bhatia, Shirshendu Ganguly}
\date{}

\begin{abstract} 
First passage percolation (FPP) on $\Z^d$ or $\R^d$ is a canonical model of  a
random metric space where the standard Euclidean geometry is distorted by
random noise. Of central interest is the length and the geometry of the
geodesic, the shortest path between points. Since the latter, owing to its
length minimization, traverses through atypically low values of the underlying
noise variables, it is an important problem to quantify the disparity between
the environment rooted at a point on the geodesic and the typical one. We
investigate this in the context of $\gamma$-Liouville Quantum Gravity (LQG) (where $\gamma \in (0,2)$ is a parameter) -- a random
Riemannian surface induced on the complex plane by the random metric tensor
$e^{2\gamma h/d_{\gamma}} ({dx^2+dy^2}),$ where $h$ is the whole plane, properly centered,
Gaussian Free Field (GFF), and $d_\gamma$ is the associated dimension. We consider
the unique infinite geodesic $\Gamma$  from the origin, parametrized by the
logarithm of its chemical length, and show that, for an almost sure realization of $h,$ the  distributions
of the appropriately scaled field and the induced metric on a ball, rooted at a point ``uniformly" sampled on $\Gamma$, converge to  deterministic measures on the space of generalized functions and continuous metrics on the unit disk respectively.  
Moreover, towards a better understanding of the limiting objects living on the
unit disk, we show that they are singular with respect to their typical
counterparts, but become absolutely continuous away from the origin. Our
arguments rely on unearthing a regeneration structure with fast decay of
correlation in the geodesic owing to coalescence and the
domain Markov property of the GFF.  While there have been significant recent advances around this question for stochastic planar growth models in the KPZ class, the present work initiates this research program in the context of LQG. 
\end{abstract}

\address{Riddhipratim Basu, International Centre for Theoretical Sciences, Tata Institute of Fundamental Research, Bangalore, India} 

\email{rbasu@icts.res.in}

\address{Manan Bhatia, International Centre for Theoretical Sciences, Tata Institute of Fundamental Research, Bangalore, India} 

\email{mananbhatia1701@gmail.com}

\address{ Shirshendu Ganguly, Department of Statistics, Evans Hall, University of California, Berkeley, CA
94720, USA} 

\email{sganguly@berkeley.edu }

\maketitle

\begin{figure}[h]
\centering
\includegraphics[scale=.15]{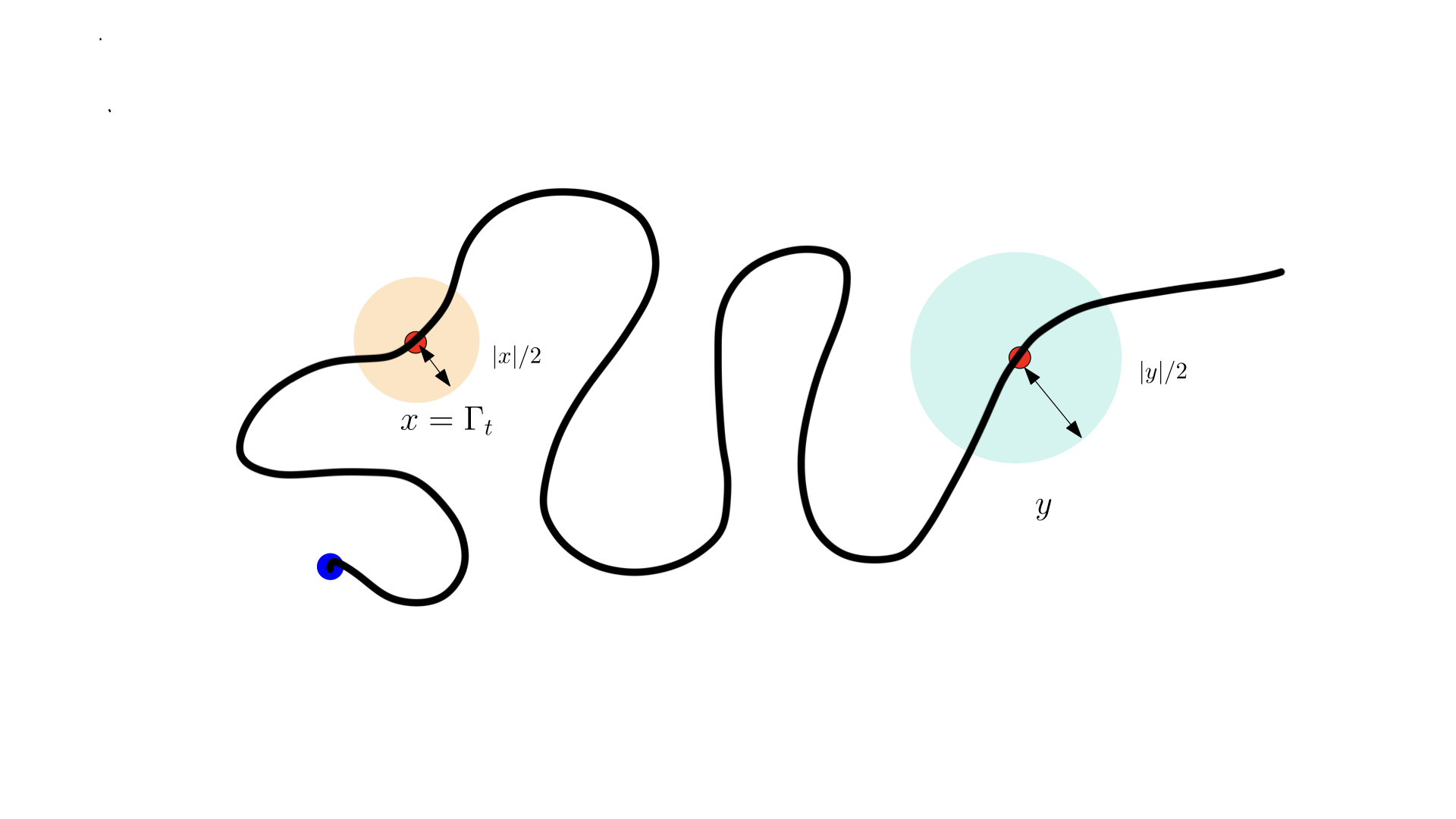}
\caption{Environment rooted along the infinite geodesic $\Gamma$ parametrized according to the log of the chemical distance from the origin. For a given realization of the whole plane GFF $h$, for a point $x=\Gamma_t$ (at chemical distance $t$ from the origin), we consider $\underline{h}_x,$ the field restricted to a ball centered around $x$ (of radius $|x|/2$ in the figure), rescaled and re-centered to fit the unit disk. The main result proves that almost surely, in the realization of the environment $h$,  the law of $\underline{h}_{\Gamma_{e^\mathfrak{t}}}$ where $\mathfrak{t}\sim
 \mathtt{Unif}(0,t)$, converges to a deterministic measure as $t\to \infty$. 
}
\label{f.emp1}
\end{figure}
\newpage
\tableofcontents
\newpage

\section{Introduction and main results}
\label{intro}
Consider a simple model of random geometry where Euclidean spaces are distorted by random noise. A
classical example capturing this is First Passage Percolation (FPP): a popular
model of fluid flow through inhomogeneous random media, where one puts random
weights on the edges of a graph and considers the first passage time between
two vertices, which is obtained by minimizing the total weight among all paths
between the two vertices. This serves as an example of a random metric space
by considering the shortest path metric in this graph with random edge
weights. FPP on Euclidean lattices was introduced by Hammersley and Welsh \cite{HW65}
in 1965 and has been studied extensively both in the statistical physics and
probability literature ever since. This model was one of the motivations of
developing the theory of sub-additive stochastic processes and the early
progresses using sub-additivity were made by Hammersley-Richardson-Kingman \cite{HW65,K73,R73} and culminated in the proof of the celebrated Cox-Durrett shape theorem \cite{CD81} establishing the first order law of large numbers behavior for passage times between far away points. 
The metric balls growing out of a point form a canonical example of stochastic
growth, modelling various natural phenomena including the spread of bacteria, propagation of flame fronts etc.\,.

Subsequently, since the seminal work of Kardar-Parisi-Zhang \cite{KPZ86}, there has been an explosion of activity in understanding planar random growth models through an i.i.d.\ noise field (either on the lattice or on $\R^2$) with deep predictions about universal behavior (associated with the so called KPZ universality class) and existence of critical exponents governing fluctuations driving most of the modern research around these models. 

Most of the research so far, however, has been centered around a handful of models, which are integrable owing to remarkable bijections to algebraic objects such as random matrices and admit exact formulae for the distributions for observables of interest. This enables the treatment of probabilistic observables by analyzing quantities such as Fredholm determinants arising naturally out of such algebraic connections. While certain variants of FPP, namely models of directed last passage percolation (LPP) do exhibit such integrable features and have played central roles in the aforementioned research advances, no model of standard FPP on $\Z^2$ is known to be exactly solvable and hence progress on this front has been notoriously slow. 

Among various objects of interest in FPP, one of the most primary ones is the geodesic or the shortest path between two points. Its length and its geometric features including its fluctuation from its  Euclidean counterpart, i.e., the straight line joining its two end points are typically studied observables. 
Note that geodesics, being the paths with the shortest length, tend to pass through edges with lower than typical noise. Thus it is of interest to investigate the difference between the typical environment and the one seen from a random point on the geodesic. 
As a first step towards this, Christopher Hoffman raised a variant of the following question (see \cite[Problem 4.4]{AIMPL}):  

\begin{question}\label{q1} Does the distribution on environments obtained by rooting it at a uniformly random vertex on the geodesic from $(0,0)$ to $(n,0)$ converge as $n$ tends to infinity? 
\end{question}
\begin{center}\textbf{This question is the central motivation behind the present work.}\end{center}

The question has attracted a lot of attention in the last couple of years ({see Section \ref{s:related}}) although Question \ref{q1} remains open in its full generality. In this paper, instead of the paradigm of two dimensional stochastic growth in the KPZ class, we consider the same question in the setting of the conformal field theory (CFT). A theory of random surfaces
called Liouville quantum gravity (LQG) was put forward in the 1980s as a canonical model of random two-dimensional Riemannian
manifolds. The subject has attracted an intense amount of mathematical
attention in recent years, because of its importance in several areas of mathematical physics
and its connections to random discrete surfaces called random planar maps. 

In contrast to the white noise or i.i.d.\ random variables driving the stochastic growth models in the KPZ universality class, the key feature of LQG is that the driving noise is the exponential of a two dimensional Gaussian Free Field (GFF), the latter being a canonical model of a log-correlated Gaussian field. While a formal definition requires a number of different elements and is postponed until later, momentarily it is useful to keep the following informal definition in mind. 

\begin{definition} For $\gamma \in (0, 2)$, a $\gamma$-Liouville quantum gravity
($\gamma$-LQG) surface is a random Riemannian manifold with random Riemannian metric tensor $e^{2\gamma/d_{\gamma} h}(dx^2+dy^2)$ where $h$ is a Gaussian free field (GFF) on some domain $U\subseteq \C$, $(dx^2+dy^2)$ is the Euclidean metric tensor on $U$ and $\d_\gamma$ is the Hausdorff dimension of the associated metric space.\footnote{The definition could seem somewhat circular given $d_{\gamma}$ is said to be the dimension of the LQG metric that the above expression purports to define. However, in fact, the exponent $d_\gamma$ was defined via various other means in \cite{ding2020}, which was then subsequently used in \cite{DDDH20, GM21} to prove the existence and the uniqueness of the LQG metric. Moreover, a-posteriori the same exponent was shown to be the Hausdorff dimension of the metric thus constructed.} 
\end{definition}

There has been an explosion of recent activity in the study of such random surfaces
with many success stories, all of which are impossible to cover here. Instead,
we simply mention the ones most relevant to the present work while pointing
the reader to \cite{GHS19} for a beautiful survey on related themes. A
rigorous construction of the LQG metric for all $\gamma \in (0,2)$ presented a major
challenge which was overcome in the series of works \cite{DDDH20, DFGPS20, 
GM19, GM20, GM21, GM19+}. Further, in
\cite{DFGPS20}, existence of exponents for natural observables such as distances between points were studied, and various bounds on the same
were presented. More recently, \cite{GPS20} carried out a related analysis of geometric properties of geodesics, the results of which are of particular importance for the present work.

The main result of the paper answers a variant of Question \ref{q1} for a canonical geodesic in the LQG metric and compares the environment seen from the geodesic to the typical environment. A study similar in spirit, where one considers the environment seen from a typical point sampled according to the LQG measure (informally, the measure with density $e^{\gamma h(\cdot)}$ with respect to Lebesgue measure; the formal definition appears in Section \ref{ss:LQG}) was already presented in the seminal work of Duplantier and Sheffield \cite{DS11}. In particular, they showed that almost surely such a point is a ``$\gamma$-thick" point, a notion introduced in the earlier work \cite{HMP10} quantifying the unusually high $h$ values. This confirms the intuition that the measure is indeed supported on higher than typical values. In light of this, the program initiated in this paper can be viewed as a metric counterpart of the above with the goal to understand the atypical-ness of the environment that the geodesic passes through. Related investigations have been carried out recently in \cite{GPS20}. Postponing further expansion on this to later, we now turn to the statements of our main results.

We start by introducing several objects that we will be relying on. 
\subsection{Definitions and notations}
We start with some notations that will be in play throughout.
We will use $\mathbb{D}$ to denote the open unit disk centered at the
		origin (in the complex plane) and more generally $\mathbb{D}_r$ to  denote the disk $\left\{ |x|<
		r \right\}$. For $z\in \C$, we use $\mathbb{D}_r(z)$ to denote the disk $\left\{
		|x-z|<r \right\}$. We analogously define the closed disks
		$\overline{\mathbb{D}}$, $\overline{\mathbb{D}}_r$ and
		$\overline{\mathbb{D}}_r(z)$. We will use $\mathbb{T}$ (resp.\ $\mathbb{T}(z)$) to
		denote the boundary of $\mathbb{D}$ (resp.\ $\mathbb{D}_1(z)$) and more generally,
		$\mathbb{T}_r$ (resp.\ $\mathbb{T}_r(z)$) will denote the boundary of $\mathbb{D}_r$ (resp.\ $\mathbb{D}_r(z)$). We
		also use $\mathbb{C}_{< r_1}$ to denote the region $\mathbb{D}_{r_1}$ to make the notation for annuli
		simpler. 
	 $\mathbb{C}_{\geq r_1}$ will denote the region
		$\mathbb{C}\setminus \mathbb{D}_{r_1}$,
		and  $\mathbb{C}_{>r_1}$ and $\mathbb{C}_{<r_1}$ are defined analogously.
		For positive numbers $r_1<r_2$,  $\mathbb{C}_{[r_1,r_2]}$
		will denote the annulus $\mathbb{C}_{\geq r_1}\cap
		\mathbb{C}_{\leq r_2}$. We will extend the notation, as in
		$\mathbb{D}_r(z)$, to denote the corresponding objects with
		center $z$. We will use the notation
		$[\![a,b]\!]$ to denote the discrete interval $[a,b]\cap
		\mathbb{Z}$ {whenever $a,b\in \mathbb{Z}$}.

\medskip	

We shall reserve the notation $h$ for the whole plane GFF on $\C$. 
The formal definition is non-trivial and will be described in detail in
Section \ref{ss:GFF}. For now let us only mention that it is a random
generalized function on $\C$, {where a generalized function on an open set
$U$ refers to an element of $\mathcal{D}'(U)$-- the continuous dual of
$\mathcal{D}(U)$, where $\mathcal{D}(U)$ is the space of compactly
supported smooth functions on $U$.} For every $\phi\in \mathcal{D}(\C)$, the action of $h$ on $\phi$ yields a Gaussian random variable and will be denoted by $(h,\phi)$. It will also be important for us that the average of $h$ with respect to to the uniform measure on $\mathbb{T}_r$ exists as a random variable for all $r>0$; it will be denoted by $\mathbf{Av}({h},\mathbb{T}_r)$ (the formal definitions will appear later).

\medskip

\textbf{Unless otherwise specifically mentioned, from now on $\gamma\in (0,2)$ will remain fixed and will be suppressed from notations.} \\

The LQG metric on $\C$ with the underlying whole plane GFF $h$ will  be denoted by $D_{h}$, i.e., $D_h(x,y)$ shall denote the LQG distance between points $x,y\in \C$. A discussion about its formal definition and properties is presented in
Section \ref{ss:LQG}. We denote by
$\Gamma(x,y)=\Gamma(x,y;h)$ the LQG
			geodesic from $x$ to $y$ for $x,y\in \C$. 
			 The existence and uniqueness of these geodesics along with other related issues are
			 discussed later in Section \ref{ss:geodesics}.
  
			 We will often work with metrics defined on open sets $U\subseteq
  \mathbb{C}$ which are continuous, by which we mean metrics which are continuous when thought of as
   functions from $U\times U$ to $\R$. This allows us to simply think of
   them as embedded in $C(U\times U)$ with the topology of uniform
   convergence on compact subsets. We will
   also require the notion of an induced metric which we now introduce.
   Given a metric $d$ on a domain $U\subseteq \mathbb{C}$ and a subdomain
   $V\subseteq U$, the induced metric $d(\cdot,\cdot;V)$ is defined for
   $z,w\in V$ by
   $d(z,w;V)= \inf_{\zeta: z\rightarrow w, \zeta\subseteq
	V}\ell(\zeta,d)$, where 
	$\zeta:z\rightarrow w$ denotes a path from $z$ to $w$, and $\ell(\zeta;d)$ denotes the length of the
	curve $\zeta$ with respect to the metric $d$. Note that a priori the latter could simply be infinite. A detailed discussion
	about path lengths and induced metrics appears in Section
	\ref{ss:LQG}.

We now introduce the central object underlying our investigation, namely the \emph{infinite geodesic from the origin.} Its existence and uniqueness is ensured by the following result.  
\begin{proposition}[{\cite[Proposition 4.4]{GPS20}}]
  \label{infgeod}
 Almost surely, for each $z \in \C$ there exists a (not necessarily unique) infinite
geodesic ray $\Gamma(z,\infty),$
started from $z$, called a $D_h$-geodesic from $z$ to $\infty$. Moreover, for each fixed $z\in \C,$ a.s., $\Gamma(z,\infty)$ is unique.
\end{proposition}
In fact, \cite[Proposition 4.4]{GPS20} provides a lot more information some
of which we recall later
as Proposition \ref{b5}. However, we will momentarily simply need the existence of the
unique infinite geodesic started from $0$ when the background field is $h$. We shall denote this by $\Gamma$.
 We parametrize
			$\Gamma$ with respect to the LQG metric induced by $h$ i.e. $D_{h}(0,\Gamma_t)=t$ and thus $D_h(\Gamma_t,\Gamma_s)=|t-s|$. In the
			sequel we shall refer to this parametrization as the
			standard or the LQG metric parametrization of the infinite geodesic. 
\medskip

We are now ready to state our main results concerning the empirical measure  of the environment rooted on the geodesic $\Gamma$ (averaging over the random location of the root), fixing an almost sure realization of $h$. In fact the results will hold for any infinite geodesic $\Gamma(z,\infty)$ (see Remark \ref{anygeod}).

\subsection{Main results}
\label{ss:main_results}
The reason one expects a positive answer to Question \ref{q1} in the classical setting of FPP/LPP is  that the underlying i.i.d.\ field suggests that the local neighborhood of a geodesic should also be approximately i.i.d.\ as one traverses along the geodesic. Due to the domain Markov property along with conformal invariance admitted by GFF, this is true for LQG  \emph{modulo rescaling}. So the first natural formulation of this question in our setting would be to parametrize the geodesic $\Gamma$ in such a way that an ``equal weight" is allotted to each scale while averaging along the geodesic where the local environments are also scaled down appropriately. This is what we set up next, the scaling mechanism followed by the log parametrization of the geodesic.

\noindent
\textbf{Scale dependent dilation:}  
To begin, let $Q=\gamma/2+2/\gamma$ and $\xi= \gamma/d_\gamma$, where the deterministic constant $d_\gamma$ is
the Hausdorff dimension (see \cite{GP19}) of $\mathbb{C}$ as a metric space
equipped with LQG
metric $D_h(\cdot, \cdot)$ coming from $h$. For a conformal map
$g:\mathbb{C}\rightarrow \mathbb{C}$, we define the random generalized
function {$h\mydot g$} by the action $(h\mydot
g,\phi)=(h,|g'|^{-2}(\phi\circ g^{-1}))$ on any
$\phi\in \mathcal{D}(\mathbb{C})$; a discussion about
this definition of composition is present in Section \ref{ss:GFF}.
By using the above definition of composition along with the coordinate change formula and the Weyl scaling (see
Proposition \ref{b3}),  for any deterministic $r>0$ and for all $z,w\in \mathbb{C},$ we have
\begin{equation}
	\label{e:b131}
	D_{h(r\cdot)}(rz,rw)=r^{\xi Q} D_{h(\cdot)}(z,w),
\end{equation}
where we use $h(r\cdot)$ as short hand for $h$ composed (in the sense of
$\mydot$) with the scaling map
$z\mapsto rz$ for $z\in \mathbb{C}$. The reason behind the above scaling relation can be found in the discussion in
\cite[(4.2)]{GPS20}.

Fix $\delta\in (0,1)$. For $x\in \C$, define a field $\underline{h}_x$ on $\mathbb{D}$ by suitably dilating
and normalizing/centering the field $h$ in a ball of size $\delta|x|$ around $x$. Namely, let 
\begin{equation}
	\label{e:field}
	\underline{h}_x=h\mydot
	\Psi_{x,\delta}-c_x(h),
\end{equation}
 where $\Psi_{x,\delta}$ is the map given  by
 $\Psi_{x,\delta}(z)=\delta|x|z+x$ for all complex numbers $z,$ and the generalized function $h\mydot
 \Psi_{x,\delta}$ on $\mathbb{D}$ is defined by the action $(h\mydot
 \Psi_{x,\delta},\phi)=(h,|\Psi_{x,\delta}'|^{-2}(\phi\circ\Psi_{x,\delta}^{-1}))$
 for any $\phi\in \mathcal{D}(\mathbb{D})$.
 The quantity $c_{x}(h)$ is chosen so that
 $\mathbf{Av}(\underline{h}_x,\mathbb{T})=0$, that is,
 $\mathbf{Av}(h,\mathbb{T}_{\delta|x|}(x))=c_x(h)$.

Having defined the field centered at any point on $\Gamma$, the averaging mechanism yielding the empirical measure is prescribed next. 

\noindent
\textbf{Log-parametrization and the empirical field:} Considering the map $s\mapsto
 \underline{h}_{\Gamma_{e^s}}$ which is defined on the entire interval
 $[0,\infty)$, for a given instance of the environment $h$, we define  
 $\emf_t$ to be  {$\underline{h}_{\Gamma_{e^\mathfrak{t}}}$}, where $\mathfrak{t}\sim
 \mathtt{Unif}(0,t)$. 

\bigskip

\noindent
\textbf{Convergence of the empirical field:}
We will often work with the H\"older spaces
$H_0^s(U)=W_0^{s,2}(U)$ for bounded domains $U\subseteq \mathbb{C}$ {and both
positive and negative values of $s$}. For each value of $s$, there is a
natural inner product $(\cdot,\cdot)_s$ on $H_0^{s}(U)$ for which it is a
Hilbert space (see \eqref{e:innprod} for more on this) and this equips 
$H_0^{s}(U)$ with a natural topology. 
{We recall that for
any $\varepsilon>0$, an element of $H_0^{-\varepsilon}(U)$ which a priori acts on
$H_0^{\varepsilon}(U)$ is determined by its action on the smaller space
$\mathcal{D}(U)$ owing to the density of the latter in the former. Thus we may identify an element of $H_0^{-\varepsilon}(U)$
by the generalized function obtained by restricting its action to
$\mathcal{D}(U)$; this identification will remain in play throughout the
paper.} With this in mind, we can write that for any $ \e_1>\e_2>0$,
\begin{equation}
  \label{e:topind}
H_0^{-\e_2}(U)\subset H_0^{-\e_1}(U)\subset
\mathcal{D}'(U),
\end{equation}
where the topology induced by
$H_0^{-\e_1}(U)$ on $H_0^{-\e_2}(U)$ is coarser than the intrinsic topology on
the latter \cite[Proposition 2.7 (1)]{Sco07}. Standard topological arguments
can be used to check the measurability of $H_0^{-\varepsilon}(U)$ as a subset of
$\mathcal{D}'(U)$ for any $\varepsilon>0$.
 For other details regarding
 Sobolev spaces relevant to the GFF, we point to the references \cite{Sco07}
 and \cite[Section 1.5]{Ber07}.

 Since $h\in \mathcal{D}'(\mathbb{C})$, we have that $\emf_{t}$ is a random
 element of $\mathcal{D}'(\mathbb{D})$. In fact, as we shall briefly
 discuss after Proposition \ref{imp4}, a.s.\ for all bounded domains $U$
 simultaneously, {$h\lvert_{U}\in \mathcal{D}'(U)$} is in the H\"older space $H_0^{-\e}(U)$
 for any fixed $\e>0$. {Here, the restriction $h\lvert_{U}$ is defined
 by $(h\lvert_{U},\phi)=(h,\phi)$ for any $\phi\in \mathcal{D}(U)$.}
 Using the above, for any $\e>0$, one can interpret
 $\emf_{t}$ as a random element of $H_0^{-\e}(\mathbb{D})$ (while $\emf_t$ implicitly depends on $\delta$, we suppress the dependence in favor of
cleaner notation).

Our first main result states  that the law of $\emf_t$ converges almost surely
 to a deterministic measure as $t\rightarrow\infty$.
 
 \begin{maintheorem}
	 \label{main1} 
	 Fix $\delta\in (0,1)$. There exists a deterministic measure
	 $\nu_{\mathbf{Field}}$ which can naturally be defined on
	 $H_0^{-\e}(\mathbb{D})\subseteq \mathcal{D}'(\mathbb{D})$ for all $\e>0,$ such that almost surely, 
	 \begin{displaymath}
	   \emf_{t}\stackrel{d}{\rightarrow}
		 \emf
	 \end{displaymath}
	 as $t\rightarrow \infty$, where $\emf \sim \nu_{\mathbf{Field}},$ and
	 the convergence is in distribution in $H_0^{-\e}(\mathbb{D})$ (and
	 thus in $\mathcal{D}'(U)$) for all $\e>0$.
 \end{maintheorem} 

\noindent
\textbf{Convergence of the empirical metric:}
Theorem \ref{main1} proves convergence of the environment around the geodesic, but we are also interested in the
 behavior of the key LQG metric in the
 local vicinity of the geodesic. In the same set up as above, we define the metric
$\underline{D}_{h,x}$ on the disk $\mathbb{D}$ by
 \begin{equation}
	 \label{e:metric}
	 \underline{D}_{{h},x}(u,v)=|\delta x|^{-\xi Q}
	 e^{-\xi
	 \mathbf{Av}(h,\mathbb{T}_{\delta|x|}(x))}D_{h}(x+\delta
	 |x|u,x+\delta |x|v;\mathbb{D}_{\delta|x|}(x)).
 \end{equation}
The terms appearing in the RHS arise from the change of variables formula and
the Weyl scaling for the LQG metric which is spelled out in Proposition \ref{b3}. To summarize in words, we center the field $h\lvert _{\mathbb{D}_{\delta|x|}(x)}$  by $\mathbf{Av}({h},\mathbb{T}_{\delta|x|}(x))$, which leads to the $e^{-\xi
	 \mathbf{Av}({h},\mathbb{T}_{\delta|x|}(x))}$ pre-factor. We now dilate $\mathbb{D}_{\delta|x|}(x)$ to $\D$ which manifests in the $|\delta x|^{-\xi Q}$ term.

 Similar to the definition of
 $\mathbf{Field}_t$, we can define $\emm_t$ as follows.
 For a given instance of the
 environment $h$, $\emm_t$ is defined to be the
 random variable
 $\underline{D}_{h,\Gamma_{e^\mathfrak{t}}}(\cdot,\cdot)$
 where $\mathfrak{t}\sim
 \mathtt{Unif}(0,t)$. Thus, almost surely, given an instance of $h$, $\emm_t$ is a random
 continuous metric on $\mathbb{D}$.

We now state the analogue of Theorem \ref{main1} for $\emm_t$ which is our second main result. 
\begin{maintheorem}
	 \label{main2}
	 There exists a deterministic measure $\nu_{\mathbf{Metric}}$ on the space
	 of continuous metrics on $\mathbb{D}$ such that if $\emm\sim
	 \nu_{\mathbf{Metric}}$, then almost surely
	 \begin{displaymath}
		 \emm_t\stackrel{d}{\rightarrow}
		 \emm	 \end{displaymath}
	 as $t\rightarrow \infty$, where the convergence is with respect to the topology of uniform
	 convergence on compact subsets of $\mathbb{D}\times \mathbb{D}$.
 \end{maintheorem}

In fact, there is a rather explicit description of a coupling of $\emf$ and $\emm.$ This is recorded later in Section \ref{ss:construct} as Theorem \ref{ac:1}.

Note that while the above results establish convergence of the local field and metric along the geodesic, another natural observable for consideration is the local LQG measure.  Though this will not be the focus of the article, 
before proceeding to our third and final main result, we briefly digress to discuss its convergence along the geodesic as well. 

 \begin{remark}[Convergence of Empirical Measure]\label{measure0}
The LQG measure $\mu_{h}$ heuristically is a random measure corresponding to the volume form $e^{\gamma
h(z)}|dz|^2$ for {$\gamma\in (0,2)$}. This can rigorously be defined by a regularization procedure (see Section \ref{ss:LQG}). Analogous to the definitions of $\emf_t$ and $\emm_t$ one can also scale and average the LQG measure along $\Gamma$ and obtain a sequence of random measures $\emv_{t}$ on $\mathbb{D}$. Convergence of $\emv_{t}$ can also be obtained using arguments similar to the ones present in this paper, but we choose not to carry out the details in this manuscript. See Remark \ref{measure1} for a somewhat more elaborate discussion. 
 \end{remark}

\noindent
\textbf{Properties of the limiting objects:}
Having obtained  $\emf$ and  $\emm,$ the natural next step is to compare them
with the whole plane GFF $h$ and the LQG metric $D_{h}$. Recall that the main motivation of the article is to quantify the disparity between the environment in the vicinity of the geodesic and the typical one owing to the fact that the
geodesic is expected to pass through points with atypically low values of $h$.

This should be reflected in the behavior of $\mathbf{Field}$ close to the origin and the following
results form a first step in this direction.

\begin{maintheorem}
	\label{mainsi1}
	$\empty$
\begin{enumerate}
\item[(i)] 
The random fields ${\emf}$ and
	$h\lvert_{\mathbb{D}}$ have mutually singular laws.
\item[(ii)] 
	\label{ac:2}
	For any $\delta'\in (0,1)$, 
	$\emf\lvert_{\mathbb{C}_{(\delta',1)}}$ is absolutely continuous
	with respect to 
	 $h\lvert_{\mathbb{C}_{(\delta',1)}}$.
	\end{enumerate}
\end{maintheorem}
Note that both $\mathbf{Field}$ and $h\lvert_{\mathbb{D}}$ in the
above are interpreted as random elements of
$\mathcal{D}'(\mathbb{D})$. Postponing formal definitions to Section \ref{ss:GFF} (refer to the
discussion just before Section \ref{ss:markov}), momentarily we simply remark
that the restriction $\emf\lvert_{\mathbb{C}_{(\delta',1)}}$ used above is
defined in the same way as $h\lvert_{\mathbb{C}_{(\delta',1)}}.$ Thus both
$\mathbf{Field}\lvert_{\mathbb{C}_{(\delta',1)}}$ and
$h\lvert_{\mathbb{C}_{(\delta',1)}}$ are random elements of
$\mathcal{D}'(\mathbb{C}_{(\delta',1)})$.
Analogous results are also true for $\mathbf{Metric}$ and are presented later as  Proposition
\ref{singprop} and Proposition \ref{ac:2.1*}. 

\begin{remark}\label{anygeod} In fact, any infinite geodesic $\Gamma(z, \infty)$ emanating from $z\in \C,$ described in Proposition \ref{infgeod}, agrees with $\Gamma$ eventually. Thus, the results of this paper continue to hold for any such infinite geodesic as well.  
\end{remark}

The problem of quantifying the source of the singularity in Theorem
\ref{mainsi1}, analogous to the result in \cite{DS11} stating that a typical
point sampled from the LQG measure is a $\gamma-$thick point, will
be taken up in future research. In \cite{Gwy20,GPS20}, related themes of the Hausdorff dimension of the infinite geodesic $\Gamma$, as well as the boundary of metric balls in LQG were explored.

\subsection*{Acknowledgements}
We thank Vasanth Pidaparthy for useful discussions and Ewain Gwynne for helpful comments. RB and MB were partially supported by a Ramanujan Fellowship (SB/S2/RJN-097/2017) from SERB and a grant from Simons Foundation (677895, R.G.). RB was also partially supported by by ICTS via project no. RTI4001 from DAE, Govt. of India, and Infosys Foundation via the Infosys-Chandrasekharan Virtual Centre for Random Geometry of TIFR.  MB acknowledges the support from the Long Term Visiting Students Program (LTVSP) at ICTS. SG was supported partially by NSF grant DMS-1855688, NSF Career grant DMS-1945172, and a Sloan Fellowship.

\section{Idea of proofs}\label{s:iop}
In this section we describe the key ideas behind our proofs of Theorems \ref{main1}, \ref{main2},  and \ref{mainsi1} and also describe some related results on the empirical distribution along the geodesic.

\subsection{Decomposition of the Geodesic}
The basic idea behind the proof of Theorems \ref{main1} and \ref{main2} is to decompose the geodesic $\Gamma$ into disjoint segments that pass
through environments that have a fast enough decay of correlation, which then
implies the geodesic has an embedded regeneration structure, albeit ``modulo
rescaling".  We describe the decomposition now relying on geodesic coalescence\footnote{In fact, this phenomenon is more widely known as \emph{confluence} in the LQG literature, introduced first by Le Gall in his seminal work \cite{gall2010}, while the term \emph{coalescence} is more prevalent in the literature on FPP on lattices, see e.g. \cite{ADH15, LN96}.}. A similar but non-quantitative version appeared in the proof of Proposition \ref{infgeod} about the
existence and uniqueness of $\Gamma$ appearing in \cite[Proposition 4.4]{GPS20}. 

\begin{figure}[h]
\centering
\includegraphics[scale=0.24]{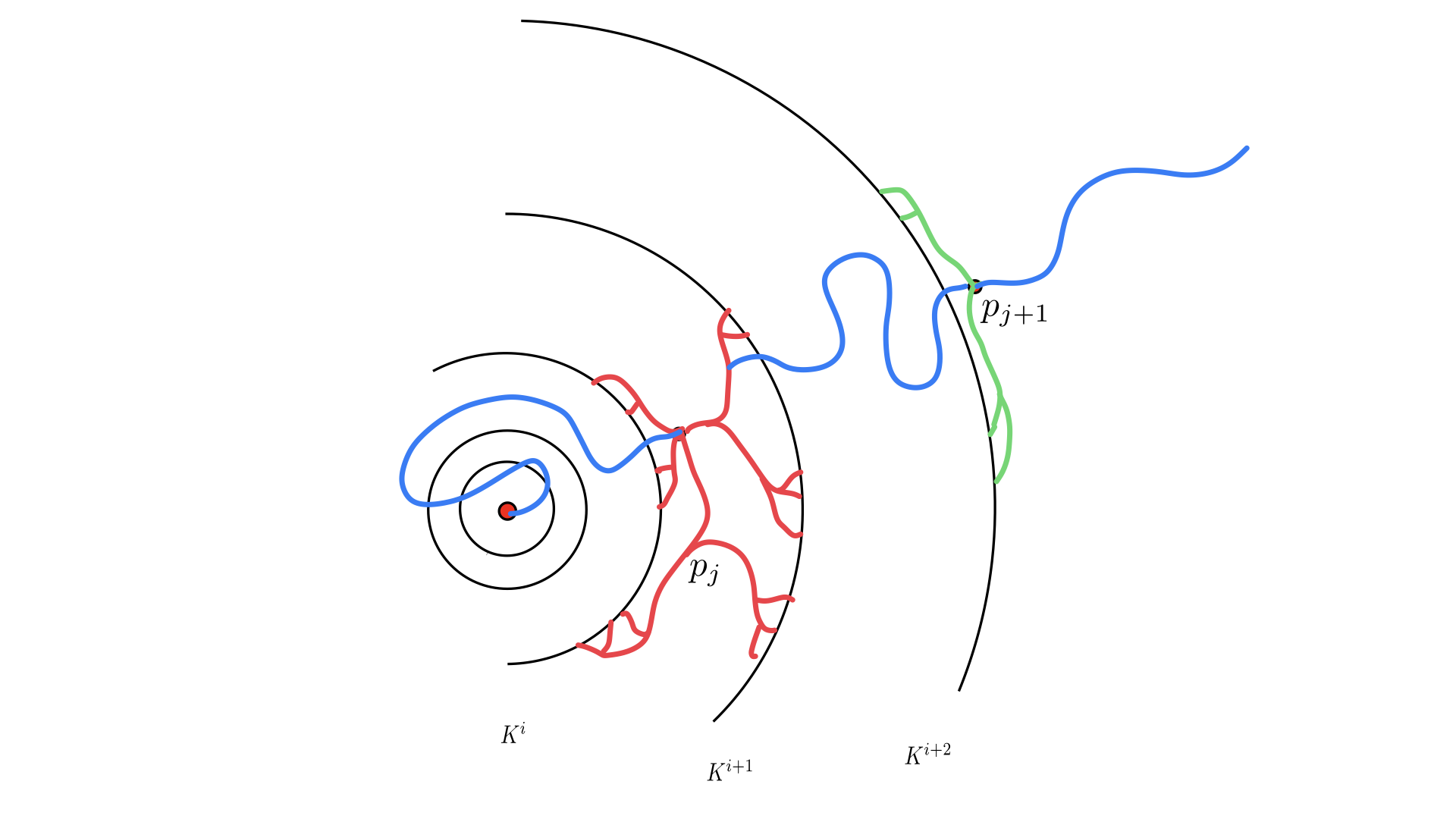}
\caption{The above figure illustrates the decomposition of $\Gamma$ into
identically distributed (modulo scaling and centering) segments with fast
decay of correlation. For each scale $i,$ there exists a constant probability
of witnessing the event ${\sf{Coal}}_i.$ In the above illustration,
${\sf{Coal}}_i$ and ${\sf{Coal}}_{i+2}$ occur but ${\sf{Coal}}_{i+1}$ does
not. $p_j$ (denoting the $j^{th}$ such coalescence point past radius $1$) is the measurable choice of a point through which geodesics
between any two points $p \in \T_{K^i}$ and $q\in \T_{K^{i+1}}$ pass through.
$\Gamma(p_j,p_{j+1})$s form segments of the desired decomposition.  Due to the domain Markov property of the GFF, the events ${\sf{Coal}}_i$ exhibit exponential decay of correlation.}
\label{f.proofsketch}
\end{figure}

Recall that we use
$\Gamma(u,v)$ to denote a, possibly non-unique, geodesic between $u$ and $v$. Consider an exponentially growing sequence $K^{i}$ and the associated annuli
$\mathbb{C}_{[K^i, K^{i+1}]}$ for some fixed large constant $K$
and for any $i\in \Z$. 
 Define the event ${\sf{Coal}}_{[r_1,r_2]}$  (where ${\sf{Coal}}$ stands for coalescence) by
\begin{align*}
  {\sf{Coal}}_{[r_1,r_2]}&= \left\{\bigcap_{u,v}  \Gamma(u,v)\neq \emptyset\right\}, \text{where the intersection is over all } 
u\in\mathbb{T}_{r_1},v\in \mathbb{T}_{r_2}\\& 
\,\,\,\,\,\,\,\text {and the possibly
	multiple geodesics between each pair } u,v.
	\end{align*} 
Letting ${\sf{Coal}}_i:={\sf{Coal}}_{[K^{i},K^{i+1}]},$ it was shown in
\cite{GPS20} that ${\sf{Coal}}_i$ occurs with probability bounded away from 0 for all $i$
as long as $K$ is large enough.
Note that the conclusion for all $i$ follows by the scale
invariance of GFF (see e.g.\ \eqref{e:Kstat}) as
long as it holds for any $i.$ Further, the probability can be made arbitrarily close to
one by choosing $K$, the growth rate of the annuli size, sufficiently large.
Given the above, it follows by a straightforward argument that with
probability one, there exists a bi-infinite sequence of radii $i_j$ such that $i_j\to -\infty$ as $j\to -\infty$ and $i_j\to \infty$ as $j\to \infty$
with the property that ${\sf{Coal}}_{i_j}$ holds for all $j$. Let
$p_j$ be a point witnessing the latter event, i.e., each geodesic from points
on $\mathbb{T}_{K^{i_j}}$ to points in $\mathbb{T}_{K^{i_{j+1}}}$ passes through
$p_j$. One can in fact choose it in a measurable way and this will be needed for our applications, but we ignore this subtlety for current expository purposes.

Once the $p_j$s are chosen, $\Gamma$ can be obtained by the concatenation of
the geodesics $\Gamma(p_{j}, p_{j+1})$. That almost surely $\Gamma(p_{j}, p_{j+1})$ is unique for all $i$ is an easy consequence of the fact that almost surely $\Gamma(0,r)$ is unique for all rational points $r \in \Q\times \Q$. Indeed, if there exists more than one geodesic between $p_j$ and $p_{j+1}$ for some $i$, the coalescence event would ensure the non-uniqueness of $\Gamma(0,z)$ for some rational point $z$ on $\mathbb{T}_r$ for a sufficiently large $r$.

\subsection{Proof sketch for Theorems \ref{main1} and \ref{main2}}
We shall prove Theorems \ref{main1} and \ref{main2} using the above decomposition by first arguing that $\{\Gamma (p_{j}, p_{j+1})\}_{i\in \Z}$ possesses a renewal structure. 

We shall provide a detailed outline for the proof of Theorem \ref{main1}; the proof of Theorem \ref{main2} is similar.  Begin by observing that the sequence  $\{h\lvert_{\mathbb{C}_{[K^i,
K^{i+1}]}}\}_{i\in\Z}$ is  stationary, modulo re-centering and re-scaling,
owing to the conformal invariance of $h$. That is, if we use $\psi_r$ to
denote the dilation map from $\mathbb{C}_{>1}$ to $\mathbb{C}_{>r}$, then for
all $i\in \mathbb{Z}$,
\begin{equation}
  \label{e:Kstat}
  h\lvert_{\mathbb{C}_{[K^i,
  K^{i+1}]}}\mydot
  \psi_{K^i}-\mathbf{Av}(h,\mathbb{T}_{K^i})\stackrel{d}{=}h\lvert_{\mathbb{C}_{[1,
  K]}}
\end{equation}
Further, and crucially, the sequence exhibits an exponential decay of
correlation, i.e., conditioning on $\mathscr{G}_i$ -- the filtration generated by
$h\lvert_{\D_{K^{i}}},$ the Radon-Nikodym derivative of the conditional law of
$(h-\mathbf{Av}(h,\mathbb{T}_{K^j}))\lvert_{\mathbb{C}_{[K^j, K^{j+1}]}}$ with respect to its unconditional law
can be shown to converge to $1$ in an $L^2$ sense as $j\to \infty$ at rate
$e^{-c(j-i)}$ for some $c>0$ (see Lemma \ref{rncond}). This is essentially due to the domain Markov
property admitted by the GFF, which states that the dependence of the law of
$(h-\mathbf{Av}(h,\mathbb{T}_{K^j}))\lvert_{\mathbb{C}_{[K^j, K^{j+1}]}}$ on
$\mathscr{G}_i$ is through $\mathfrak{h}-\mathbf{Av}(h,\mathbb{T}_{K^i})$, where
$\mathfrak{h}$ is the harmonic extension of the restriction of $h$ on
$\T_{K^i}$. The fact that $\mathfrak{h}-\mathbf{Av}(h,\mathbb{T}_{K^i})$ converges to $0$ as one moves away from $\D_{K^{i}}$ to $\infty$ can then be used to prove the desired correlation decay.  

The correlation decay described above indicates that different segments of the geodesic act
approximately independently, thereby leading to the possibility of invoking
law of large numbers type results to obtain Theorem \ref{main1}. However, at
this point, observe that the log-parametrization of the geodesic induces a \emph{long range} dependence in the following way.  
Recall from the set-up of Theorem \ref{main1} that we considered the
map $s\mapsto
 \underline{h}_{\Gamma_{e^s}}$, and  defined  $\emf_t$ to be the law of the random
 variable $\underline{h}_{e^\mathfrak{t}}$, where $\mathfrak{t}\sim
 \mathtt{Unif}(0,t)$. In the standard parametrization of $\Gamma$ according to the LQG distance, let $X_{i}$ denote the
amount of LQG distance gained by the geodesic segment $\Gamma (p_{j},
p_{j+1})$ if there exists $j$ with $p_j\in \C_{[K^{i}, K^{i+1}]}$, and $0$ otherwise.  The scale invariance of $\Gamma$ (established recently in
\cite{GPS20})
and $h$, and the known decomposition of the latter  into a radial part
(Brownian motion) and a lateral part (see e.g., above \eqref{e:b4.1*}), together imply that $X_{i}=e^{\xi B(\log K^i)} K^{i \xi
Q} Y_i$, where $B(\cdot)$ is a standard Brownian motion and the term $K^{i \xi
Q}$ arises due to Weyl scaling (see Lemma \ref{b4}). Moreover, the $Y_i$s are
$O(1)$ random variables that form a stationary sequence with exponentially
decaying correlations. 

In the remaining discussion, to convey the key idea we assume that $X_i$ is non-zero. In this case, the chemical distance along the geodesic increases from 
$L_i=\sum_{\ell=-\infty}^{i-1}X_\ell$ to $L_{i+1}=\sum_{\ell=-\infty}^{i}X_\ell$ from the segment $\Gamma(0,p_j)$ to $\Gamma(0,p_{j+1})$. 
Hence, the increment of the corresponding log-length is 
\begin{equation}\label{log-para1}
G_i=\log (L_{i+1})-\log(L_i)=\log\left(1+\frac{L_{i+1}-L_i}{L_i}\right).
\end{equation}
This quantifies the long range dependence induced by the log-parametrization.
Nonetheless, the explicit form of the random variable $X_i$ allows us to obtain the following representation of $G_i$ 
\begin{displaymath}
	G_i=\log \left(
	1+\frac{Y_i}{\sum_{\ell=-\infty}^{i-1} Y_\ell e^{\xi Q(\log
	K^\ell-\log K^i)+\xi(B(\log K^\ell)-B(\log K^i))}} \right)
	\end{displaymath} 
This along with
Brownian fluctuation estimates allow us to show that $\{G_{i}\}_{i\in
\mathbb{Z}}$ itself is a stationary sequence with
exponential decay of correlations.	
	
One can then use the theory of stationary processes with fast decay of correlation and appeal to abstract results in functional analysis to deduce Theorem \ref{main1}. Specifically, we first establish the convergence of ``finite
dimensional'' distributions by showing, almost surely in the randomness of $h$, the convergence of
 the joint law of the random vector $$\left(
 (\emf_t,\phi_1),\dots,(\emf_t,\phi_\ell) \right)$$ as
 $t\rightarrow \infty$ for any $\phi_1,\dots,\phi_{\ell}\in
 \mathcal{D}(\mathbb{D})$ to a deterministic measure on $\R^{\ell}$, and then show the appropriate tightness. For the
 former, we will show, almost surely in the realization of $h,$ the convergence of
 the joint characteristic function
 \begin{displaymath}
	 \mathbb{E}_t\left[ \exp\big( \sum_{j=1}^{\ell} \mathbf{i}\lambda_j
	 (\emf_t,\phi_j) \big)\right]
 \end{displaymath}
 along with establishing the continuity of the above expression at
 $(\lambda_1,\dots,\lambda_{\ell})=(0,\dots,0)$, 
 where given a realization of $h,$ $\P_t$ and $\E_t$ are used to denote the law and associated expectation $\emf_t$ (where the randomness is induced by $\mathfrak{t}$) respectively.
	 
Having established the convergence of finite dimensional distributions, the
required tightness is obtained by controlling the tail for
$\|\emf_t\|_{H_0^{-\varepsilon}(\mathbb{D})}$ for any $\varepsilon>0$ using similar bounds for $h$ as in
Proposition \ref{imp3*}.
Standard functional analysis arguments involving the topology of
$H_0^{-\varepsilon}(\mathbb{D})$ now complete the proof. 

Theorems \ref{main2}
follows using similar arguments.

\subsection{Proof sketch of Theorem \ref{mainsi1}}
We next move on to a sketch of the proof of Theorem \ref{mainsi1}. Recall that it
makes two claims: (i) the mutual singularity of $\emf$ and $h\lvert_{\mathbb{D}}$, and (ii) the absolute continuity away from the origin, i.e., absolute continuity of $\emf\lvert_{\C_{(\delta',1)}}$ w.r.t.\ $h\lvert_{\C_{(\delta',1)}}.$

\begin{figure}[h]
\centering
\includegraphics[scale=0.2]{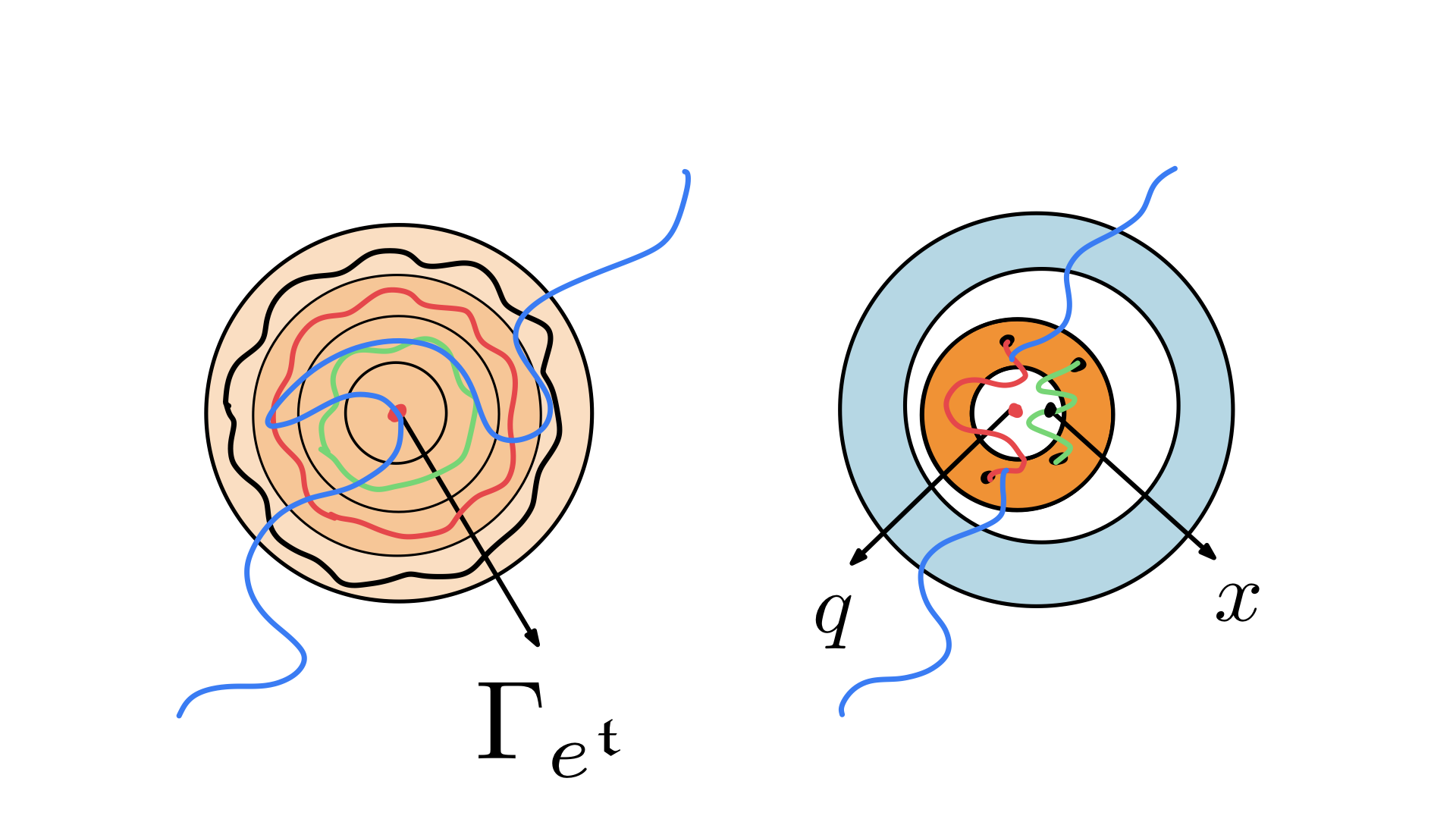}
\caption{\textbf{Left: Sketch of the proof of singularity}. Note that for any point
$x=\Gamma_{e^{\mathfrak{t}}},$ in a local neighborhood as depicted, no annulus around $x$
admits a shortcut, i.e.\ a path around with low weight which the geodesic is
incentivized to follow, since we know that, necessarily, the geodesic $\Gamma$ (marked in blue) passes
through $x.$ However, for a typical environment, centered at the origin,
almost surely there exist such shortcuts, owing to the independence across
scales and the fact that for each scale, such a path exists with constant
chance. \textbf{Right:  Sketch of the proof of absolute continuity.} We consider balls
$\mathbb{D}_{\varepsilon_1}(q)$ and $\mathbb{D}_{\varepsilon_2}(q)$ with $\e_1< \e_2$
rationals (the small white and the enclosing orange regions centered at $q$) and $q\in \Q^2.$ The figure depicts the geodesics in the induced
metric on $\mathbb{D}_{\varepsilon_2}(q)$ between rational points in the latter and their intersections with $\mathbb{D}_{\varepsilon_1}(q)$. This gives a
countable collection of segments $\zeta_1, \zeta_2,\ldots$,  measurable with respect to the noise in
$\mathbb{D}_{\varepsilon_2}(q).$ For an arbitrary point $x$ on one of these segments, the field
on the annulus
of radius $\delta'$, the blue region, is distributed as GFF conditional
on
$h\lvert_{\mathbb{D}_{\varepsilon_2}(q)}$ which by the domain Markov property
is absolutely continuous to its unconditional distribution. On the event that the semi-infinite geodesic $\Gamma$ (marked in blue) enters $\mathbb{D}_{\varepsilon_1}(q)$, its intersection with the latter is contained in the union of the $\zeta_i$\,s.} 
\label{f.proofsketch2}
\end{figure}

\subsubsection{Singularity of the fields}
We start by recalling for the whole plane GFF $h$ and the LQG metric $D_h(\cdot,\cdot),$  the induced metric $D_h(\cdot,\cdot; \D)$ and the
restricted field $h\lvert_\D$. It turns out that the former is a  function
of the latter in a sense made precise in Section \ref{ss:LQG}, and hence can be denoted as $D_{h\lvert_{\mathbb{D}}}(\cdot,\cdot).$

We will in fact first prove the following
proposition concerning $\emm$ instead. 
\begin{proposition}\label{singprop}
  The random metrics $\emm$ and $D_{h\lvert_{\mathbb{D}}}(\cdot,\cdot)$ have
  mutually singular laws. 
\end{proposition}

Towards the proof of Proposition \ref{singprop}, since $\emm_{t}$
is the same as $\underline{D}_{h,\Gamma_{e^\mathfrak{t}}}$
 where $\mathfrak{t}\sim
 \mathtt{Unif}(0,t)$ we get that for the latter, there is a geodesic between
 two points on $\mathbb{T}$ passing \emph{through} the origin (see the left
 panel in Figure \ref{f.proofsketch2}).  
However, the latter event does not happen for $h$ almost
surely, a fact that has already been used multiple times in the literature, in
particular to prove that the LQG metric $D_{h}(\cdot,\cdot)$ does not admit any
bigeodesics \cite[Lemma 4.5]{GPS20}. This can be used to assert singularity. 

To deduce Theorem \ref{mainsi1} (1) from Proposition \ref{singprop}, we use a result which is of independent interest. 
As mentioned above, $D_h(\cdot,\cdot; \mathbb{D})$ is a function
of $h\lvert_{\mathbb{D}}$, say denoted by $\Psi_{\mathbb{D}} 
:\mathcal{D}'(\mathbb{D})\rightarrow C(\mathbb{D}\times \mathbb{D})$ such that almost surely in the realization
of $h$, we have $\Psi_{\mathbb{D}}(h\lvert_{\mathbb{D}})=D_h(\cdot,\cdot; \mathbb{D})$. 
Further, the arguments in \cite{GM21,DDDH20} proving the above in fact define $\Psi_{\mathbb{D}}$
only on a measure one set of fields in $\mathcal{D}'(\mathbb{D})$ while
extending it arbitrarily in a measurable way to the complement. 
However such a statement, though known to be true for any \emph{fixed} domain $U$, is not known for all $U$ simultaneously. 
Consequently, it is not a priori clear if one can expect such a statement allowing us to construct $\emm$ from $\emf$ since they are obtained by sampling the local environment in random domains rooted along the geodesic. 
Nonetheless the next result shows that 
$\Psi_\mathbb{D}$ can in fact be chosen such that the latter almost surely acts on any realization of $\emf$ to yield $\emm$.

\begin{proposition}\label{mainmeas1} The measurable function $\Psi_\mathbb{D}:
  \mathcal{D}'(\D)\to C(\D\times \D)
  $, a version of the $h\lvert_{\mathbb{D}}\to D_h(\cdot,\cdot; \mathbb{D})$
  correspondence, can be chosen such that there exists a coupling of $\emf$
  and $\emm,$ under which $\Psi_\mathbb{D}(\emf)=\emm$ almost surely.
\end{proposition}

Clearly, the first part of Theorem \ref{mainsi1} follows from Proposition
\ref{mainmeas1} along with Proposition \ref{singprop}.

\subsubsection{Absolute continuity}
\label{ss:acidea}
The proof of the fact that for any $\delta'\in (0,1)$, the restriction $\emf
\lvert_{\mathbb{C}_{(\delta',1)}}$ is absolutely continuous with respect to
$h\lvert_{\mathbb{C}_{(\delta',1)}}$ is more complicated and is built on the
following idea. {As in Figure \ref{f.proofsketch2}}, consider two
concentric balls centered at a rational point $q\in \C \setminus \left\{
0 \right\}$ with rational radii $\e_1<\e_2$ with $2\e_2<\frac{\delta'}{2}$. 
	Now the argument  hinges on predicting the intersection of the
	geodesic $\Gamma$ with $\mathbb{D}_{\e_1}(q)$ by simply using $h\lvert_{\mathbb{D}_{\e_2}(q)}.$ 
	We will prove that there exists a countable collection of paths
	$\zeta_{1}, \zeta_2,\ldots,$ all subsets of ${\mathbb{D}_{\e_2}(q)}$ and
	measurable with respect to $h\lvert_{\mathbb{D}_{\e_2}(q)}$, such that $\Gamma \cap
	\mathbb{D}_{\e_1}(q)$ is a subset of  $\bigcup_i \zeta_i.$ At this point, since
	$\e_2$ was chosen to be small enough, given $h\lvert_{\mathbb{D}_{\e_2}(q)},$ and
	fixing $i,$ if $x\in \zeta_i,$ then
	$\underline{h}_{{x}}\lvert_{\mathbb{C}_{(\delta',1)}} $  is a
	projection
	of $h\lvert_{\C_{>\e_2}(q)}.$ Now by the domain Markov
	property, the law of the latter, even conditional on
	$h\lvert_{\mathbb{D}_{\e_2}(q)},$ is absolutely continuous to the
	unconditional distribution. The last fact is why it was crucial to
	obtain the $\zeta_i$s as measurable with respect to only
	$h\lvert_{\mathbb{D}_{\varepsilon_2}(q)}$.

There is one subtlety that we have overlooked so far.  Namely, that the
empirical distribution is guided by the log-parametrization of $\Gamma$, which
is \emph{not} locally determined and has long range dependencies as discussed
in the proof outline for Theorem \ref{main1}. However, this does not pose a real problem and we will not elaborate on this further.

We end with a short description of how to obtain the paths
$\zeta_1,\zeta_2,\ldots$\,. 	
We use the following facts. 
\begin{itemize}	
\item	Almost surely, for any $0<s<t$, given any balls $B_1$ and $B_2$ around $\Gamma_s$ and $\Gamma_t$ respectively, 	there exists further balls $B_1'\subset B_1$ and $B_2' \subset B_2$ such that for any points $w_1\in B'_1$ and $w_2 \in B'_2,$ the geodesic (geodesics) between $w_1$ and $w_2$ agree with $\Gamma$ outside $B_1\cup B_2.$ 
	 \item For points $z_1, z_2 \in \mathbb{D}_{\e_2}(q)$, if the geodesic
	   $\Gamma(z_1,z_2)$ does not exit $\mathbb{D}_{\e_2}(q)$ then
	   {$\Gamma_{\mathbb{D}_{\e_2}(q)}(z_1,z_2)$} exists and agrees with the former,
	   where the latter is the  `local' geodesic in the metric defined by
	   the local field $h\lvert_{\mathbb{D}_{\e_2}(q)}$ (see Lemma \ref{intrin}).  
	 
\end{itemize}	 
	 Using the above, we simply take $\zeta_i$ to be the list of local
	 geodesics (if they exist) between $w_1$ and $w_2$, where the latter range over all rational points in $\mathbb{D}_{\e_2}(q)\setminus \mathbb{D}_{\e_1}(q).$

As a final point, we remark that in the application of the usual Portmanteau
theorem, statements of the kind
$\P_t(\emf_t\lvert_{\mathbb{C}_{(\delta',1)}}\in A)=0$ for a measurable $A$ do not
immediately imply
$\nu_{\emf}(A)=0$ since
$A$ can have non-trivial boundary measure.  However to avoid this issue, we
use a representation of $\emf$ as the expected empirical distribution along a
finite geodesic in $\C_{> 1}$ induced by a \emph{size-biased} version (as one
might expect from classical renewal theory) of $h\lvert_{\C_{>
1}}$ (see Figure \ref{f.proofsketch3} and its caption for a quick description. The formal statement and proof appears in Section \ref{ss:construct}; see Theorem \ref{ac:1}).
This allows us to bypass applications of abstract results such as the Portmanteau's theorem. 
Moreover, this also helps in the proof of singularity, since generally, usage of
geodesics in statements involving weak convergence is delicate as they can
change drastically on slight perturbations of the underlying noise. 

\begin{figure}[htbp!]
\centering
\includegraphics[scale=0.15]{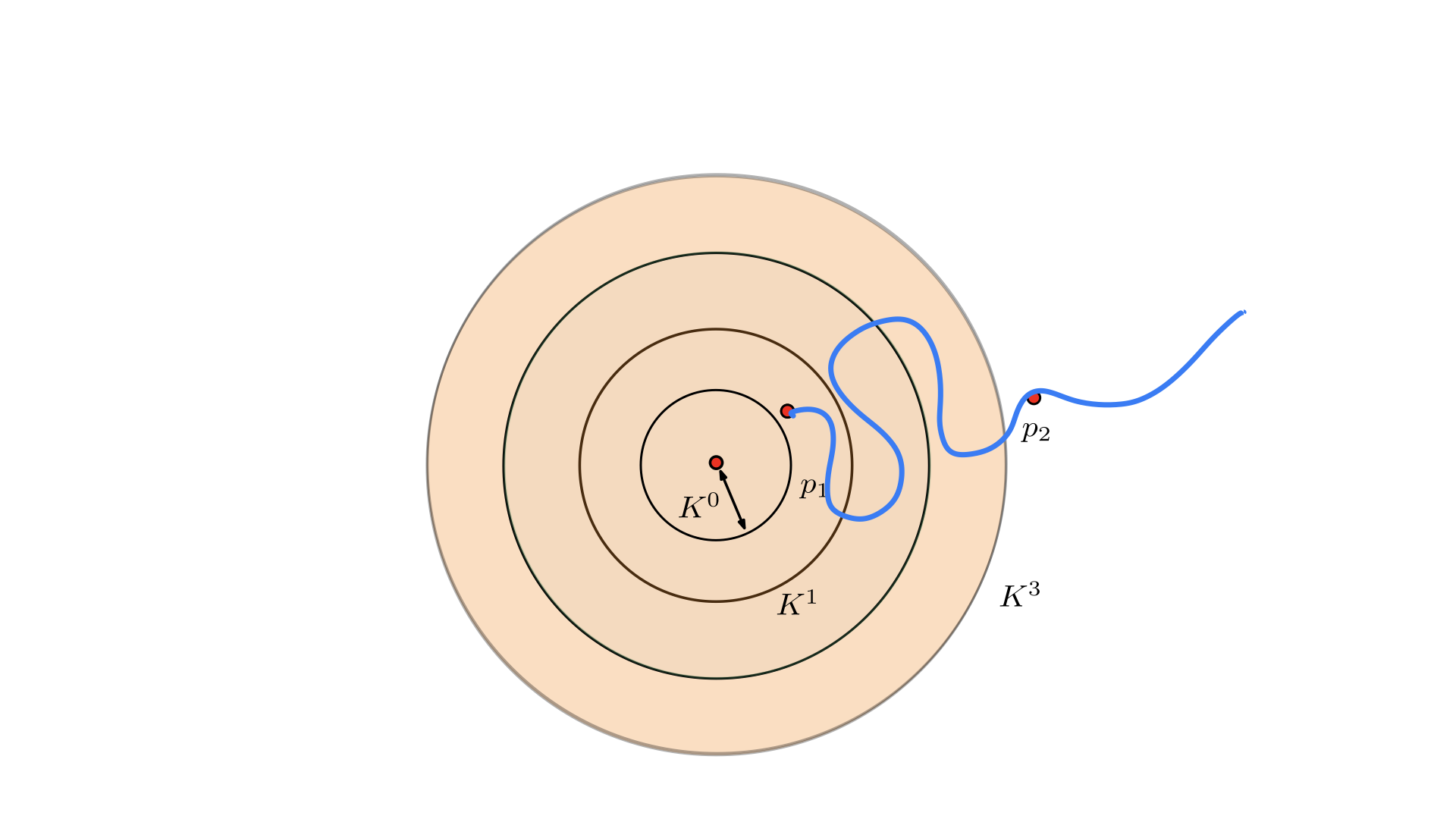}
\caption{Using renewal theory, the limiting field and similarly the metric has the following representation involving a size biasing. Let $p_1$ be `the' coalescence point in $\C_{[1, K]}$ in case it exists and let $p_2$ be the next such coalescence point. In the figure this exists in $\C_{[K^3,K^4]}$. Let $G_0=\log (D_h(0,p_2))-\log (D_h(0,p_1))$ in case $p_1$ exists and $0$ otherwise. Then the limiting field $\emf$ is the expectation of the random measure obtained in the following way: sample $h$ according to its law biased by $\frac{G_0}{\E G_0}$ and then looking at the empirical law of $\underline{h}_{x}$ where $x$ is a point on the geodesic whose log-chemical distance from the origin is a uniformly chosen number from the interval $[\log (D_h(0,p_1)),\log (D_h(0,p_2))].$}
\label{f.proofsketch3}
\end{figure}

\subsection{Related work on empirical distributions}
\label{s:related}
To contrast the ideas and the methods in this paper with the ones appearing in recent related works, we next present a brief account of some of these important advances made in
understanding the empirical distribution in models of lattice first and last
passage percolation. In recent work \cite{B20}, by arguments hinging on
abstract convex analysis, Erik Bates proved that the answer to a simplified version of Question
\ref{q1} (where one simply considers the value of the noise \emph{on} the geodesic) in the setting of lattice first passage percolation (where i.i.d.\ 
non-negative conductances are placed on the edges) is \emph{yes} for a `dense' class
of conductance distributions, for various notions of denseness. In a preprint
\cite{JLS20} posted following {Bates's} work, 
Janjigian, Lam and Shen analyze the tail behavior for any weak limit point of
the set of empirical distributions by comparing them to the tail behavior of edge conductance variables. As an application, they show that any such limit point must be absolutely continuous with respect to the original weight distribution. 

In the related but separate realm of exactly solvable models of planar last passage percolation, two recent works address the same problem, but are able to obtain a rather precise understanding owing to the availability of various `integrable' inputs for such models. 
In \cite{MSZ21}, Martin, Sly and Zhang study the empirical distribution of
exponential last passage percolation in $\Z^2$ and use connections to particle systems
to exactly pin down the limiting empirical distribution. In a different work
\cite{DSV20}, Dauvergne, Sarkar and Vir\'ag studied the same problem in the
Directed Landscape, a space time random energy field conjectured to be the
universal scaling limit for many last passage percolation models, constructed
in the recent breakthrough paper \cite{DOV18}. Using coalescence of geodesics
(a key device also in our work) and strong Brownian properties of geodesic
weight profiles as the endpoint is varied, which are known for the directed landscape (a
fact also expected to be true in models of FPP), \cite{DSV20} obtains an exact description of the limiting empirical distribution in terms of an LPP problem driven by the Directed Landscape, with certain Brownian boundary data. 

Note that while we prove the convergence of empirical measure of scaled down local metrics, our results do not capture the microscopic local information along the geodesics unlike the above mentioned results where irrespective of how far along the geodesic a point is, one considers a ball of a fixed radius. This is out of reach of the methods in this paper and will be taken up in future research. As a first step one might aim to prove that the limit of $\emm(\cdot,\cdot; \D_{\e})$ exists after appropriate scaling, say to make the diameter one, as $\e\to 0.$

\subsection*{Organisation of the paper}
The rest of the paper is organized as follows. In Section \ref{s:basics}, we formally define Gaussian free fields, LQG metrics and geodesics and recall the relevant results from the literature that we rely on. Section \ref{s:prelim} records certain Radon-Nikodym derivative estimates for GFF with different boundary conditions. In Section \ref{s:convergence} we construct a stationary sequence of fields with exponential decay of correlations that will be critical to the proof of Theorems \ref{main1} and \ref{main2}. In Section \ref{disjseg} we decompose the geodesic into a sequence of approximately i.i.d.\ segments (upto scaling) and in Sections \ref{s:stat_env} and \ref{s:tightness} we use that decomposition to prove Theorems \ref{main1} and \ref{main2}. Section \ref{ss:construct} provides a construction of $\emf$ and $\emm$ via a finite geodesic and a coupling between the two. {Section \ref{s:singularity} and Section \ref{s:ac} are devoted to the proofs of the first and second parts of Theorem \ref{mainsi1} respectively.} Finally, in Section \ref{s:f-field}, we show that $\emm$ is almost surely a function of $\emf$.

\section{Preliminaries}
\label{s:basics}
In this section we shall recall the formal definition of GFF, basic properties of the associated LQG metric and geodesics and also certain crucial results from the literature that we shall use throughout this article.  
\subsection{Gaussian Free Field}
\label{ss:GFF}
For the uninitiated reader we include a brief description of the mathematical
issues one encounters in defining the GFF rigorously on sub-domains of $\C$  and how one addresses them.
For a more thorough treatment, see \cite{Ber07,WP20}. Before beginning, we
recall that throughout the paper, we shall reserve the letter $h$ to denote a
whole plane GFF normalized to have $\mathbf{Av}(h,\mathbb{T})=0$. On the other
hand $\mathtt{h}$ will not denote anything specific, but will in general be
used as a variable for random generalised functions; often, $\mathtt{h}$ will
be a GFF on some domain $U\subsetneq \mathbb{C}$.

\subsubsection{GFF with zero boundary condition}
Let $U\subsetneq \mathbb{C}$ be a Greenian domain (an open set such that a
Brownian motion started at every $z\in U$ hits $\partial U$ almost surely). We will
use $\mathcal{D}(U)$ to denote the space of smooth compactly supported functions
in $U$ with the following topology: a sequence of functions $g_n\rightarrow 0$
if there exists a compact set $A$ such that the $g_n$ are supported on $A$ for
all $n$ and $g_n$ along with all their derivatives converge to $0$ uniformly on
$A$. Define
$H_0^1(U)$ to be the Hilbert space closure of $\mathcal{D}(U)$ with respect to
the inner product
\begin{equation}
	\label{e:b1}
	(f,g)_\nabla=\frac{1}{2\pi}\int_U\nabla f\cdot \nabla g.
\end{equation}
We call the above inner product as the Dirichlet inner product corresponding
to $U$. The corresponding norm is denoted by $\|\cdot\|_\nabla$, and the
square of the norm is called the Dirichlet energy. To avoid
clutter, the dependence on $U$ has been suppressed in the notation. The
space $H_0^1(U)$ is a separable Hilbert space (see \cite[Exercise
11.13]{Leo17}) and thus if
$\left\{ f_1,f_2,\dots \right\}$ form an orthonormal basis of $H_0^1(U)$, 
we can define the zero boundary GFF by the formal sum
\begin{equation}
	\label{e:b2}
	\mathtt{h}=\sum_{i=1}^\infty X_i f_i,
\end{equation}
where $\left\{ X_i \right\}$ is a sequence of i.i.d.\ standard Gaussian random
variables. That is, for any $\phi\in \mathcal{D}(U)$, we can define
\begin{equation}
	\label{e:b3*}
	(\mathtt{h},\phi)_\nabla=
\sum_{i=1}^\infty X_i (f_i,\phi)_\nabla
\end{equation}
 which is distributed as
$N(0,\|\phi\|_\nabla^2)$. The ordinary $L^2$ inner products
$(\mathtt{h},\phi)$ can now be defined by 
\begin{equation}
	\label{e:b3}
	(\mathtt{h},\phi)=-2\pi(\mathtt{h},\Delta^{-1}
\phi)_\nabla
\end{equation}
 where $\Delta^{-1}$ is the inverse Laplacian with zero boundary
conditions.

Recall that for a bounded domain $U$, one can define the H\"older spaces
$H_0^s(U)$ (see for e.g.\ \cite[Section 1.5]{Ber17}) for any $s\in \mathbb{R}$ as the Hilbert space closure of
$\mathcal{D}(U)$ with respect to the inner product 
\begin{equation}
	\label{e:innprod}
	(f,g)_s=\frac{1}{(2\pi)^s} \left(  (-\Delta)^{s/2}f,(-\Delta)^{s/2}g
	\right),
\end{equation}
where, since $U$ is bounded, the fractional powers of the Laplacian can be easily defined by
choosing a suitable eigenbasis for $-\Delta$. The above definition agrees
with the previous definition for $H_0^1(U)$ in the case when $U$ is a bounded
Greenian domain.
As shown in \cite[Theorem
1.23]{Ber07}, in case of bounded $U$, the series \eqref{e:b2} makes
sense as a generalized function on $U$ which is in the H\"older space
$H_0^{-\varepsilon}(U)$ for every $\varepsilon>0$. For unbounded $U$, the restriction
of $\mathtt{h}$ to any bounded subdomain $V$ makes sense as a generalized function
in $H_0^{-\varepsilon}(V)$. In either case, $\mathtt{h}$ makes sense as a random element of the continuous dual
$\mathcal{D}'(U)$ via the mapping $\phi\mapsto (\mathtt{h},\phi)$.

Before discussing the conformal invariance of the GFF, we introduce some
notation. If we have two open sets $U,V\subseteq \mathbb{C}$ with a conformal map
$g:U\rightarrow V$, then any generalized function $\widetilde{\mathtt{h}}\in
\mathcal{D}'(V)$ can be ``composed'' with $g$ to yield $\widetilde{\mathtt{h}}\mydot
g\in \mathcal{D}'(U)$ which is defined by $(\widetilde{\mathtt{h}}\mydot
g,\phi)=(\widetilde{\mathtt{h}},|g'|^{-2}(\phi\circ g^{-1}) )$.
Note that this is different from the usual definition of composition
$\circ$ for functions, and as in the last equation, we will need to use both
of them.
We remark that if we
consider the signed measure on $U$ having density $\phi$ with respect to the
Lebesgue measure, then its pushforward to $V$ has the density
$|g'|^{-2}(\phi\circ g^{-1})$ with respect to the Lebesgue measure on $V$. Now, the
zero boundary GFF is in fact conformally invariant in the following sense: if we
have a zero boundary GFF $\mathtt{h}$ on a domain $V$ and a conformal map $g$
from a domain $U$ to $V$, then $\mathtt{h}\mydot g$ is distributed according to a
zero boundary GFF on $U$. This can be seen as a consequence of the fact that
the Dirichlet norm itself is conformally invariant (see  e.g.\
\cite[Section 1.9]{Ber07}), that is, we have that
$\|\phi\circ g^{-1}\|_\nabla=\|\phi\|_\nabla$ for any function $\phi$ on $U$ satisfying
$\|\phi\|_\nabla<\infty$.

In fact \cite[Section 1.7]{Ber07}, 
one can define $(\mathtt{h},\phi)$ even for $\phi\in
H_0^{-1}(U)$, although unlike the case of $\mathcal{D}(U)$, this can only be defined almost surely once $\phi$ is
fixed. In the case when $\phi$ is the uniform probability measure on the
circle $\mathbb{T}_r(z)$, we denote $(\mathtt{h},\phi)$ by
$\mathbf{Av}(\mathtt{h},\mathbb{T}_r(z))$, and we call the collection $\left\{
\mathbf{Av}(\mathtt{h},\mathbb{T}_r(z)):z\in U, r<\mathtt{dist}(z,\partial U) \right\}$ as
the circle average process corresponding to $\mathtt{h}$. As shown in
\cite[Section 3.1]{DS11}, Kolmogorov's criterion can be used to show that
there exists a modification of the circle average process such that almost
surely,
the mapping $(z,r)\mapsto \mathbf{Av}(\mathtt{h},\mathbb{T}_r(z))$ is continuous (see \cite[Proposition 3.20]{WP20}) on the set $\left\{ (z,r) \in U\times (0,\infty): r<d(z,\partial
U)\right\}$. It is further true that for any fixed $z\in D$ and any fixed
$t_0$ satisfying $\mathbb{D}_{e^{-t_0}}(z)\subseteq U$, we have that for
$t\geq t_0$,
$\mathbf{Av}(\mathtt{h},\mathbb{T}_{e^{-t}}(z))$ is distributed according to a
standard Brownian motion started from time $t_0$.

\subsubsection{Whole plane GFF}
In the case $U=\mathbb{C}$, $\|\cdot\|_\nabla$ would no longer be a norm on
$H_0^1(\mathbb{C})$ since any
constant function has zero Dirichlet energy. To make it into a norm, consider
$\|\cdot\|_\nabla$ on 
the space
$\widetilde{H}^1(\mathbb{C})$, the Hilbert space completion of the set of
functions
$\phi\in \mathcal{D}(\mathbb{C})$ which satisfy the additional constraint
$\int \phi=0$. If $\left\{ f_1,f_2,\dots \right\}$ denotes an orthonormal
basis of $\widetilde{H}^1(\mathbb{C})$, then we define the whole plane GFF $h$ by the
formal sum
\begin{equation}
	\label{e:b4}
	h=\sum_{i=1}^\infty X_i f_i.
\end{equation}
In a manner similar to \eqref{e:b3*} and 
\eqref{e:b3}, we can now first define $(h,\phi)_\nabla$ and subsequently
define $(h,\phi)=-2\pi(h,\Delta^{-1}
\phi)_\nabla$ where $\Delta^{-1}$ is the inverse Laplacian normalized to
satisfy $\int_\mathbb{C}\Delta^{-1} \phi(z)dz=0$. This defines $(h,\phi)$ for all $\phi\in
\mathcal{D}(\mathbb{C})$ which satisfy $\int_\mathbb{C} \phi=0$ and one can
interpret $h$ as a random element of
$\mathcal{D}'(\mathbb{C})$ modulo constants, that is, $h$ is a random
element of $\mathcal{D}'(\mathbb{C})/\sim$ with $\sim$ denoting the equivalence relation where two elements are defined
to be equivalent if they differ by a constant. When defined modulo constants the whole plane GFF $h$
satisfies a conformal invariance property (see \cite[Section 3.2.2]{GHS19}) similar to the one satisfied by the
zero boundary GFF.

To make sense of $h$ as a random
element of $\mathcal{D}'(\mathbb{C})$ itself, one needs to fix a particular normalization which
entails defining $(h,\phi_0)=c$ for some distinguished $\phi_0$ satisfying
$\int \phi_0=1$ (and some constant $c$) and then
defining $(h,\phi)=(h,\phi-(\int \phi)\phi_0)+c\int \phi$. As in the zero
boundary case, $h$ now makes sense as a random element of
$\mathcal{D}'(\mathbb{C})$. Throughout the paper, we will consider the whole
plane GFF $h$ with the normalization $(h,\phi_0)=1$, where $\phi_0$ denotes the
uniform probability measure on the unit circle $\mathbb{T}$. In our previous notation, we take
$h$ to deterministically satisfy
$\mathbf{Av}(h,\mathbb{T})=0$. After fixing the above normalization, $h$ is no
longer exactly conformally invariant. However, using conformal invariance at
the level of a generalized function
modulo constants, one can deduce that $h$ satisfies

\begin{equation}
	\label{e:b4.1}
	h(r\cdot)-\mathbf{Av}(h,\mathbb{T}_r)\stackrel{d}{=}h(\cdot)
\end{equation}
 for any
deterministic $r>0$; to see this, note that both the left and right sides
have average zero on the circle $\mathbb{T}$. 

In this paper, we will often have the setting where the object of interest is
$h\lvert_{\mathbb{C}_{>1}}$ instead of $h$. A random generalized function
$\mathtt{h}$ on $\mathbb{C}_{>1}$ with the same distribution as
$h\lvert_{\mathbb{C}_{>1}}$ will be referred to as a ``\textbf{whole plane GFF
marginal}''. We note that $\mathtt{h}$ also inherits its
circle average process from the corresponding construction for $h$. Thus the
mapping $(z,r)\mapsto \mathbf{Av}(\mathtt{h},\mathbb{T}_r(z))$ makes sense as
a random function which is continuous on the set $\left\{ (z,r)\in
\mathbb{C}_{>1}\times (0,\infty):r<\mathtt{dist}(z,\mathbb{T}) \right\}$.

As was the case for the zero boundary GFF, the same proof from \cite[Section 3.1]{DS11}
shows that the circle
average process for the whole plane GFF $h$ admits a bi-continuous
modification and we always work with this modification. That is, we
take $h$ such that it almost surely satisfies that the mapping $(z,r)\mapsto
\mathbf{Av}(h,\mathbb{T}_r(z))$ is continuous on the set $\left\{
(z,r)\in \mathbb{C}\times (0,\infty)
\right\}$. Similar to the case for the zero boundary GFF, we have that for any
fixed $z\in \mathbb{C}$,
$\mathbf{Av}(h,\mathbb{T}_{e^{-t}}(z))-\mathbf{Av}(h,\mathbb{T}(z))$ is
distributed according to a two-sided standard Brownian motion. In fact, if we
use $\mathbf{Av}(h,\mathbb{T}_{|\cdot|})$ to denote the function $z\mapsto
\mathbf{Av}(h,\mathbb{T}_{|z|})$, then one
has that $\mathbf{Av}(h,\mathbb{T}_{|\cdot|})$ and
$h-\mathbf{Av}(h,\mathbb{T}_{|\cdot|})$ are independent of each other, and this
gives a decomposition of $h$ into the so-called radial and lateral parts. In light of
this, we define $B(t)$, a standard two-sided Brownian motion, by 
\begin{equation}
	\label{e:b4.1*}
	B(t)=\mathbf{Av}(h,\mathbb{T}_{e^t}).
\end{equation}
Similarly, for a whole plane GFF marginal $\mathtt{h}$, we define
$B_\mathtt{h}(t)=\mathbf{Av}(\mathtt{h},\mathbb{T}_{e^t})$ for all $t\geq 0$
and we note that $B_\mathtt{h}$ is a standard one-sided Brownian motion.

At this point, we make a technical remark regarding how the scale invariance
\eqref{e:b4.1} ties up with the above discussion related to the circle
average process. Though
\eqref{e:b4.1} is written with both sides being elements of
$\mathcal{D}'(\mathbb{C})$, standard arguments can be used to show that
scale invariance holds true for both the sides augmented with the circle
average process. Thus, when the operation
$\mathbf{Av}(\cdot,\mathbb{T}_s(z))$ is applied to either side in
\eqref{e:b4.1}, the output is a random continuous function (continuous in both $s,z$)
which has the same distribution.

For most of the paper, we will consider fields defined on the domain $\mathbb{C}_{>1}$. Often we would want to rescale a field $\mathtt{h}$ defined on
$\mathbb{C}_{>r}$ to a field defined on
$\mathbb{C}_{>1}$ (recall the scaling in \eqref{e:field}). For this purpose, we define the rescaling map
$\psi_r\colon \mathbb{C}_{>1}\mapsto \mathbb{C}_{>r}$ by $x\mapsto rx$. Using
this notation, $\mathtt{h}\mydot \psi_r$ will define an element of
$\mathcal{D}'(\mathbb{C}_{>1})$ for any $\mathtt{h}\in
\mathcal{D}'(\mathbb{C}_{>r})$. We now recall the
definition of restrictions of generalized functions from Section
\ref{ss:main_results}, that is, for domains $U\subseteq V$ and a generalized
function $\mathtt{h}$ on $V$, we define $\mathtt{h}\in \mathcal{D}'(U)$
by the action
$(\mathtt{h}\lvert_U,\phi)=(\mathtt{h},\phi)$ for all $\phi\in
\mathcal{D}(U)$.

Given a zero boundary/whole plane GFF $\mathtt{h}$ on a domain containing $\mathbb{C}_{\geq r}$ for some
$r>0$, we can define the centered restricted field
$\widehat{\mathtt{h}}_r=\mathtt{h}\lvert_{\mathbb{C}_{>r}}-\mathbf{Av}(\mathtt{h},\mathbb{T}_r)$
on the domain $\mathbb{C}_{>r}$. Here, we use the ``hat'' symbol in
$\widehat{\mathtt{h}}_r$ to
emphasize that we are centering the field by subtracting off the circle
average $\mathbf{Av}(\mathtt{h},\mathbb{T}_r)$. We will often consider the field
$\widehat{\mathtt{h}}_r\mydot \psi_r$ defined on the domain $\mathbb{C}_{>1}$ when
$\mathtt{h}$ is a whole plane GFF marginal, and in this case, due to the
conformal invariance as in \eqref{e:b4.1}, $\widehat{\mathtt{h}}_r\mydot \psi_r$ is a
whole plane GFF marginal as well.

\subsubsection{The Markov property}
\label{ss:markov}
The GFF enjoys a Markov property which
informally states that if $A$ is a closed set, then the conditional law of the
field on $A^c$ given its value on $A$ is equal to a zero boundary GFF on $A^c$
plus the harmonic extension of the value of the field on $\partial A$. To
prepare for a precise statement and the conditioning, we first define the
relevant $\sigma$-algebras. For any random generalized function $\mathtt{h}$ on an
open set $V$, and for any open set $U\subseteq V$, we define the $\sigma$-
algebra $\mathscr{F}_{U}(\mathtt{h})$ by
\begin{equation}
	\label{e:b5}
	\mathscr{F}_{U}(\mathtt{h})=\sigma\left(  \left\{
	(\mathtt{h},\phi):\phi\in \mathcal{D}(U) \right\} \right).
\end{equation}
For the case when $U=V$, we simply abbreviate
$\mathscr{F}_U(\mathtt{h})$ by $\sigma(\mathtt{h})$. For any closed (as a subset of $V$) set $A\subseteq V$, consider the open set
$A^\delta$ defined by
$A^\delta=\left\{ z:\mathtt{dist}(z,A)<\delta \right\}\cap V$. Define
$\mathscr{F}_{A}(\mathtt{h})$ by
\begin{equation}
	\label{e:b6}
	\mathscr{F}_{A}(\mathtt{h})=\bigcap_{\delta>0}\sigma\left( \left\{
	(\mathtt{h},\phi):\phi\in \mathcal{D}(A^\delta) \right\}\right)=\bigcap_{\delta>0}
	\mathscr{F}_{A^\delta}(\mathtt{h}).
\end{equation}

We note that in case the generalized function $\mathtt{h}$ is only defined
modulo an additive constant, the definitions in \eqref{e:b5} and \eqref{e:b6}
are modified to have the additional constraint $\int_\mathbb{C} \phi=0$ to
ensure that $(\mathtt{h},\phi)$ is well-defined.

By stating that a random variable is
measurable with respect to $\mathtt{h}\lvert_{U}$ (resp.\ $\mathtt{h}\lvert_{A}$), we
mean that the random variable is measurable with respect to
$\mathscr{F}_{U}(\mathtt{h})$ (resp.\ $\mathscr{F}_{A}(\mathtt{h})$).
\paragraph{\textbf{Comments on the notation}}
In the later sections, we will often work with different random generalized
functions $\mathtt{h}$ and for this reason, we make the dependence on
$\mathtt{h}$ explicit in \eqref{e:b5} and \eqref{e:b6}. We will also often
have the setting where we have a measurable map
$X:\mathcal{D}'(U)\rightarrow
\mathbb{C}$ and in this case, $X(\mathtt{h})$ will be a random variable. In
the case where $X$ takes values in $\left\{ 0,1 \right\}$,
$X(\mathtt{h})$ would be the indicator function corresponding to an event $E$ and we might for
instance denote this event by $E[\mathtt{h}]$. We explicitly keep
$[\mathtt{h}]$ in the notation for events to keep
track of the field for which the event is considered i.e., an event $E[\mathtt{h}]$ would be measurable with
respect to $\sigma(\mathtt{h})$.

The Markov property allows us to give meaning to the boundary condition
induced by a GFF on the boundary of a closed set $A$ via its well defined
harmonic extension to $A^c$. We now record a formal statement in the setting of the
whole plane GFF viewed modulo an additive constant.

\begin{lemma}[{\cite[Lemma 2.1]{GMS19}}]
	\label{b2}
	Let $A\subseteq \mathbb{C}$ be a closed set such that $A^c$ is a Greenian
	domain. Use $h'$ to denote $h$ viewed modulo an additive constant. We
	then have the
	decomposition
	$h'=\mathtt{h}_1'+\mathtt{h}_2$, where $\mathtt{h}_1'$ is a random
	generalized function viewed modulo an additive constant which is harmonic on $A^c$ and is measurable with
	respect to
	$\mathscr{F}_A(h')$\footnote{Note that the definition \eqref{e:b6}
	does not directly apply here as $h'$ is only defined modulo an
	additive constant. Instead, we define
	$\mathscr{F}_{A}(h')=\bigcap_{\delta>0}\sigma\left( \left\{
	(h',\phi):\phi\in \mathcal{D}(A^\delta) \text{ with } \int\phi=0
	\right\}\right)$.}. 
	On the other hand, $\mathtt{h}_2$ is independent of $\mathscr{F}_A(h')$ and is distributed according
	to a zero boundary GFF on $A^c$ and is identically $0$ on $A$.
\end{lemma}

In this paper, we will always work with sets $A$ of the form
$A=\mathbb{C}_{\leq r}$ for some $r\ge1$ and will repeatedly condition on
restrictions of the type $h\lvert_{\mathbb{C}_{\leq r}}$. We thus introduce
convenient notation for the corresponding filtration. Define the
filtration $\mathscr{F}=\{\sF_r\}_{r\ge 0}$ by
$\mathscr{F}_r=\mathscr{F}_{\mathbb{C}_{\leq r}}(h)$, where the latter was
defined in \eqref{e:b6}.
Apart from the whole plane GFF $h$, we will need to define a corresponding
filtration for a whole plane GFF marginal $\mathtt{h}$ defined on
$\mathbb{C}_{>1}$. We define the filtration
$\mathscr{F}(\mathtt{h})=\left\{ \mathscr{F}_r(\mathtt{h}) \right\}_{r>1}$ by
$\mathscr{F}_r(\mathtt{h})=\mathscr{F}_{\mathbb{C}_{(1,r]}}(\mathtt{h})$. We now restate Lemma \ref{b2} for
this case in a form that will be convenient for us.
\begin{lemma}
	\label{markov*}
	For any $r>0$, we have the following decomposition of fields on
	$\mathbb{C}_{>r}$.
	\begin{equation}
		\label{e:markov*}
		\widehat{h}_r=\widehat{\mathbf{Hr}}_{h,r}+\mathtt{h}_2.
	\end{equation}
	Here, $\widehat{\mathbf{Hr}}_{h,r}$ is a random harmonic function on
	$\mathbb{C}_{>r}$ which is measurable with respect to
	$\mathscr{F}_r(h)$ and satisfies
	$\mathbf{Av}(\widehat{\mathbf{Hr}}_{h,r},\mathbb{T}_s)=0$ for all $s>r$
	along with $\lim_{z\rightarrow
	\infty}\widehat{\mathbf{Hr}}_{h,r}(z)= 0$.
	On the other hand, $\mathtt{h}_2$ is independent of
	$\mathscr{F}_r(h)$ and is distributed as a zero boundary GFF on
	$\mathbb{C}_{>r}$.
\end{lemma}
\begin{proof}
	We first apply Lemma \ref{b2} with $A=\mathbb{C}_{\leq r}$ to obtain the decomposition
	$h'=\mathtt{h}_1'+\mathtt{h}_2$ and then fix the additive constant of
	$h',\mathtt{h}_1'$ by enforcing them to have average zero on the circle
	$\mathbb{T}_r$. This yields the decomposition \eqref{e:markov*}. Note
	that $\mathtt{h}_1'$ is harmonic in the region $\mathbb{C}_{>r}$ and
	thus on normalising it to have average zero on $\mathbb{T}_r$, we obtain
	by the mean value property of harmonic functions that it has average
	zero on all circles $\mathbb{T}_s$ for $s>r$ as well. This implies that
	$\mathbf{Av}(\widehat{\mathbf{Hr}}_{h,r},\mathbb{T}_s)=0$ for all $s>r$.
	
      To see that $\lim_{z\rightarrow
	\infty}\widehat{\mathbf{Hr}}_{h,r}(z)=0$, we first argue that
	$\widehat{\mathbf{Hr}}_{h,r}$ can be extended to a function on
	$\mathbb{C}\cup\left\{ \infty \right\}$. We begin by noting the fact that any finite energy harmonic function on
	$\mathbb{C}$ can be extended harmonically to $\mathbb{C}\cup\left\{ \infty
	\right\}$, which by Liouville's theorem implies that it must be a constant.
	This together with the argument in the proof of \cite[Lemma 1.36]{BP21} used
	to treat domains without $C^1$ boundary (by taking the domain to be
	$\mathbb{C}\subseteq \mathbb{C}\cup \left\{ \infty \right\}$ shows that $H_0^1(\mathbb{C})$ can
	be identified with $H_0^1(\mathbb{C}\cup\left\{ \infty
	\right\}))$, when both
	the spaces are viewed modulo additive constants. Thus $h$ extends as a
	generalized function modulo an additive constant on $\mathbb{C}\cup \left\{ \infty
	\right\}$. A version of the Markov property analogous to Lemma \ref{b2}
	can now be applied with the domain
	$\mathbb{C}_{>r}\cup \left\{ \infty \right\}\subseteq
	\mathbb{C}\cup \left\{ \infty \right\}$ to obtain a
	harmonic function modulo an additive constant on $\mathbb{C}_{>r}\cup \left\{ \infty \right\}$. The elliptic regularity argument in the proof
	of \cite[Lemma 1.36]{BP21}
	can be used to show that viewed modulo additive constants, the above harmonic function when restricted to $\mathbb{C}_{>r}$ agrees with
	$\widehat{\mathbf{Hr}}_{h,r}$. Thus $\widehat{\mathbf{Hr}}_{h,r}$ can be extended
	harmonically to
	$\infty$. 
	Together with the condition
	$\mathbf{Av}(\widehat{\mathbf{Hr}}_{h,r},\mathbb{T}_s)=0$ for all $s>r$
	implies that $\lim_{z\rightarrow
	\infty}\widehat{\mathbf{Hr}}_{h,r}(z)= 0$.

\end{proof}
We will often work with a whole plane GFF marginal $\mathtt{h}$ instead of
$h$, and we thus state the above Markov property in this setting.
Note that for a whole plane
GFF marginal $\mathtt{h}$, the following lemma
also defines $\widehat{\mathbf{Hr}}_{\mathtt{h}}$ which we call the harmonic extension of the boundary
condition of $\mathtt{h}$.
\begin{lemma}
	\label{markovv}
	Let $\mathtt{h}$ be a whole plane GFF marginal. Then we
	have the independent decomposition $\mathtt{h}=
	\widehat{\mathbf{Hr}}_{\mathtt{h}}+\mathtt{h}_2$, where $\mathtt{h}_2$ is a
	zero boundary GFF on $\mathbb{C}_{>1}$ and $\widehat{\mathbf{Hr}}_{\mathtt{h}}$
	is a random harmonic function on $\mathbb{C}_{>1}$ satisfying
	$ \lim_{z\rightarrow
	\infty}\widehat{\mathbf{Hr}}_{\mathtt{h}}(z)= 0$.
\end{lemma}
\begin{proof}
	This follows by applying Lemma \ref{markov*} with $r=1$.
\end{proof}

As we have been doing so far, we put the ``hat'' in $\widehat{\mathbf{Hr}}_{\mathtt{h}}$ is to
emphasise that it is centered, in the sense that we have
$\mathbf{Av}(\widehat{\mathbf{Hr}}_{\mathtt{h}},\mathbb{T}_{r})=0$ for all
$r> 1$.

In our applications of the Markov property from Lemma \ref{markovv}, it
will be useful to have a quantity describing how strong the effect of the
harmonic extension $\widehat{\mathbf{Hr}}_\mathtt{h}$ is, we now proceed to define it. For some fixed
$\varepsilon>0$, we define ${\sf{Rough}}_{\varepsilon}(\mathtt{h})$ by
\begin{equation}
	\label{e:b8}
	{\sf{Rough}}_\varepsilon(\mathtt{h})= \sup_{w\in \mathbb{T}_{(1+\varepsilon
	)}}|\widehat{\mathbf{Hr}}_{\mathtt{h}}(w)|.
\end{equation}
Such a quantity was
considered in \cite{MQ20}, we think of this as the roughness of the boundary condition corresponding
to $\mathtt{h}$.

\subsection{LQG metric and measure}
\label{ss:LQG}
Recall from before that given a GFF $\mathtt{h}$ on an open set $U\subseteq \mathbb{C}$, the LQG measure
$\mu_{\mathtt{h}}$ heuristically corresponds to the random measure $e^{\gamma
\mathtt{h}(z)}|dz|^2$ for $\gamma\in (0,2)$. It can be rigorously defined by a regularization
procedure which we briefly outline. Regularize $\mathtt{h}$ by replacing it
with the function $z\mapsto
\mathbf{Av}(\mathtt{h},\mathbb{T}_\varepsilon(z))$ and consider the measure
$\mu^\varepsilon(\mathtt{h})$ defined by
\begin{equation}
	\label{e:b9}
	\mu^\varepsilon_\mathtt{h}=e^{\gamma
	\mathbf{Av}(\mathtt{h},\mathbb{T}_\varepsilon(z))}\varepsilon^{\gamma^2/2}|dz|^2.
\end{equation}
From the results in \cite{SW16}, it follows that almost surely, $\mu^\varepsilon(\mathtt{h})$
converges weakly as $\varepsilon\rightarrow 0$ to a measure which we denote as
$\mu_\mathtt{h}$. It is in fact true that both $\mathtt{h}$ and
$\mu_\mathtt{h}$ determine each other measurably \cite{BSS14} and thus one can
define the map $\mathtt{h}\mapsto \mu_\mathtt{h}$ as a measurable map from
$\mathcal{D}'(U)$ to the space of measures on $U$, where the latter is
equipped with the $\sigma$-algebra generated by the maps $\mu\mapsto \mu(V)$
for all open subsets $V\subseteq U$.

Before coming to the LQG metric, we introduce some terminology from the
literature (see \cite{GM19+}). By a GFF plus a continuous function $\mathtt{h}$ on a domain
$U\subseteq \mathbb{C}$, we mean a random generalized function $\mathtt{h}$
which is coupled with a random continuous function $f$ on $U$ such that
$\mathtt{h}-f$ is distributed as a GFF (zero-boundary or whole plane) on
$U$.

The question of the existence and uniqueness of the LQG metric turned out to be much harder than the
case of the LQG measure. It was recently established \cite{DDDH20, DFGPS20,
GM21} that one can define the LQG metric uniquely by using a regularization
procedure analogous to \eqref{e:b9} and taking weak limits. Further, the metric can be obtained
from the GFF measurably in the sense that there is a measurable map $\mathtt{h}\mapsto
D_\mathtt{h}(\cdot,\cdot)$ from $\mathcal{D}'(U)$ with the weak-$*$ topology to the space of continuous
metrics on $U$ thought of as a subset of $C(U\times U)$ with
the topology of uniform convergence on compact subsets, such that whenever
$\mathtt{h}$ is a GFF plus a continuous function, the regularized sequence
of pre-limiting metrics corresponding to $\mathtt{h}$
converge in probability to the random metric $D_\mathtt{h}(\cdot,\cdot)$.
We note that there is some arbitrary choice involved
in the definition of the measurable map $D_{(\cdot)}$. However, if $D_{(\cdot)}$ and
$\widetilde{D}_{(\cdot)}$ are two measurable maps such that if the regularized pre-limiting metrics for
$\mathtt{h}$ converge in
probability to both
$D_\mathtt{h}(\cdot,\cdot)$ and $\widetilde{D}_\mathtt{h}(\cdot,\cdot)$  for any GFF plus a continuous
function $\mathtt{h}$, then
$D_\mathtt{h}(\cdot,\cdot)\stackrel{\mathrm{a.s.}}{=}\widetilde{D}_\mathtt{h}(\cdot,\cdot)$.

We
refer to the above map $\mathtt{h}\rightarrow D_\mathtt{h}$ as the field-metric correspondence. Heuristically, the LQG metric can be
interpreted as the Riemannian metric tensor $e^{2\xi \mathtt{h}}(dx^2+dy^2)$ where recall that $\xi= \gamma/d_\gamma$, where the deterministic constant $d_\gamma$ is
the Hausdorff dimension \cite{GP19} of $\mathbb{C}$ as a metric space with the LQG
metric coming from $h$. Using the LQG metric, one can define
the lengths of continuous curves $\zeta\colon[a,b]\rightarrow U$ by
\begin{equation}
	\label{e:b10}
	\ell(\zeta,D_\mathtt{h})=
	\sup_{\mathcal{P}}\sum_{i=0}^{|\mathcal{P}|-1}D_\mathtt{h}(\zeta(t_i),\zeta(t_{i+1})),
\end{equation}
where the supremum is over all partitions $\mathcal{P}=\left\{
t_0,t_1,\dots,t_{|\mathcal{P}|} 
\right\}$ of $[a,b]$. For any open set $V\subseteq U$, one can now define
the induced metric $D_\mathtt{h}(\cdot,\cdot;V)$ by 
\begin{equation}
	\label{e:b9.1}
	D_\mathtt{h}(z,w;V)=\inf_{\zeta: z\rightarrow w, \zeta\subseteq
	V}\ell(\zeta,D_h).
\end{equation}
It is a general fact about metrics (see Lemma \ref{intrin}) that if $\zeta:[a,b]\rightarrow V\subseteq
U$ is a continuous curve, then the length of $\zeta$ with respect to a metric on $U$ is
the same as its length with respect to the induced metric on $V$. In other
words, in our setting, we have
$\sup_{\mathcal{P}}\sum_{i=0}^{|\mathcal{P}|}D_\mathtt{h}(\zeta(t_i),\zeta(t_{i+1}))=\sup_{\mathcal{P}}\sum_{i=0}^{|\mathcal{P}|}D_\mathtt{h}(\zeta(t_i),\zeta(t_{i+1});V)$
and this justifies why we do not introduce a variant of the definition
\eqref{e:b10} for induced metrics.
Geodesics are now defined to be the curves attaining the
minimum length between two points; note that they might not always exist. The following properties of the LQG metric
taken from \cite{GM19+}
will be crucial to us. 
\begin{proposition}[{\cite[Definition 1.1]{GM19+}}]
	\label{b3}
	The LQG metric $D_\mathtt{h}(\cdot,\cdot)$ satisfies the following
	properties for $\mathtt{h}$, any  GFF plus a continuous function on
	$U$.
	\begin{enumerate}
		\item \textbf{Length space.} The set $U$ equipped with $D_\mathtt{h}(\cdot,\cdot)$ almost
		surely forms a length space. That is, almost surely, for any points
		$z,w\in U$, $D_\mathtt{h}(z,w)$ is equal to
		$\inf_{\zeta:z\rightarrow w}\ell(\zeta,\mathtt{h})$, where the infimum
		is over all continuous paths $\zeta$ in $U$ from $z$ to $w$.

\item \textbf{Locality.} For any fixed open set $V\subseteq U$, we almost surely have that the induced metric
	$D_\mathtt{h}(\cdot,\cdot;V)$ satisfies
	$D_\mathtt{h}(\cdot,\cdot;V)=D_{\mathtt{h}\lvert_{V}}(\cdot,\cdot)$. In
	particular,
	$D_\mathtt{h}(\cdot,\cdot;V)$ is a.s.\ equal to a metric
	measurable with respect to
	$\mathtt{h}\lvert_{V}$.
\item \textbf{Weyl Scaling.} For a continuous function $f\colon
		U\rightarrow \mathbb{R}$, define the metric $(e^{\xi f}
		\cdot D_\mathtt{h})(\cdot,\cdot)$
		by
		\begin{equation}
			\label{e:b11}
			(e^{\xi f} \cdot D_\mathtt{h})(z,w)=\inf_{\zeta: z\rightarrow
			w}\int_0^{\ell(\zeta,D_\mathtt{h})}e^{\xi f(\zeta(t))} dt
		\end{equation}
		where the infimum is over all paths $\zeta\colon z\rightarrow w$
		parametrized by $D_\mathtt{h}$-length. Then we have that almost
		surely, for all continuous functions $f\colon U\rightarrow
		\mathbb{R}$ simultaneously, $e^{\xi
		f}\cdot D_\mathtt{h}(\cdot,\cdot)=D_{\mathtt{h}+f}(\cdot,\cdot)$. 
	\item \textbf{Coordinate change formula.} Let $V\subseteq \mathbb{C}$ be
		an open set and let
	$\phi\colon U\rightarrow V$ be a deterministic conformal map. Then we have that almost
	surely, for all $z,w\in U$,
	\begin{equation}
		\label{e:b12}
		D_\mathtt{h}(z,w)=D_{\mathtt{h}\mydot \phi^{-1}+Q\log
		|(\phi^{-1})'|}\left(\phi(z),\phi(w)\right),
	\end{equation}
	where $(\phi^{-1})'$ denotes the derivative of the conformal map
	$\phi^{-1}$ and $Q=\gamma/2+2/\gamma$.
\end{enumerate}
\end{proposition}

By using the coordinate change formula and the Weyl scaling for the special case
$U=\mathbb{C}$ and $\mathtt{h}=h$, we have that for any deterministic $r>0$
and all $z,w\in \mathbb{C}$,
\begin{equation}
	\label{e:b13}
	D_{h(r\cdot)}(rz,rw)=r^{\xi Q} D_{h(\cdot)}(z,w).
\end{equation}
Another application of the Weyl scaling
yields the following expression which will
be used repeatedly throughout the paper. Before recording it as a lemma, we recall
that we had defined $B(t)=\mathbf{Av}(h,\mathbb{T}_{e^t})$ in
\eqref{e:b4.1*}.

\begin{lemma}
	\label{b4}

Let $\mathtt{h}$ be a whole plane GFF marginal. Then for any fixed $r>1$,
almost surely, we
have that 
	\begin{displaymath}
		D_{\mathtt{h}\lvert_{\mathbb{C}_{>
	r}}}(r\cdot,r\cdot)=r^{\xi Q}e^{\xi
	\mathbf{Av}(\mathtt{h},\mathbb{T}_r)}D_{\widehat{\mathtt{h}}_r\mydot
	\psi_r}(\cdot,\cdot).
	\end{displaymath}
Similarly, for the whole plane GFF $h$, we almost surely have
\begin{displaymath}
		D_{h\lvert_{\mathbb{C}_{>
	r}}}(r\cdot,r\cdot)=r^{\xi Q}e^{\xi B(\log
	r)}D_{\widehat{h}_r\mydot
	\psi_r}(\cdot,\cdot).
	\end{displaymath}

\end{lemma}

\begin{proof}
	We have that
\begin{align}
	\label{e:b14}
	D_{\mathtt{h}\lvert_{\mathbb{C}_{>
	r}}}(r\cdot,r\cdot)&=(r^{\xi Q}\cdot
	D_{\mathtt{h}\mydot
\psi_r})(\cdot,\cdot)=r^{\xi Q}e^{\xi \mathbf{Av}(\mathtt{h},\mathbb{T}_r)}
D_{\widehat{\mathtt{h}}_r\mydot
\psi_r}(\cdot,\cdot).
\end{align}
The first equality is obtained by using the coordinate change formula and Weyl
scaling from Proposition \ref{b3}, while the second equality requires only
Weyl scaling.
\end{proof}

Before moving on to the properties of the geodesics, we first record a
couple useful lemmas about metrics and induced metrics.

\begin{lemma}
	\label{intrin}
	Let $d$ be a metric on an open set $V$ and $U\subseteq V$ be an
	open subset. Let $\zeta:[a,b]\rightarrow U$ be a path. Then we have that
	$\ell(\zeta;d(\cdot,\cdot;U))=\ell(\zeta;d)$, and as a particular
	consequence, $d(\cdot,\cdot;U)$ is a length metric. Also, if $\zeta$ is a
	geodesic for $d$ then it is also a geodesic for $d(\cdot,\cdot;U)$.
\end{lemma}
\begin{proof}
	The first statement appears as \cite[Proposition 2.3.12]{Burago}. To
	obtain the second statement, note that for any $x,y\in U$, we have that
	$d(x,y;U)\geq d(x,y)$. By the first statement and the fact that $\zeta$
	is a geodesic for $d$, we have that
	$d(x,y)=\ell(\zeta;d)=\ell(\zeta,d(\cdot,\cdot;U))$ and these both imply
	that $\zeta$ is a geodesic for $d(\cdot,\cdot;U)$ as well.
\end{proof}

\begin{lemma}
	\label{lengthcts}
	Let $d$ be a continuous length metric on an open set $V$ and $U\subseteq
	V$ be an open subset. Then $d(\cdot,\cdot;V)$ is a continuous length
	metric on $V$. As a consequence, almost surely,
	$D_h(\cdot,\cdot;U)$ is a length metric simultaneously for all
	sets $U\subseteq \mathbb{C}$.
\end{lemma}
\begin{proof}
	For the first statement, see \cite[Lemma 1.1]{GM19}. The second
	statement follows by using that $D_h$ is a.s.\ a length metric on
	$\mathbb{C}$.
\end{proof}

\subsubsection{Properties of the geodesics}
\label{ss:geodesics}

Regarding geodesics, as a consequence of \cite[Corollary 2.5.20]{Burago},
almost surely, any pair of
points $z,w\in \mathbb{C}$ have at least one $D_h$ geodesic
between them. Moreover, for any two fixed points $z,w\in
\mathbb{C}$ (and hence all pairs of rational points simultaneously), there
exists a unique $D_h$ geodesic from $z$ to $w$ \cite[Theorem
1.2]{MQ20}. Note that for any
GFF $\mathtt{h}$ on an open set $U\subsetneq \mathbb{C}$ and the
corresponding metric $D_\mathtt{h}(\cdot,\cdot)$, it is possible
that geodesics do not exist. However, by adapting the same argument as in \cite[Theorem
1.2]{MQ20}, it can be shown that between any two fixed points, at most one
such geodesic exists almost surely. In case there does exist a unique geodesic between two points $z,w\in U$, we denote it by
$\Gamma(z,w;\mathtt{h})$.

We will always parametrise the above geodesics to cover LQG distance at rate $1$. That
is, we define $\Gamma_t(z,w;\mathtt{h})$ such that we have
	$D_\mathtt{h}(z,\Gamma_t(z,w;\mathtt{h}))=t$
for all $0\leq t\leq D_\mathtt{h}(z,w)$.

It was recently shown in \cite{GPS20} that there almost surely exists a
unique geodesic $\Gamma$ from $0$ to $\infty$ for the whole plane GFF
$h$ and this geodesic is also scale invariant in law. This geodesic is of central importance to us and we now record the above fact as a proposition. A random path {$\zeta:[0,\infty)\rightarrow \C$} coupled with $h$ is said to be a geodesic
from $0$ to $\infty$ if it satisfies
$\zeta(0)=0$, $\zeta(t)\rightarrow \infty$ as $t\rightarrow \infty$ and $\zeta\lvert_{[a,b]}$ is a $D_h(\cdot,\cdot)$ geodesic between
$\zeta(a)$ and $\zeta(b)$ for all $0\leq a<b<\infty$.
\begin{proposition}[{\cite[Proposition 4.4]{GPS20}}]
	\label{b5}
	Almost surely, there exists a unique geodesic $\Gamma$ from ${0}$ to
	$\infty$ for $D_h$, the LQG metric coming from the whole plane GFF $h$.
	When parametrised by the LQG length with respect to
	$D_h$, $\Gamma$ is scale invariant in the following sense:
	\begin{displaymath}
		\left( h(r\cdot)-B(\log r),r^{-1}\Gamma_{(r^{\xi
		Q}e^{\xi B(\log
		r)}\cdot)}\right)\stackrel{d}{=}\left(h,\Gamma_{(\cdot)}\right).
	\end{displaymath}
\end{proposition}

We will often simply refer to $\Gamma$ as the infinite geodesic.

\subsection{Inputs from the literature} We now state a series of
results appearing in the literature which we will be using throughout the sequel. These results are mainly related to the smoothness of harmonic functions as
well as H\"older regularity for the LQG metric.

Recall the roughness ${\sf{Rough}}_\varepsilon(\mathtt{h})$ as defined in
\eqref{e:b8}. The following lemma from \cite{MQ20} gives Gaussian tails for
the roughness.
\label{ss:import}
\begin{proposition}[{\cite[Lemma 4.4]{MQ20}}]
	\label{1.15}
	For any fixed $\varepsilon>0$, and a whole plane GFF marginal $\mathtt{h}$,
	we have that there exist positive constants $c,C$ depending on
	$\varepsilon$ such that for all $t>0$,
	\begin{displaymath}
		\mathbb{P}\left({\sf{Rough}}_{\varepsilon}(\mathtt{h})\geq t \right)\leq
		Ce^{-ct^2}.
	\end{displaymath}
\end{proposition}
In \cite{MQ20}, the above proposition is stated in the setting where
$h\lvert_{\mathbb{C}_{<1}}$ is being explored inwards towards the origin
starting from the circle $\mathbb{T}$, whereas in our setting, the whole
plane GFF marginal corresponds to
$h\lvert_{\mathbb{C}_{>1}}$ being explored outwards towards $\infty$ starting from the
circle $\mathbb{T}$. It is easy to see that Proposition \ref{1.15} does hold
in the latter setting also because applying the conformal map $z\mapsto 1/z$
takes $\mathbb{C}_{>1}$ to $\mathbb{C}_{<1}$ and
$h\lvert_{\mathbb{C}_{>1}}$ to
$h\lvert_{\mathbb{C}_{<1}}$ by the conformal invariance of $h$ viewed modulo a
constant.

The following proposition from \cite{GM20} uses the correlation structure of
the GFF to iterate events depending on the recentered field in disjoint
concentric annuli.
\begin{proposition}[{\cite[Lemma 2.12]{GM20}}]
	\label{1.16}
	Fix $1<\alpha_1<\alpha_2<\infty$. Let $\left\{ r_i
	\right\}_{i\in \mathbb{N}\cup\left\{ 0 \right\}}$ be an increasing
	sequence of positive real numbers such that $r_{i+1}/r_i\geq
	\alpha_2$ for $i\in \mathbb{N}\cup\left\{  0\right\}$ and let
	$\left\{ E_{r_i} \right\}_{i\in \mathbb{N}\cup\left\{ 0
	\right\}}$ be events such that $E_{r_i}\in \sigma \left( \left(
	h-\mathbf{Av}(h,\mathbb{T}_{r_i})\right)\lvert_{\mathbb{C}_{(\alpha_1r_i,\alpha_2r_i)}}
	\right)$ for each $i\in \mathbb{N}\cup \left\{ 0 \right\}$. For $I\in
	\mathbb{N}$, let $N(I)$ be the number of $i\in [\![1,I]\!]$ for which
	$E_{r_{i}}$ occurs. If there exists a $p\in (0,1)$
	such that 
	\begin{displaymath}
		\mathbb{P}\left( E_{r_i} \right)\geq p 
	\end{displaymath}
	for all $i\in \mathbb{N}\cup\left\{ 0 \right\}$, then there exist
	$c>0,b\in (0,1)$ and $C>0$ depending on $p,\alpha_1,\alpha_2$ such that
	\begin{displaymath}
		\mathbb{P}\left( N(I)<bI \right)\leq Ce^{-cI}
	\end{displaymath}
\end{proposition}

The following proposition is a special case of a stronger result on the finiteness
of the
moments of the LQG-diameter of bounded sets.
\begin{proposition}[{\cite[Theorem 1.8]{DFGPS20}}]
	\label{imp2}
	Let $A$ be a compact connected set with more than one point. Then we have
	that
	\begin{displaymath}
		\mathbb{E}\left[ \left( \sup_{u,v\in A}D_h(u,v) \right)^p \right]<\infty
	\end{displaymath}
	for all $p\in (0,\frac{4d_\gamma}{\gamma^2})$.
\end{proposition}
The following proposition is a strong quantitative version of the statement that the
LQG metric is locally H\"older continuous with respect to the Euclidean
metric. 
\begin{proposition}[{\cite[Lemma
	3.20]{DFGPS20}}]
	\label{imp3}
	For each $\chi\in (0,\xi(Q-2))$ and each compact set $A$, there exists a
	$p>0$ such that for all $\varepsilon$ small enough and for all
	$r>0$,
	\begin{displaymath}
		\mathbb{P}\left( r^{-\xi Q}e^{-\xi\mathbf{Av}(h,\mathbb{T}_r)}
		\sup_{u,v\in rA;|u-v|\leq \varepsilon r}D_h(u,v;\mathbb{D}_{ 2|u-v|}(u)) \leq
		\left|\frac{u-v}{r}\right|^\chi\right)\geq 1-\varepsilon^p.
	\end{displaymath}
\end{proposition}
Similarly, the LQG metric is also H\"older lower bounded by the Euclidean
metric and this is recorded in the following proposition from \cite{DFGPS20}.
\begin{proposition}[{\cite[Lemma
	3.18]{DFGPS20}}]
	\label{imp3*}
	For each $\chi'>\xi(Q-2)$ and each compact set $A$, there exists a
	$p>0$ such that for all $\varepsilon$ small enough and all $r>0$, 
	\begin{displaymath}
		\mathbb{P}\left( r^{-\xi Q}e^{-\xi\mathbf{Av}(h,\mathbb{T}_r)}
		\inf_{u,v\in rA;|u-v|\leq \varepsilon r}D_h(u,v) \geq
		\left|\frac{u-v}{r}\right|^{\chi'}\right)\geq 1-\varepsilon^p.
	\end{displaymath}
\end{proposition}
Later, to prove the tightness of $\emf_t$, we will require the following uniform estimates
on the $H_0^{-\varepsilon}$ norm of the restriction of $h$ to bounded subsets.
\begin{proposition}
	\label{imp4}
	Fix an $\varepsilon>0$ and a bounded open set $U$. For each $N>0$, there exists a
	constant $c_N$ satisfying $c_N\rightarrow 0$ as $N\rightarrow \infty$
	such that for all $r>0$, we have
	\begin{displaymath}
		\mathbb{P}\left(
		\|\left(h-\mathbf{Av}(h,\mathbb{T}_r)\right)\lvert_{rU}\|_{H_0^{-\varepsilon}(rU)}
		\leq N\right)\geq 1-c_N.
	\end{displaymath}
\end{proposition}
We now indicate how the above can be obtained. For a fixed bounded domain $V\subseteq \mathbb{C}$, we know
from the discussion in Section \ref{ss:GFF} that the law of a zero boundary
GFF on $V$ is supported on $H_0^{-\varepsilon}(V)$ for all $\varepsilon>0$.
Now note that for any bounded domains $V\subseteq W$ with
$\mathtt{dist}(V,\partial W)>0$, when viewed modulo additive
constants, $h\lvert_{V}$ is mutually absolutely
continuous with respect to $\mathtt{h}'\lvert_V$, where $\mathtt{h}'$ is a
zero boundary GFF on
$W$ (see
\cite[Proposition 2.11]{MS17}). For a given bounded domain $U$, we can
choose $V,W$ as above which in addition satisfy $U\cup \mathbb{T}\subseteq
V\subseteq W$. We now normalize the generalized functions modulo constants in the
above absolute continuity statement to have zero average on
$\mathbb{T}$ and thus obtain
that $h\lvert_V$ is mutually absolutely continuous with respect to
$\mathtt{h}'\lvert_{V}-\mathbf{Av}(\mathtt{h}',\mathbb{T})$. We know that
$\mathtt{h}'\in H_0^{-\varepsilon}(W)$ and this implies that
$\mathtt{h}'\lvert_V\in H_0^{-\varepsilon}(V)$. Since
$\mathbf{Av}(\mathtt{h}',\mathbb{T})$ is a constant, it can be thought of as
an element of 
$L^2(V)\subseteq H^{-\varepsilon}_0(V)$, and thus we obtain that
$\mathtt{h}'\lvert_{V}-\mathbf{Av}(\mathtt{h}',\mathbb{T})\in
H^{-\varepsilon}_0(V)$ and thereby $h\lvert_V\in H^{-\varepsilon}_0(V)$ by the
mutual absolute continuity and the measurability of
$H^{-\varepsilon}_0(V)$ as a subset of $\mathcal{D}'(V)$. Since $U\subseteq V$, this immediately
implies that $h\lvert_U\in H^{-\varepsilon}_0(U)$. This combined with the scale invariance of $h$ as in
\eqref{e:b4.1} can be used to obtain Proposition \ref{imp4}.

In particular, Proposition \ref{imp4} implies that for any fixed bounded domain $U$,
$h\lvert_{rU}\in H_0^{-\varepsilon}(U)$ almost surely. Thus almost surely, $h\lvert_{\mathbb{D}_n}\in
H_0^{-\varepsilon}(\mathbb{D}_n)$ for all $n\geq 1$. Finally, any bounded domain
$U$ satisfies $U\subseteq \mathbb{D}_n$ for some $n$ large enough and this
finally leads to the conclusion that almost surely, simultaneously for all
bounded $U\subseteq \mathbb{C}$, $h\lvert_{U}\in H_0^{-\varepsilon}(U)$. An
important consequence of
this is that for almost every instance of $h$, the laws of the variables
$\mathbf{Field}_t$ are supported on $H_0^{-\varepsilon}(\mathbb{D})$ for all
$t>0$.

Before proceeding further, we make a few comments regarding measurability issues
	which will be useful later.

\subsection{Comments regarding measurability}
	\label{ss:meas}
	The measurability of some of the objects and
	operations introduced earlier (e.g.\ length metrics, induced metrics, geodesics) is not
	a-priori clear and in this section, we indicate such arguments. 
	\paragraph{\textbf{Set of length metrics as a subset of $C(U\times U)$}}
	 To begin, we sketch the proof of the fact that
	continuous length metrics on an open set $U$ themselves form a measurable set of $C(U\times
	U)$. First, it is easy to check that the space of continuous metrics is itself a
	measurable subset of $C(U\times U)$. Now fix $x,y\in U$ and a sequence of
	compact sets $S_n$ increasing to $U$ and consider the
	set 
	\begin{displaymath}
		E^{x,y}_{n,\varepsilon}=\left\{ \text{continuous metrics } d \text{ on } U: \exists
	\zeta:x\rightarrow y ,\zeta\subseteq S_n, \ell(\zeta;d)\leq
	(1+\varepsilon)d(x,y) \right\}.
	\end{displaymath}
	Note that $\left\{ \text{continuous length
	metrics on U} \right\}=\bigcap_{x,y\in \mathbb{Q}^2\cap U}\bigcap_{m\in \mathbb{N}}\bigcup_{n\in
	\mathbb{N}}E^{x,y}_{n,2^{-m}}:=E$. To see this, consider some $x,y\in U$ and
	choose
rational points $x_n,y_n$ such that $x_n\rightarrow x,y_n\rightarrow y$
sufficiently rapidly in the $d$ metric. On the event $E$, we can choose paths $\zeta:x_1\rightarrow y_1,\zeta^x_n:x_{n+1}\rightarrow
x_n,\zeta^y_n:y_{n}\rightarrow
y_{n+1}$ such that $\ell(\zeta;d)\leq (1+\varepsilon)d(x_1,y_1)$, $\ell(\zeta^x_n;d)\leq
(1+\varepsilon)d(x_{n+1},x_n)$, $\ell(\zeta^y_n;d)\leq
(1+\varepsilon)d(y_n,y_{n+1})$. Consider the concatenation $\zeta^*=\dots
\cdot\zeta^x_2\cdot \zeta^x_1\cdot \zeta \cdot \zeta^y_1 \cdot
\zeta^y_2\dots$ and note that $\zeta^*$ is a path from $x$ to $y$ and by the
above inequalities, it can be shown that $\ell(\zeta^*;d)\leq
(1+O(\varepsilon))d(x,y)$.

We now sketch the argument for the measurability of $\bigcap_{m\in \mathbb{N}}\bigcup_{n\in
	\mathbb{N}}E^{x,y}_{n,2^{-m}}$ for
fixed $x,y\in \mathbb{Q}^2\cap U$. We define $E^{x,y}_{n,\varepsilon,i}$ to be the set 
	\begin{align*}
		E^{x,y}_{n,\varepsilon,i}=&\Big\{ \text{continuous metrics } d \text{ on
		} U: \exists f_i:[0,1]\cap \frac{1}{2^i}\mathbb{Z}\rightarrow \mathbb{Q}^2\cap S_n \text{ such
		that }f_i(0)=x,f_i(1)=y \nonumber\\
		&\qquad \qquad \text{ and } d(f_i(j/2^i),f_i((j+1)/2^i)\leq (1+\varepsilon)d(x,y)/2^i
		\text{ for all } j\Big\}.
	\end{align*}
	The measurability of each set $E^{x,y}_{n,\varepsilon,i}$ is clear by
	definition and one can show
	that $E^{x,y}_{n,\varepsilon}\supseteq \bigcap_{i\in
	\mathbb{N}}E^{x,y}_{n,\varepsilon,i} \supseteq E^{x,y}_{n,\varepsilon/2}$. To see $E^{x,y}_{n,\varepsilon/2}\subseteq \bigcap_{i\in
	\mathbb{N}}E^{x,y}_{n,\varepsilon,i}$, note that on $E^{x,y}_{n,\varepsilon/2}$,
	we can obtain an approximating path $\zeta:x\rightarrow y$, parametrise it
	by arc-length and then define $f_i(j/2^i)$ to be a rational approximant to
	$\zeta(\frac{j}{2^i}\ell(\zeta;d))$. To show
	$E^{x,y}_{n,\varepsilon}\supseteq \bigcap_{i\in
	\mathbb{N}}E^{x,y}_{n,\varepsilon,i}$, note that on the set on the R.H.S., we
	can use the compactness of $S_n$ to obtain a convergent subsequence for
	$f_i(1/2)$. By doing a diagonal argument and repeating this for all scales instead of just the first
	scale, we obtain a function $f:\left\{ \text{dyadic rationals in }
	[0,1]\right\}\rightarrow \mathbb{Q}^2\cap S_n$ such that $d(f(z_1),f(z_2))\leq
	(1+\varepsilon)d(x,y)|z_2-z_1|$ for all dyadic rationals $z_1,z_2\in [0,1]$.
	Since $S_n$ is compact and since $d$ is continuous, $f$ can be extended to a map $f:[0,1]\rightarrow
	S_n$ while preserving the Lipschitz property. The curve $\zeta(t)=f(t)$ now
	satisfies the property required in the definition of
	$E^{x,y}_{n,\varepsilon}$. Thus for each $m\in \mathbb{N}$, we have $\bigcup_{n\in
	\mathbb{N}}\bigcap_{i\in \mathbb{N}}E^{x,y}_{n,2^{-m},i}  \supseteq \bigcup_{n\in
	\mathbb{N}}E^{x,y}_{n,2^{-m}}\supseteq \bigcup_{n\in
	\mathbb{N}}\bigcap_{i\in \mathbb{N}}E^{x,y}_{n,2^{-(m+1)},i}$. Thus $\bigcap_{m\in \mathbb{N}}\bigcup_{n\in
	\mathbb{N}}E^{x,y}_{n,2^{-m}}=\bigcap_{m\in \mathbb{N}}\bigcup_{n\in
	\mathbb{N}}\bigcap_{i\in \mathbb{N}}E^{x,y}_{n,2^{-m},i}$ which is
	measurable because each set $E^{x,y}_{n,2^{-m},i}$ is measurable.

\paragraph{\textbf{The induction map $d\mapsto d(\cdot,\cdot;V)$ for a
length metric $d$}}
	It can also be shown that
	for any open $V\subseteq U$, the map $d\mapsto d(\cdot,\cdot;V)$ is
	measurable from $\left\{ d\in C(U\times U):d \text{ is a length
	metric} \right\}$ to $C(V\times V)$. The definition of the induced
	metric is an infimum over lengths of paths and for a length metric, any
	path can be made to pass through rational points while incurring a small
	error in length. Using this, we can write the map $d\mapsto
	d(\cdot,\cdot;V)$ in terms of the countable set $\left\{ d(x,y)
	\right\}_{x,y\in \mathbb{Q}^2\cap \mathbb{D}}$, and this will yield
	measurability. We will use this fact to obtain the measurability of the
	set
	\begin{displaymath}
		\left\{ d\in  C(\mathbb{D}\times \mathbb{D}): d \text{ is a
	length metric and } d\lvert_{\mathbb{C}_{(0,1)}\times
	\mathbb{C}_{(0,1)}}\neq d(\cdot,\cdot;\mathbb{C}_{(0,1)}) \right\}
	\end{displaymath}
	in \eqref{e:acevent4.1} in Section \ref{s:f-field} later.

	\paragraph{\textbf{The $D_h$ geodesic between two fixed points in
	$\mathbb{C}$}}One might also wonder if the
	geodesic between two fixed points $p, q \in \C$ is indeed a measurable
	object, i.e.\ , it can be determined measurably as a function of the
	underlying field $h.$  It can indeed be up to null sets. Though we will not elaborate much on this issue throughout the paper, we outline here a broad sketch of how one might go about proving it in the context of the whole plane GFF. 
Note that the geodesic being a closed set, it suffices to determine for any closed rational ball $B$ (with center and radius being rational elements) whether the geodesic intersects $B$ or not. Now in case of the latter, we have the following equality which is a measurable event (by the locality of induced metrics on deterministic sets), 
$$D_h(p,q)=D_h(p,q;B^c).$$ Since the set of $B$ is countable, one concludes measurability. 
Variants of this argument apply for the case where one  considers the geodesics in the induced metric on some set.

\paragraph{\textbf{$D_h(\cdot,\cdot;U)$ geodesics or $D_\mathtt{h}$ geodesics for a GFF $\mathtt{h}$ on
$U\subsetneq \mathbb{C}$}} As a consequence of locality (Proposition
\ref{b3} (2)), we have $D_h(\cdot,\cdot;U)=D_{h\lvert_{U}}(\cdot,\cdot)$
almost surely.
Since we can define $\mathtt{h}=h\lvert_{U}$, we only discuss $D_\mathtt{h}$ geodesics for a GFF $\mathtt{h}$ on
$U\subsetneq \mathbb{C}$.

As mentioned in Section \ref{ss:geodesics}, it is possible that
$D_\mathtt{h}$ geodesics do not
exist but if such a geodesic does exist between two fixed points $u,v\in U$,
then it is a.s.\ unique. Once we establish that the event
$\left\{ \text{there exists a $D_\mathtt{h}$-geodesic  from $u$ to $v$} \right\}$ is
itself measurable, we can use the argument in the previous
paragraph to determine the geodesic as a measurable
object. 

We now indicate an argument for the measurability of $\left\{ \text{there
exists a $D_\mathtt{h}$-geodesic  from $u$ to $v$} \right\}$. We let $S_n$ be a sequence of compact sets increasing
to $U$, then we show that the above event can be written as the increasing union $\bigcup_{n\in
\mathbb{N}}\left\{ D_\mathtt{h}(u,v;S_n) =D_\mathtt{h}(u,v)\right\}$ and the
proof can then be completed by using the measurability of the induction map.
To see the above equality of events, note that if a $D_\mathtt{h}$-geodesic
between $u,v$ exists, then $D_\mathtt{h}(u,v;S_n)
=D_\mathtt{h}(u,v)$ for any $n$ large enough such that $S_n$ contains the
geodesic. For the other direction, if we have $D_\mathtt{h}(u,v;S_n)
=D_\mathtt{h}(u,v)$ for some $n$ then since $D_\mathtt{h}$ is a.s.\ a length
metric, we can find paths $\zeta_i:u\rightarrow v$ such that
$\zeta_i\subseteq S_n$ and $\ell(\zeta_i;D_\mathtt{h})\rightarrow
D_\mathtt{h}(u,v)$ as $i\rightarrow \infty$. Since $S_n$ is compact, by the
Arzela-Ascoli theorem, we
can obtain a path $\zeta:u\rightarrow v$ such that $\zeta\subseteq S_n$ and
$\zeta_i\rightarrow \zeta$ uniformly. Now by the lower
semi-continuity of length as a functional on paths \cite[Proposition 2.3.4
(iv)]{Burago}, we have $\ell(\zeta;D_\mathtt{h})\leq
\liminf_{i\rightarrow \infty} \ell(\zeta_i;D_\mathtt{h})=D_\mathtt{h}(u,v)$.
However, by the triangle inequality, $\ell(\zeta;D_\mathtt{h})\geq
D_\mathtt{h}(u,v)$ and thus
$\ell(\zeta;D_\mathtt{h})=D_\mathtt{h}(u,v)$ which implies
that $\zeta\subseteq S_n\subseteq U$ is a $D_\mathtt{h}$-geodesic from $u$ to
$v$.

\paragraph{\textbf{Shortcut sets $A_\varepsilon$ from Section
\ref{s:singularity}}}
In Section
	\ref{s:singularity} we will work with the sets
	\begin{align}
	A_{\varepsilon}:=& \Big\{ d\in C(\mathbb{D}\times \mathbb{D}): d
	\text{ is a length metric and there exists some
	path }\zeta \text{ in } \mathbb{C}_{(3\varepsilon,4\varepsilon)}\nonumber\\
	&\text{ which disconnects }\mathbb{T}_{3\varepsilon}
	\text{ and
	$\mathbb{T}_{4\varepsilon}$ and satisfies }
	\ell(\zeta;d)<d(\mathbb{T}_{2\varepsilon},\mathbb{T}_{3\varepsilon}) \Big\}.
\end{align}
	Note that the above definition is in terms of lengths of paths which are
	a-priori not measurable quantities. However if
	we restrict to the class of length metrics, then we can reroute $\zeta$ to a new path $\zeta^*$ which say
	satisfies $\ell(\zeta;d)\leq \ell(\zeta;d)+\varepsilon_1$ and also passes
	through a sequence of rational points $r_1,\dots,r_n$ such that
	$|r_i-r_{i+1}|\leq \varepsilon_2$, $|x-r_1|\leq \varepsilon_2$ and
	$|r_n-y|\leq \varepsilon_2$. By this reasoning, the relevant sets
	$A_\varepsilon$ can be
	described in terms of the countable set $\left\{ d(x,y)
	\right\}_{x,y\in \mathbb{Q}^2\cap \mathbb{D}}$ and this will yield measurability
	since $A_\varepsilon$ is by definition a subset of $\left\{ d\in C(\mathbb{D}\times \mathbb{D}):d \text{ is a
	length metric.} \right\}$.

\section{Comparison between GFFs with different boundary conditions}
\label{s:prelim}

In this section we collect a few estimates regarding the effect of boundary
conditions on the law of a GFF. 
The estimates we present are tailored to our purposes, but the proofs are all standard. On a first reading, the reader can safely skip this section and refer back to it later as needed.

The following lemma (see e.g.\ \cite[Lemma 3.14]{WP20})
states the fact that a GFF plus a deterministic
finite Dirichlet energy function $f$ is mutually absolutely continuous with
respect to the GFF. 

\begin{lemma}
	\label{2.0*}
	Let $\mathtt{h}'$ be a zero boundary GFF on
	$\mathbb{C}_{>1}$ and let
	$f\in H_0^1(\mathbb{C}_{>1})$. 
If we use
	$\nu_{\mathtt{h}'}$ and $\nu_{\mathtt{h}'+f}$ to denote the laws of the
	respective fields, then we have,
	\begin{gather*}
		\frac{d\nu_{\mathtt{h}'+f}}{d\nu_{\mathtt{h}'}}=\exp\left(
		(\mathtt{h}',f)_{\nabla}-\frac{1}{2}(f,f)_\nabla \right).
	\end{gather*}
\end{lemma}

Lemma \ref{2.0*} indicates that the moments of the Radon-Nikodym derivative
between $\mathtt{h}'+f$ and $\mathtt{h}'$ should be controlled by the Dirichlet
energy $\|f\|^2_\nabla$. We will mostly work with harmonic functions $f$ and
will have bounds on the supremum $\sup_w |f(w)|$. The next lemma gives an
example on how to use bounds on the supremum to get corresponding bounds for the Dirichlet
energy; the lemma is frame to suit later applications. We will often use the
$L^\infty$ norm notation $\|\cdot\|_\infty$ to
denote the supremum of continuous functions on their domains of definition.

\begin{lemma}
	\label{2.1}
	Let $f$ be a harmonic function on $\mathbb{C}_{>1}$ satisfying that
	$\lim_{z\rightarrow\infty}f(z)$ exists and is finite.
	There exists a constant $c$ such that 
	\begin{displaymath}
		\|f\lvert_{\mathbb{C}_{>r}}\|_\nabla^2\leq c
		\|f\lvert_{\mathbb{C}_{>r/2}}\|_\infty^2
	\end{displaymath}
	for all $r>2$.
\end{lemma}

\begin{proof}
	Consider the inversion map $z\mapsto 1/z$ which we denote by
	$\eta$. To show that $f\lvert_{\mathbb{C}_{>r/2}}$ is bounded, it suffices to show that
	$(f\circ \eta)\lvert_{\mathbb{D}_s}$ is bounded for every fixed $s\in
	(0,1)$. Since $\eta$
	is harmonic and $f$ is holomorphic, $f\circ \eta$ is harmonic on
	$\mathbb{C}_{(0,1)}$. By assumption, it also satisfies
	$\lim_{z\rightarrow 0}(f\circ \eta)(z)<\infty$ and this can be used along with
	the mean value property of harmonic functions to conclude that $f\circ
	\eta$ is in fact harmonic on $\mathbb{D}_{1}$. Now harmonic functions are
	bounded as long as one stays away from the boundary of the
	domain of definition. Thus $(f\circ \eta)\lvert_{\mathbb{D}_s}$ is bounded for every fixed $s\in
	(0,1)$.
	
	 By the conformal invariance of Dirichlet energy (see e.g.\
	\cite[Section 1.9]{Ber07}), and the fact that  $\eta$
	conformally maps the
	domains $\mathbb{C}_{>r}$ and $\mathbb{C}_{(0,1/r)}$ to each other, we have that
	$\|f\lvert_{\mathbb{C}_{> r}}\|_\nabla^2=\|(f\circ
	\eta)
	\lvert_{\mathbb{C}_{<1/r}}\|_\nabla^2$. For
	any harmonic function $g$ defined on $\mathbb{C}_{<1}$, any $r>2$ and any $w\in
	\mathbb{D}_{1/r}$, we know that (see \cite[Theorem 2.10]{GT01})
	\begin{equation}
		\label{e:2.-1}
		|\nabla g(w)|\leq c_1
		(\text{dist}(\mathbb{T}_{1/r},\mathbb{T}_{2/r}))^{-1}
		\sup_{v\in \mathbb{D}_{2/r}}|g(v)-g(0)|=c_1r\sup_{v\in \mathbb{D}_{2/r}}|g(v)-g(0)|
	\end{equation}
	for all $r>0$ and some constant $c_1>0$. From the reasoning in the first
	paragraph, we have that $f\circ \eta$ is a harmonic
	function defined on $\mathbb{C}_{<1}$. Thus by using \eqref{e:2.-1}
	with the choice $g=f\circ \eta$, we obtain
	\begin{equation}
		\label{e:2.0}
		\|(f\circ
	\eta)
	\lvert_{\mathbb{C}_{<1/r}}\|_\nabla^2=\int_{\mathbb{D}_{1/r}}|\nabla
	(f\circ \eta)(w)|^2 dw \leq \frac{\pi}{4}r^{-2} \left(c_1 r
	\|f\lvert_{\mathbb{C}_{>r/2}}\|_\infty\right)^{2}\leq
	c_2\|f\lvert_{\mathbb{C}_{>r/2}}\|_\infty^2,
	\end{equation}
	and this completes the proof.
\end{proof}

The inequality \eqref{e:2.-1} from the above lemma can be used to
show that to reduce the effect of the boundary
condition of $\mathtt{h}$ to
$s^{-1} {\sf{Rough}_\varepsilon(\mathtt{h})}$, it is sufficient to go
$O(s)$ distance away from $0$.
\begin{proposition}
	\label{3.3}
	For any fixed $\varepsilon>0$, there exists a constant
	$\kappa=\kappa(\varepsilon)>0$ such that for a whole plane
	GFF marginal $\mathtt{h}$ and any $s\geq 1+2\varepsilon$, the harmonic extension
	$\widehat{\mathbf{Hr}}_{\mathtt{h}}$ satisfies
	\begin{displaymath}
		\sup_{w:|w|\geq
		 s}|\widehat{\mathbf{Hr}}_{\mathtt{h}}(w)|\leq
		 \frac{\kappa}{s}{\sf{Rough}}_\varepsilon(\mathtt{h}).
	\end{displaymath}
\end{proposition}
\begin{proof}
	We use the same notation as in the proof of Lemma \ref{2.1} and set
	$g=\widehat{\mathbf{Hr}}_{\mathtt{h}}\lvert_{\mathbb{C}_{>1+\varepsilon}}\circ
	\eta$. Similar to \eqref{e:2.-1}, we can write for all $w\in
	\mathbb{D}_{s}$, 
	\begin{align*}
			|\nabla
	g(1/w)|&\leq c_1
	(\text{dist}(\mathbb{T}_{1/s},\mathbb{T}_{1/(1+\varepsilon)}))^{-1}
	\sup_{v\in
	\mathbb{D}_{1/(1+\varepsilon)}}|g(v)-g(0)|\nonumber\\
	&=c_1\left(\frac{1}{1+\varepsilon}-\frac{1}{s}\right)^{-1}\mathsf{Rough}_\varepsilon(\mathtt{h})\leq
	c_2 \mathsf{Rough}_\varepsilon(\mathtt{h})
	\end{align*}
	 for some constant $c_2$ depending on $\varepsilon$.
	Here, we used $g(0)=0$ to obtain the equality and used $s\geq
	1+2\varepsilon$ to obtain the last inequality. The proof
	can now be completed by joining $1/w$, $0$ by a straight line and
	using the bound on $\nabla g$ on this line to obtain the needed bound on $\sup_{|w|\geq s}|g(1/w)|=\sup_{w:|w|\geq
		 s}|\widehat{\mathbf{Hr}}_{\mathtt{h}}(w)|$.
\end{proof}

In the next lemma (c.f.\ \cite[Lemma 4.1]{MQ20}), we use Lemmas \ref{2.0*} and \ref{2.1} to bound the
moments of the Radon-Nikodym derivative between a GFF and its sum with a
deterministic harmonic function, both restricted to a subdomain away from the
boundary. The lemma can be stated more generally but we frame it to suit subsequent applications.

\begin{lemma}
	\label{2.2}
	Let $\mathtt{h}'$ be a zero boundary GFF on the domain $\mathbb{C}_{>
	1}$ and let $g$ be a deterministic harmonic function on
	$\mathbb{C}_{>1}$ which satisfies that $\lim_{z\rightarrow \infty}g(z)$
	exists and is finite. 
	On the
	domain $\mathbb{C}_{>1}$, consider the generalized function
	$\mathtt{h}'_{r}= \mathtt{h}'\mydot \psi_{r}$ and the harmonic
	function $g_{r}=g\mydot
	\psi_r$. Then we have that the fields $\mathtt{h}'_{r}$ and
	$\mathtt{h}'_{r}+g_{r}$, defined on $\mathbb{C}_{>
	1}$, are absolutely continuous with
	respect to each other and the respective Radon-Nikodym derivatives
	satisfy, for all $p\in \mathbb{R}\setminus (0,1)$, for some absolute
	constant $c_1>0$ and for all $r>4$,
	\begin{gather*}
		\mathbb{E}_{\nu_{\mathtt{h}'_{r}}}\left[
		\left(\frac{d\nu_{(\mathtt{h}'_{r}+g_{r})}}{d\nu_{\mathtt{h}'_{r}}}\right)^p
		\right]\leq \exp\left( c_1(p^2-p) \|g\lvert_{\mathbb{C}_{>r/4}}\|_\infty^2 \right).
	\end{gather*}
The same bounds hold for the Radon-Nikodym derivative between the laws of
$\mathtt{h}'_r+g_r-\mathbf{Av}(\mathtt{h}'+g,\mathbb{T}_r)$ and
$\mathtt{h}'_r-\mathbf{Av}(\mathtt{h}',\mathbb{T}_r)$.
\end{lemma}

\begin{proof}
	With a constant $c'$ being chosen large enough, fix a smooth function $\phi$  defined on $\mathbb{C}_{>1}$ with $0\leq \phi\leq 1$
	such that $\phi\lvert_{\mathbb{C}_{\geq r}}=1$,
	$\phi\lvert_{\mathbb{C}_{(1,r/2]}}=0$ and $|\nabla \phi|\leq
	c'r^{-1}$
	everywhere. Our choice of $\phi$ ensures that for any deterministic function $g$ defined on $\C_{>
	1}$, we have $(\phi
	g)\mydot \psi_{r}=g_{r}$, where both the sides are defined on
	$\mathbb{C}_{>1}$.
	By Lemma \ref{2.0*}, for any fixed function $g\in
	H_0^1(\mathbb{C}_{>1})$, we have
	\begin{equation}
		\label{e:2.1}
		\frac{d\nu_{\mathtt{h}'+\phi g}}{d\nu_{\mathtt{h}'}}=\exp\left(
			(\mathtt{h}',\phi g)_{\nabla}-\frac{1}{2}\|\phi g\|_{\nabla}^2
			\right).
	\end{equation}

	We will use the above with $g$ being a harmonic function on
	$\mathbb{C}_{>1}$ satisfying $\lim_{z\rightarrow \infty}g(z)<\infty$, as in
	the statement of the lemma. To be able to use \eqref{e:2.1}, we need to
	bound $\|\phi g\|_{\nabla}^2$ for such $g$, and this is done as follows.
	\begin{align}
		\label{e:2.2.1}
		\|\phi g\|_{\nabla}^2&=\|(\phi
		g)\lvert_{\mathbb{C}_{(r/2,r)}}\|_{\nabla}^2+\|(\phi
		g)\lvert_{\mathbb{C}_{>r}}\|_{\nabla}^2\nonumber\\
		& \leq 2\int_{\mathbb{C}_{(r/2,r)}} \left|(g\nabla \phi
		)\right|^2+ 2\int_{\mathbb{C}_{(r/2,r)}} \left|(\phi \nabla g
		)\right|^2+\|
		g\lvert_{\mathbb{C}_{>r}}\|_{\nabla}^2\nonumber\\
		&\leq 2\int_{\mathbb{C}_{(r/2,r)}} \left|(g\nabla \phi
		)\right|^2+2\|
		g\lvert_{\mathbb{C}_{(r/2,r)}}\|^2_{\nabla}+\|
		g\lvert_{\mathbb{C}_{>r}}\|_{\nabla}^2\nonumber\\
		&\leq 2\int_{\mathbb{C}_{(r/2,r)}} \left|(g\nabla \phi
		)\right|^2+3\|
		g\lvert_{\mathbb{C}_{>r/2}}\|^2_{\nabla}\nonumber\\
		&\leq  c_3\int_{\mathbb{C}_{(r/2,r)}}\|g\lvert_{\mathbb{C}_{>r/4}}\|_\infty^2 r^{-2} + c_4
		\|g\lvert_{\mathbb{C}_{>r/4}}\|_\infty^2\nonumber\\
		&\leq  c_5\|g\lvert_{\mathbb{C}_{>r/4}}\|_\infty^2
	\end{align}
	for some $c_5>0$. The last term in the first equality is obtained by
	using that $\phi=1$ on the region $\mathbb{C}_{>r}$. The second term in the second inequality is
	obtained by using that $\phi\leq 1$. The fourth inequality is obtained by using that $|\nabla \phi|\leq
	c'r^{-1}$ everywhere along with Lemma \ref{2.1}.

Now, note that $(\mathtt{h}',\phi
	g)_\nabla$ has the law of a centered Gaussian random
	variable with variance $\|\phi g\|_{\nabla}^2$. Since
	$\mathbb{E}e^{tZ}=e^{t^2/2}$ for $Z\sim N(0,1)$, we have that 
	\begin{equation}
		\label{e:2.3}
			\mathbb{E}_{\nu_{\mathtt{h}'}}\left[
			\left(\frac{d\nu_{\mathtt{h}'+\phi g}}{d\nu_{\mathtt{h}'}}\right)^p
			\right]=\exp\left( (p^2-p) \|\phi g\|^2_{\nabla}/2
			\right)\leq \exp\left( c(p^2-p)\|g\lvert_{\mathbb{C}_{>r/4}}\|_\infty^2 \right),
	\end{equation}
	where the last inequality follows by \eqref{e:2.2.1}. 
	
	We now note a general property about pushforwards of measures and
	Radon-Nikodym derivatives that will be useful. If we have two measures $\nu_1,\nu_2$ on a measurable space
$(\Omega,\mathscr{A})$
which are mutually absolutely continuous, then for any deterministic measurable function $\pi:\Omega\rightarrow
\Omega$, the pushforward measures $\pi^*\nu_1,\pi^*\nu_2$ are mutually
absolutely continuous as well. With
$\sigma(\pi)\subseteq \mathscr{A}$ denoting the sigma algebra generated by the map $\pi$, It can be
checked that if $X\sim \nu_1$, then
$\frac{d(\pi^*\nu_{2})}{d(\pi^*\nu_{1})}(\pi(X))=\mathbb{E}\left[
\frac{d\nu_{2}}{d\nu_{1}}(X)\big\vert \sigma(\pi)
\right]$. By using this along with conditional Jensen's inequality, for any
convex function $\theta:\mathbb{R}_+\mapsto \mathbb{R}_+$, we have
\begin{equation}
	\label{e:abscts}
	\mathbb{E}\left[
	\theta\left(\frac{d(\pi^*\nu_{2})}{d(\pi^*\nu_{1})}(\pi(X))\right)
	\right]\leq \mathbb{E}\left[ \theta
	\left(\frac{d\nu_{2}}{d\nu_{1}}(X)\right)\right].	
\end{equation}
We use the above with $(\Omega,\mathscr{A})$ being
$\mathcal{D}'(\mathbb{C}_{>1})$ with the natural $\sigma$-algebra and $\pi$ taking a field $\widetilde{\mathtt{h}}$ on
$\mathbb{C}_{>1}$ to
$\widetilde{\mathtt{h}}_r$. The
measures $\nu_1$ and $\nu_2$ are defined to be the laws of $\mathtt{h}'$
and $\mathtt{h}'+\phi g$ respectively, and we note that
$(\mathtt{h}'+\phi g)_r=\mathtt{h}'_r+g_r$ since $g$ and $\phi g$ agree in the
region $\mathbb{C}_{\geq r}$. We
define $\theta(x)=x^p$ which is convex on $\mathbb{R}_+$ for $p\notin (0,1)$. Thus we obtain
\begin{equation}
		\label{e:2.4}
		\mathbb{E}_{\nu_{\mathtt{h}'_{r}}}\left[
	\left(\frac{d\nu_{(\mathtt{h}'_{r}+g_{r})}}{d\nu_{\mathtt{h}'_{r}}}(\mathtt{h}_r')\right)^p
		\right]		
			\leq 
			\mathbb{E}_{\nu_{\mathtt{h}'}}\left[
			\left(\frac{d\nu_{\mathtt{h}'+\phi
			g}}{d\nu_{\mathtt{h}'}}(\mathtt{h}')\right)^p
			\right]\leq \exp\left(
			c(p^2-p)\|g\lvert_{\mathbb{C}_{>r/4}}\|_\infty^2
			\right),
		\end{equation}
		where the last inequality comes from \eqref{e:2.3}.

To obtain the corresponding statement regarding the Radon-Nikodym derivatives
between between the laws of
$\mathtt{h}'_r+g_r-\mathbf{Av}(\mathtt{h}'+g,\mathbb{T}_r)$ and
$\mathtt{h}'_r-\mathbf{Av}(\mathtt{h}',\mathbb{T}_r)$, we use
\eqref{e:abscts} with the same definitions of
$(\Omega,\mathscr{A}),\nu_1,\nu_2,\theta$ but define the map $\pi$ to take a field
$\widetilde{\mathtt{h}}$ on $\mathbb{C}_{>1}$ and yield
$\widetilde{\mathtt{h}}_r-\mathbf{Av}(\widetilde{\mathtt{h}},\mathbb{T}_r)$.
This completes the proof.

\end{proof}

\begin{lemma}
	\label{3.9}
	Let $\mathtt{h}'$ be a zero boundary GFF on $\mathbb{C}_{>1}$
	and let $\mathtt{h}$ be a whole plane GFF marginal. If we
	denote the law of  $\widehat{\mathtt{h}'}_r\mydot
	\psi_{r}$ by $\nu_r$ and the law of $\mathtt{h}$ by $\nu$, then we have
	that for any fixed $p\in \mathbb{R}\setminus (0,1)$, 
	there exists a positive constant $c$ such that for all $r$ large enough,
	\begin{gather*}
		\mathbb{E}_{\nu_r}\left[ \left|\frac{d\nu}{d\nu_r}\right|^p
		\right]\leq 1+ cr^{-2}.
	\end{gather*}
\end{lemma}

\begin{proof}
By the Markov property, we have the independent decomposition
	\begin{displaymath}
		\widehat{\mathtt{h}}_r\mydot
		\psi_r=\widehat{\mathtt{h}'}_r\mydot \psi_r
		+\widehat{\mathbf{Hr}}_{\mathtt{h}}\mydot \psi_{r}
	\end{displaymath}
Note that the left hand side is distributed as a whole plane GFF marginal for
any $r>0$ by the scale invariance property of the whole plane GFF as in
\eqref{e:b4.1}. Thus we note that $\nu=\nu_{\widehat{\mathtt{h}}_r\mydot
\psi_r}$ and $\nu_r=\nu_{\widehat{\mathtt{h}'}_r\mydot \psi_r}$.
	 
	 An application of Proposition \ref{1.15} (with $\varepsilon=1$) and
	Proposition \ref{3.3}
	implies that there exist constants $C,c$ such that for any $t>0$, with
	probability at least $1-Ce^{-ct^2}$,
	$\|\widehat{\mathbf{Hr}}_\mathtt{h}\lvert_{\mathbb{C}_{>r}}\|_\infty\leq
	\kappa t/r$ for all $r>1+2\varepsilon=3$.
	In other words, 
there
	exists a positive constant $C_2$ and an abstract coupling of
	an exponential random variable $X$ with $\mathtt{h}$
such that for any $r>3$,
	\begin{equation}
		 \label{ee:5.1}
		 \|\widehat{\mathbf{Hr}}_\mathtt{h}\lvert_{\mathbb{C}_{>r}}\|_\infty\leq
	\kappa r^{-1} \sqrt{X+C_2}
	 \end{equation}
	holds almost surely.
	 
	 Define $\lambda$ such that we have $X\sim
	\mathtt{Exp}(\lambda)$. We define the function
	$g=\widehat{\mathbf{Hr}}_\mathtt{h}$ and use $\nu\vert g$ to denote the
	law of $\widehat{\mathtt{h}}_r\mydot
		\psi_r$ conditional on $g$; we recall that $g$ is independent of
		$\mathtt{h}'$. Note that $\lim_{z\rightarrow
		\infty}g(z)=0$ by Lemma \ref{markovv} and
		this will provide the finite limit condition in the upcoming
		application of Lemma \ref{2.2}. By using the conditional Jensen's
		inequality with the convex function $x\mapsto x^p$ on
		$\mathbb{R}^+$ and applying Lemma \ref{2.2} after conditioning
		on $g$, we have
		\begin{equation*}
		\mathbb{E}_{\nu_r}\left[
		\left(\frac{d\nu}{d\nu_r}\right)^p\right]\leq
		\mathbb{E}_{g}\left[\mathbb{E}_{\widehat{\mathtt{h}'}_r\mydot \psi_r}\left[
		\left(\frac{d(\nu\lvert g)}{d\nu_{\widehat{\mathtt{h}'}\mydot
		\psi_r}}\right)^p
		\right]\right]\leq  \mathbb{E}[e^{c_1 (r/4)^{-2} (X+C_1)}]=
		\frac{\lambda e^{c_1C_1r^{-2}}}{\lambda-c_2r^{-2}}\leq 1+ c_3 r^{-2}.
	\end{equation*}
	
The second inequality uses Lemma \ref{2.2} along with an
application of \eqref{ee:5.1} with $r$ replaced by $r/4$. The equality
after uses that $X$ is an
	exponential random variable and that $r$ is sufficiently large (hence the moment generating function is finite), while the last inequality uses again that $r$
	is taken
	large.
\end{proof}

\section{Stationary bi-infinite sequence of fields with exponential decay of
correlations}
\label{s:convergence}
As indicated in Section \ref{s:iop}, we will explore the
field $h$ progressively from $\mathbf{0}$ to $\infty$ by looking at the field restricted to disks of exponentially
growing radii. We choose the sequence of radii to grow as $K^i$ for some constant $K>1$ which will remain as a
parameter throughout the paper.
For notational convenience, we define 
\begin{equation}
	\label{e:S0}
	S_i=K^i
\end{equation}
for all $i\in \mathbb{Z}$.
Recall further that for any $r>0$, $\widehat{h}_r$ refers to the field
$h\lvert_{\mathbb{C}_{>r}}-\mathbf{Av}(h,\mathbb{T}_r)$. 
We now define the random generalized function $\mathcal{H}_i$ on
$\mathbb{C}_{>1}$, which we call the recentered field observed at scale
$i$, as follows.
\begin{equation}
	\label{ee:0}
	\mathcal{H}_i=\widehat{h}_{S_i}\mydot \psi_{S_i},
\end{equation}
In the above, we used the scaling maps $\psi_r$ which were defined at the end
of Section \ref{ss:markov}. Before moving on, we first  point out the later
lemmas that will fix the choice of the parameter $K$ in \eqref{e:S0}.

\subsubsection*{{\textbf{Choice of the parameter $K$}}}
\label{sss:K}
The parameter $K$ is
fixed large enough to satisfy the conclusions of Lemma \ref{rncond} and
Lemma \ref{3.10.01}. The former will lead to a correlation decay statement for
the sequence $\left\{ \mathcal{H}_i \right\}_{i\in \mathbb{Z}}$, while the
latter will imply that the coalescence events ${\sf{Coal}}$
from Section \ref{s:iop} occur at infinitely many scales.

\bigskip

The goal of this section is to show that the bi-infinite sequence $\mathcal{H}=\left\{ \mathcal{H}_i
\right\}_{i\in \mathbb{Z}}$ is stationary and has exponential decay of
correlations. We first quickly note the stationarity of $\mathcal{H}$.

\begin{lemma}
	\label{3.7.0.1}
	For each $k\in \mathbb{N}{\cup\left\{ 0 \right\}}$, the law of
	$(\mathcal{H}_i,\dots,\mathcal{H}_{i+k})$ is the same for all $i\in
	\mathbb{Z}$.
\end{lemma}

\begin{proof}
	By the stationarity of the whole plane GFF as in \eqref{e:b4.1},
	$\mathcal{H}_i\stackrel{d}{=}\mathcal{H}_0$ and now note that
	$(\mathcal{H}_i,\mathcal{H}_{i+1},\dots,\mathcal{H}_{i+k})=\left(
	\mathcal{H}_i,(\widehat{\mathcal{H}_i})_{S_1}\mydot \psi_{S_1}\cdots,
	(\widehat{\mathcal{H}_{i}})_{S_k}\mydot \psi_{S_k} \right)$ is measurably
	determined by just $\mathcal{H}_i$.
	Thus 
	\begin{equation}
		\label{ee:1.1}
		(\mathcal{H}_i,\mathcal{H}_{i+1},\dots,\mathcal{H}_{i+k})\stackrel{d}{=}(\mathcal{H}_0,\mathcal{H}_{1},\dots,\mathcal{H}_{k})
	\end{equation}
	and this completes the proof.
\end{proof}

The focus of the next subsection will be to obtain correlation decay
statements for the stationary sequence $\mathcal{H}$. Before doing so, we
first define some discrete filtrations which will be useful for us
throughout the paper. Recall the discussion after Lemma \ref{markovv} on the
filtration $\mathscr{F}$ for $h$ and the corresponding filtration $\mathscr{F}(\mathtt{h})$ for a whole
plane GFF marginal $\mathtt{h}$. We now define
the filtration $\mathscr{G}(\mathtt{h})$ by
$\mathscr{G}_i(\mathtt{h})=\mathscr{F}_{S_i}(\mathtt{h})$ for
all $i\in \Z$. Similarly, for the whole plane GFF $h$ itself, we define the
filtration $\mathscr{G}$ by $\mathscr{G}_i=\mathscr{F}_{S_i}$ for all $i\in
\mathbb{Z}$.
\subsection{Decay of correlation and convergence results}
\label{ss:corr}

Having defined $\mathcal{H}$, we now proceed towards obtaining the desired decay of correlation estimates which are recorded as Lemmas \ref{3.10} and \ref{'3.163} for certain `local' and `non-local' observables respectively. 

We start with a simple lemma
regarding the measurability of the roughness as defined in \eqref{e:b8}.
\begin{lemma}
	\label{roughm}
	For any fixed $\varepsilon>0$ and all $i\in \mathbb{Z}$, we have that
	${\sf{Rough}}_\varepsilon(\mathcal{H}_i)\in \mathscr{G}_i$.
\end{lemma}
\begin{proof}
By definition, we have that
${\sf{Rough}}_\varepsilon(\mathcal{H}_i)=\sup_{w\in \mathbb{T}_{(1+\varepsilon
)}}|\widehat{\mathbf{Hr}}_{\mathcal{H}_i}(w)|=\sup_{w\in
\mathbb{Q}^2\cap \mathbb{T}_{(1+\varepsilon
)}}|\widehat{\mathbf{Hr}}_{\mathcal{H}_i}(w)|$, and we note that the latter
is measurable with respect to
$\mathscr{G}_i=\mathscr{F}_{\overline{\mathbb{D}}_{S_i}}(h)$ by an
application of Lemma \ref{markov*}. Note that we used 
$\widehat{\mathbf{Hr}}_{\mathcal{H}_i}={\widehat{\mathbf{Hr}}_{h,S_i}}\mydot
\psi_{S_i}$.
\end{proof}

The main ingredient for the proof of the above-mentioned Lemma \ref{3.10} is
a decorrelation statement at the level of fields which states that for all
$j>i$, we have that
with probability at least $1-Ce^{-c(j-i)}$, the field $\mathcal{H}_j$ given
$\mathscr{G}_i$ is close to the whole plane GFF marginal $\mathtt{h}$ in
the sense that the corresponding Radon Nikodym derivative
$\frac{d\nu_{(\mathcal{H}_j\vert\mathscr{G}_i)}}{d\nu}$ is close to $1$ in
${L}^2$.

\begin{lemma}
	\label{rncond}
	Let $\nu$ denote the law of $\mathtt{h}$, a whole plane GFF marginal on
	$\mathbb{C}_{>1}$. There exist positive constants $c,C,c', K_0$ such that if the
	parameter $K$ satisfies $K\geq K_0$ then for any $0\leq i< j$
	there exists an event $E_{i,j}\in \mathscr{G}_i$ satisfying
	\begin{displaymath}
		\mathbb{P}\left( E_{i,j} \right)\geq 1-C'e^{-c'(j-i)}
	\end{displaymath}
	such that on $E_{i,j}$, we have that
	$\nu_{(\mathcal{H}_j\vert\mathscr{G}_i)}$, the
	conditional law of $\mathcal{H}_{j}$ given
	$\mathscr{G}_i$ satisfies 
	\begin{equation}
		\label{e:3.5.6}
		\mathbb{E}_{\nu} \left[
		\left|\frac{d\nu_{(\mathcal{H}_j\vert\mathscr{G}_i)}}{d\nu} -1\right|^2
		\right] \leq  Ce^{-c(j-i)}.
	\end{equation}
	Note that the $\frac{d\nu_{(\mathcal{H}_j\vert\mathscr{G}_i)}}{d\nu}$ in
	the above expression is a random variable measurable with
	respect to $\mathscr{G}_i$.
\end{lemma}
\begin{proof}
	We will often use the roughness as defined in \eqref{e:b8} and thus
	begin by fixing $\varepsilon=1$ for the entirety of the proof. By Proposition \ref{1.15} and the stationarity of the whole plane GFF as in
	\eqref{e:b4.1}, there exist constants $c_1,C_1$ such that we have 
	\begin{equation}
		\label{ee:1}
		\mathbb{P}\left({\sf{Rough}}_\varepsilon(\mathcal{H}_i)\geq t \right)\leq C_1e^{-c_1t^2}
	\end{equation}
	for all $i$. Define the event $E_{i,j}$ by 
	\begin{equation}
		\label{e:Eij}
		E_{i,j}=\left\{
		{\sf{Rough}}_\varepsilon(\mathcal{H}_i)\leq
		K^{(j-i)/2} \right\}
	\end{equation}
	and note that by an application of Lemma \ref{roughm}, $E_{i,j}\in
	\mathscr{G}_i$ and as a consequence of \eqref{ee:1}, 
	\begin{equation}
		\label{ee:2}
		\mathbb{P}\left( E_{i,j}\right)\geq 1- C_1e^{-c_1K^{j-i}}.
	\end{equation}
	Now the Markov property (Lemma \ref{markov*}, Lemma
	\ref{markovv}), yields the following independent
	decomposition.
	\begin{equation}
		\label{ee:3}
		\mathcal{H}_j= \widehat{\mathbf{Hr}}_{\mathcal{H}_i}\mydot
		\psi_{K^{j-i}}+\widehat{\mathtt{h}'}_{K^{j-i}}\mydot \psi_{K^{j-i}}.
	\end{equation}
	Here, $\widehat{\mathbf{Hr}}_{\mathcal{H}_i}\mydot
	\psi_{K^{j-i}}$ is harmonic on $\mathbb{C}_{>1}$ and is
	measurable with respect to $\mathscr{G}_i$. On the other
	hand, $\mathtt{h}'$ is a zero boundary GFF on $\mathbb{C}_{>1}$ and is
	independent of $\mathscr{G}_i$.

	By \eqref{ee:2} and Proposition \ref{3.3}, on an event
	$E_{i,j}$ satisfying $\mathbb{P}\left( E_{i,j}
	\right)\geq 1-C_1e^{-c_1K^{j-i}}$, for some constant $\kappa>0$,
	\begin{equation}
		\label{ee:4}
		\|\widehat{\mathbf{Hr}}_{\mathcal{H}_i}\lvert_{\mathbb{C}_{>r}}\|_\infty\leq
		\kappa K^{(j-i)/2}/r
	\end{equation}
	for all $r>1+2\varepsilon=3$. The above with $r=K^{j-i}/4$  implies that on the event
	$E_{i,j}$, we have
\begin{equation}
	\label{ee:5}
	\|\widehat{\mathbf{Hr}}_{\mathcal{H}_i}\lvert_{\mathbb{C}_{>K^{j-i}/4}}\|_\infty \leq 4\kappa K^{-(j-i)/2}.
\end{equation}

	Let $\nu_0$ denote the law of $\widehat{\mathtt{h}'}_{K^{j-i}}\mydot
	\psi_{K^{j-i}}$. On the event $E_{i,j}$, we have
	 \begin{align}
		\label{e:2.9}
		\mathbb{E}_{\nu} \left[
		\left(\frac{d\nu_{(\mathcal{H}_j\vert\mathscr{G}_i)}}{d\nu}\right)^2
		\right]=\mathbb{E}_{\nu_0} \left[
		\left(\frac{d\nu_{(\mathcal{H}_j\vert\mathscr{G}_i)}}{d\nu_0}\right)^2\left(\frac{d\nu}{d\nu_0}\right)^{-2} \left(\frac{d\nu}{d\nu_0}\right)
		\right] &\leq \mathbb{E}_{\nu_0} \left[
		\left(\frac{d\nu_{(\mathcal{H}_j\vert\mathscr{G}_i)}}{d\nu_0}\right)^4\right]^{1/2}
		\mathbb{E}_{\nu_0}\left[ \left(\frac{d\nu}{d\nu_0}\right)^{-2}
		\right]^{1/2}\nonumber\\
		&\leq\left[ e^{c_1K^{-(j-i)} } \right]^{1/2}\left[
		1+C_2K^{-2(j-i)}
		\right]^{1/2}\nonumber\\
		&\leq 1+C_3K^{-(j-i)}\leq 1+Ce^{-c(j-i)}.
	\end{align}

	The first line consists of simple manipulations of Radon-Nikodym densities
	along with an application of the Cauchy-Schwarz inequality. The first term
	in the second line is obtained by using Lemma \ref{2.2} with $p=4$,
	$r=K^{j-i}$ and
	$g=\widehat{\mathbf{Hr}}_\mathtt{h}$; note that we use that
	$\|g\lvert_{\mathbb{C}_{>r/4}}\|_\infty\leq 4\kappa K^{-(j-i)/2}$ on the event $E_{i,j}$ and
	this is because of \eqref{ee:5}. The second term in the second line uses
	that $K$ is taken to be sufficiently large and is
	obtained by an application of Lemma \ref{3.9} with $p=-2$ and $r=K^{j-i}$.
	We now notice that \eqref{e:2.9} implies
	\begin{displaymath}
	\mathbb{E}_{\nu} \left[
		\left|\frac{d\nu_{(\mathcal{H}_j\vert\mathscr{G}_i)}}{d\nu}
		-1\right|^2\right]=\mathbb{E}_{\nu} \left[ \left(\frac{d\nu_{(\mathcal{H}_j\vert\mathscr{G}_i)}}{d\nu}\right)^2
		\right]-1\leq C'e^{-c'(j-i)}
	\end{displaymath}
	and this completes the proof.
\end{proof}

We now obtain the decay of correlations for `local'
observables. Local roughly means that the value of the observable at scale $i$
is a function purely of the field $\mathcal{H}_i$ and also does not depend
on the information contained in the fields $\mathcal{H}_j$ for $j$ much
larger than $i$; the local observable is
denoted by $X$ in the following lemma.
\begin{lemma}
	\label{3.10}
	Let $\tau:\mathcal{D}'(\mathbb{C}_{>1})\rightarrow \mathbb{R}_+$ be a measurable
	function such that for the whole plane GFF marginal $\mathtt{h}$ on $\mathbb{C}_{>1}$,
	$\tau(\mathtt{h})$ is an almost surely finite stopping time for the filtration
	$\mathscr{G}(\mathtt{h})$. Also assume that $\tau$ satisfies the property
	that there exist constants $c,C$ such that 
	\begin{equation}
		\label{e:ST}
		\mathbb{P}\left( \tau(\mathtt{h})-m\geq t \right)\leq Ce^{-ct}
	\end{equation}
	for all $t\geq 0$ and some $m\in \mathbb{N}$.
	Let $X:\mathcal{D}'(\mathbb{C}_{>1})\rightarrow \mathbb{R}$ be a measurable
	function such that for {a} whole plane GFF marginal $\mathtt{h}$ on
	$\mathbb{C}_{>1}$, the random variable $X(\mathtt{h})$ is measurable with respect to
	$\mathscr{G}_{\tau(\mathtt{h})}(\mathtt{h})$. We also require that the
	random variable
	$X(\mathtt{h})$ have finite fourth moment, that is,
	\begin{equation}
		\label{e:moment}
		\mathbb{E}X(\mathtt{h})^4<\infty.
	\end{equation}

	Then the
	stationary sequence $\left\{ X(\mathcal{H}_i) \right\}_{i\in \mathbb{Z}}$
	has exponential decay of correlations, i.e., there exist some
	constants $c',C'$ not depending on $m$ such that for all $i,j\in \mathbb{Z}$, we have
	\begin{equation}
		\label{e:CovX}
	\left|\mathrm{Cov}\left(X(\mathcal{H}_i),X(\mathcal{H}_j)\right)\right|\leq
		C'e^{-c'(|i-j|-2m)}.
	\end{equation}

\end{lemma}

\begin{proof}

	Without loss of generality, we can assume that $(j-i)/2$ is an
	integer. We use Lemma \ref{rncond} with the whole plane GFF marginal
	$\mathcal{H}_i$ and $m+(j-i)/2,(j-i)$ instead of $i,j$ 
	 to
	obtain that
	there is an event $\mathcal{E}_{i,j}=E_{m+(j-i)/2,j-i}\in
	\mathscr{G}_{m+(j-i)/2}(\mathcal{H}_i)$
	satisfying 
	\begin{equation}
		\label{ee:3.5.6}
		\mathbb{P}\left( \mathcal{E}_{i,j}^c \right)\leq
		C_1e^{-c_1(j-i-2m)}
	\end{equation}
	such that on $\mathcal{E}_{i,j}$, the
	conditional law $\nu_{i,j}$ of  $\mathcal{H}_{j}$ given
	$\mathscr{G}_{m+(j-i)/2}(\mathcal{H}_i)$ satisfies for some positive constants
	$C_2,c_2$,

	\begin{equation*}
		\mathbb{E}_{\nu} \left[ \left|\frac{d\nu_{i,j}}{d\nu} -1\right|^2
		\right]\leq C_2e^{-c_2(j-i-2m)}.
	\end{equation*}

	Using \eqref{e:ST}, we can say that there
	exist constants $c_1'$, $C_1'$ such that for all $0\leq i\leq j$, we
	have
	\begin{equation}
		\label{e:3.5.2*}
		\mathbb{P}\left( \tau(\mathcal{H}_i) \leq m+(j-i)/2 \right)\geq
		1-C_1'e^{-c_1' (j-i)}.
	\end{equation}
	Now note that by applying the Cauchy-Schwarz inequality twice and
	using the fact that
	$\mathcal{H}_i\stackrel{d}{=}\mathtt{h}$ for all $i$, we have the following.
	\begin{align}
		&\mathrm{Cov}\left(X(\mathcal{H}_i),X(\mathcal{H}_j)\right)\nonumber\\
		&\leq
		\text{Cov}\left(X(\mathcal{H}_i)\mathbbm{1}_{\tau(\mathcal{H}_i)\leq
		m+(j-i)/2}\mathbbm{1}_{\mathcal{E}_{i,j}},X(\mathcal{H}_j)\right)+\text{Cov}\left(X(\mathcal{H}_i)\mathbbm{1}_{\mathcal{E}_{i,j}^c\cup\left\{\tau(\mathcal{H}_i)>
		m+(j-i)/2\right\}},X(\mathcal{H}_j)\right)\nonumber\\
		& \leq \text{Cov}\left(X(\mathcal{H}_i)\mathbbm{1}_{\tau(\mathcal{H}_i)\leq
		m+(j-i)/2}\mathbbm{1}_{\mathcal{E}_{i,j}},X(\mathcal{H}_j)\right)\nonumber\\
		&\qquad+\left(\mathbb{E}[X(\mathcal{H}_i)^4]\right)^{1/4}\mathbb{P}\left(
		\mathcal{E}_{i,j}^c\cup\left\{\tau(\mathcal{H}_i)>
		m+(j-i)/2\right\} \right)^{1/4}\left(\mathbb{E}[X(\mathcal{H}_j)^2]\right)^{1/2}\nonumber\\
		& \leq \text{Cov}\left(X(\mathcal{H}_i)\mathbbm{1}_{\tau(\mathcal{H}_i)\leq
		m+(j-i)/2}\mathbbm{1}_{\mathcal{E}_{i,j}},X(\mathcal{H}_j)\right)+\mathbb{E}[X(\mathtt{h})^4]^{3/4}\left(\mathbb{P}\left(
		\mathcal{E}_{i,j}^c \right)+\mathbb{P}\left(
		\tau(\mathcal{H}_i)>m+(j-i)/2 \right)\right)^{1/4}.
		\label{e:3.5.21*}
	\end{align}

	Thus by using \eqref{e:3.5.2*}, \eqref{ee:3.5.6} along with the assumption \eqref{e:moment}, to prove
	\eqref{e:CovX} we only need
	to show that there exist constants $c,C$ such that
	\begin{equation}
		\label{e:3.5.3}
		\text{Cov}\left(X(\mathcal{H}_i)\mathbbm{1}_{\tau(\mathcal{H}_i)\leq
		m+(j-i)/2}\mathbbm{1}_{\mathcal{E}_{i,j}},X(\mathcal{H}_j)\right)\leq
		Ce^{-c(j-i-2m)}.
	\end{equation}

	Now by using that $\mathcal{E}_{i,j}\in
	\mathscr{G}_{m+(j-i)/2}(\mathcal{H}_i)$ along with the fact that
	$\tau(\mathtt{h})$ is a stopping time for $\mathscr{G}(\mathtt{h})$, we have that $X(\mathcal{H}_i)\mathbbm{1}_{\{\tau(\mathcal{H}_i)\leq
		m+(j-i)/2\}}\mathbbm{1}_{\mathcal{E}_{i,j}}$ is measurable with respect to
	$\mathscr{G}_{m+(j-i)/2}(\mathcal{H}_i)$ and thus
	\begin{align}
		\label{e:3.5.4}
		\mathbb{E}\left[
		X(\mathcal{H}_i)\mathbbm{1}_{\{\tau(\mathcal{H}_i)\leq
		m+(j-i)/2\}}\mathbbm{1}_{\mathcal{E}_{i,j}}X(\mathcal{H}_j)
		\right]=\mathbb{E}\left[X(\mathcal{H}_i)\mathbbm{1}_{\{\tau(\mathcal{H}_i)\leq
		m+(j-i)/2\}}\mathbbm{1}_{\mathcal{E}_{i,j}}\mathbb{E}\left[ X(\mathcal{H}_j)|
		\mathscr{G}_{(j-i)/2}(\mathcal{H}_i) \right]\right].
	\end{align}

On the event $\mathcal{E}_{i,j}$, we have that
	\begin{align}
		\left|\mathbb{E}\left[ X(\mathcal{H}_j)|
		\mathscr{G}_{m+(j-i)/2}(\mathcal{H}_i) \right]-\mathbb{E}\left[
		X(\mathcal{H}_j) \right] \right|^2 &\leq
		\mathbb{E}[(X(\mathtt{h}))^2]\mathbb{E}_{\nu} \left[ \left|\frac{d\nu_{i,j}}{d\nu} -1\right|^2
		\right]\nonumber\\
		\label{e:3.5.7}
		&\leq C_1\mathbb{E}\left[ (X(\mathtt{h}))^2 \right]
		e^{-c_1(j-i-2m)}.
	\end{align}
	
	We are now in a position to show \eqref{e:3.5.3}. By \eqref{e:3.5.4} and
	the Cauchy-Schwarz inequality, we
	can write
	\begin{align}	
		&\left|
		\text{Cov}\left(X(\mathcal{H}_i)\mathbbm{1}_{\{\tau(\mathcal{H}_i)\leq
		m+(j-i)/2\}}\mathbbm{1}_{\mathcal{E}_{i,j}},X(\mathcal{H}_j)\right)\right|^2\nonumber\\
		&\leq
		\mathbb{E}\left[\Big|X(\mathcal{H}_i)\mathbbm{1}_{\{\tau(\mathcal{H}_i)\leq
		m+(j-i)/2\}}\Big|	\Big|\mathbbm{1}_{\mathcal{E}_{i,j}}\mathbb{E}\left[
		X(\mathcal{H}_j)|
		\mathscr{G}_{m+(j-i)/2}(\mathcal{H}_i) \right]-\mathbb{E}\left[
		X(\mathcal{H}_j) \Big] \right|\right]^2\nonumber\\
		&\leq \mathbb{E}\left[
		\Big|X(\mathcal{H}_i)\mathbbm{1}_{\{\tau(\mathcal{H}_i)\leq
		m+(j-i)/2\}}\Big|^2 \right] \mathbb{E}\left[ \Big|\mathbb{E}\left[
		X(\mathcal{H}_j)|
		\mathscr{G}_{(j-i)/2}(\mathcal{H}_i) \right]-\mathbb{E}\left[
		X(\mathcal{H}_j) \right] \Big|^2; \mathcal{E}_{i,j} \right]\nonumber\\
		&\leq c_4\mathbb{E}\left[ |X(\mathcal{H}_i)|^2 \right]
		\mathbb{E}\left[ |X(\mathcal{H}_i)|^2 \right]e^{-c_1(j-i-2m)}\nonumber\\
		&\leq c_5\mathbb{E}\left[ |X(\mathcal{H}_i)|^2 \right]^2
		e^{-c_1(j-i-2m)},
	\end{align}
	where we have used \eqref{e:3.5.7} to obtain the third inequality. Note
	that the second moment of $X(\mathcal{H}_i)$ is finite by \eqref{e:moment}. This completes the proof of the exponential decay
	of correlations for $\left\{ X(\mathcal{H}_i) \right\}_{i\in
	\mathbb{Z}}$.
\end{proof}

To use a law of large numbers argument for our application later,
we will require the stationarity and exponential decay of correlations for
observables which are dependent on all past scales and are thus in particular are not
local. As described in Section \ref{s:iop}, in our setting, these observables arise from the nature of the
log-parametrization of the infinite geodesic $\Gamma$ when computing
the empirical quantities (see \eqref{log-para1}). 
Though not local, the observables that we will use will still be ``almost
local'' in the sense that the contribution to the observable at scale $i$
coming from scales less than $i-j$ would decay like
$e^{-cj}$. In the remainder of this subsection, we set up a technical framework to obtain
exponential decay of correlations for such observables.

We first introduce some notation. 
For any bi-infinite sequence $\mathcal{J}=\left\{ \mathcal{J}_i
\right\}_{i\in \mathbb{Z}}$, and any ${n\in \mathbb{Z}}$, we use
$\mathcal{J}^n$ to denote the shifted sequence defined by
\begin{equation}\label{shift1}
\mathcal{J}^n_i=\mathcal{J}_{n+i}.
\end{equation}

{Also, we will often use calligraphic letters to denote the bi-infinite sequence
formed by the random variables denoted by the non-calligraphic counterpart} (e.g.\
for a sequence of variables $\left\{ J_i \right\}_{i\in
\mathbb{Z}}$, we will, without commenting, use $\mathcal{J}$ to
denote the sequence $\mathcal{J}_i=J_i$). 

Now let $f:(\mathcal{D}'(\mathbb{C}_{>1}))^{\mathbb{Z}}\rightarrow
\mathbb{C}^\mathbb{Z}$ be a measurable function 
which satisfies
\begin{equation}
	\label{e:3.161}
	f(\mathcal{H})^n=f( \mathcal{H}^n)
\end{equation}
almost surely for all $n\in \mathbb{Z}$.
Further, for each $m\in \mathbb{N}$, let
$X_m:\mathcal{D}'(\mathbb{C}_{>1})\rightarrow \mathbb{C}$ be a
deterministic measurable
function. Using the latter, we define the sequence of functions $\left\{ f_m \right\}_{m\in
\mathbb{N}}$ with
$f_m:(\mathcal{D}'(\mathbb{C}_{>1}))^{\mathbb{Z}}\rightarrow
\mathbb{C}^\mathbb{Z}$ (note that $f_m$ takes a sequence as an input and
outputs a sequence) such that for any $\mathtt{H}\in
(\mathcal{D}'(\mathbb{C}_{>1}))^{\mathbb{Z}}$, we have
\begin{equation}
	\label{ee:e}
	[f_m(\mathtt{H})]_i=X_m(\mathtt{H}_{i-1-m}).
\end{equation}
Note that the above definition in particular ensures that 
\begin{equation}
	\label{e:3.162}
	f_m(\mathcal{H})^n=f_m( \mathcal{H}^n)
\end{equation}
almost surely for all $n\in \mathbb{Z}$. This along with \eqref{e:3.161}
implies that the $f_m(\mathcal{H})$ and
$f(\mathcal{H})$ are stationary bi-infinite sequences themselves and we now
record this as a lemma.
\begin{lemma}
	\label{3.163-}
	With the above definitions, we have that $f(\mathcal{H})$ and
	$f_m(\mathcal{H})$ are stationary sequences for all $m\in \mathbb{N}$.
\end{lemma}
\begin{proof}
	The proof follows by using \eqref{e:3.161} and \eqref{e:3.162} along
	with the stationarity of $\mathcal{H}$.
\end{proof}
Before stating the main decorrelation result we need a bit more preparation.
For each $m$, let $\tau_m:\mathcal{D}'(\mathbb{C}_{>1})\rightarrow
\mathbb{R}_+$ be a deterministic
measurable function 
such that for {a} whole plane GFF marginal $\mathtt{h}$, $\tau_m(\mathtt{h})$ is a stopping time for the
filtration $\mathscr{G}(\mathtt{h})$ and
$X_m(\mathtt{h})$ is measurable
with respect to
$\mathscr{G}_{\tau_m(\mathtt{h})}(\mathtt{h})$. Also,
note that Lemma \ref{3.163-} implies that the laws of the
variables $[f(\mathcal{H})]_i$ and $[f_m(\mathcal{H})]_i$ do not depend on
$i$. We can now state the main lemma of this subsection.
\begin{lemma}
	\label{3.163}
	Let $f$ and $\left\{ f_m \right\}_{m\in\mathbb{N}}$ as defined above
	also satisfy that 
	\begin{displaymath}
		f_m(\mathcal{H})\rightarrow f(\mathcal{H})
	\end{displaymath}
	coordinate-wise almost surely as $m\to \infty$. For a whole plane GFF marginal
	$\mathtt{h}$, assume that
	$\mathbb{E}|X_m(\mathtt{h})|^4$ and thus
	$\mathbb{E}|[f_m(\mathcal{H})]_i|^4$ are bounded
	uniformly in $m$ and that there exist constants
	$c_1,C_1$ such that
	\begin{displaymath}
		\mathbb{P}\left( \tau_m(\mathtt{h})-m\geq t \right) \leq C_1e^{-c_1t} 
	\end{displaymath}
	for all $t>0$ and $m\in \mathbb{N}$.
Suppose that there exist positive
	constants $c_2,C_2$
	satisfying
	\begin{equation}
		\label{e:g1}
		\mathbb{E}\left| [f_m(\mathcal{H})]_i -[f(\mathcal{H})]_i
		\right|^4\leq C_2e^{-c_2m}
	\end{equation}
	for all $m$. Then there exist positive constants
	$c_3,C_3$
	such that we have
	\begin{displaymath}
		\mathrm{Cov}\left( [f(\mathcal{H})]_i,[f(\mathcal{H})]_j \right)\leq
		C_3e^{-c_3|i-j|}
	\end{displaymath}
	for all $i,j\in \mathbb{Z}$.

\end{lemma}

\begin{proof}

	\noindent The first step will be to prove
	\textit{exponential decay of correlations for the sequence
	$f_m(\mathcal{H})$ for a fixed $m$, with the decay constants being
	uniform in $m$:}
	the proof will be an application of Lemma \ref{3.10}. 
Note that $\mathbb{E}[X_m(\mathtt{h})^4]$ is bounded uniformly in $m$ by
assumption and 	
the conditions in the definition
	of $\tau_m$ ensure that $X_m(\mathtt{h})\in
	\mathscr{G}_{\tau_m(\mathtt{h})}(\mathtt{h})$ as required for an application of Lemma
	\ref{3.10}. 
	By using Lemma
	\ref{3.10} along with the assumption on the exponential tails of
	$\tau_m-m$, there
	exist positive constants $C,c,$ independent of $m,$ such that
	\begin{equation}
		\label{e:3.9.31}
		\mathrm{Cov}\left(
		X_m(\mathcal{H}_0),X_m(\mathcal{H}_i)
		\right)\leq Ce^{-c(i-2m)}.
	\end{equation}
	and hence by using the stationarity of $f_m(\mathcal{H})$,
	\begin{equation}
		\label{e*:3.15.22}
		\mathrm{Cov}\left(
		([f_m(\mathcal{H})]_0,[f_m(\mathcal{H})]_i
		\right)\leq Ce^{-c(i-2m)}.
	\end{equation}
	for all $i$. 
	
	By the above along with the stationarity of $f_m(\mathcal{H})$, for all
	$i,j\in \mathbb{Z}$ with $|j-i|\geq 4m$,
	\begin{equation}
		\label{e:f_mdecay}
		\mathrm{Cov}\left(
		([f_m(\mathcal{H})]_i,[f_m(\mathcal{H})]_j
		\right)\leq Ce^{-c|j-i|}.
	\end{equation}
	This will be used in the next part of the proof where we obtain the correlation
	decay statement for $f(\mathcal{H})$ itself.
	
	\medskip

\noindent
\textit{Exponential decay of correlations for
$f(\mathcal{H})$:} 
First observe that as a consequence of \eqref{e:g1}, for some
positive constants
	$C_1,c_1$, 
	\begin{equation}
		\label{e*:3.15.211}
		\Var\left(
		[f(\mathcal{H})]_i-[f_m(\mathcal{H})]_i\right)\leq
		C_1e^{-c_1m}
		\end{equation}
		for all $i$. 
The proof of the exponential decay of correlations for
	$f(\mathcal{H})$ can now be completed in a
		straightforward manner by the Cauchy-Schwarz inequality. In the
		following, we make the choice $m=|j-i|/4$ and obtain for
		any $i,j\geq 0$,
		\begin{align}
			\label{e:finCov}
			&\mathrm{Cov}\left( [f(\mathcal{H})]_i,[f(\mathcal{H})]_j
			\right)\\
			\nonumber
			&\leq \mathrm{Cov}\left(
			[f_m(\mathcal{H})]_i,[f_m(\mathcal{H})]_j
			\right) +\sqrt{\mathrm{Var}\left(
			[f(\mathcal{H})]_i-[f_m(\mathcal{H})]_i \right)\mathrm{Var}\left(
			[f(\mathcal{H})]_j-[f_m(\mathcal{H})]_j \right)} \nonumber\\
			&+\sqrt{\mathrm{Var}\left(
			[f(\mathcal{H})]_i-[f_m(\mathcal{H})]_i \right)\mathrm{Var}\left(
			[f_m(\mathcal{H})]_j \right)}+ \sqrt{\mathrm{Var}\left(
			[f_m(\mathcal{H})]_i \right)\mathrm{Var}\left(
			[f(\mathcal{H})]_j-[f_m(\mathcal{H})]_j \right)}\nonumber\\
			\nonumber
			&\leq C'e^{-c'|j-i|}.
		\end{align}
The first inequality involves using the bi-linearity of
$\mathrm{Cov}(\cdot,\cdot)$ along with the Cauchy-Schwarz inequality while for the final bound, we use \eqref{e:f_mdecay} to bound the first
term along with \eqref{e:g1} and the fact that
$\mathbb{E}|[f_m(\mathcal{H})]_i|^4$ and hence
$\mathrm{Var}([f_m(\mathcal{H})]_i)$ is bounded uniformly in $m$.
Note that the choice $m=|j-i|/4$ allows us to obtain the needed bound
$e^{-c'|j-i|}$ from \eqref{e:g1}.
\end{proof}

We also state a weaker version of the above lemma for situations where we do
not need its full power. Let $X:\mathcal{D}'(\mathbb{C}_{>1})\rightarrow \mathbb{C}$ be a 
deterministic measurable function and let
$f:(\mathcal{D}'(\mathbb{C}_{>1}))^{\mathbb{Z}}\rightarrow
\mathbb{C}^\mathbb{Z}$ be a deterministic measurable {function} which satisfies
\begin{displaymath}
	[f(\mathcal{H})]_i=X(\mathcal{H}_i)
\end{displaymath}
almost surely for all $i\in \mathbb{Z}$.
Let $\tau:\mathcal{D}'(\mathbb{C}_{>1})\rightarrow \mathbb{R}_+$ be a
measurable function such that for a whole plane GFF marginal $\mathtt{h}$,
we have that $\tau(\mathtt{h})$
is a stopping time for the filtration $\mathscr{G}(\mathtt{h})$ and
$X(\mathtt{h})$ is measurable with respect to
$\mathscr{G}_{\tau(\mathtt{h})}(\mathtt{h})$. In this setting, we have the
following lemma.

\begin{lemma}
	\label{'3.163}
	Assume that $\mathbb{E} |X(\mathtt{h})|^4$ is finite and that there
	exist constants $c,C$ such that 
	\begin{displaymath}
		\mathbb{P}\left( \tau(\mathtt{h})\geq t \right)\leq Ce^{-c t}
	\end{displaymath}
	for all $t>0$. Then $f(\mathcal{H})$ is a stationary
	sequence indexed by $\mathbb{Z}$ and has exponential decay of
	correlations. 
\end{lemma}
\begin{proof}
	 Use Lemma \ref{3.163} in the degenerate case where
	$f_m,f$ are defined by
	$[f(\mathcal{H})]_i=[f_m(\mathcal{H})]_i=X(\mathcal{H}_i)$ a.s.\, i.e., $X_m(\mathcal{H}_{i-1-m})=X(\mathcal{H}_i)$ for all $m\in
	\mathbb{N}$ and $i \in \mathbb{Z}$.
\end{proof}

Before concluding this section, we make a note about the measurable
functions $f,f_m,X,X_m$ that will be defined whenever we invoke Lemma \ref{3.163} and
Lemma \ref{'3.163}. In general, we will take $X(\mathtt{h}),
X_m(\mathtt{h})$ to be some
observables of a GFF $\mathtt{h}$ and it will typically be defined in
terms of the coalescence properties of the metric $D_\mathtt{h}$. This will almost
surely define $X(\mathtt{h}),X_m(\mathtt{h})$ for a whole plane GFF marginal $\mathtt{h}$.
To define these as measurable maps from $\mathcal{D}'(\mathbb{C}_{>1})$ to
$\mathbb{C}$, an arbitrary extension is made outside the
support of $\mathtt{h}$. We will similarly first define 
$f(\mathcal{H})$ almost surely and then arbitrarily extend outside the support of
$\mathcal{H}$ to obtain a measurable function. This will be done
repeatedly throughout the paper whenever the results of Section
\ref{ss:corr} are used and for the sake of brevity we will not further include any comment regarding measurability.

\section{
Renewal theory for the infinite geodesic}\label{disjseg}

We now implement the first step of the strategy outlined in Section \ref{s:iop}  for studying
the environment seen from the geodesic $\Gamma$, namely, to partition $\Gamma$ into segments such that the local environments around these
segments are approximately independent and identically distributed.

Since the section is somewhat long, we begin with a brief roadmap. In
Section \ref{ss:confluence}, we show that a slightly modified version of the coalescence event
$\mathsf{Coal}$ from Section \ref{s:iop} occurs
almost surely at infinitely many scales and in Section \ref{ss:segments}, obtain a sequence of coalescence points $\left\{ p_i \right\}_{i\in
\mathbb{Z}}$ to decompose $\Gamma$ into  the geodesic segments $\Gamma(p_i,p_{i+1})$. Section \ref{ss:Y_i}
contains some technical estimates for the power law tails for the lengths of
the segments. In Section \ref{ss:lengths}, we introduce the quantity
$G_i$ (recall from \eqref{log-para1}) which allows us to keep track of the $\log$ length of the segments. This is crucial because of the $\log$ parametrization used
to define the empirical quantities. Noting that the variables $G_i$
depend on all the past scales before $i$, we define a finite proxy 
$\underline{G}_{m,i}$ depending only on the past $m+1$ scales. In Section \ref{ss:alpha},
we show that the latter approximate $G_i$
up to an $O(e^{-cm})$ error. The last subsection is devoted to showing
eighth moment bounds for $G_i$ and $\underline{G}_{m,i}$ and this will be
used later in Section \ref{s:stat_env} where we show a correlation decay
statement for the $G_i$ as an application of Lemma \ref{3.163}.

\subsection{ The coalescence events}
\label{ss:confluence}

The difficulty a-priori is that the geodesic $\Gamma$ has no
clear
sense of direction as it goes from $0$ to $\infty$, as in, if
$\Gamma_t\in \mathbb{C}_{>r}$, it might still be true that $\Gamma_s\in
\mathbb{C}_{<r}$ for some $s>t$, thus necessitating a proper control over these
`backtracks'. To address this, we will look at the events $F_r$, the coalescence event at radius
$r$ (see Figure \ref{f.proofsketch}), whose occurrence roughly allows us to think of
$\Gamma\lvert_{\mathbb{C}_{(r,Kr)}}$ as a directed geodesic. These
events $F_r$ will intuitively be the same as the ${\sf{Coal}}$ events
defined in Section \ref{s:iop}, but with certain additional
technical {measurability conditions}, which is why we choose a distinct labeling.

We first define the events $F_r[\mathtt{h}]$ for a whole plane GFF
marginal $\mathtt{h}$ via the
following lemma from \cite{GPS20}. The lemma is phrased in terms of the
parameter $K$ that we have been using, the value of which will be chosen and
fixed shortly. We note that the lemma is stated in the original source for a
whole plane GFF $h$ but the same construction works for our setting of a
whole plane GFF marginal $\mathtt{h}$ by using the locality of the LQG metric.
\begin{lemma}[{\cite[Lemma 4.8 and proof]{GPS20}}]
	\label{3.10.0}
	Given a whole plane GFF marginal $\mathtt{h}$ 
	on $\mathbb{C}_{>1}$, for any $r\geq 1$, there
	exists a random point $p(\mathtt{h},r)\in
	\mathbb{C}_{(K^{1/4}r,K^{1/2}r)}$ and an event 
	$F_r[\mathtt{h}]$ which are both measurable functions of
	$\mathtt{h}\lvert_{\mathbb{C}_{(K^{1/8}r,Kr)}}$ viewed modulo an additive
	constant, such that on $F_r[\mathtt{h}]$, the
	following is true almost surely.
	\begin{enumerate}
		\item There exists at least one geodesic
			$\Gamma(u,v;\mathtt{h})$ for every pair
			of points $u,v\in \mathbb{C}_{[K^{1/4}r,K^{1/2}r]}$ 
			Furthermore, all such geodesics $\Gamma(u,v;\mathtt{h})\subseteq
			\mathbb{C}_{(K^{1/8}r,Kr)}$.
		\item  $p(\mathtt{h},r)\in \bigcap_{u,v}  \Gamma(u,v;
	\mathtt{h})\neq \emptyset$, where the intersection is over all $u\in
	\mathbb{T}_{K^{1/4}r},v\in \mathbb{T}_{K^{1/2}r}$ and the possibly
	multiple geodesics between each pair $u,v$.
	\end{enumerate}
	Also, for any fixed $r>1$, we have that the events
	$F_r[\mathtt{h}]$ and
	$F_1[\widehat{\mathtt{h}}_r\mydot \psi_r]$ are equal
	upto a null set (by the coordinate change formula recorded in
	Proposition \ref{b3} (4)). Similarly, we have that almost surely,
	$p(\mathtt{h},r)=rp(\widehat{\mathtt{h}}_r\mydot\psi_r,1)$.
\end{lemma}
The two statements at the end of the above lemma are not explicitly stated in
the original source \cite{GPS20} but they follow by using
the scale invariance of the GFF and inspecting the proof in a straightforward manner.
In the above lemma, we have that the
events
$F_r[\mathtt{h}]$ and $F_1[\widehat{\mathtt{h}}_r\mydot \psi_r]$ are equal
upto null sets. This scenario, where two events are equal upto null sets
will occur repeatedly throughout the paper. \textbf{From now on, we will not mention
this explicitly and will simply write that two events are equal when we mean
that they are equal upto a null set.}

In case of the whole plane GFF $h$, we can define the coalescence
event by
$F_r=F_1[\widehat{h}_r\mydot \psi_r]$; note that if $r\geq 1$, we
can write $F_r=F_r[\mathcal{H}_0]$
by using the last part of
the above lemma. Similarly, we define the coalescence point
$p(r)=rp(\widehat{h}_r\mydot \psi_r,1)$. As a consequence of the above lemma, we also have that the
events $F_r[\mathtt{h}]$ and thus $F_r$ have the same probability for all $r\geq 1$ and the next result says that this probability is positive if $K$ is large enough.
\begin{lemma}{\cite[Lemma 4.8]{GPS20}}
	\label{3.10.01}
	The probability $\mathbb{P}(F_r[\mathtt{h}])$ is independent of $r$ and if $K$ is
	chosen large enough, we have
	$\mathbb{P}(F_r[\mathtt{h}])>0$.
\end{lemma}
\begin{proof}
	The fact that $\mathbb{P}(F_r[\mathtt{h}])$ is not dependent on $r$ is a consequence of the equality of
	$F_r[\mathtt{h}]$ and $F_1[\widehat{\mathtt{h}}_r\mydot \psi_r]$ from Lemma
	\ref{3.10.0} along with the scale invariance of the whole plane GFF
	viewed modulo an additive constant as in \eqref{e:b4.1}. From the results in \cite[Lemma
	4,8, Lemma 4.6]{GPS20}, one has that $\mathbb{P}(F_r[\mathtt{h}])>0$ for all $K$
	large enough.
\end{proof}

{{At this point, we can fix the value of the parameter $K$ according to the
discussion right before Lemma \ref{3.7.0.1} in Section \ref{sss:K}.
}}
\begin{center}
\textbf{This choice of $K$ will remain fixed throughout the rest of the paper.}
\end{center}
 
Since we will later be working with lengths of the
geodesic segments between two successive coalescence points, it will be convenient to consider an event
where along with $F_r[\mathtt{h}]$ we also demand the stronger condition that such geodesic segments have lengths
bounded away from $0$. This event will be denoted as $E_r[\mathtt{h}]$
and is defined for all $r\geq 1$ in terms of a parameter $\rho$ as
\begin{equation}
	\label{e:E_r}
	E_r[\mathtt{h}]=F_r[\mathtt{h}]\cap \left\{
	r^{-\xi Q}  e^{-\xi \mathbf{Av}(\mathtt{h},\mathbb{T}_r)}
	D_\mathtt{h}(p(\mathtt{h},r),\mathbb{T}_{K^{1/2}r})\geq \rho
	\right\}.
\end{equation}
Before proceeding further we include a brief heuristic discussion motivating the above definition. Suppose that both $E_{r_1}[\mathtt{h}]$ and $E_{r_2}[\mathtt{h}]$
occurred for some fixed positive real $r_1,r_2$ satisfying $Kr_1<r_2$.
Then the second event in the intersection defining $E_r[\mathtt{h}]$ in \eqref{e:E_r} allows us to
say that any geodesic
$\zeta=\Gamma(p(\mathtt{h},r_1),p(\mathtt{h},r_2);\mathtt{h})$ has
length at least $\rho$ in the metric corresponding to the `recentered field viewed from $r$', by
which we mean that
$\ell(r^{-1}\zeta;D_{\widehat{\mathtt{h}}_r\mydot \psi_r})\geq \rho$. A
relevant
detail that we avoid discussing here is that the occurrence of the coalescence event
$F_{r_1}[\mathtt{h}]$ ensures that $\zeta\subseteq \mathbb{C}_{>{r_1}}$; it will appear later in
Lemma \ref{3.12.1}.

An important property of the above definition of
$E_r[\mathtt{h}]$ is that it is measurable with respect to
$\mathtt{h}\lvert_{\mathbb{C}_{(K^{1/8}r,Kr)}}$ viewed modulo an additive
constant and this is shown in the
following lemma.
\begin{lemma}
	\label{Emeas}
	Let $\mathtt{h}$ be a whole plane GFF marginal. Then
	$E_r[\mathtt{h}]$ is measurable with respect to  $
	\left(\mathtt{h}-\mathbf{Av}(\mathtt{h},\mathbb{T}_{r})\right)\lvert_{\mathbb{C}_{(K^{1/8}r,Kr)}}
	$. Also, for any fixed $r>1$,  the events
	$E_r[\mathtt{h}]$ and
	$E_1[\widehat{\mathtt{h}}_r\mydot \psi_r]$ are equal.
	
\end{lemma}
\begin{proof}
	First note that we can equivalently show that $E_r[\mathtt{h}]$ is measurable with respect to
	$\mathtt{h}\lvert_{\mathbb{C}_{(K^{1/8}r,Kr)}}$ viewed modulo an
	additive constant.
	The above statement for the events $F_r[\mathtt{h}]$ instead of
	$E_r[\mathtt{h}]$ is true by Lemma
	\ref{3.10.0}. We will locally use $A_r[\mathtt{h}]$ to denote the event $$\left\{
	r^{-\xi Q}  e^{-\xi \mathbf{Av}(\mathtt{h},\mathbb{T}_r)}
	D_\mathtt{h}(p(\mathtt{h},r),\mathbb{T}_{K^{1/2}r})\geq \rho
\right\}$$ and prove the same statement for $E_r[\mathtt{h}]=F_r[\mathtt{h}]\cap
A_r[\mathtt{h}]$. 
		Observe that by the locality of the LQG metric as in Proposition
		\ref{b3} (2), we have that
		\begin{equation}
			\label{e:ll}
				D_\mathtt{h}(\cdot,\cdot;\mathbb{C}_{(K^{1/8}r,Kr)})=D_{\mathtt{h}\lvert_{\mathbb{C}_{(K^{1/8}r,Kr)}}}
		\end{equation}
		almost surely.
	By  Lemma
		\ref{3.10.0} (1) along with the fact that $p(\mathtt{h},r)\in
		\mathbb{C}_{(K^{1/4}r,K^{1/2}r)}$, on the event
		$F_r[\mathtt{h}]$, there exists at least one
		geodesic $\Gamma(p(\mathtt{h},r),v;\mathtt{h})$ for all $v\in
		\mathbb{T}_{K^{1/2} r}$ and all these geodesics satisfy  $\Gamma(p(\mathtt{h},r),v;\mathtt{h})\subseteq
		\mathbb{C}_{(K^{1/8}r,Kr)}$. By using this and \eqref{e:ll}, along with
		the fact that
			$D_{\mathtt{h}\lvert_{\mathbb{C}_{(K^{1/8}r,Kr)}}}$ and
			$p(\mathtt{h},r)$ are measurable
			with respect to $\mathtt{h}\lvert_{\mathbb{C}_{(K^{1/8}r,Kr)}}$,
			we obtain that on the event $F_r[\mathtt{h}]$, the random
			variable $r^{-\xi Q} 
	D_\mathtt{h}(p(\mathtt{h},r),\mathbb{T}_{K^{1/2}r})$ is measurable with respect to
	$\mathtt{h}\lvert_{\mathbb{C}_{(K^{1/8}r,Kr)}}$. This implies that on
	the event $F_r[\mathtt{h}]$, the random variable $r^{-\xi Q}  e^{-\xi \mathbf{Av}(\mathtt{h},\mathbb{T}_r)}
	D_\mathtt{h}(p(\mathtt{h},r),\mathbb{T}_{K^{1/2}r})$ and subsequently the
	event $A_r[\mathtt{h}]$ are measurable with respect to $\sigma\left( \mathtt{h}\lvert_{\mathbb{C}_{(K^{1/8}r,Kr)}},
	\mathbf{Av}(\mathtt{h},\mathbb{T}_r)\right)$. 

		Now note that as a consequence of Weyl Scaling (see Proposition
		\ref{b3} (3)) along with the measurability of $p(\mathtt{h},r)$
		with respect to $\mathtt{h}\lvert_{\mathbb{C}_{(K^{1/8}r,Kr)}}$
		viewed modulo an additive constant (see Lemma \ref{3.10.0}), almost
		surely, the
	quantity $e^{-\xi \mathbf{Av}(\mathtt{h}+c,\mathbb{T}_r)}
	D_\mathtt{h+c}(p(\mathtt{h},r),\mathbb{T}_{K^{1/2}r})$ does not depend
	on the choice of $c\in \mathbb{R}$. This shows that the event $A_r[\mathtt{h}]$ is measurable
		with respect to $\mathtt{h}$ viewed modulo an additive constant. This
		along with the result of the previous paragraph implies that on the event
		$F_r[\mathtt{h}]$, the random variable $\mathbbm{1}(A_r[\mathtt{h}])$ is equal to a
		function measurable with respect to
		$\mathtt{h}\lvert_{\mathbb{C}_{(K^{1/8}r,Kr)}}$ viewed modulo an
		additive constant. Since $F_r[\mathtt{h}]$ itself is measurable with
		respect to the same, we have that $E_r[\mathtt{h}]=F_r[\mathtt{h}]\cap
		A_r[\mathtt{h}]$ is measurable with respect to 	$\mathtt{h}\lvert_{\mathbb{C}_{(K^{1/8}r,Kr)}}$ viewed modulo an
		additive constant. This completes the proof of the first assertion in
		the statement of the lemma, and we now move on to proving the second
		assertion.
		
		In view of the corresponding assertion for the $F_r$s in
		Lemma \ref{3.10.0} and the expression
		$E_r[\mathtt{h}]=F_r[\mathtt{h}]\cap A_r[\mathtt{h}]$, we need only
		show that $A_r[\mathtt{h}]$ is equal to
		$A_1(\widehat{\mathtt{h}}_r\mydot \psi_r)$.
		By Lemma
		\ref{3.10.0}, we also have that
		$p(\mathtt{h},r)=rp(\widehat{\mathtt{h}}_r\mydot\psi_r,1)$ almost
		surely. This, along
		with Weyl Scaling and the coordinate change formula, implies that
		\begin{displaymath}
			A_r[\mathtt{h}]=\left\{
	r^{-\xi Q}  e^{-\xi \mathbf{Av}(\mathtt{h},\mathbb{T}_r)}
	D_\mathtt{h}(rp(\widehat{\mathtt{h}}_r\mydot\psi_r,1),\mathbb{T}_{K^{1/2}r})\geq \rho
	\right\}=\left\{
	D_{\widehat{\mathtt{h}}_r\mydot\psi_r}(p(\widehat{\mathtt{h}}_r\mydot\psi_r,1),\mathbb{T}_{K^{1/2}})\geq
	\rho\right\}=A_1[\widehat{\mathtt{h}}_r\mydot \psi_r]
		\end{displaymath}
and this completes the proof.
		\end{proof}

We now define the coalescence events for the whole plane GFF $h$ itself. We
define the event $E_r$ for any $r>0$ by $E_r=E_1[\widehat{h}_r\mydot
\psi_r]$; note that if $r\geq 1$, we
can write $E_r=E_r[\mathcal{H}_0]$
by an application of the
second assertion in Lemma \ref{Emeas}.

Note that we still have an undetermined parameter
$\rho$ in the definition of $E_r[\mathtt{h}]$ which is determined by using the
following lemma to ensure $\mathbb{P}(E_r)>0$.
\begin{lemma}
	\label{3.10.1}
	The probability $\mathbb{P}(E_r)$ is independent of $r$ 
	and if $\rho$ is chosen small enough, we have $\mathbb{P}(E_r)>0$.
\end{lemma}
\begin{proof}
	To see that $\mathbb{P}(E_r)$ is not dependent on $r$, recall the
	definition $E_r=E_1[\widehat{h}_r\mydot \psi_r]$ and note that
	$\widehat{h}_r\mydot \psi_r$ is a whole plane GFF marginal for fixed $r>0$.
	To obtain that the above probability is positive, first notice that for
	a whole plane GFF marginal $\mathtt{h}$ on $\mathbb{C}_{>1}$, we have
	\begin{displaymath}
		\mathbb{P}\left(
	r^{-\xi Q}  e^{-\xi \mathbf{Av}(\mathtt{h},\mathbb{T}_r)}
	D_\mathtt{h}(p(\mathtt{h},r),\mathbb{T}_{K^{1/2}r}))\geq \rho
	\right)\rightarrow 1
	\end{displaymath}
	as $\rho\rightarrow 0$. This uses that
	$D_\mathtt{h}(p(\mathtt{h},r),\mathbb{T}_{Kr}))>0$ almost surely
	since $D_\mathtt{h}$ is almost surely a metric inducing the Euclidean
	topology on $\mathbb{C}_{>1}$ and $\mathbb{T}$ is compact. The proof can be completed by substituting
	$\mathtt{h}=\mathcal{H}_0$, invoking Lemma
	\ref{3.10.01} and using a union bound.
\end{proof}

Having defined the events $E_r$, we next show that 
that events $E_{S_i}$, where $S_i$ was defined to be $K^i$ in \eqref{e:S0}, occur infinitely often.

\begin{lemma}
	\label{3.11}

Let $\mathtt{h}$ be a whole plane GFF marginal on $\mathbb{C}_{>1}$. Then the
events $E_{S_i}[\mathtt{h}]$ almost surely occur
infinitely often. Moreover, {there exist constants $c,C$ such that }
\begin{displaymath}
	\mathbb{P}\left( \inf_{i>0}\left\{ E_{S_i}[\mathtt{h}] \text{ occurs}
	\right\}\geq j \right)\leq Ce^{-cj}.
\end{displaymath}

\end{lemma}

\begin{proof}

	By Lemma \ref{3.10.1}, we have that $\mathbb{P}\left(
	E_{S_i}[\mathtt{h}]
	\right)$ is the same for all $i$ and is positive. Let $p$ denote the
	above probability. We will use Proposition \ref{1.16} with
	$\alpha_1=K^{1/8}$ and $\alpha_2=K$. By Lemma \ref{Emeas}, we
	have that $E_{S_i}[\mathtt{h}]\in \sigma \left( \left(
	(\mathtt{h}-\mathbf{Av}(\mathtt{h},\mathbb{T}_{S_i})\lvert_{\mathbb{C}_{(\alpha_1S_i,\alpha_2S_i)}}
	\right)
	\right)$ and this justifies the use of Proposition \ref{1.16}. Thus there exist positive constants $a,b,c$ such that
	if $N(I)$ denotes the number of $i\in [\![1,I]\!]$ for which the event
	$E_{S_i}[\mathtt{h}]$ occurs then 
	\begin{displaymath}
		\mathbb{P}\left( N(I)<bI \right)\leq ce^{-aI}.
	\end{displaymath}
	The lemma is now a straightforward consequence of the above and the Borel-Cantelli lemma.
\end{proof}
We have the following analogous statement for the whole plane GFF $h$.
\begin{lemma}
	\label{3.11*}
	The events $E_{S_i}$ almost surely occur infinitely often
	both along negative and positive integer values of $i$.
\end{lemma}
\begin{proof}
	The case of positive integer values of $i$ is covered by Lemma
	\ref{3.11}. For the negative values of $i$, the same argument works if we use \cite[Lemma 2.11]{GM20} instead of
	Proposition \ref{1.16}; the former result is about iterating events on
	annuli shrinking towards $0$ whereas the latter result does the same for
	annuli growing towards infinity.
\end{proof}

For  $\mathtt{h}$, in order to keep track of the
indices for which the events $E_{S_i}[\mathtt{h}]$ occur, we
introduce the sequence of random variables $\mathcal{P}(\mathtt{h})$ indexed by
{$\mathbb{N}\cup\left\{ 0 \right\}$ defined by
$\mathcal{P}_i(\mathtt{h})=\mathbbm{1}_{E_{S_{i}}[\mathtt{h}]}$}. For the case
of the whole plane GFF $h$, we instead define the
corresponding sequence as a bi-infinite sequence $\mathcal{P}$ such that
$\mathcal{P}_i=E_{S_i}$ for $i\in \mathbb{Z}$.

Next, we introduce the
notation $\eta(j,\mathtt{h})$ for the number of $1$s present in
$\mathcal{P}(\mathtt{h})$ until the
index $j$. That is, we define
\begin{equation}
	\label{e:3.5.12.0}
	\eta(j,\mathtt{h})=
		\sum_{k=0}^j \mathcal{P}_k(\mathtt{h})-1 \text{
		for } j\in \mathbb{N}\cup \left\{ 0 \right\}.
\end{equation}

Here, we have put an additional $-1$ because we would like the first
coalescence point occurring after $S_0=1$ to have the index $0$. In the case
when we are working with the whole plane GFF $h$, we recall that we have
coalescence points on both sides of the circle $\mathbb{T}$ and thus
define

\begin{equation}
	\label{e:3.5.12.0*}
	\eta(j)=
	\begin{cases}
		\sum_{k=0}^j \mathcal{P}_k-1 &\text{
		if } j\geq 0,\\
		\sum_{k=-j}^{-1}-\mathcal{P}_k-1 &\text{
		if } j<0.
	\end{cases}
\end{equation}

Using the above definitions, we can define the
inverse $\eta^{-1}$ such that for a whole plane GFF marginal $\mathtt{h}$,
 for any $j\geq 0$,  $\eta^{-1}(j,\mathtt{h})$ is the smallest
$j_1\geq 0$ such that $\eta(j_1,\mathtt{h})=j$.
Similarly, in the case of the whole plane GFF $h$, we define
$\eta^{-1}(j)$ as the smallest
$j_1\in \mathbb{Z}$ such that $\eta(j_1)=j$.

At this point, we introduce some more notation to keep track of the first
coalescence point occurring after scale $i$. For any $i\in \mathbb{Z}$, we define $\iota(i)$ by
\begin{equation}
	\label{e:iota}
	\iota(i)=\inf_{j\geq i}\left\{ E_{S_j} \text{ occurs} \right\}.
\end{equation}
Similarly, for a whole plane GFF marginal $\mathtt{h}$, we can define
$\iota(i,\mathtt{h})$ for all $i\in \mathbb{N}\cup\left\{ 0 \right\}$ by
$\iota(i,\mathtt{h})= \inf_{j\geq i}\left\{ E_{S_j}[\mathtt{h}] \text{ occurs}
\right\}$. Note that a particular consequence of Lemma \ref{3.11} and
the equality $E_{r}=E_1[\widehat{h}_r\mydot \psi_r]$ is that
there exist constants $C,c$ such that $\mathbb{P}\left( \iota(i)\geq i + t
\right)\leq Ce^{-ct}$. 
Recall the convention recorded before Lemma \ref{3.163} that if we have
some bi-infinite sequence $\mathcal{J}=\left\{ \mathcal{J}_i
\right\}_{i\in \mathbb{Z}}$, then we use $\mathcal{J}^n$ to denote the
shifted sequence defined by $\mathcal{J}^n_i=\mathcal{J}_{n+i}$. Next, we show that the bi-infinite sequence $\mathcal{P}$, a factor of the sequence
$\mathcal{H}$, is stationary with exponential decay of correlations.

\begin{lemma}
	\label{3.12}
	The bi-infinite sequence $\mathcal{P}$ is stationary with exponential decay of
	correlations.

\end{lemma}
\begin{proof}
	The proof will be an application of Lemma \ref{'3.163}. 	
For a whole plane GFF marginal $\mathtt{h}$, define $X(\mathtt{h})$ by 
\begin{equation}
	\label{eeee:5}
	X(\mathtt{h})= \mathbbm{1}_{E_{S_0}[\mathtt{h}]}.
\end{equation}
We arbitrarily extend $X$ outside the support of $\mathtt{h}$ to
	yield a measurable map
	$f:\mathcal{D}'(\mathbb{C}_{>1})\rightarrow
	\mathbb{C}$.
Define the function $f$ which takes a
	bi-infinite sequence $\mathtt{H}=\left\{ \mathtt{H}_i
	\right\}_{i\in \mathbb{Z}}\in
	(\mathcal{D}'(\mathbb{C}_{>1}))^{\mathbb{Z}}$ and yields an element of
	$\mathbb{C}^\mathbb{Z}$ by the map
	\begin{equation}
		\label{3.5.12.1}
		[f(\mathtt{H})]_i=X(\mathtt{H}_i).
	\end{equation}
It is easy to see that we have
\begin{equation}
	\label{e:3.5.12.2}
	\mathcal{P}_i=[f(\mathcal{H})]_i.
\end{equation}
almost surely for all $i\in \mathbb{Z}$.

Since $|X(\mathtt{h})|$ is at most $1$, $\mathbb{E}|X(\mathtt{h})|^4$ is uniformly bounded for all $i$. As can be seen in the statement of Lemma \ref{'3.163}, it remains to
define $\tau(\mathtt{h})$ such that $\tau(\cdot)$ is
deterministic, $\tau(\mathtt{h})$ is a stopping time for $\mathscr{G}(\mathtt{h})$,
$X(\mathtt{h})$ is measurable with respect to
$\mathscr{G}_{\tau(\mathtt{h})}(\mathtt{h})$, and $\tau(\mathtt{h})$ has
exponential tails. In this case, we can simply define $\tau(\mathtt{h})$ to be
$1$ deterministically and it can be seen that the above properties are
satisfied. Thus the use of Lemma \ref{'3.163} is justified and the proof is
complete.
\end{proof}

\subsection{Partitioning $\Gamma$ into segments}
\label{ss:segments}
Note that the definition of $\mathcal{P}$ indicates a natural choice for
partitioning the geodesic $\Gamma$ into smaller segments.
Indeed,
for $i\in \mathbb{N}\cup\{0\}$, we define
$p_i(\mathtt{h})=p(\mathtt{h},S_{\eta^{-1}(i,\mathtt{h})})$, and we call this
the $i$th coalescence point for $\mathtt{h}$. In case of the whole plane GFF
$h$, we define $p_i:=p(S_{\eta^{-1}(i)})$ and think of $p_i$ as the $i$th coalescence point observed when
exploring $h$ radially outwards with points outside $\mathbb{D}$ being
numbered starting from $0$ and the points inside $\mathbb{D}$ being numbered
negative. Lemma \ref{3.10.0} has the following
consequence for the points $p_i$.
\begin{lemma}
	\label{3.12.0}
	Almost surely in the randomness of $h$, we have that
	$p_{\eta(i)+j}=S_ip_j(\mathcal{H}_i)$ for all $j\in
	\mathbb{N}\cup\left\{ 0 \right\}$ and for $i\in
	\mathbb{Z}$ satisfying $\mathcal{P}_i=1$.
\end{lemma}
The following lemma records the partition of $\Gamma$ into $\Gamma(p_i,p_{i+1})$ and
controls the radial backtrack of the latter. 
\begin{lemma}
	\label{3.12.1}
	Almost surely, for all $i\in \mathbb{Z}$, there exists a
	unique geodesic $\Gamma(p_i,p_{i+1})$ between $p_i$ and $p_{i+1}$ and this
	unique geodesic also satisfies 
	\begin{equation}
		\label{e:georest1}
		\Gamma(p_i,p_{i+1})\subseteq
	\mathbb{C}_{( K^{1/4}
	S_{\eta^{-1}(i)},K^{1/2}S_{\eta^{-1}(i+1)})}\subseteq \mathbb{C}_{(
	S_{\eta^{-1}(i)},S_{\eta^{-1}(i+1)+1})}.
	\end{equation}
	Similarly, {for a whole plane GFF marginal} $\mathtt{h}$, there exists a
	unique geodesic
	$\Gamma(p_i(\mathtt{h}),p_{i+1}(\mathtt{h});\mathtt{h})$ between
	$p_i(\mathtt{h})$ and $p_{i+1}(\mathtt{h})$ for
	all $i\geq 0$ and this geodesic satisfies
	\begin{equation}
		\label{e:georest2}
		\Gamma(p_i(\mathtt{h}),p_{i+1}(\mathtt{h});\mathtt{h})\subseteq
	\mathbb{C}_{( K^{1/4}
	S_{\eta^{-1}(i,\mathtt{h})},K^{1/2}S_{\eta^{-1}(i+1,\mathtt{h})})}\subseteq \mathbb{C}_{(
	S_{\eta^{-1}(i,\mathtt{h})},S_{\eta^{-1}(i+1,\mathtt{h})+1})}.
\end{equation}
	Further, almost surely
	$\Gamma\setminus \left\{ 0 \right\}=\bigcup_{i\in \mathbb{Z}}\Gamma(p_i,p_{i+1})$,
	where the segments $\Gamma(p_i,p_{i+1})$ are disjoint except at
	their endpoints.
\end{lemma}
\begin{proof}	First note that by the almost sure uniqueness of $\Gamma$ from
	Proposition \ref{b5},  almost
	surely, for any two points $z,w$ on $\Gamma$, the
restriction of $\Gamma$ is the unique geodesic between $z,w$. Now suppose that
the event
$E_{S_j}$ occurred for some $j$. It is easy to see that on this event, $p(S_j)\in \Gamma$ because $\Gamma$ is a
geodesic which crosses the annulus $\mathbb{C}_{(K^{1/4}S_j,K^{1/2}S_j)}$. {Say that
$p(S_j)=\Gamma_{t_0}$; note that this is satisfied for a
unique value of $t_0$.} Now, we claim that for all $s>t_0$, we have that $\Gamma_s \in
\mathbb{C}_{>K^{1/4}S_j}$. Indeed, this true simply because if $\Gamma_{s_0}\in
\mathbb{C}_{\leq K^{1/4}S_j}$ for some $s_0>t_0$, then we can use the definition of
$E_{S_j}$ along with the fact that $\Gamma_t\rightarrow \infty$ as
$t\rightarrow \infty$ to say that there exists $t_1>s_0$ satisfying
$\Gamma_{t_1}=p(S_j)$, but this contradicts the
almost sure uniqueness of the infinite geodesic as mentioned earlier. This
shows that $\Gamma(p_i,p_{i+1})\subseteq
\mathbb{C}_{>K^{1/4}S_{\eta^{-1}(i)}}$ for all $i$. A similar argument shows that on the event $E_{S_j}$, if we have
$p(S_j)=\Gamma_{t_0}$, then for all $s<t_0$, we have $\Gamma_s\subseteq
\mathbb{C}_{<K^{1/2}S_j}$. This shows that $\Gamma(p_i,p_{i+1})\subseteq
\mathbb{C}_{<K^{1/2}S_{\eta^{-1}(i+1)}}$ and completes the proof of the
statement. 

We now come to the the corresponding statement for a whole plane GFF marginal
$\mathtt{h}$. To see the a.s.\ existence and uniqueness of
$\Gamma(p_i(\mathtt{h}),p_{i+1}(\mathtt{h});\mathtt{h})$ by using the
corresponding statement for $h$, we begin by constructing $\mathtt{h}$ in the
same space as $h$ via
$\mathtt{h}=h\lvert_{\mathbb{C}_{>1}}$. We thus have $p_j=p_j(\mathtt{h})$ and
$\eta^{-1}(j)=\eta^{-1}(j,\mathtt{h})$ for all $j\geq 0$. By
\eqref{e:georest1}, $\Gamma(p_i,p_{i+1})\subseteq \mathbb{C}_{>1}$ and
locality, $\Gamma(p_i,p_{i+1})$ is also a
$D_{h\lvert_{\mathbb{C}_{>1}}}=D_\mathtt{h}$ geodesic. Thus $
\Gamma(p_i(\mathtt{h}),p_{i+1}(\mathtt{h});\mathtt{h})=\Gamma(p_i,p_{i+1})$.
Now \eqref{e:georest2} follows by using \eqref{e:georest1} for $i\geq 0$.

To see that $\Gamma\setminus \left\{ 0 \right\}=\bigcup_{i\in
\mathbb{Z}}\Gamma(p_i,p_{i+1})$ almost surely, note that $p_i\rightarrow 0$
as $i\rightarrow -\infty$ and $p_i\rightarrow \infty$ as
$i\rightarrow \infty$.
\end{proof}

Next, we develop some notation to keep track of the lengths of the segments
of the infinite geodesic and this will be useful later when we deal with the
$\log$-parametrization of $\Gamma$. 
{We define for all
$i\in \mathbb{N}\cup \left\{ 0 \right\}$ and a whole plane GFF marginal}
$\mathtt{h}$,
\begin{equation}
	\label{e:3.5.14}
	(L_{i+1}-L_i)(\mathtt{h}):=D_\mathtt{h}(p_{i}(\mathtt{h}),p_{i+1}(\mathtt{h})).
\end{equation}
that is, $(L_{i+1}-L_i)(\mathtt{h})$ is the LQG length between the $i$th and $(i+1)$th
coalescence points of $\mathtt{h}$.
For the whole plane GFF $h$, analogously, let, for
all $i \in \mathbb{Z}$ 
\begin{equation}
	\label{e:3.5.13}
{L_{i+1}-L_i:=D_h(p_{i},p_{i+1}).}
\end{equation}
Thus \begin{equation}
	\label{e:3.5.13*}
	L_i=\sum_{j=-\infty}^{i-1}(L_{j+1}-L_j).
\end{equation}
We now observe that $L_i=D_h(0,p_i)$
using 
Lemma \ref{3.12.1} along with the
continuity of the LQG metric with respect to the Euclidean metric and that
$p_{i}\rightarrow 0$ as $i\rightarrow -\infty$.

For a whole plane GFF marginal $\mathtt{h}$, we define 
{\begin{equation}
	\label{e:3.5.15^}
	Y_i(\mathtt{h}):= \mathcal{P}_{i}(\mathtt{h})
	\left(L_1-L_0\right) (\widehat{\mathtt{h}}_{S_i}\mydot
	\psi_{S_i}).
\end{equation}}
{For the case of the whole plane GFF $h$, analogously we now introduce the variables $Y_i$,} by defining
$Y_i=Y_0(\mathcal{H}_i)$ for all $i\in \mathbb{Z}$ and
thus by definition, we can write
\begin{equation}
	\label{e:3.5.15}
	Y_i= \mathcal{P}_{i}
	\left(L_1-L_0\right) (\mathcal{H}_i) =\mathcal{P}_i
	D_{\mathcal{H}_i}(p_0(\mathcal{H}_i),p_1(\mathcal{H}_i)).
\end{equation}
The definition is formulated in this way to
ensure that $Y_i$ is measurable with respect to $\mathcal{H}_i$.
Indeed, this is because $\mathcal{P}_{i}=
\mathcal{P}_{0}(\mathcal{H}_i)$. The reasoning behind adding the
$\mathcal{P}_i$ in the above definition is that we want each of the
segments of $\Gamma$ to contribute to exactly one $Y_i$; this will be
convenient later as it will allow us to simply sum over all $i$ while
ensuring that each segment is accounted for exactly once. We also note that the sequence $Y_i$ is stationary and we shall later show in
Lemma \ref{3.15.1} that if we instead consider the sequence $\left\{ \log Y_i
\right\}_{i\in \mathbb{Z}}$ then we also have exponential decay of correlations.

The following lemma is a consequence of the coordinate change formula
(Proposition \ref{b3} (4)) and relates the LQG length of a segment to the
corresponding value $Y_i$.
\begin{lemma}
	\label{3.12.2}
	For all $i\in \mathbb{Z}$, almost surely,
	\begin{displaymath}
			\mathcal{P}_{i}\left(
	L_{\eta(i)+1}-L_{\eta(i)}\right)=e^{\xi Q\log S_i
	+\xi B(\log S_i)}Y_i.
	\end{displaymath}
\end{lemma}
\begin{proof}
	By \eqref{e:3.5.15}, both the sides are $0$ in case
	$\mathcal{P}_i=0$ and thus we can assume that
	$\mathcal{P}_i=1$. 
	Thus, using the coordinate change formula from Lemma \ref{b4} together with Lemma \ref{3.12.0}, we obtain the claimed equality.
\end{proof}

For a whole plane GFF marginal $\mathtt{h}$, on using the definition of the event $E_r[\mathtt{h}]$ as in \eqref{e:E_r}, we
have that the variables $Y_i(\mathtt{h})$ are uniformly bounded below when
they are positive: with the constant $\rho$ coming from \eqref{e:E_r}, \begin{equation}
	\label{e:3.15.1.1}
	Y_i(\mathtt{h})\geq \rho \mathcal{P}_{i}(\mathtt{h})
\end{equation}
holds deterministically for all $i\geq 0$ (Indeed, this was the reason
we switched from $F_r[\mathtt{h}]$ to $E_r[\mathtt{h}]$.).
Following our convention, we
now define a bi-infinite sequence $\mathcal{Y}$ such that $\mathcal{Y}_i=Y_i$. In order
to obtain a correlation decay statement for the stationary sequence $\mathcal{Y}$ analogous to Lemma \ref{3.12} (actually, we
will consider $\log \mathcal{Y}$ instead of $\mathcal{Y}$ where $(\log \mathcal{Y})_{i}=\log Y_{i}$), we shall need
tail bounds for $Y_i$ to invoke Lemma \ref{'3.163}, and this is what we obtain now.
\subsection{Power law tails for $Y_i$}
\label{ss:Y_i}
Since $Y_i$s are identically distributed, it suffices to obtain power law tail estimates for
$Y_0$. Let $J$ denote the index of the second $1$ observed in the
sequence $\mathcal{P}$, that is, we define
$J=\eta^{-1}(1)$, where $\eta$ and its inverse were defined in the
discussion around \eqref{e:3.5.12.0}.

\begin{lemma}
	\label{3.13}
	We have that {$p_1\in \overline{\mathbb{D}}_{\leq
	S_{J+1}}$}. Also, there exist constants $C,c$ such that we have
$\mathbb{P}\left( J\geq t \right)\leq Ce^{-ct}$
	for all $t\geq 0$.	

\end{lemma}
\begin{proof}
	We know that $p_1\in
	\mathbb{C}_{[S_{J},KS_{J}]}$ since
	$J=\eta^{-1}(1)$ and this shows that $p_1\in \mathbb{C}_{\leq
	S_{J+1}}$. The tail bound for $J$ is a consequence
	of Lemma \ref{3.11}. 	
\end{proof}

Let  $\mathfrak{D}_r$ refer to the LQG diameter of the disk
$\overline{\mathbb{D}}_r$, that is,
\begin{equation}
	\label{eqn:3.5.151}
	\mathfrak{D}_r=\sup_{x,y\in
	\overline{\mathbb{D}}_r}D_h(x,y).
\end{equation}
We can relate the distribution of $\mathfrak{D}_r$ for different values of $r$
by using the scale invariance of $h$ along with the Weyl scaling and the coordinate change
formula recorded in Proposition \ref{b3}. Indeed for the case of a
deterministic positive radius $r$,
\begin{equation}
	\label{e:3.13.2}
	\mathfrak{D}_r\stackrel{d}{=}e^{\xi Q \log r + \xi B(\log
	r)}\mathfrak{D}_1.
\end{equation}
We now complete the proof of the power law tail for $Y_0$.
\begin{lemma}
	\label{3.13.3}
	There exists a positive constant $c$ such that $\mathbb{P}\left( Y_0\geq
	t
	\right)\leq t^{-c}$ for all $i$ and all $t$ large enough. Thus there
	exists a constant $\theta>0$ such that
	\begin{displaymath}
		\mathbb{E}Y_0^\theta<\infty.
	\end{displaymath}
\end{lemma}
\begin{proof}
	Fix a $\delta>0$ satisfying $2\xi Q\delta \log K<1/2$. By Lemma
	\ref{3.13},
$p_0(\mathcal{H}_0),p_1(\mathcal{H}_0)\in  \overline{\mathbb{D}}_{\leq
S_{J+1}}$ and this implies $D_h(p_0(\mathcal{H}_0),p_1(\mathcal{H}_0))\leq
\mathfrak{D}_{J+1}$. By using the above along with $\mathcal{P}_0\leq 1$,
we can write for some constant $c>0$,
\begin{align}
	\label{e:repl}
	\mathbb{P}\left( Y_0\geq t \right)&\leq \mathbb{P}\left(
	J+1\geq \delta \log t \right)+\mathbb{P}\left(
	\mathfrak{D}_{K^{\delta\log
	t}}\geq t \right) \\\nonumber
	 &\leq\mathbb{P}\left(
	J+1\geq \delta\log t \right)+\mathbb{P}\left(
	\exp\left(\xi B(\log K^{\delta\log t})\right) \mathfrak{D}_{1}\geq
	tK^{-\xi Q\delta\log t} \right)\\\nonumber
	&\leq\mathbb{P}\left(
	J+1\geq \delta\log t \right)+\mathbb{P}\left( B(\log K^{\delta\log t})\geq
		Q\log (K^{\delta \log t})
		\right)+\mathbb{P}\left( \mathfrak{D}_{1}\geq
	tK^{-2\xi Q\delta\log t} \right)\\\nonumber
	&\leq\mathbb{P}\left(
	J+1\geq \delta\log t \right)+\mathbb{P}\left( B(\log K^{\delta\log t})\geq
		Q\log (K^{\delta \log t})
		\right)+\mathbb{P}\left( \mathfrak{D}_{1}\geq
	t^{1-2\xi Q\delta\log K} \right)\\\nonumber
	&\leq\mathbb{P}\left(
	J+1\geq \delta\log t \right)+\mathbb{P}\left( B(\log K^{\delta\log t})\geq
		Q\log (K^{\delta \log t})
		\right)+\mathbb{P}\left( \mathfrak{D}_{1}\geq
	t^{1/2} \right)\\\nonumber
	&\leq t^{-c}.
\end{align}
The second line is obtained by using \eqref{e:3.13.2} while the fifth line
is obtained by using that $2\xi Q\delta \log K<1/2$. The first term in the
fifth line is bounded using Lemma \ref{3.13}, while the second
term is bounded using the subexponential tails for a Gaussian and the bound for the third
term is a consequence of Markov's inequality along with the moment bound for
$\mathfrak{D}_1$ 
delivered by Proposition \ref{imp2}.
\end{proof}

\subsection{The $\log$-lengths of the segments}
\label{ss:lengths}
As mentioned in Section \ref{s:iop}, we will consider the $\log$-length between any two successive coalescence
points $p_j$ and $p_{j+1}$. We use $G_i$ to denote a variant of the latter which we formally define as follows.
\begin{equation}
	\label{e:3.17}
	G_i:=\mathcal{P}_{i}\left(
	\log L_{\eta(i)+1}-\log L_{\eta(i)}\right)
\end{equation}
In other words, $G_i$ is non-zero only if there is a coalescence point at
scale $i$. In case $G_i$ is non-zero,
it is equal $\log D_h(0,p_{\eta(i)+1})-\log
D_h(0,p_{\eta(i)})$, where recall that $p_{\eta(i)}$ and
$p_{\eta(i)+1}$ are the first and second coalescence points in the region
$\mathbb{C}_{>S_i}$. 
We will later in Lemma \ref{3.15.1.0} show that the bi-infinite sequence $\mathcal{G}$ corresponding
to the $G_i$s is a stationary sequence with exponential decay of
correlations. Thus the latter formalizes the notion that in the log-parametrization, the infinite geodesic
$\Gamma$ spends approximately i.i.d.\ length between two successive coalescence points which, as alluded to in Section \ref{s:iop} is the reason for using the $\log $
parametrization while considering the empirical observables.

Unlike the variables $\cP_i=\mathcal{P}_0(\mathcal{H}_i)$ the variables $G_i$ depend on all the past scales before $i$. This is quantified in the  following lemma.
\begin{lemma}$\empty$
	\label{3.15.0}
	\begin{displaymath}
	G_i=\log \left(
	1+\frac{Y_i}{\sum_{j=-\infty}^{i-1} Y_j e^{\xi Q(\log
	S_j-\log S_i)+\xi(B(\log S_j)-B(\log S_i))}} \right)
	\end{displaymath}
\end{lemma}
\begin{proof}
		First note that in case $\mathcal{P}_i=0$, we have $Y_i=0$ by 
	definition (see \eqref{e:3.5.15}) and thus the expression in the
	statement of the lemma yields $G_i=0$ which matches with \eqref{e:3.17}. Next assuming $\mathcal{P}_i=1$, by an
	application of Lemma \ref{3.12.2}, we have
	\begin{gather}
		L_{\eta(i)+1}-L_{\eta(i)}=e^{\xi Q\log S_i
	+\xi B(\log S_i)}Y_i,\,\,
	L_{\eta(i)}=\sum_{j=-\infty}^{i-1} Y_j e^{\xi Q\log
	S_j+\xi B(\log S_j)},
	\label{e:3.171}
	\end{gather}
	which, substituted into  $\log L_{\eta(i)+1}-\log L_{\eta(i)}=\log\left(1+\frac{L_{\eta(i)+1}-L_{\eta(i)}}{L_{\eta(i)}}\right)$ yields the
	result. 
\end{proof}
As seen from the above lemma, the expression for $G_i$ depends on the values
of $Y_j$ for all $j<i$ which makes it difficult to directly apply
Lemma \ref{'3.163}. To serve as a warm-up, we first discuss an easier setting in which Lemma \ref{'3.163} directly applies. Instead of
$G_i=\mathcal{P}_{i}\left(
	\log L_{\eta(i)+1}-\log L_{\eta(i)}\right)$, we consider
	\begin{displaymath}
		\mathcal{P}_{i}
	\log\left(L_1-L_0\right) (\mathcal{H}_i)=\log Y_i
	\end{displaymath}
	when $Y_i\neq 0$. Since $Y_i=Y_0(\mathcal{H}_i)$, the above depends only on $\mathcal{H}_i$.

	Notice that $\log \mathcal{Y}_i$ has the problem of being $-\infty$ in
	the case when $\mathcal{Y}_i=0$. To handle this, we adopt the convention 
$0\times(-\infty)=0$ and introduce the sequence $\mathbbm{1}(\mathcal{Y})$
such that $(\mathbbm{1}(\mathcal{Y}))_i=\mathbbm{1}_{\mathcal{Y}_i\neq 0}$.
The
following lemma, an analogue of Lemma \ref{3.12}, is now formulated for
the sequence (coordinate-wise product) $\mathbbm{1}(\mathcal{Y})\log \mathcal{Y}$.
\begin{lemma}
	\label{3.15.1}

The sequence $\mathbbm{1}(\mathcal{Y})\log
\mathcal{Y}$ is stationary with
exponential decay of correlations.
\end{lemma}
\begin{proof}
	The structure of the proof is the same as the proof of Lemma
	\ref{3.12} and consists on checking the conditions needed to invoke
	Lemma \ref{'3.163} and we will directly use notation from the
	setting of Lemma \ref{'3.163}. For a whole plane GFF marginal
	$\mathtt{h}$, we define
	\begin{equation}
		\label{e:3.1711}
		X(\mathtt{h})=\mathbbm{1}_{Y_0(\mathtt{h})\neq 0}\log
		Y_0(\mathtt{h}).
	\end{equation}
	and we arbitrarily extend this definition outside the support of the GFF
	to yield a measurable function defined on
	$\mathcal{D}'(\mathbb{C}_{>1})$.
	The first condition that needs to be checked is that
	$\mathbbm{1}_{Y_0(\mathtt{h})\neq 0}\log Y_0(\mathtt{h})$ has finite fourth
	moment. This is true because $Y_0=Y_0(\mathtt{h})$ has power law tails by
	Lemma \ref{3.13.3} along with the fact that on the event $\left\{
	Y_0(\mathtt{h})\neq
	0 \right\}$, we have $Y_0(\mathtt{h})\geq \rho$ by \eqref{e:3.15.1.1}. 
	
We finally need to define a measurable function
$\tau: \mathcal{D}'(\mathbb{C}_{>1})\rightarrow \mathbb{R}_+$ such that the conditions needed to invoke Lemma
	\ref{'3.163} are satisfied. Notice here that the definition
	$\tau\equiv 1$
	from the proof of Lemma \ref{3.12} does not work in this setting. This is
	because for a whole plane GFF marginal $\mathtt{h}$, $Y_0(\mathtt{h})$
	when positive is equal to the
	$D_\mathtt{h}(p_0(\mathtt{h}),p_1(\mathtt{h}))$ and the latter might be
	arbitrarily large. We now
	define $\tau(\mathtt{h})$ for a whole plane GFF marginal $\mathtt{h}$ by 
	\begin{equation}
		\label{e:S'}
		\tau(\mathtt{h})=1+\iota(1,\mathtt{h})=1+\inf_{i\geq 1}\left\{
		E_{S_i}[\mathtt{h}]~\mathrm{
	 occurs } \right\},
	\end{equation}
	and we arbitrarily extend this definition outside the support of the GFF to
	yield a measurable function on 
	$\mathcal{D}'(\mathbb{C}_{>1})$. By Lemma \ref{3.11}, $\tau(\mathtt{h})$ has exponential
	 tails. 
	 
	 One condition that we need to check is that $\tau(\mathtt{h})$ is a
	 stopping time for $\mathscr{G}(\mathtt{h})$. Note that by Lemma
	 \ref{Emeas},
	 the coalescence event $E_r[\mathtt{h}]$ is measurable with respect to
	 $\left(\mathtt{h}-\mathbf{Av}(\mathtt{h},\mathbb{T}_{r})\right)\lvert_{\mathbb{C}_{(K^{1/8}r,Kr)}}$ for any $r>1$ which in
	 particular implies that
	 $E_{S_i}[\mathtt{h}]$ is measurable with respect to
	 $\mathscr{G}_{i+1}(\mathtt{h})=\sigma(\mathtt{h}\lvert_{\mathbb{C}_{(1,
	 S_{i+1}]}})$. The above implies that
	 \begin{equation}
		 \label{e:SST}
		 \left\{ \tau(\mathtt{h})\leq j \right\}=\left\{
		 E_{S_i}[\mathtt{h}] \text{ occurs
		 for some } 1\leq i\leq j-1\right\}\in \mathscr{G}_j(\mathtt{h})
	 \end{equation} as desired. 
	 The final condition that we need to verify is that $X(\mathtt{h})$ is measurable with respect to
$\mathscr{G}_{\tau(\mathtt{h})}(\mathtt{h})$. We we will use that on the event $\left\{ Y_0(\mathtt{h})\neq 0 \right\}=\left\{
	 \mathcal{P}_0(\mathtt{h})=1 \right\}$, we have
	 $\tau(\mathtt{h})=1+\eta^{-1}(1,\mathtt{h})$. Thus by the definition
	 of the stopped $\sigma$-algebra
	 $\mathscr{G}_{\tau(\mathtt{h})}(\mathtt{h})$, it suffices to show that for any $s>0$ and $n\in \mathbb{N}\cup
	 \left\{ 0 \right\}$, the event $\mathcal{E}_{s,n}[\mathtt{h}]=\left\{
Y_0(\mathtt{h})\in (0,s) \right\}\cap \left\{ \tau(\mathtt{h})\leq n \right\}=\left\{
Y_0(\mathtt{h})\in (0,s) \right\}\cap \left\{ 1+\eta^{-1}(1,\mathtt{h})\leq n \right\}$
is measurable with respect to $\mathscr{G}_{n}(\mathtt{h})$. To see this,
note that, on $\mathcal{E}_{s,n}[\mathtt{h}]$, we have
$D_\mathtt{h}(p_0(\mathtt{h}),p_1(\mathtt{h}))=D_{\mathtt{h}\lvert_
{\mathbb{C}_{(1,S_n)}}}(p_0(\mathtt{h}),p_1(\mathtt{h}))$. The above equality is
obtained by using locality along with the fact (Lemma
\ref{3.12.1}) that there is a.s.\ a unique $D_\mathtt{h}$
geodesic $\Gamma(p_0(\mathtt{h}),p_1(\mathtt{h});\mathtt{h})\subseteq
\mathbb{C}_{(1,1+\eta^{-1}(1,\mathtt{h}))}$ and the latter is a
subset of
$\mathbb{C}_{(1,S_n)}$ on the event $\mathcal{E}_{s,n}[\mathtt{h}]$. We
further note that
both $p_0(\mathtt{h})$ and $p_1(\mathtt{h})$ are measurable with respect to
$\mathtt{h}\lvert_
{\mathbb{C}_{(1,S_n)}}$ on the event $\mathcal{E}_{s,n}[\mathtt{h}]$ as a consequence of Lemma
\ref{3.10.0}. This completes the verification of the measurability of
$X(\mathtt{h})$ with respect to
$\mathscr{G}_{\tau(\mathtt{h})}(\mathtt{h})$. Thus we can now apply Lemma
\ref{'3.163} and complete the proof.
\end{proof}

As noted above, Lemma \ref{'3.163} does not directly apply to the
case of $G_i$ and we will use Lemma \ref{3.163} instead. To implement this, we
will approximate $G_i$ by 
$\underline{G}_{m,i}$ which will only depend on $m+1$ many scales to the
past of $\mathcal{H}_i$. For notational
convenience, we first define the variables $D_i$ by
\begin{equation}
	\label{e:3.172}
	D_i=\xi Q(\log S_i-\log S_{i+1}) +\xi
(B(\log S_i)-B(\log S_{i+1}))=\xi
(B(\log S_i)-B(\log S_{i+1}))-\xi Q\log K;
\end{equation}
note that the sequence $\left\{ D_k \right\}_{k\in \mathbb{Z}}$ is
stationary by the i.i.d.\ nature of Brownian increments. We will interpret $D_i$ as the increments $\underline{B}(\log
S_{i+1})-\underline{B}(\log
S_i)$ of a Brownian motion with negative drift, i.e.,
$\underline{B}(t)=-\xi B(t)-\xi Q t$. Analogously we  can also define the variables $D_i(\mathtt{h})$ for a
whole plane GFF marginal $\mathtt{h}$ for all $i\geq 0$ by replacing the Brownian motion $B$ by
$B_\mathtt{h}$ (recall from the discussion following \eqref{e:b4.1*}) in the expression \eqref{e:3.172}.

With the above notational preparation, we rewrite the expression in Lemma \ref{3.15.0} as
\begin{equation}
	\label{e:3.5.181}
	G_i=\log \left(
	1+\frac{Y_i}{\sum_{j=-\infty}^{i-1} Y_j
	e^{\sum_{k=j}^{i-1}D_k}}
	\right).
\end{equation}
Given the above, we define the  finitary versions of ${G}_{i}$ by
\begin{equation}
	\label{e:3.15.21}
	\underline{G}_{m,i}:= 
	\begin{cases}
	\log \left(
	1+\frac{Y_i}{\sum_{j=i-1-m}^{i-1} Y_j
	e^{\sum_{k=j}^{i-1}D_k}} \right), &\text{if } Y_j>0
	\text{ for some } i-1-m\leq j \leq i-1\\
	0, &\text{otherwise.}
	\end{cases}
\end{equation}
Note that $\underline{G}_{m,i}$ depends
only on the fields $\left\{ \mathcal{H}_{i-1-m},\dots,\mathcal{H}_{i}
\right\}$, that is, only $m+1$ scales to the past of the current scale $i$.
One also has the following alternate expression for $\underline{G}_{m,i}$.
\begin{lemma}
	\label{*3.15}
	On the event $\left\{\underline{G}_{m,i}>0\right\}$, we
have
\begin{displaymath}
	\underline{G}_{m,i}=\mathcal{P}_i\left( \log 
	D_h(p_{\iota(i-1-m)},p_{\eta(i)+1}) - \log 
	D_h(p_{\iota(i-1-m)},p_{\eta(i)}) \right)
\end{displaymath}
\end{lemma}
\begin{proof}
	Note that restricting to the event $\left\{\underline{G}_{m,i}>0\right\}$ implies that
	$\underline{G}_{m,i}$ is equal to the first expression in
	\eqref{e:3.15.21}. The proof now follows by noting that
	$D_h(0,p_{\iota(i-1-m)})= \sum_{j=-\infty}^{{i-m-2}} Y_j
	e^{\xi Q\log
	S_j+\xi B(\log S_j)}$, and reversing the steps in the proof of Lemma
	\ref{3.15.0}. 
\end{proof}

To be able to use Lemma \ref{3.163}, we finally need to
assert that $\underline{G}_{m,i}$ gets close to $G_{i}$ as
$m\rightarrow\infty$.  Looking at the expression for $G_i$ and recalling that the $D_i$ are the
increments of a Brownian motion with a negative drift, one would expect
that replacing $\sum_{j=-\infty}^{i-1} Y_j
	e^{\sum_{k=j}^{i-1}D_i}$ by $\sum_{j=i-1-m}^{i-1} Y_j
	e^{\sum_{k=j}^{i-1}D_i}$ incurs only a $(1-O(e^{-cm}))$ multiplicative error. This is what we prove next. Define the relevant quantity $\alpha_{m,n}$ for all
	$m\in \mathbb{N},n\in \mathbb{Z}$ by
		\begin{equation}
			\label{e:3.15.223}
			\alpha_{m,n}=\frac{\sum_{j=-\infty}^{n-1-m} Y_j
	e^{\sum_{k=j}^{n-1}D_k}}{\sum_{j=-\infty}^{n-1} Y_j
	e^{\sum_{k=j}^{n-1}D_k}}.
		\end{equation}

	\subsection{Showing that $\alpha_{m,n}$ is $O(e^{-cm})$}
	\label{ss:alpha}
We start with the following proposition.
	\begin{proposition}
		\label{3.15.1.-1.2}
		There exist constants $c_1,c_2,C_1,C_2$ positive such that for all
		$m\in \mathbb{N}, n\in \mathbb{Z}$, we have
		\begin{displaymath}
			\mathbb{P}\left( \alpha_{m,n}\geq C_1e^{-c_1 m}\right)\leq
			C_2e^{-c_2m}.
		\end{displaymath}
	\end{proposition}
	Since $\alpha_{m,n}\leq 1$ deterministically, we immediately get the following moment estimate for $\alpha_{m,n}$. 
	
	\begin{lemma}
		\label{3.15.1.-1.21}
		For any $p>0$, we have that there exist positive constants $c,C$
		depending on $p$ such that for all $m\in \mathbb{N},n\in
		\mathbb{Z}$,
		\begin{displaymath}
			\mathbb{E}\alpha_{m,n}^p\leq Ce^{-cm}.
		\end{displaymath}
	\end{lemma}

	The proof of Proposition \ref{3.15.1.-1.2} is rather technical and we now
	give a brief roadmap. To show that $\alpha_{m,n}$ is
	small, one needs to show that the numerator is small and the
	denominator is relatively large. We use the power law tails for $Y_{i}$ coming from Section \ref{ss:Y_i} to show that the $Y_i$s are not too large in Lemma \ref{3.15.1.-1.3} and using this to upper bound the numerator in Lemma \ref{1*}. Lemma \ref{3.15.1.-1.4} shows that the
	denominator is large. These are then combined to complete the proof.

\begin{lemma}
	\label{3.15.1.-1.3}
	For all $m\in \mathbb{N},n\in \mathbb{Z}$, $-\infty\leq j\leq n-1-m$, and any $\delta>0$ we have
	\begin{displaymath}
		\mathbb{P}\left( Y_j \geq e^{ -\delta \sum_{k=j}^{n-1}D_k}\right) \leq
		Ce^{-c(n-j)}
	\end{displaymath}
	for some positive constants $c,C$ depending on $\delta$.
\end{lemma}
\begin{proof}
	For any $\alpha_1>0$, we have
	\begin{align}
		\mathbb{P}\left( Y_j \geq e^{ -\delta\sum_{k=j}^{n-1}D_k
		}\right)&\leq \mathbb{P}\left(
		\sum_{k=j}^{n-1}D_k\geq-\alpha_1  (n-j) \log K \right)+
		\mathbb{P}\left( Y_j\geq e^{-\alpha_1 \delta (n-j)\log K} \right)\nonumber\\
		\label{e:3.5.183}
		&\leq \mathbb{P}\left(
		\underline{B}(\log S_{n})-\underline{B}(\log S_j)\geq-\alpha_1
		(n-j) \right)+
		e^{-c_1\alpha_1 (n-j)\log K}\nonumber\\
		&\leq \mathbb{P}\left(
		B(\log S_{n})-B(\log S_j)\geq (\xi Q-\alpha_1)
		(n-j)\log K\right) + e^{-c_2\alpha_1 (n-j)},
	\end{align}
	where $c_1$ is a constant depending on $\delta$ and $K$ as obtained from Lemma
	\ref{3.13.3}. 
	Now fix
	$\alpha_1=\xi Q/2$ and observe that the first term in \eqref{e:3.5.183}, by
	a standard estimate for the Gaussian tail, can be bounded by

	\begin{equation}
		\label{eeee:1}
		\mathbb{P}\left(
		B(\log S_{n})-B(\log S_j)\geq \frac{\xi Q}{2}
		(n-j)\log K\right)\leq C_3e^{-c_3(n-j)}.
	\end{equation}
	On combining \eqref{eeee:1} and \eqref{e:3.5.183}, the proof is complete.	
\end{proof}

\begin{lemma}
	\label{1*}
	For any $\delta\in (0,1)$, there exist positive
	constants $c_1,c_2,C_1$ such that
	\begin{displaymath}
		\mathbb{P}\left( \sum_{j=-\infty}^{n-1-m} Y_j
		e^{\sum_{k=j}^{n-1}D_k} \geq
		c_2\exp\left\{ -(1-\delta)\xi Q(\log S_n-\log S_{n-1-m})
		\right\}\right)\leq C_1e^{-c_1m}
	\end{displaymath}
	for all $m\in \mathbb{N}, n\in \mathbb{Z}$.
\end{lemma}
\begin{proof}
	First, note that by a union bound, it suffices to show the following two
	inequalities in order to complete the proof:
	\begin{gather}
		\label{e*:2}
		\sum_{j=-\infty}^{n-1-m}\mathbb{P}\left( Y_j \geq e^{
		-\frac{\delta}{2} \sum_{k=j}^{n-1}D_k}\right)\leq C_1'e^{-c_1'm},\\
	\label{e*:2.1}
	\mathbb{P}\left( \sum_{j=-\infty}^{n-1-m}
		e^{(1-\delta/2)\sum_{k=j}^{n-1}D_k} \geq
		c_2\exp\left\{ -(1-\delta)\xi Q(\log S_n-\log S_{n-1-m})
		\right\}\right)\leq C_2'e^{-c_2'm}.
	\end{gather}
	The first inequality is true for any fixed value of $\delta$ by an
	application of Lemma \ref{3.15.1.-1.3}. To show \eqref{e*:2.1},  note that since
	$(1-\delta)\leq (1-\delta/2)^2$, it suffices to instead show that
	\begin{equation}
		\label{e*:2.1.1}
			\mathbb{P}\left( \sum_{j=-\infty}^{n-1-m}
		e^{(1-\delta/2)\sum_{k=j}^{n-1}D_k} \geq
	c_2\exp\left\{ -(1-\delta/2)^2\xi Q(\log S_n-\log S_{n-1-m})
	\right\}\right)\leq C_2'e^{-c_2'm}.
	\end{equation}
	To obtain the bound \eqref{e*:2.1.1}, we will
	use that with high probability, $\sum_{k=j}^{n-1}D_k\leq
	-(1-\delta/2)\xi Q
		(n-j)\log K$ for all $j\in[\![-1,n-1-m]\!]$ by  standard estimates for Brownian motion.
			\begin{align}
		\label{e1:1}
		&\mathbb{P}\left( \max_{
		j\in[\![-\infty,n-1-m]\!]}\left\{\sum_{k=j}^{n-1}D_k+(1-\delta/2)\xi Q
		(\log S_n-\log S_{j})\right\} \geq 0\right)\nonumber\\
		&= \mathbb{P}\left( \max_{
		j\in[\![-\infty,n-1-m]\!]}\left\{B(\log S_j)-B(\log
		S_n)-\frac{\delta}{2} Q
		(\log S_n-\log S_{j})\right\} \geq 0\right)\nonumber\\
		&=\mathbb{P}\left( \max_{
		j\in[\![-\infty,n-1-m]\!]}\left\{B(\log S_n-\log
		S_j)-\frac{\delta}{2} Q
		(\log S_n-\log S_{j})\right\} \geq 0 \right)\nonumber\\
		&\leq C_3'e^{-c_3'm\log K}\leq C_4'e^{-c_4'm}
	\end{align}

	We have used the Markov property and symmetry of Brownian motion to
	obtain the second equality. For the penultimate inequality, we are just
	using that for all $M$ large,
	\begin{displaymath}
		\mathbb{P}\left( B(t)\leq t \text{ for all } t\in [M,\infty)
		\right)\geq 1-C_4'e^{-c_4' M},
	\end{displaymath}
	and this can be obtained by using the reflection principle after
	performing a time inversion. 

	In view of \eqref{e1:1}, to complete the proof of
	\eqref{e*:2.1.1}, it suffices to show that for some choice of $c_2$, 
	\begin{equation}
		\label{e*:2.2}
		 \sum_{j=-\infty}^{n-1-m}
		\exp\left\{-(1-\delta/2)^2\xi Q(\log S_n - \log
		S_j)\right\} <
	c_2\exp\left\{ -(1-\delta/2)^2\xi Q(\log S_n-\log
		S_{n-1-m}) \right\}.
	\end{equation}
	By using that $\log S_{i+1}-\log S_i=\log K$ for all $i$, we obtain that
	the above holds with $c_2=\sum_{j=0}^{\infty}
	e^{ -\xi Q(1-\delta/2)^2j\log K}$ and this completes the proof.
\end{proof}

\begin{lemma}
	\label{m1}
	There exist
	constants $c,C>0$ such that 
	\begin{displaymath}
		\mathbb{P}\left(
		Y_j=0 \text{ for all } j\in [\![n-m-1,n-1]\!] \right) \leq Ce^{-cm}
	\end{displaymath}
	 for all $m\in \mathbb{N},n\in \mathbb{Z}$.
\end{lemma}
\begin{proof}
	Note that 
	\begin{align}
		\label{e*:2.31}
		\left\{Y_j=0 \text{ for all } j\in [\![n-m-1,n-1]\!]
		\right\}&= \left\{  \inf_{i> n-m-1}\left\{ E_{S_i}~\mathrm{
		occurs } \right\} -(n-m-1)> m  \right\}\nonumber\\
		&=	\left\{  \inf_{i\geq 0}\left\{ E_{S_i}(\mathcal{H}_{n-m-1})~\mathrm{
		occurs } \right\}> m  \right\}
	\end{align}
	and the result now follows by an application of Lemma \ref{3.11}.
\end{proof}

\begin{lemma}
	\label{3.15.1.-1.4}
	For any $\delta\in (0,1)$, we have that there exist positive constants
	$c_1,c_2,C_1$ depending on $\delta$ such that
	\begin{displaymath}
		\mathbb{P}\left( \sum_{j=n-1-m/2}^{n-1} Y_j
		e^{\sum_{k=j}^{n-1}D_k} \leq
		c_2\exp\left\{-(1+\delta)\xi Q(\log S_{n}-\log
		S_{n-1-m/2})\right\}\right)\leq C_1e^{-c_1m}.
	\end{displaymath}
	for all $m\in \mathbb{N}, n\in \mathbb{Z}$.
\end{lemma}
\begin{proof}We have
	\begin{equation}
		\label{e:3.5.186}
		\sum_{j=n-1-m/2}^{n-1} Y_j
		e^{\sum_{k=j}^{n-1}D_k}\geq \left(\sum_{j=n-1-m/2}^{n-1} Y_j\right) \exp\left(\min_{j\in[\![n-1-m/2,n-1]\!]}\left\{
		\sum_{k=j}^{n-1}D_k\right\} \right).
	\end{equation}
 To avoid
unnecessary clutter, we define $A_{m,n}$ by 
\begin{equation}
	\label{e:3.5.1891}
	A_{m,n}=	\min_{j\in[\![n-1-m/2,n-1]\!]}\left\{
		\sum_{k=j}^{n-1}D_k\right\}.
\end{equation}
By \eqref{e:3.5.186}, it suffices to show the following two inequalities:
	for some positive constants $c_1,c,C_1',C_2'$
	\begin{gather}
		\label{e:3.5.187}
		\mathbb{P}\left(\sum_{j=n-1-m/2}^{n-1} Y_j\leq c_2\right)\leq
		C_1'e^{-c_1'm},\\
		\label{e:3.5.1871}
		\mathbb{P}\left( A_{m,n}\leq -(1+\delta)\xi Q(\log S_{n}-\log
		S_{n-1-m/2}) \right)\leq C_2'e^{-c_2'm}.
	\end{gather}

	We first obtain the inequality \eqref{e:3.5.187} with $c_2=\rho/2$,
	where recall $\rho$ in \eqref{e:3.15.1.1} was such that $Y_i\geq \rho$ on the event
	$\left\{ Y_i\neq 0 \right\}$ for any $i\geq 0$. Using Lemma
	\ref{m1} with $m/2$ in place of $m$ yields 
	\begin{equation}
		\label{e:3.5.1872}
		\mathbb{P}\left(\sum_{j=n-1-m/2}^{n-1} Y_j<\rho\right)\leq
		C_1'e^{-c_1'm}
	\end{equation}
	and this implies \eqref{e:3.5.187}. It now remains to show \eqref{e:3.5.1871}. Note that by the definition of the $D_k$s, we have 
	\begin{equation}
		\label{e:3.5.18931}
		A_{m,n}\geq \xi\min_{j\in[\![n-1-m/2,n-1]\!]}\left\{  B(\log S_j)-
		B(\log S_{n})\right\}-\xi Q\left( \log S_{n}-\log S_{n-1-m/2} \right).
	\end{equation}
	Again, by using the Markov property of Brownian motion along with the reflection
	principle, 
	\begin{align}
		\label{e:3.5.1894}
		&\mathbb{P}\left( \min_{j\in[\![n-1-m/2,n-1]\!]}\left\{  B(\log S_j)-
		B(\log S_{n})\right\} \leq -\xi Q \delta\left( \log S_{n}-\log
		S_{n-1-m/2} \right) \right)\nonumber\\
		&\leq
		C_3'e^{-c_3'm\log K /2}=C_4'e^{-c_4' m},
	\end{align}
	The above implies
	\begin{equation}
		\label{e:3.5.1895}
		\mathbb{P}\left( A_{m,n}\leq -(1+\delta)\xi Q(\log S_{n}-\log
		S_{n-1-m/2}) \right)\leq C_4'e^{-c_4' m}
	\end{equation}
	and this completes the proof.
\end{proof}

We can now complete the proof of Proposition \ref{3.15.1.-1.2}.

\begin{proof}[Proof of Proposition \ref{3.15.1.-1.2}]
	It suffices to show that there exist constants $c,c',C,C'$ positive such
	that
	\begin{equation}
		\label{e:3.5.1892}
		\mathbb{P}\left( \frac{\sum_{j=-\infty}^{n-1-m} Y_j
	e^{\sum_{k=j}^{n-1}D_k}}{\sum_{j=n-1-m/2}^{n-1} Y_j
	e^{\sum_{k=j}^{n-1}D_k}}\geq Ce^{-cm}  \right)\leq C'e^{-c'm}.
	\end{equation}

Fix $\delta=1/5$ and note that this satisfies $1-\delta>(1+\delta)/2$ and
thus this implies that there are positive constants $c_1,C_1$ such that we have
\begin{equation}
	\label{eeee:2}
	\frac{\exp\left\{ -(1-\delta)\xi Q(\log S_n-\log
		S_{n-1-m}) \right\}}{\exp\left\{-(1+\delta)\xi Q(\log S_{n}-\log
		S_{n-1-m/2})\right\}}= C_1e^{-c_1m}.
\end{equation}
	By using Lemma \ref{3.15.1.-1.4} and Lemma \ref{1*} for $\delta=1/5$ along
	with \eqref{eeee:2}, we obtain that for some constants $C_2,c_2',C_2'$ 
	
	\begin{equation}
		\label{e:3.5.1893}
	\mathbb{P}\left( \frac{\sum_{j=-\infty}^{n-1-m} Y_j
	e^{\sum_{k=j}^{n-1}D_k}}{\sum_{j=n-1-m/2}^{n-1} Y_j
	e^{\sum_{k=j}^{n-1}D_k}}\geq C_1C_2e^{-c_1m}  \right)\leq C_2'e^{-c_2' m}.
	\end{equation}
	and this shows \eqref{e:3.5.1892} and completes the proof.
\end{proof}

\subsection{Moment bounds for the $G_i$}
\label{ss:G_imoment}
In this subsection, we prove moment bounds for $G_i$ and
$\underline{G}_{m,i}$ and these will be useful in Section
\ref{s:stat_env} when applying Lemma
\ref{3.163}. The bulk of the subsection is devoted to proving the following
lemma.

\begin{lemma}
	\label{3.15.1.-1.5}
	There exist constants $c,C$ such that for all $x>0$ and all $i$, we
	have
	\begin{displaymath}
		\mathbb{P}\left( G_i\geq x \right)\leq Ce^{-cx}.
	\end{displaymath}
	In particular, the eighth moment of $G_i$ is bounded.
	That is, there exists a positive constant $C_1$ such that 
	\begin{displaymath}
		\mathbb{E}G_i^8\leq C_1
	\end{displaymath}
	for all $i$.
\end{lemma}

	Before moving on to the proof, we remark that Lemma
	\ref{3.15.1.-1.5} deals with the moment bound for $G_i$; the
corresponding moment bound for
the finitely dependent approximants $\underline{G}_{m,i}$ is stated in Lemma
\ref{3.15.-1.8}.

\begin{lemma}
	\label{3.15.1.-1.6}
	There exist positive constants $c,c',C'$ such that for any
	$x>0$ and all $i,j\in \mathbb{Z}$ with $i-1-x\leq j\leq i-1$, 
	\begin{displaymath}
		\mathbb{P}\left( e^{\sum_{k=j}^{i-1}D_k}\leq e^{-c x}
		\right)\leq C'e^{-c'x}.
	\end{displaymath}
\end{lemma}

\begin{proof}
	 The intuition behind the condition $i-1-x\leq j\leq i-1$ in the
	 statement of the lemma is that the $D_i$ are the increments of
	 $\underline{B}$ which has negative drift, and thus $|i-j|$ must be kept small to
	 prevent the negative drift contribution to
	 $\sum_{k=j}^{i-1}D_k$ from overtaking $x$. We now come to the
	 proof.
	 
	 Without loss of generality, we can assume that $x$ is a positive
	 integer. For any $j$ satisfying
	 $i-1-x\leq j\leq i-1$, for any $c>0$,
	 $c_2,c_3,c_4,c_5$,
	 \begin{align}
		\mathbb{P}\left( e^{\sum_{k=j}^{i-1}D_k}\leq e^{-c x}
		\right)&=\mathbb{P}\left( \underline{B}(\log S_i)-\underline{B}(\log
		S_j)\leq -c x \right)\nonumber\\
		&= \mathbb{P}\left( B(\log S_i)-B(\log
		S_j)\leq \xi Q(i-j)\log K -c x \right)\nonumber\\
		&\leq  \mathbb{P}\left( B(\log S_i)-B(\log
		S_j)\leq (\xi Q \log K -c)x +\xi Q \log K \right)
\label{eeee:3}
	\end{align}
	Choosing $c=2\xi Q \log K$, and using an estimate for Gaussian
	variables, we obtain
	\begin{equation}
		\label{}
		\mathbb{P}\left( B(\log S_i)-B(\log
		S_j)\leq (\xi Q \log K -c)x +\xi Q \log K \right)\leq
		C_1e^{-c_1x^2/(i-j)}\leq C_2e^{-c_2x}
	\end{equation}
	where we used that $i-j\leq x+1$ to obtain the last inequality. This
	combined with \eqref{eeee:3} completes the proof.

\end{proof}

 \begin{lemma}
	 \label{3.15.1.-1.7}
	 There exist positive constants $c,c',C'$ such that 
	 \begin{displaymath}
		 \mathbb{P}\left(\sum_{j=-\infty}^{i-1} Y_j
		 e^{\sum_{k=j}^{i-1}D_k}\leq e^{-cx} \right)\leq C'e^{-c' x}
	 \end{displaymath}
	 for all $x$ positive, uniformly in $i\in \mathbb{Z}$. Thus, the random variables $
	 \left(\sum_{j=\infty}^{i-1} Y_j
	 e^{\sum_{k=j}^{i-1}D_k} \right)^{-1}$ have uniform (in $i$) power law tails.
 \end{lemma}
 \begin{proof}
	It suffices to show the statement for positive integer values of $x$.
	For some positive constants $c_1,c_2,c_3$, 	 \begin{align}
		 \label{e:G_4thmom1}
		  &\mathbb{P}\left(\sum_{j=-\infty}^{i-1} Y_j
		  e^{\sum_{k=j}^{i-1}D_k}\leq e^{-cx} \right)\nonumber\\
		  &\leq \mathbb{P}\left(
		  Y_j=0 \text{ for all } j\in [\![i-1-x,i-1]\!] \right)+\sum_{j=i-1-x}^{i-1}
		  \mathbb{P}\left( e^{\sum_{k=j}^{i-1}D_k}\leq \frac{1}{\rho}e^{-c
		  x} \right)\nonumber\\
		  &\leq C_1e^{-c_1 x}+C_2x e^{-c_2 x}\leq C_3e^{-c_3x}
	\end{align}
	 To get the second inequality above, the second term was bounded by a
	 direct application of Lemma \ref{3.15.1.-1.6} by choosing $c$ to be large
	 enough. On the other hand, the first term is bounded by using Lemma
	 \ref{m1}. This completes the proof. 
 \end{proof}

 We now finish the proof of Lemma
 \ref{3.15.1.-1.5}.
 \begin{proof}[Proof of Lemma \ref{3.15.1.-1.5}]
	 We have that 
	 \begin{equation}
		 \label{e:G_4thmom5}
		 G_i\leq \log (1+Y_i)+\log \left(1+\left(\sum_{j=-\infty}^{i-1} Y_j
		 e^{\sum_{k=j}^{i-1}D_k}\right)^{-1}\right).
	 \end{equation}
	 The first term has an exponential tail (uniformly in $i$) by using Lemma
	 \ref{3.13.3}. The second term has an exponential tail (uniformly in $i$)
	 by using Lemma \ref{3.15.1.-1.7}. The bound on the fourth moment of $G_i$
	 is an immediate consequence of the exponential tail. 
 \end{proof}

 We will also need a uniform (in $m$)
 bound for the fourth moment
 the finitely dependent variables $\underline{G}_{m,i}$.  Before stating the lemma,
 we recall that for each $m$, the sequence $\underline{G}_m=\left\{
 \underline{G}_{m,i}
 \right\}_{i\in \mathbb{Z}}$ is stationary and thus the variables
 $\underline{G}_{m,i}$ have the same distribution for all $i$.

 \begin{lemma}
	 \label{3.15.-1.8}
	 The random variables $\underline{G}_{m,i}$ have uniform (in $m$) exponential tails. That
	is, there exist constants $c,C$ such that for all $x>0$ and all $m\in
	\mathbb{N},i\in
	\mathbb{Z}$, we
	have
	\begin{displaymath}
		\mathbb{P}\left( \underline{G}_{m,i}\geq x \right)\leq Ce^{-cx}.
	\end{displaymath}
	In particular, there exists a positive constant $C_1$ such that for all $i\in\mathbb{Z},m\in \mathbb{N}$
	\begin{displaymath}
		\mathbb{E}\underline{G}_{m,i}^8\leq C_1.
	\end{displaymath}
 \end{lemma}

 \begin{proof}Since the proof is very similar to the proof of Lemma
 \ref{3.15.1.-1.5}, we will be brief.
	 Let $A$ denote the event that $Y_j$ is non-zero for at least one
	 $j\in[\![i-1-m,i-1]\!]$. Using that $\underline{G}_{m,i}$ is
	 deterministically $0$ on $A^c$ and imitating \eqref{e:G_4thmom5} and the rest of the proof of Lemma
	 \ref{3.15.1.-1.5}, it suffices to show that there exist positive
	 constants $c,c',C'$ such that
	 \begin{displaymath}
		 \mathbb{P}\left(\left\{\sum_{j=i-1-m}^{i-1} Y_j
		 e^{\sum_{k=j}^{i-1}D_k}\leq e^{-cx}\right\}\cap A \right)\leq C'e^{-c' x}
	 \end{displaymath}
	 for all $x>0$ uniformly in $i,m$. Note that it suffices to show
	 the above for all positive integers $x$  and indeed, for the same, by an application of Lemma \ref{3.15.1.-1.6},
	 \begin{align}
		 \mathbb{P}\left(\left\{\sum_{j=i-1-m}^{i-1} Y_j
		 e^{\sum_{k=j}^{i-1}D_k}\leq e^{-cx}\right\}\cap A \right)
		 &\leq \sum_{j=(i-1-x)\lor(i-1-m)}^{i-1}
		  \mathbb{P}\left( e^{\sum_{k=j}^{i-1}D_k}\leq \frac{1}{\rho}e^{-c
		  x} \right)\nonumber\\
		&\leq C_1xe^{-c_1x} \leq C_2e^{-c_2 x},
		\label{e:G_4thmom6}
	 \end{align}
	 and this finishes the proof.
 \end{proof}

\section{Stationarity for the environment around the geodesic segments}
\label{s:stat_env}

 In this section, we use the decomposition of $\Gamma$ into segments to
 show the finite dimensional convergence in the setting of Theorem
 \ref{main1} and Theorem \ref{main2}; tightness will
 be shown in the next section and that will complete the proof.
As explained in Section \ref{s:iop}, the proof of the finite dimensional distributional convergence for both the field and the metric relies on writing the characteristic function of the
 empirical observables as the sum of asymptotically i.i.d.\ variables.

 Recall that given a realisation of $h$,
 $\mathbf{Field}_t$ and $\mathbf{Metric}_t$ are still random quantities where the randomness comes from the variable $\mathfrak{t}\sim
 \mathtt{Unif}(0,t)$. For a fixed realisation of $h$, we will denote the expectation
 with respect to $\mathfrak{t}$ by the symbol
 $\mathbb{E}_t$ to distinguish them from $\mathbb{E}$ where one also averages
 over the randomness of $h$, e.g. if $\phi\in
 \mathcal{D}(\mathbb{D})$, then $\mathbb{E}_t\left[(\emf_t,
 \phi)\right]$ is a function of $t$ as well as the realization of $h$.
 
To prove Theorem \ref{main1}, the first step is to show the convergence of
 the joint law of the random vector $\left(
 (\emf_t,\phi_1),\dots,(\emf_t,\phi_l) \right)$ as
 $t\rightarrow \infty$ for any $\phi_1,\dots,\phi_l\in
 \mathcal{D}(\mathbb{D})$. We will do this by proving the convergence of
 the characteristic function 
 \begin{equation}
	 \label{e:char}
	 \mathbb{E}_t\left[ \exp\left( \sum_{j=1}^l \mathbf{i}\lambda_j
	 (\emf_t,\phi_j) \right) \right]
 \end{equation}
 and by showing the continuity of the above expression at
 $(\lambda_1,\dots,\lambda_l)=(0,\dots,0)$. This will complete the proof of
 the convergence of finite dimensional distributions. To upgrade this to
 convergence in
 $H_0^{-\varepsilon}$, as promised in the statement of Theorem \ref{main1}, it remains to show 
 tightness; this will be done in Section
 \ref{ss:tightnessfield}. {Before moving on, we write out the explicit form of
 the characteristic function from \eqref{e:char} by using the definition of
 $\emf_t$, i.e., the local field around a point on
 $\Gamma$ with its $\log$-LQG distance from $0$ being uniformly distributed in
 $(0,t)$}.
 \begin{equation}
	 \label{e:3.5.19}
	 \mathbb{E}_t\left[ \exp\left( \sum_{j=1}^l \mathbf{i}\lambda_j
	 (\emf_t,\phi_j) \right) \right]=\frac{1}{t}\int_0^t
	 e^{\mathbf{i}\sum_{j=1}^l
	 \lambda_j(\underline{h}_{\Gamma_{e^s}},\phi_j)}ds=\frac{1}{t}\int_1^{e^t}\frac{1}{s}
	 e^{\mathbf{i}\sum_{j=1}^l 
	 \lambda_j(\underline{h}_{\Gamma_{s}},\phi_j)}ds.
 \end{equation}
 As mentioned earlier in Section \ref{s:iop}, the proof technique for the
 case of $\emm$ is the same and hence, analogues of the above
expression for the case of the latter will also appear. To
keep our notation sufficiently general so as to directly apply to both
these cases, let
$\Upsilon_{1},\dots,\Upsilon_{l}$ be any measurable real valued functions. The
domain of these will change according to the
specific application and will be one among these two:
\begin{itemize}
	\item $\mathcal{D}'(\mathbb{D})$ with the weak-$*$ topology and the
		Borel $\sigma$-algebra in case
		we are looking at the empirical field.
	\item Continuous metrics on $\mathbb{D}$ with the uniform topology and the
		Borel $\sigma$-algebra in
		case we are looking at the empirical metric.
	\end{itemize}
	For the sake of clarity, we first only work with the case of the empirical field and
then later, at the end of the section, describe how the same arguments translate to the case of the
empirical metric as well. For the case of the field, we will work with the
following choice of the functions $\Upsilon_1,\dots,\Upsilon_l$.
\begin{equation}
	\label{upsilon1}
	\Upsilon_j(\mathtt{h})=(\mathtt{h},\phi_j) \text{ for functions }
	\phi_1,\dots,{\phi_l}\in \mathcal{D}(\mathbb{D})\text{ and } \mathtt{h}\in
		\mathcal{D}'(\mathbb{D}).
\end{equation}

Recall that $\eta(\cdot)$ was defined in \eqref{e:3.5.12.0}  by
$\eta(j)=\sum_{k=0}^j\mathcal{P}_k-1$ i.e., the value
$\eta(j)$ is such that the radius $S_j$ lies between the $\eta(j)$th and the
$(\eta(j)+1)$th coalescence points. For any
$\lambda_1,\dots,\lambda_l\in \mathbb{R}$, we define the following random variables
\begin{gather}
	Z_i^\mathrm{field}(\left\{
\lambda_j \right\}_{1\leq j\leq l},\left\{ \Upsilon_j \right\}_{1\leq j\leq
l}):=\mathcal{P}_{i}\int_{\log L_{\eta(i)}}^{\log
	L_{\eta(i)+1}}
	e^{\mathbf{i}\sum_{j=1}^l
	\lambda_j\Upsilon_j(\underline{h}_{\Gamma_{e^s}})}ds=
	\mathcal{P}_{i}\int_{ L_{\eta(i)}}^{ L_{\eta(i)+1}}
	\frac{1}{s}e^{\mathbf{i}\sum_{j=1}^l
	\lambda_j\Upsilon_j(\underline{h}_{\Gamma_{s}})}ds,\label{e:3.5.20}
\end{gather}

Since we will be only working with the case of the field for now and
because the $\left\{
\lambda_j \right\}_{1\leq j\leq l}$,
$\left\{ \Upsilon_j \right\}_{1\leq j \leq l}$ will remain fixed
throughout, we abbreviate $Z_i^\text{field}(\left\{
\lambda_j \right\}_{1\leq j\leq l},\left\{ \Upsilon_j \right\}_{1\leq j\leq
l})$ to $Z_i$.

The reason for having the $\mathcal{P}_i$ coefficient in the above
expressions is that we want to later write the characteristic function from \eqref{e:3.5.19} as the sum of the $Z_i$s and
adding the $\mathcal{P}_i$ coefficient ensures that each segment of the
geodesic contributes only once to the expression; in
other words, $Z_i$ is non-zero if and only if the coalescence event
$E_{S_i}$ occurs in which case $Z_i$ describes the contribution of the
segment $\Gamma(p_{\eta(i)},p_{\eta(i)+1})$ to the characteristic
function. We will then show that $Z_i$s themselves
form a stationary sequence with exponential decay of
correlations which via a standard law of
large numbers argument implies an almost sure convergence of the
characteristic function from \eqref{e:3.5.19}. Note that this is very
similar to what was described for the variables $G_i$ at the beginning of
Section \ref{ss:lengths}. 
Indeed, the $G_i$ variables are actually a special
case of the $Z_i$ variables realized by choosing $l=1$ and
$\Upsilon_1=0$ in the definition of $Z_i$. 

Following our convention, we use $\mathcal{Z}$ to denote the bi-infinite
sequence corresponding to the $Z_i$s. In Lemma \ref{3.16}, we shall use Lemma \ref{3.163} to show that $\mathcal{Z}$ is a stationary
sequence with exponential decay of correlations.

As we had outlined for the special case of $\mathcal{G}$ in Section \ref{ss:lengths}, the reason
to use Lemma \ref{3.163} instead of Lemma \ref{'3.163} to prove the
decay of correlations in $\mathcal{Z}$ is that, due to the
log-parametrization, the variables $Z_i$ are actually dependent on all past
scales before $i$. In the case of the $G_i$, this can be seen from Lemma
\ref{3.15.0} and we now give an analogous lemma for $Z_i$. Before coming the
lemma, we introduce some notation-- for a whole plane GFF marginal
$\mathtt{h}$, we use $\Gamma_s(\mathtt{h})$ as a shorthand for the geodesic
$\Gamma_{s}\left(p_0(\mathtt{h}),p_1(\mathtt{h});\mathtt{h}\right)$ from
Lemma \ref{3.12.1}, parametrised
according to the LQG length coming from $D_\mathtt{h}$. Also, we now give an
extension of the ``underline'' notation $\underline{h}_x$ and
$\underline{D}_{\mathtt{h}',x}$ from \eqref{e:field}, \eqref{e:metric} to fields other
than $h$. For a GFF plus continuous function $\mathtt{h}'$ on an open set
$U\subseteq \mathbb{C}$, and a point $x\in U$ satisfying
$\mathbb{D}_{\delta|x|}(x)\subseteq U$, we define
\begin{gather}
	\label{e:fieldg}
	\underline{\mathtt{h}'}_x=\mathtt{h}\mydot \Psi_{x,\delta}-c_x(\mathtt{h}')
\end{gather}
where $\Psi_{x,\delta}$ is defined by 
$\Psi_{x,\delta}(z)=\delta|x|z+x$ for $z\in \mathbb{C}$, and
$c_{x}(\mathtt{h}')$ is chosen so that
$\mathbf{Av}(\underline{\mathtt{h}'}_x,\mathbb{T})=0$, i.e., 
$\mathbf{Av}(\mathtt{h}',\mathbb{T}_{\delta|x|}(x))=c_x(\mathtt{h}')$. We can
analogously define
\begin{equation}
	\label{e:metricg}
	\underline{D}_{\mathtt{h}',x}(u,v)=|\delta x|^{-\xi Q}
	 e^{-\xi
	 \mathbf{Av}({\mathtt{h}'},\mathbb{T}_{\delta|x|}(x))}D_{{\mathtt{h}'}}(x+\delta
	 |x|u,x+\delta |x|v;\mathbb{D}_{\delta|x|}(x)).
\end{equation}
We now come to the aforementioned analogue of Lemma \ref{3.15.0} for the
$Z_i$.
\begin{lemma}
	\label{3.15.0*}
	We have that
	\begin{displaymath}
		Z_i=\int_{0}^{
	Y_i}
	\frac{e^{\mathbf{i}\sum_{j=1}^l
	\lambda_j\Upsilon_j(\underline{\mathcal{H}_i}_{\Gamma_{s}(\mathcal{H}_i)})}}{\sum_{j=-\infty}^{i-1} Y_j
	e^{\sum_{k=j}^{i-1}D_i}+s}ds.
	\end{displaymath}
\end{lemma}

\begin{proof}

Observe that almost surely, simultaneously for all $x,y\in \mathbb{C}_{>S_i}$  such that the geodesic
$\Gamma(x,y)$
does not
intersect $\mathbb{T}_{S_i}$, 
\begin{equation}
	\label{e:ceq}
	D_h(x,y)=D_h(x,y;\mathbb{C}_{>S_i})=D_{h\lvert_{\mathbb{C}_{>S_i}}}(x,y)=e^{\xi Q\log S_i
	+\xi B(\log S_i)}D_{\mathcal{H}_i}(x/S_i,y/S_i).
\end{equation}
The first equality above uses Lemma \ref{intrin} while  the second equality follows from the locality of
the LQG metric as in Proposition \ref{b3} (2) and the third equality is an
application of Lemma \ref{b4} along with the definition
$\mathcal{H}_i=\widehat{h}_{S_i}\mydot \psi_{S_i}$. With $t= e^{\xi Q\log S_i
	+\xi B(\log
	S_i)}s$ and $\tilde{t}$
denoting $L_{\eta(i)}+t$, we now use \eqref{e:ceq} to
note that on the event $E_{S_i}$, we have that almost surely, simultaneously
for all $s\in (0,D_{\mathcal{H}_i}(p_0(\mathcal{H}_i),p_1(\mathcal{H}_i)) )=
(0,Y_i )$,
\begin{equation}
	\label{changeof}
	\Gamma_{\tilde{t}}=S_i\Gamma_s(\mathcal{H}_i).
\end{equation}
 To see this, first note
that Lemma \ref{3.12.0} implies that on the event $E_{S_i}$, we have
$p_{\eta(i)}(h)=S_i p_0(\mathcal{H}_i)$ and
$p_{\eta(i)+1}(h)=S_i
p_1(\mathcal{H}_i)$. Further, on $E_{S_i}$, the geodesic $\Gamma(p_{\eta(i)}(h),p_{\eta(i)+1}(h))$
cannot enter $\mathbb{C}_{\leq S_i}$ by Lemma
\ref{3.12.1}. By using the above along with \eqref{e:ceq}, we obtain
\eqref{changeof}.

	 By
using
\eqref{e:3.5.20} and the argument from Lemma
\ref{3.15.0}, we obtain
\begin{align}
Z_i&= \mathcal{P}_{i}\int_{0}^{
	L_{\eta(i)+1}-L_{\eta(i)}}
	\frac{1}{L_{\eta(i)}+t}e^{\mathbf{i}\sum_{j=1}^l
	\lambda_j\Upsilon_j(\underline{h}_{\Gamma_{\tilde{t}}})}dt\nonumber\\
	&=\int_{0}^{\mathcal{P}_{i}(
	L_{\eta(i)+1}-L_{\eta(i)})}
	\frac{1}{\sum_{j=-\infty}^{i-1}e^{\xi Q\log S_j
	+\xi B(\log
	S_j)}Y_j+t}e^{\mathbf{i} \sum_{j=1}^l
	\lambda_j\Upsilon_j(\underline{h}_{\Gamma_{\tilde{t}}})}dt\nonumber\\
	&=\int_{0}^{e^{\xi Q\log S_i
	+\xi B(\log
	S_i)}Y_i}
	\frac{1}{\sum_{j=-\infty}^{i-1}e^{\xi Q\log S_j
	+\xi B(\log
	S_j)}Y_j+t}e^{\mathbf{i} \sum_{j=1}^l
	\lambda_j\Upsilon_j(\underline{h}_{\Gamma_{\tilde{t}}})}dt\nonumber\\
	& \label{e:3.8} =\int_{0}^{
	Y_i}
	\frac{e^{\xi Q\log S_i
	+\xi B(\log S_i)}}{\sum_{j=-\infty}^{i-1}e^{\xi Q\log S_j
	+\xi B(\log S_j)}Y_j+e^{\xi Q\log S_i
	+\xi B(\log
	S_i)}s}e^{\mathbf{i}\sum_{j=1}^l
	\lambda_j\Upsilon_j(\underline{\mathcal{H}_i}_{\Gamma_{s}(\mathcal{H}_{i})})}ds
\end{align}
The penultimate equality is an application of Lemma \ref{3.12.2}. The last
equality involving the appearance of $\Gamma_s(\mathcal{H}_i)$ is an application of the change of variables $t= e^{\xi Q\log S_i
	+\xi B(\log
	S_i)}s$; note that
one only needs to consider the case when $\mathcal{P}_{i}=1$, i.e.,
$E_{S_i}$ occurs, in which case we have that by using \eqref{changeof} that
$\underline{h}_{\Gamma_{\tilde{t}}}=\underline{\mathcal{H}_i}_{\Gamma_s(\mathcal{H}_i)}$ almost surely, simultaneously for all $s\in
(0,Y_i)$.

We can now
write \eqref{e:3.8} by using the $D_i$s (recall from \eqref{e:3.172}) in the following manner.
\begin{align}
	Z_i&= \int_{0}^{
	Y_i}
	\frac{e^{\mathbf{i}\sum_{j=1}^l
	\lambda_j\Upsilon_j(\underline{\mathcal{H}_i}_{\Gamma_{s}(\mathcal{H}_{i})})}}{\sum_{j=-\infty}^{i-1}e^{\xi Q(\log S_j-\log S_i)
	+\xi (B(\log S_j)-B(\log
	S_i))}Y_j+s}ds\nonumber\\
	\label{e:3.9}
	&=\int_{0}^{
	Y_i}
	\frac{e^{\mathbf{i}\sum_{j=1}^l
	\lambda_j\Upsilon_j(\underline{\mathcal{H}_i}_{\Gamma_{s}(\mathcal{H}_{i})})}}{\sum_{j=-\infty}^{i-1} Y_j
	e^{\sum_{k=j}^{i-1}D_i}+s}ds.
\end{align}
This completes the proof.
\end{proof}
We spend the rest of this section introducing quantities and estimates for the
$Z_i$ that we have already seen in the case of the variables $G_i$. In
particular, for $i\in \mathbb{Z}$, we define a $m$-dependent analogue of $Z_i$, which we
call $\underline{Z}_{m,i}$. These are analogous to the variables
$\underline{G}_{m,i}$ which were defined in \eqref{e:3.15.21} and
are in fact the same in the special case when $l=1$ and $\Upsilon_1=0$. 
\begin{equation}
	\label{e:3.9.1}
	\underline{Z}_{m,i}= \begin{cases}
		 \int_{0}^{
Y_i}
	\frac{1}{\sum_{j=i-1-m}^{i-1}Y_j
	e^{\sum_{k=j}^{i-1}D_i}+s}e^{\mathbf{i}\sum_{j=1}^l
	\lambda_j\Upsilon_j
	(\underline{\mathcal{H}_i}_{\Gamma_{s}(\mathcal{H}_i)})}ds &\text{if } Y_j>0
	\text{ for some } i-1-m\leq j \leq i-1,\\
	0 &\text{otherwise}.
\end{cases}
\end{equation}
Analogous to Lemma \ref{*3.15}, we can use a change of variables
argument to obtain the following alternate expression for
the variables
$\underline{Z}_{m,i}$.  

\begin{lemma}
	\label{**3.15}
	On the event $\left\{\underline{Z}_{m,i}\neq 0\right\}$, we have
	\begin{displaymath}
		\underline{Z}_{m,i}=\int_{\log D_h(p_{\iota(i-1-m)},p_{\eta(i)})}^{\log D_h(p_{\iota(i-1-m)},p_{\eta(i)+1})}
	\exp\left(\mathbf{i}\sum_{j=1}^l
	\lambda_j\Upsilon_j\left(\underline{h}_{\Gamma_{e^s+D_h(p_{\iota(i-1-m)},p_{\eta(i)})}}\right)\right)ds.
	\end{displaymath}
\end{lemma}

In Section \ref{ss:G_imoment}, we showed uniform $8$th moment bounds for the
$G_i$ and the $\underline{G}_{m,i}$ in Lemmas \ref{3.15.1.-1.5} and
\ref{3.15.-1.8} respectively. We now obtain analogous
estimates for $Z_i$ and $\underline{Z}_{m,i}$. 

\begin{lemma}
	\label{Z_4thmom}
	There exists a constant $C>0$ such that we have 
	\begin{displaymath}
		\mathbb{E}|Z_i|^8\leq C
	\end{displaymath}
	uniformly in $i,\left\{ \lambda_j \right\}_{1\leq j\leq l}$ and $\left\{
	\Upsilon_j
	\right\}_{1\leq j \leq l}$. 
\end{lemma}
\begin{proof}
	By the definition of $Z_i$ in \eqref{e:3.5.20}, note that
	\begin{displaymath}
		|Z_i|\leq \mathcal{P}_i(\log
	L_{\eta(i)+1}-\log
	L_{\eta(i)})=G_i,
	\end{displaymath}
	and the result now follows by using Lemma \ref{3.15.1.-1.5}.
\end{proof}

\begin{lemma}
	\label{Z_4thmom1}
	There exists a constant $C>0$ such that we have 
	\begin{displaymath}
		\mathbb{E}|\underline{Z}_{m,i}|^8\leq C
	\end{displaymath}
	uniformly in $i$,$m$,$\left\{ \lambda_j \right\}_{1\leq j\leq l}$ and $\left\{
	\Upsilon_j
	\right\}_{1\leq j \leq l}$. 
\end{lemma}
\begin{proof}
	By using Lemma \ref{**3.15}, we have that 
\begin{equation}
		\label{e:Z_4thmom2}
|\underline{Z}_{m,i}|\leq\mathcal{P}_i\left( \log 
	D_h(p_{\iota(i-1-m)},p_{\eta(i)+1}) - \log 
	D_h(p_{\iota(i-1-m)},p_{\eta(i)}) \right)
	\end{equation}

Observe that the right hand side is nothing but $\underline{G}_{m,i}$ by Lemma \ref{*3.15} and the result now follows from Lemma \ref{3.15.-1.8}.
\end{proof}

\subsection{Stationarity and exponential decay of correlations for the sequences
$\mathcal{Z}$ and $\mathcal{G}$}
In this subsection, we show that the sequence $\mathcal{Z}$ is a stationary
sequence with exponential decay of correlations. Since $\mathcal{G}$ is a
special case of the sequence $\mathcal{Z}$, this will imply the same
result for the sequence $\mathcal{G}$ as well. 

As mentioned earlier, we cannot use
Lemma \ref{'3.163} directly since the variables $Z_i$ depend on all the fields occurring before the scale $i$ as can be seen in the expression from
Lemma \ref{3.15.0*}. For this reason, we have introduced the variables
$\underline{Z}_{m,i}$ which depend only on the fields $\left\{
\mathcal{H}_{i-1-m},\dots,\mathcal{H}_i \right\}$. We will apply Lemma \ref{3.163} with
appropriate choices of $f_m,f$ such that we can interpret $f_m(\mathcal{H})$
as $\underline{\mathcal{Z}}_{m}$ and $f(\mathcal{H})$ as
$\mathcal{Z}$, where
by $\underline{\mathcal{Z}}_m$, we mean the bi-infinite sequence corresponding
to the sequence $(\underline{\mathcal{Z}}_m)_{i\in \Z}:=\left\{
\underline{Z}_{m,i} \right\}_{i\in\mathbb{Z}}$ for a fixed value
of $m$. 

In order to invoke Lemma
\ref{3.163}, we have
to show \eqref{e:g1}, which in this setting is the same as showing that
$\mathbb{E}\left|\underline{Z}_{m,i}-Z_{i}\right|^4=O(e^{-cm})$ uniformly in
$i$. This will be the main additional estimate used in the proof of Lemma \ref{3.16}.

\begin{lemma}
	\label{3.16}
	The bi-infinite sequence $\mathcal{Z}$ is stationary with
	exponential decay of correlations.

\end{lemma}

\begin{proof}
	The proof will consist of an application of Lemma \ref{3.163} and we
	borrow all the notation from its setting.
	We first define $f$ as follows to ensure that we have $\mathcal{Z}=f(\mathcal{H})$ almost surely.
\begin{equation}
	\label{e*:3.5.190}
	[f(\mathcal{H})]_i=
	\begin{cases}
		\int_{0}^{
	Y_0(\mathcal{H}_i)}
	\frac{1}{\sum_{j=-\infty}^{i-1} 	Y_0(\mathcal{H}_j)
	e^{\sum_{k=j}^{i-1}D_0(\mathcal{H}_k)}+s}e^{\mathbf{i}\sum_{j=1}^l
	\lambda_j\Upsilon_j
	(\underline{\mathcal{H}_i}_{\Gamma_{s}(\mathcal{H}_i)})}ds, &\text{if }
	Y_0(\mathcal{H}_j)>0
	\text{ for some } -\infty<j \leq i-1,\\
	0, &\text{otherwise}.
\end{cases}
\end{equation}
	We want to now similarly define $X_m$ such that the corresponding $f_m$
	obtained by using \eqref{ee:e} satisfies
	$\underline{\mathcal{Z}}_{m}=f_m(\mathcal{H})$ almost surely. We will
	have the following definition of $f_m$.
\begin{equation}
	\label{e*:3.5.191}
	[f_m(\mathcal{H})]_i=
	\begin{cases}
		\int_{0}^{
	Y_0(\mathcal{H}_i)}
	\frac{1}{\sum_{j=i-1-m}^{i-1} 	Y_0(\mathcal{H}_j)
	e^{\sum_{k=j}^{i-1}D_0(\mathcal{H}_k)}+s}e^{\mathbf{i}\sum_{j=1}^l
	\lambda_j\Upsilon_j
	(\underline{\mathcal{H}_i}_{\Gamma_{s}(\mathcal{H}_{i})})}ds, &\text{if }
	Y_0(\mathcal{H}_j)>0
	\text{ for some }\\ &\qquad i-1-m \leq j \leq i-1,\\
	0, &\text{otherwise}.
\end{cases}
\end{equation}
The corresponding $X_m$ acts on a whole plane
GFF marginal $\mathtt{h}$ in the following manner.
\begin{equation}
	\label{e*:3.5.191*}
	X_m(\mathtt{h})=
	\begin{cases}
		\int_{0}^{
		Y_{m+1}(\mathtt{h})}
		\frac{1}{\sum_{j=0}^{m} 	Y_j(\mathtt{h})
		e^{\sum_{k=j}^{m}D_j(\mathtt{h})}+s}e^{\mathbf{i}\sum_{j=1}^l
		\lambda_j\Upsilon_j (\underline{\mathtt{h}}_{\Gamma_{s}(\mathtt{h})})}ds, &\text{if }
	Y_j(\mathtt{h})>0
	\text{ for some }  0 \leq j \leq m,\\
	0, &\text{otherwise}.
\end{cases}
\end{equation}

By using the above definitions and the expressions
\eqref{e:3.9} and \eqref{e:3.9.1}, it is easy to see that we have
$f(\mathcal{H})^n=f( \mathcal{H}^n)$ and
$f_m(\mathcal{H})^n=f_m( \mathcal{H}^n)$ for all $m,n$ almost
surely. This verifies the corresponding conditions
\eqref{e:3.161} and \eqref{e:3.162} that $f_m$ and $f$ need to satisfy in
order to use Lemma \ref{3.163}. Also, as a consequence of Lemma \ref{3.163-},
this shows that the sequences $\underline{\mathcal{Z}}_{m}=\left\{ \underline{Z}_{m,i} \right\}_{i\in
\mathbb{Z}}$ for all $m$ and $\mathcal{Z}=\left\{ Z_i
\right\}_{i\in \mathbb{Z}}$ are stationary.

For a whole plane GFF marginal $\mathtt{h}$, we define $\tau_m(\mathtt{h})$ by 
\begin{equation}
	\label{e:S'defn}
	\tau_m(\mathtt{h})=\iota(m+2,\mathtt{h})=1+\inf_{i\geq m+2}\left\{
	i:E_{S_i}[\mathtt{h}] \text{
	occurs}  \right\}.
\end{equation}
We first verify the 
measurability conditions needed in Lemma \ref{3.163}, i.e., we need to check
that $\tau_m(\mathtt{h})$ is a stopping time for the
filtration $\mathscr{G}(\mathtt{h})$ 
and $X_m(\mathtt{h})$ is a measurable with respect to
$\mathscr{G}_{\tau_m(\mathtt{h})}(\mathtt{h})$. The proof of the former 
fact is
exactly the same as the corresponding statement in \eqref{e:SST} in the proof of
Lemma \ref{3.15.1} and we do not repeat it.

The proof that $X_m(\mathtt{h})$ is a measurable with respect to
$\mathscr{G}_{\tau_m(\mathtt{h})}(\mathtt{h})$  is
similar to the corresponding step in the proof of Lemma \ref{3.15.1} and we
thus go through it quickly.
It suffices to show that all the variables $Y_j(\mathtt{h})$ and
$D_j(\mathtt{h})$ for
$0\leq j \leq m+1$ are measurable with respect to
$\mathscr{G}_{\tau_m(\mathtt{h})}(\mathtt{h})$. Dealing with $Y_j(\mathtt{h})$ first, we fix some $j\in
[\![0,m+1]\!]$ and show that for any $s>0$ and $n\in
\mathbb{N}$, the event $\mathcal{E}_{s,n}[\mathtt{h}]=\left\{
Y_j(\mathtt{h}) \in (0,s) \right\}\cap\left\{ \tau_m(\mathtt{h})\leq n
\right\}$ is measurable with respect to $\mathscr{G}_n(\mathtt{h})$. This is
true because, on $\mathcal{E}_{s,n}[\mathtt{h}]$, by locality and Lemma
\ref{3.12.1},  $D_\mathtt{h}(p_0(\widehat{\mathtt{h}}_{S_j}\mydot
\psi_{S_j}),p_1(\widehat{\mathtt{h}}_{S_j}\mydot
\psi_{S_j}))=D_{\mathtt{\mathtt{h}}\lvert_{\mathbb{C}_{(1,S_n)}}}(p_0(\widehat{\mathtt{h}}_{S_j}\mydot
\psi_{S_j}),p_1(\widehat{\mathtt{h}}_{S_j}\mydot
\psi_{S_j}))$ and we note that both $p_0(\widehat{\mathtt{h}}_{S_j}\mydot
\psi_{S_j}),p_1(\widehat{\mathtt{h}}_{S_j}\mydot
\psi_{S_j})$ are measurable with respect to $\mathcal{E}_{s,n}[\mathtt{h}]$ as a
consequence of Lemma \ref{3.10.0}. This completes the justification of the
measurability of $Y_j(\mathtt{h})$ with respect to
$\mathscr{G}_{\tau_m(\mathtt{h})}(\mathtt{h})$ for all $0\leq j\leq m+1$.

As for the variables $D_j(\mathtt{h})$, the Brownian increments
$B_\mathtt{h}(\log S_{j})-B_\mathtt{h}(\log S_{j+1})$ are measurable with
respect to $\mathscr{G}_{m+2}(\mathtt{h})$ for all $0\leq j\leq m+1$ and we note
that $m+2\leq \tau_m(\mathtt{h})$ deterministically. Thus $D_j(\mathtt{h})$
are measurable with respect to $\mathscr{G}_{\tau_m(\mathtt{h})}(\mathtt{h})$
for all $0\leq j\leq m+1$. Putting everything together, we conclude the measurability of $X_m(\mathtt{h})$ with respect to
$\mathscr{G}_{\tau_m(\mathtt{h})}(\mathtt{h})$.

	 Another condition that we need to verify to use Lemma \ref{3.163} is that
	 the variables $\tau_m(\mathtt{h})-m$ have exponential tails uniformly in
	 $m$; this is an immediate consequence of Lemma \ref{3.11}.
The condition for
	 $\mathbb{E}|X_m(\mathtt{h})|^4$ is implied by  Lemma
\ref{Z_4thmom1}, by which there exists a constant $C$ such that
\begin{displaymath}
	\mathbb{E}|[f_m(\mathcal{H})]_i|^4=\mathbb{E}|X_m(\mathtt{h})|^4\leq C
\end{displaymath}
for all $m\in \mathbb{N}$ and $i\in \mathbb{Z}$.

In view of Lemma \ref{3.163}, to complete the proof, it remains to show that there exist constants $C,c$ such that
	\begin{equation}
		\label{e:expm}
		\mathbb{E}\left| [f_m(\mathcal{H})]_i -[f(\mathcal{H})]_i
		\right|^4\leq Ce^{-cm}
	\end{equation}
	for all $m,i$ which is the same as showing that
$\mathbb{E}\left|\underline{Z}_{m,i}-Z_{i}\right|^4\leq Ce^{-cm},$ for all $m,i$.
Towards this, let $\varphi$ be defined on the interval $[0, Y_i]$ 
by
\begin{equation}
	\label{e:3.9.5}
	\varphi(s)=\frac{1}{\sum_{j=  i-1-m}^{i-1}Y_j
e^{\sum_{k=j}^{i-1}D_k}+s}-\frac{1}{\sum_{j= -\infty}^{i-1}Y_j
	e^{\sum_{k=j}^{i-1}D_k}+s}.
\end{equation}
It is easy to see that
\begin{equation}
	\label{e:3.9.6}
	\varphi(s) \leq \frac{\alpha_{m,i}}{\sum_{j=i-1-m}^{i-1}Y_j
e^{\sum_{k=j}^{i-1}D_k}+s}
\end{equation}
where $\alpha_{m,i}$ was defined in
\eqref{e:3.15.223}. Let $A$ denote the event that $Y_j$ is non-zero for at least one
	 $j\in[\![i-1-m,i-1]\!]$; note that this is the same as the
	 event $\left\{ \underline{Z}_{m,i}\neq 0 \right\}$ as can be seen from
	 \eqref{e:3.9.1}. We now write
	\begin{align}
		\mathbb{E}\left|\underline{Z}_{m,i}-Z_{i}\right|^4
		&=\mathbb{E}[\left|\underline{Z}_{m,i}-Z_{i}\right|^4;A]+\mathbb{E}[\left|\underline{Z}_{m,i}-Z_{i}\right|^4;A^c]\nonumber\\		&= \mathbb{E}\left[\left|
		\int_{0}^{
	Y_n}
	\varphi(s)e^{\mathbf{i}\sum_{j=1}^l
	\lambda_j\Upsilon_j
	(\underline{\mathcal{H}_i}_{\Gamma_{s}(\mathcal{H}_i)})}ds
\right|^4;A\right]+\mathbb{E}\left[
|Z_i|^4;A^c\right]\nonumber\\
&\leq \mathbb{E}\left[\alpha_{m,i}^4 \underline{Z}_{m,i}^4;A\right]  +\sqrt{\mathbb{P}(A^c)}
\sqrt{\mathbb{E} Z_i^8} \nonumber\\
&\leq 16\sqrt{\mathbb{E} \alpha_{m,i}^8
}\sqrt{\mathbb{E}\underline{Z}_{m,i}^8}+\sqrt{\mathbb{P}(A^c)}
\sqrt{\mathbb{E} Z_i^8 }\nonumber\\
&\leq C_1e^{-c_1m}
\label{e:3.9.7}
	\end{align}
for some positive constant $c_1$. The first term in the third line has been obtained by
using \eqref{e:3.9.6} along with the definition \eqref{e:3.9.1} for
$\underline{Z}_{m,i}$. To obtain the last inequality above, the exponential
decay (in $m$) of $\mathbb{E}|\alpha_{m,i}|^8$, $\mathbb{P}(A^c)$ 
are obtained from Lemma \ref{3.15.1.-1.21} and
Lemma \ref{m1}
	respectively. The uniform boundedness
	(in $m,i$) of
	$\mathbb{E}Z_i^8$ and $\mathbb{E} \underline{Z}_{m,i}^8$ follow from Lemma
	\ref{Z_4thmom} and Lemma \ref{Z_4thmom1} respectively. This
	completes the justification of \eqref{e:expm} and thus of all the conditions needed to apply Lemma \ref{3.163}, which completes the proof of the lemma.	
\end{proof}

We would like to point out that the proof of Lemma \ref{3.16} yields the
following bound for $\mathbb{E}|\underline{G}_{m,i}-G_{m,i}|^4$ which will be
useful later.
\begin{lemma}
	\label{m2}
	There exist positive constants $C,c$ such that we have
	\begin{displaymath}
		\mathbb{E}|\underline{G}_{m,i}-G_{i}|^4\leq Ce^{-cm}
	\end{displaymath}
	uniformly in $m,i$.
\end{lemma}
\begin{proof}
	The proof follows by \eqref{e:3.9.7} and recalling that $G_i=Z_i$ and $\underline{G}_{m,i}=\underline{Z}_{m,i}$ if the
	$Z_i$s are defined with $l=1$ and $\Upsilon_1=0$. 
\end{proof}

Since the $G_i$ variables are a special case of the $Z_i$ variables, an analogue of Lemma
\ref{3.16} holds for the former, which we record. 
\begin{lemma}
	\label{3.15.1.0}
We have that the sequence $\mathcal{G}$ is stationary and has
exponential decay of correlations.
\end{lemma}

\subsection{Convergence of finite dimensional distributions via an SLLN argument}
To conclude distributional convergence from convergence of characteristic
functions later, we will need the continuity of the limiting characteristic
function at $0$. For this purpose, we record the following lemma.

\begin{lemma}
	\label{3.17}
	For any choice of $\Upsilon_1,\dots,\Upsilon_l$ as in \eqref{upsilon1}, 
	$\mathbb{E}Z_0$ as a function of
	$(\lambda_1,\dots,\lambda_l)$ is
	continuous at $(\lambda_1,\dots,\lambda_l)=(0,\dots,0)$.
\end{lemma}

\begin{proof}
	By using that the $G_i$s are a special case of the
	$Z_i(\left\{ \lambda_1,\dots,\lambda_l \right\})$s corresponding to $(\lambda_1,\dots,\lambda_l)=(0,\dots,0)$, we
	have that $Z_0(\left\{ \lambda_1,\dots,\lambda_l \right\})\rightarrow G_0$ almost
	surely as $(\lambda_1,\dots,\lambda_l)\rightarrow (0,\dots,0)$.
	We know from Lemma \ref{Z_4thmom} that the random variables
	$\left\{Z_0\right\}$ are $L^8$ bounded and thus
	uniformly integrable as a family in $(\lambda_1,\dots,\lambda_l)$. 
This justifies the interchange
	\begin{displaymath}
		\lim_{(\lambda_1,\dots,\lambda_l)\rightarrow (0,\dots,0)} \mathbb{E}
		Z_0(\left\{ \lambda_1,\dots,\lambda_l \right\})= \mathbb{E}\lim_{(\lambda_1,\dots,\lambda_l)\rightarrow
		(0,\dots,0)} Z_0(\left\{ \lambda_1,\dots,\lambda_l \right\})=\mathbb{E}G_0
	\end{displaymath}
	and completes the proof of the lemma. 
\end{proof}

The following lemma is based on a law of large numbers argument. It will
finally be used to argue the almost sure convergence of the characteristic
function ${\mathbb{E}_t}\left[ \exp\left( \sum_{j=1}^l \mathbf{i}\lambda_j\Upsilon_j
	 (\mathbf{Field}_t) \right) \right]$ as $t\rightarrow \infty$ in the
	 proof of Proposition \ref{3.21}.
\begin{lemma}
	\label{3.19}
$\lim_{n\rightarrow
		\infty}\frac{\sum_{i=0}^n Z_{i}}{n}\rightarrow
		\mathbb{E}Z_0$
	almost surely. 
\end{lemma}
\begin{proof}
Note that $\mathcal{Z}=\left\{ Z_i \right\}_{i\in \mathbb{Z}}$ is a stationary with
exponential decay of correlations by an application of Lemma \ref{3.16}. The
result now follows by an application of the strong law of large numbers for
stationary sequences with exponential decay of correlations (see \cite[Theorem
6]{Lyo88}).
\end{proof}
As a special case of the above lemma, we obtain the following lemma for the
$G_i$ variables.
\begin{lemma}
	\label{3.20}
	$\lim_{n\rightarrow
		\infty}\frac{\sum_{i=0}^n G_{i}}{n}\rightarrow\mathbb{E}G_0
$	almost surely.
\end{lemma}

The following lemma is the main ingredient in the proofs of Theorem
\ref{main1} and Theorem \ref{main2}.

\begin{proposition}
	\label{3.21}
For any $\phi_1,\dots,\phi_l\in \mathcal{D}(\mathbb{D})$, almost surely in the
	randomness of $h$, the law of the random vector
	\begin{displaymath}
		\left(
 (\emf_t,\phi_1),\dots,(\emf_t,\phi_l) \right)
	\end{displaymath}
	converges in distribution to a {deterministic measure} (depending on the
	choice of $\phi_1,\dots,\phi_l$) on $\mathbb{R}^l$ as
	$t\rightarrow \infty$.
\end{proposition}
\begin{proof}
	As in \eqref{upsilon1}, we use the shorthand
	$\Upsilon_l(\mathtt{h})=(\mathtt{h},\phi_j)$ for $\mathtt{h}\in
	\mathcal{D}'(\mathbb{D})$ and $1\leq j\leq l$. We will show that the characteristic function
 \begin{equation}
	 \label{e:3.15.42}
	 \mathbb{E}_t\left[ \exp\left( \sum_{j=1}^l \mathbf{i}\lambda_j\Upsilon_j
	 (\mathbf{Field}_t) \right) \right]=\frac{1}{t}\int_0^t
	 e^{\mathbf{i}\sum_{j=1}^l
	 \lambda_j\Upsilon_j(\underline{h}_{\Gamma_{e^s}})}ds	 
 \end{equation}
 converges almost surely as $t\rightarrow\infty$ to a deterministic function
 of the $\left\{ \lambda_j \right\}_{1\leq j\leq l}$ and the $\left\{
 \Upsilon_j
 \right\}_{1\leq j\leq l}$. Proving the continuity of the characteristic
 function at $(\lambda_1,\dots,\lambda_l)=(0,\dots,0)$ would complete the proof. 
 
 We will be using standard arguments from renewal theory. For a given positive $t$, use $i(t)$ to denote $j$, where $j$ is the largest non-negative integer
 such that $\log L_{\eta(j)}\leq t$ i.e., $i(t)$ is the random positive integer
 satisfying 
 \begin{equation}
	 \label{e:3.15.421}
	\log L_{\eta(i(t))}\leq t<\log L_{\eta(i(t)+1)}.
 \end{equation}
 We begin by showing that almost surely as $t\rightarrow \infty$
	 \begin{equation}
		 \label{e:3.15.47}
		 \frac{t}{i(t)}\rightarrow \mathbb{E}G_0.
	 \end{equation}
	 Note that we can write
	 \begin{equation}
		 \label{e:3.15.48}
		 \frac{\log L_{\eta(i(t))}-\log L_{\eta(0)}}{i(t)}\leq \frac{t-\log
		 L_{\eta(0)}}{i(t)}\leq \frac{\log L_{\eta(i(t)+1)}-\log L_{\eta(0)}}{i(t)}
	 \end{equation}
	 which implies that
	 \begin{equation}
		 \label{e:3.15.49}
		 \frac{1}{i(t)}\sum_{j=0}^{i(t)-1}G_i\leq \frac{t-\log
		 L_{\eta(0)}}{i(t)}\leq \frac{1}{i(t)}\sum_{j=0}^{i(t)}G_i
	 \end{equation}
	 and Lemma \ref{3.20} now yields
	$\frac{t-\log L_{\eta(0)}}{i(t)}\rightarrow \mathbb{E}G_0.$	 {Observing that} $\frac{\log L_{\eta(0)}}{i(t)}\rightarrow 0$ almost surely as
	 $t\rightarrow \infty$ completes the proof of \eqref{e:3.15.47}. Before
	 moving on, we record that the above also shows that almost surely as $t\rightarrow \infty$
	 \begin{equation}
		 \label{e:3.15.50*}
		\frac{1}{i(t)} (t-\log L_{\eta(i(t))})\rightarrow 0.
	 \end{equation}
	 We now show that
	 \begin{equation}
		 \label{e:3.15.43}
		 \frac{1}{i(t)}\int_0^t
	 e^{\mathbf{i}\sum_{j=1}^l
	 \lambda_j\Upsilon_j(\underline{h}_{\Gamma_{e^s}})}ds\rightarrow \mathbb{E}
	 Z_0
	 \end{equation}
	 almost surely. First notice that we can write
	 \begin{equation}
		 \label{e:3.15.44}
		 \frac{1}{i(t)}\int_{\log L_{\eta(0)}}^{\log L_{\eta(i(t))}}
	 e^{\mathbf{i}\sum_{j=1}^l
	 \lambda_j\Upsilon_j(\underline{h}_{\Gamma_{e^s}})}ds=
	 \frac{1}{i(t)}\sum_{j=0}^{i(t)}Z_j.
	 \end{equation}
	 Since $i(t)\rightarrow \infty$ as $t\rightarrow \infty$ almost surely, we have that 
	 \begin{equation}
		 \label{e:3.15.46}
		 \frac{1}{i(t)}\int_{\log L_{\eta(0)} }^{\log L_{\eta(i(t))}}
	 e^{\mathbf{i}\sum_{j=1}^l
	 \lambda_j\Upsilon_j(\underline{h}_{\Gamma_{e^s}})}ds\rightarrow
	 \mathbb{E}Z_0
	 \end{equation}
	 invoking Lemma \ref{3.19}.
	 To
	 finish showing \eqref{e:3.15.43}, we need to show that changing the
	 limits of integration from $(0,t)$ in \eqref{e:3.15.43} to $(\log L_{\eta(0)},\log
	 L_{\eta(i(t))})$ does not change the limiting quantity as $t\rightarrow \infty$.
	 Indeed, we can write
	 \begin{displaymath}
		 \left|\frac{1}{i(t)}\int_{0}^{\log L_{\eta(0)}}
	 e^{\mathbf{i}\sum_{j=1}^l
	 \lambda_j\Upsilon_j(\underline{h}_{\Gamma_{e^s}})}ds\right|\leq
	 \frac{\log L_{\eta(0)}}{i(t)}\rightarrow 0
	 \end{displaymath}
	 as $t\rightarrow \infty$. Finally, it remains to show that
	 \begin{displaymath}
		 \frac{1}{i(t)}\int_{\log L_{\eta(i(t))}}^{t}
	 e^{\mathbf{i}\sum_{j=1}^l
	 \lambda_j\Upsilon_j(\underline{h}_{\Gamma_{e^s}})}ds\rightarrow 0
	 \end{displaymath}
	 as $t\rightarrow \infty$. To see this, note that the above expression is
	 bounded in absolute value by $\frac{1}{i(t)} (t-\log L_{\eta(i(t))})$
	 which almost surely goes to $0$ as $t\rightarrow \infty$ by \eqref{e:3.15.50*}. This completes the proof of
	 \eqref{e:3.15.43}.
	
	 We now write the characteristic function from \eqref{e:3.15.42} as
	 \begin{equation}
		 \label{e:3.15.51}
		\frac{1}{i(t)}\int_0^t
	 e^{\mathbf{i}\sum_{j=1}^l
	 \lambda_j\Upsilon_j(\underline{h}_{\Gamma_{e^s}})}ds\times
	 \frac{i(t)}{t},
	 \end{equation}
	 and this on combining with \eqref{e:3.15.43} and \eqref{e:3.15.47}
	 immediately implies that

	 \begin{equation}
		 \label{e:3.15.52}
		  \mathbb{E}_t\left[ \exp\left( \sum_{j=1}^l \mathbf{i}\lambda_j
		  \Upsilon_j(\mathbf{Field}_t) \right) \right]\rightarrow
		  \frac{\mathbb{E}
		  Z_0}{\mathbb{E}G_0}
	 \end{equation}
	 almost surely as $t\rightarrow \infty$, and this completes the proof of
	 the convergence of the characteristic function to a deterministic
	 function of $\left\{ \lambda_j \right\}_{1\leq j\leq l}$ and
	 $\left\{ \Upsilon_j
 \right\}_{1\leq j\leq l}$.

To finish, it remains to show that $\frac{\mathbb{E}
		  Z_0}{\mathbb{E}G_0}$ is continuous
		  at $(\lambda_1,\dots,\lambda_l)=(0,\dots,0)$. Here,
		  $\mathbb{E}G_0$ does not depend on
		  $(\lambda_1,\dots,\lambda_l)$ while the numerator $\mathbb{E}
		  Z_0$ is continuous at
		  $(\lambda_1,\dots,\lambda_l)=(0,\dots,0)$ by Lemma
		  \ref{3.17}. This completes the proof.
\end{proof}
In the remainder of this section, we discuss the corresponding
results for the case of the empirical metric, we will work with the following choices
of the functions
$\Upsilon_1,\dots,\Upsilon_l$ from the beginning of this section.
\begin{equation}
	\label{upsilon2}
	\Upsilon_j(d)=d(x_j,y_j) \text{ for a continuous metric } d \text{ on
	}\mathbb{D}
	\text{ and points } x_j,y_j\in \mathbb{D} \text{ where } j\in \left\{ 1,\dots,l
	\right\}.
\end{equation}
As in the case of the field, we can define
\begin{equation}
	Z_i^\mathrm{metric}(\left\{
\lambda_j \right\}_{1\leq j\leq l},\left\{ \Upsilon_j \right\}_{1\leq j\leq
l}):=\mathcal{P}_{i}\int_{\log L_{\eta(i)}}^{\log
	L_{\eta(i)+1}}
	e^{\mathbf{i}\sum_{j=1}^l
	\lambda_j\Upsilon_j(\underline{D}_{h,\Gamma_{e^s}})}ds\label{e:3.5.20*}
\end{equation}
and the above quantity will be abbreviated to $Z_i^\text{metric}$. All the
results and proofs of this section directly apply if $Z_i$ is replaced by
$Z_i^\text{metric}$. Before closing this section, we state the main results
of this section for the case of the metric.
\begin{lemma}
	\label{metric_rep}
	We have
	\begin{displaymath}
		Z^\mathrm{metric}_i=\int_{0}^{
	Y_i}
	\frac{e^{\mathbf{i}\sum_{j=1}^l
	\lambda_j\Upsilon_j(\underline{D}_{\mathcal{H}_i,\Gamma_{s}(\mathcal{H}_i)})}}{\sum_{j=-\infty}^{i-1} Y_j
	e^{\sum_{k=j}^{i-1}D_i}+s}ds,
	\end{displaymath}
	where we recall that $\Gamma_{s}(\mathcal{H}_i)$ was defined to be a
	shorthand for
	$\Gamma_{s}\left(p_0(\mathcal{H}_i),p_1(\mathcal{H}_i);\mathcal{H}_i\right)$
	parametrised according to LQG length with respect to
	$D_{\mathcal{H}_i}$.
\end{lemma}
\begin{lemma}
	\label{metric:decay}
	The bi-infinite sequence $\left\{ Z^\mathrm{metric}_i \right\}_{i\in
	\mathbb{Z}}$ is stationary with exponential decay of correlations.
\end{lemma}
\begin{lemma}
	\label{metric_conv1}
$\lim_{n\rightarrow
\infty}\frac{\sum_{i=0}^n Z^\mathrm{metric}_{i}}{n}\rightarrow
\mathbb{E}Z^\mathrm{metric}_0$
	almost surely. 
\end{lemma}
\begin{proposition}
	\label{metric:3.21}
	For any $\Upsilon_1,\dots,\Upsilon_l$ as in \eqref{upsilon2}, almost surely in the
	randomness of $h$, the law of the random vector
	\begin{displaymath}
		\left(
 \Upsilon_1(\emm_t),\dots,\Upsilon_l(\emm_t) \right)
	\end{displaymath}
	converges in distribution to a deterministic measure (depending on the
	choice of $\Upsilon_1,\dots,\Upsilon_l$) on $\mathbb{R}^l$ as
	$t\rightarrow \infty$.
\end{proposition}

\section{Tightness and the proofs of Theorems
\ref{main1} and Theorem \ref{main2}}
\label{s:tightness}

With the above preparation, we now move on to the proof of Theorems \ref{main1} and \ref{main2}. We start with the latter.
 \subsection{Tightness of the empirical metrics and the proof of Theorem
 \ref{main2}}
 \label{ss:metrictight}
 The proof of Theorem \ref{main2} can be completed by using Proposition
 \ref{metric:3.21}
 to show the convergence of finite dimensional distributions with an
 additional step needed to show tightness. Before delving into the details, we briefly outline the strategy.

 Recall that for a given instance of $h$,
 $\mathbf{Metric}_t$ is a random element of the space $C(\mathbb{D}\times
 \mathbb{D})$, where the randomness comes from $\mathfrak{t}\sim
 \mathtt{Unif}(0,t)$; in fact, by using Lemma \ref{lengthcts} along with the fact
 that the space of continuous length metrics is a measurable subset of
 $C(\mathbb{D}\times \mathbb{D})$ (see Section \ref{ss:meas}), we have that
 for any fixed $t$, $\mathbf{Metric}_t$ is
 supported on the set of continuous length metrics on $\mathbb{D}$. It remains to show the tightness of this family of measures in the space $C(\mathbb{D}\times \mathbb{D})$. However, to handle certain technicalities we go about it in the following two steps.

\textbf{Step 1:} 
We define, for any $\varepsilon\in
 (0,1),$ 
 $\mathbf{Metric}_t^{\varepsilon}$, a slightly modified version of
 $\mathbf{Metric}_t$, namely the restriction:
 \begin{equation}
	 \label{e:3.15.53}
	 \mathbf{Metric}_t^{\varepsilon}=\mathbf{Metric}_t\lvert_{\overline{\mathbb{D}}_{(1-\varepsilon)}\times
	 \overline{\mathbb{D}}_{(1-\varepsilon)}}.
 \end{equation}
 Hence for each instance of $h$, the law of $\mathbf{Metric}_t^{\varepsilon}$ is a
measure on the space $C(\overline{\mathbb{D}}_{(1-\varepsilon)}\times
 \overline{\mathbb{D}}_{(1-\varepsilon)})$. The essential reason behind the definition is that  uniform H\"older
 estimates from \cite{DFGPS20} only apply directly to $\mathbf{Metric}_t$ as long as
 one stays away from the boundary of $\mathbb{D}$. 
 
We first show that there exists a random continuous
 metric $\mathbf{Metric}^\varepsilon$ on
 $\overline{\mathbb{D}}_{(1-\varepsilon)}$ such that almost surely in the randomness of
 $h$, we have
 \begin{equation}
	 \label{e:weak}
	 \mathbf{Metric}^\varepsilon_t\stackrel{d}{\rightarrow}\mathbf{Metric}^\varepsilon,
 \end{equation}
 where the weak convergence is with respect to the topology of uniform
 convergence on $C(\overline{\mathbb{D}}_{(1-\varepsilon)}\times
 \overline{\mathbb{D}}_{(1-\varepsilon)})$. 
 
\textbf{Step 2:} 
 Next, we check that the
 $\mathbf{Metric}^\varepsilon$ satisfy a consistency property as $\varepsilon$
 varies and then define the
	 metric $\mathbf{Metric}$ on $\mathbb{D}$ by  $\mathbf{Metric}\lvert_{\overline{\mathbb{D}}_{(1-\varepsilon)}\times
	 \overline{\mathbb{D}}_{(1-\varepsilon)}}=\mathbf{Metric}^\varepsilon$ for any
	 $\varepsilon\in (0,1)$. Now, \eqref{e:weak} will imply that almost surely
	 in the randomness of $h$, we have
	 $\mathbf{Metric}_t\stackrel{d}{\rightarrow}\mathbf{Metric}$
	 where the convergence is now with respect to the topology of uniform
	 convergence on compact subsets of $\mathbb{D}\times \mathbb{D}$ which will complete
the proof of Theorem \ref{main2}.

\bigskip 

For Step 1, we shall need a tightness statement for
 $\mathbf{Metric}^\varepsilon_t$. This will follow from an
 equicontinuity result
which will be a local H\"older
 continuity estimate similar to Proposition \ref{imp3} (see Lemma \ref{3.25} for a precise statement). Recall that Proposition \ref{imp3} says that for
 any fixed compact set, with probability going to $1$ polynomially fast as
 $\alpha$ goes to $0$, the metric
 $D_h(\cdot,\cdot,\mathbb{D}_{2\alpha(z)})$ is H\"older
 continuous as a metric on the ball
 $\mathbb{D}_{2\alpha}(z)$ simultaneously for all such balls contained in the
 compact set. To obtain the corresponding estimate
 for $\mathbf{Metric}_t^\varepsilon$, we will use the {decomposition of the infinite geodesic}
 $\Gamma$ into segments established in Section \ref{disjseg}. We will use Proposition \ref{imp3} to assert that the empirical metric around a fixed segment satisfies a similar 
 local H\"older continuity estimate and then use the fast decay of correlations for the segments to conclude that with
 high probability, a large
 fraction of the segments contributing to $\mathbf{Metric}_t^\varepsilon$ have
H\"older continuous local metrics which will translate to
 saying that with high probability in the randomness of $h$, a large
 fraction of points on $\Gamma$ having their $\log$-LQG distance from the
 origin in $(0,t)$
have their empirical metrics as H\"older continuous, and this will be enough to imply the desired tightness of the metrics $\mathbf{Metric}^\varepsilon_t$ for any fixed
 $\varepsilon$.

 Note that $\mathbf{Metric}^\varepsilon$, the right hand side in
 \eqref{e:weak}, is a weak limit of metrics and is thus a-priori only a
 pseudo metric (one way to see this is to use a Skorokhod embedding to convert
 weak convergence to almost sure convergence; note that a Skorokhod
 embedding exists as the space
 $C(\overline{\mathbb{D}}_{1-\varepsilon}\times\overline{\mathbb{D}}_{1-\varepsilon}))$
 with the $\sup$-norm metric is separable). To show that it is in fact a metric, we will use Proposition
 \ref{imp3*} to show a H\"older continuity lower bound for
 $\mathbf{Metric}^\varepsilon$ with respect to the Euclidean metric.
 
 In view of Propositions \ref{imp3} and \ref{imp3*}, choose any
 $\chi \in (0,\xi(Q-2))$ and $\chi'>\xi(Q-2)$; this choice will remain fixed throughout
 this section. For any $\mathfrak{d}\in(0,1)$ and any $r_1,r_2$ satisfying
 $1<r_1/(1-\delta)<r_2$, define the
 event $A^{\mathfrak{d},\varepsilon}_{r_1,r_2}[\mathtt{h}]$ for a whole plane GFF
 marginal
 $\mathtt{h}$ by

\begin{equation}
	 \label{e:3.15.54}
	 A^{\mathfrak{d},\varepsilon}_{r_1,r_2}[\mathtt{h}]:=\left\{ \left| u-v
	 \right|^{\chi'}\leq\underline{D}_{\mathtt{h},w}(u,v)\leq
	 \left| u-v \right|^{\chi}\text{ for all }
	 w\in \mathbb{C}_{[r_1/(1-\delta),r_2]}\text{ and } u,v\in \overline{\mathbb{D}}_{(1-\varepsilon)}\text{ with } |u-v|\leq \mathfrak{d}
	 \right\}.
 \end{equation}

By replacing $\mathtt{h}$ by $h$ in the above definition, we can similarly
define the corresponding event $A_{r_1,r_2}^{\mathfrak{d},\varepsilon}$ for all
$0\leq r_1/(1-\delta)<r_2$. Note that the parameter $\delta \in (0,1)$
above is the one used in the definition of
the empirical quantities from Section \ref{ss:main_results}; 
this parameter $\delta$ is always fixed and hence suppressed in the
notation $A^{\mathfrak{d},\varepsilon}_{r_1,r_2}[\mathtt{h}]$. 

The event in \eqref{e:3.15.54} ensures that all
points in the annulus $\mathbb{C}_{[r_1/(1-\delta),r_2]}$ have H\"older
continuous metrics in their vicinity; by Propositions
\ref{imp3} and \ref{imp3*} this event will be shown to have probability going to $1$ as
$\mathfrak{d}\rightarrow 0$. The usefulness of the above event is that on
the intersection of the same with the event that one of the segments of $\Gamma$ is contained entirely
within $\mathbb{C}_{[r_1/(1-\delta),r_2]}$, all points on this segment have their corresponding local metrics  H\"older continuous. The
reason for the $\frac{r_1}{1-\delta}$ term is that we
want $ A^{\mathfrak{d},\varepsilon}_{r_1,r_2}[\mathtt{h}]$ to be measurable
with respect to $\sigma\left( \mathtt{h}\lvert_{\mathbb{C}_{>r_1}}
\right)$. For any point $w\in \mathbb{C}_{[r_1/(1-\delta),r_2]}$, we have
that $\mathbb{D}_{\delta|w|}(w)\subseteq \mathbb{C}_{>r_1}$ and this
ensures the desired measurability.

Since the use of the parameters $\varepsilon$ and $\mathfrak{d}$ is local to
this section, we now suppress them and simply write
$A_{r_1,r_2}[\mathtt{h}]$ instead of
$A^{\mathfrak{d},\varepsilon}_{r_1,r_2}[\mathtt{h}]$. In the same vein, we will next define variables
$I_i$ and $J_i$ for $i\in \Z$ whose dependence on  $\varepsilon$ and $\mathfrak{d}$ will be suppressed as well.

We begin
with an informal description and describe how they will aid in establishing
tightness. The variable $I_i$ will be the indicator of the event
$A_{S_i,S_{\iota(i+1)+1}}$. By the definition of
$\iota(i)$ from \eqref{e:iota}, thus, on the event $\left\{ I_i=1 \right\}\cap \left\{
\mathcal{P}_i=1 \right\}$, the segment of the geodesic $\Gamma$ starting at
$p_{\eta(i)}$, the coalescence point occurring just after the radius $S_i$, is contained in the region
$\mathbb{C}_{[r_1/(1-\delta),r_2]}$ with $r_1=S_i$ and
$r_2=S_{\iota(i+1)+1}$. Thus by the preceding discussion, all points on this
segment have their corresponding local metrics H\"older
continuous.

Recall that we want to show that most points on $\Gamma$ have their corresponding local metrics H\"older
continuous. By the above, argument the segment of $\Gamma$ starting at $p_{\eta(i)}$ contributes a total of $J_i=G_iI_i$ in $\log$-LQG
length (recall $G_i$ from \eqref{e:3.17}) to the the total $\log$-LQG length
contributed by points $\Gamma_s$ on the infinite geodesic
with H\"older continuous local metrics $\underline{D}_{\Gamma_s,h}$ around them.

We now record the above
definitions of $I_i$ and $J_i$ formally.
 \begin{gather}
	 I_{i}=\mathbbm{1}\left(A_{S_i,S_{\iota(i+1)+1}}\right),\nonumber\\
	 J_{i}=G_i I_{i},
	 \label{e:3.15.55}
 \end{gather}
 We can similarly define the variables $I_i(\mathtt{h})$ and
 $J_i(\mathtt{h})$ for a whole plane GFF marginal $\mathtt{h}$ for all
 $i\in \mathbb{N}\cup\left\{ 0 \right\}$. 

 It is straightforward to see that the variable $I_{0}(\mathtt{h})$ is
 measurable with respect to
 $\mathscr{G}_{\iota(1,\mathtt{h})+1}(\mathtt{h})$. The following
 important lemma yields that the variable $I_i$, which is a-priori measurable only with
 respect to $\sigma(h)$, is actually a function only of the rescaled field
 $\mathcal{H}_i$.

 \begin{lemma}
	 \label{I_irescaled}
	 We have $I_i=I_0(\mathcal{H}_i)$ almost surely for all $i\in
	 \mathbb{Z}$.
 \end{lemma}
 \begin{proof}
	 Observe that we have the a.s.\ equality
	 $S_{\iota(i+1)+1}/S_i=S_{\iota(1,\mathcal{H}_i)+1}$ for all $i \in
	 \mathbb{Z}$. In view of this, it suffices to show that almost
	 surely, for all $w\in
	 \mathbb{C}_{>1/(1-\delta)}$ simultaneously, the equality
	 $\underline{D}_{\mathcal{H}_i,w}=\underline{D}_{h,S_iw}$ holds. To see this, we first note that by the coordinate change
	 formula along with locality, almost surely,
	 $D_{\mathcal{H}_i}(x,y)=(S_i)^{-\xi Q}e^{-\xi
	 \mathbf{Av}(h,\mathbb{T}_{S_i})}D_h(S_ix,S_iy;\mathbb{C}_{>S_i})$ for all
	 $x,y\in \mathbb{C}_{>1}$ simultaneously. Since
	 $\mathbb{D}_{\delta|w|}(w)\subseteq \mathbb{C}_{>1}$ for all $w\in
	 \mathbb{C}_{>1/(1-\delta)}$, the above implies that almost surely,
	 simultaneously for all $w\in \mathbb{C}_{>1/(1-\delta)}$ and
	 $u,v\in \mathbb{D}$, 
	 \begin{displaymath}
	 D_{\mathcal{H}_i}(w+\delta
	 |w|u,w+\delta |w|v;\mathbb{D}_{\delta|w|}(w))= (S_i)^{-\xi Q}e^{-\xi
	 \mathbf{Av}(h,\mathbb{T}_{S_i})}D_{h}(S_iw+\delta
	 |S_iw|u,S_iw+\delta |S_iw|v;\mathbb{D}_{\delta|S_iw|}(S_iw)).	 
	 \end{displaymath}
	 On multiplying both sides by $(\delta|w|)^{-\xi Q}$ and re-writing the
	 equation in terms of the $\underline{D}$
	 notation from \eqref{e:metric}, we obtain that almost surely, for all
	 $w\in \mathbb{C}_{>1/(1-\delta)}$ simultaneously,
	 \begin{equation}
		 \label{showexp0}
		 \underline{D}_{\mathcal{H}_i,w}=e^{-\xi\left(
	 \mathbf{Av}(h,\mathbb{T}_{S_i})-\mathbf{Av}({h},\mathbb{T}_{\delta|S_i
	 w|}(S_i
	 w))+\mathbf{Av}(\mathcal{H}_i,\mathbb{T}_{\delta|w|}(w))\right)} \underline{D}_{h,S_iw}
	 \end{equation}
	To complete the proof, it suffices to prove that the quantity in the
	exponent on the right hand side is zero. To see this, note that $\mathbf{Av}({h},\mathbb{T}_{\delta|S_i
	 w|}(S_i
	 w))=\mathbf{Av}(h\lvert_{\mathbb{C}_{>S_i}}\mydot \psi_{S_i},\mathbb{T}_{\delta|
	 w|}(
	 w))=\mathbf{Av}(h,\mathbb{T}_{S_i})+\mathbf{Av}(\mathcal{H}_i,\mathbb{T}_{\delta|w|}(w))$,
	 where we just used the definition
	 $\mathcal{H}_i=h\lvert_{\mathbb{C}_{>S_i}}\mydot
	 \psi_{S_i}-\mathbf{Av}(h,\mathbb{T}_{S_i})$. The above combined with
	 \eqref{showexp0} yields that almost surely, the equality
	 $\underline{D}_{\mathcal{H}_i,w}=\underline{D}_{h,S_iw}$ holds for all $w\in
	 \mathbb{C}_{>1/(1-\delta)}$ simultaneously and completes the proof.
 \end{proof}

For any fixed $\mathfrak{d} \in (0,1)$ and $\varepsilon$, we define the sequences
 $\mathcal{I},\mathcal{J}$  by 
 $\mathcal{I}_{i}=I_{i}$ and
 $\mathcal{J}_{i}=J_{i}$. 
 We have
 the following analogues of Lemma \ref{3.16} for the sequences $\mathcal{I}$
 and $\mathcal{J}$.
 \begin{lemma}
	 \label{3.211}
	 For any fixed
	 $\mathfrak{d}\in (0,1)$, 
	 the bi-infinite sequence
	 $\mathcal{I}=\left\{I_{i}\right\}_{i\in\mathbb{Z}}$ is  stationary with exponential
	 decay of correlations.
 \end{lemma}

 \begin{proof}
	 The proof consists of an application of Lemma \ref{'3.163}. In view of
	 the a.s.\ equality $I_i=I_0(\mathcal{H}_i)$ from Lemma
	 \ref{I_irescaled}, we define, for a
	 whole plane GFF marginal $\mathtt{h}$,
	 $X(\mathtt{h})=I_{0}(\mathtt{h})$ and 
	 $\tau(\mathtt{h})=\iota(1,\mathtt{h})+1$. Note that  the same
	 definition of $\tau$ was used in the proof of Lemma \ref{3.15.1} and recall that we already showed then that $\tau(\mathtt{h})$ is a
	 stopping time for $\mathscr{G}(\mathtt{h})$. The measurability of
	 $X(\mathtt{h})$ with respect to
	 $\mathscr{G}_{\tau(\mathtt{h})}(\mathtt{h})$ follows by the comment on the
	 measurability of $I_0(\mathtt{h})$ after \eqref{e:3.15.55}.
	 To justify the application of Lemma \ref{'3.163}, we need to 
	 show the finiteness of $\mathbb{E}X(\mathtt{h})^4$, but that is obvious because
	 $\mathcal{I}_0(\mathtt{h})\leq 1$. The required tail bound for
	 $\tau(\mathtt{h})$ follows by an application of Lemma \ref{3.11}. We can now invoke Lemma \ref{'3.163} and
	 this completes the proof.
 \end{proof}
 \begin{lemma}
	 \label{3.22}
	 For any fixed
	 $\mathfrak{d}\in (0,1),$ the bi-infinite sequence
	 $\mathcal{J}=\left\{J_{i}\right\}_{i\in\mathbb{Z}}$ is
	 stationary, exhibiting exponential
	 decay of correlations.
\end{lemma}

\begin{proof}
		 The proof consists of an application of Lemma \ref{3.163} and proceeds
	 along the same lines as the proof of Lemma \ref{3.16}. Recall that
	 $J_i=G_iI_i$ and thus the definitions of
	 $[f(\mathcal{H})]_i,f_m(\mathcal{H})]_i,X_m(\mathtt{h})$ here would be
	 $I_0(\mathcal{H}_i),I_0(\mathcal{H}_i),I_{m+1}(\mathtt{h})$ times the
	 corresponding definitions for the $G_i$ appearing in Lemma \ref{3.16}
	 (recall that $G_i$ is a special case of $Z_i$). We omit further details to avoid repetition.
\end{proof}

The following lemma states that as $t\rightarrow \infty$, the limiting fraction of the segments of $\Gamma$ such that all points on them have a H\"older continuous local metric around them weighted by their corresponding log-LQG lengths, is at least $\E J_0$.

\begin{lemma}
	\label{3.24}
	For any $\varepsilon>0,\mathfrak{d}\in (0,1)$, almost surely, $		\lim_{n\rightarrow
		\infty}\frac{\sum_{i=0}^n J_{i}}{n}\rightarrow
		\mathbb{E}J_0.$
\end{lemma}
\begin{proof}
	The proof is the same as that of Lemma \ref{3.19}.
\end{proof}
Along with Proposition \ref{3.21}, the following lemma coming from the H\"older
continuity estimate in Proposition \ref{imp3} will be the main ingredient
in completing the proof of Theorem \ref{main2}.
\begin{lemma}
	\label{3.23}
	For each $\varepsilon\in (0,1)$ fixed, we have that
	$\mathbb{P}\left( I_0=1
		\right)\rightarrow 1
	$
	as $\mathfrak{d}\rightarrow 0$. In other words, for each
	$\varepsilon>0$,
	$I_0\stackrel{d}{\rightarrow}
	1$ as $\mathfrak{d}\rightarrow 0$.
	\end{lemma}
Since the proof of Lemma \ref{3.23} is technical, we postpone it for now and
use it to complete the proof of Theorem \ref{main2} first. The following lemma
strongly relies on Lemma \ref{3.23} and is what will directly be used in the
proof of Theorem \ref{main2}. 

\begin{lemma}
	\label{3.25}
	For every fixed $\varepsilon>0$ and any $\mathfrak{d}\in(0,1)$, there exists a constant
	$c_\mathfrak{d}$ satisfying $c_\mathfrak{d}\rightarrow 0$ as
	$\mathfrak{d}\rightarrow 0$ such that almost surely in the
	randomness of $h$,
	\begin{displaymath}
		\liminf_{t\rightarrow \infty}\mathbb{P}_t\left( |x-y|^{\chi'}\leq \mathbf{Metric}_t^\varepsilon(x,y)
		\leq |x-y|^\chi \text{ for all } x,y \in
		\overline{\mathbb{D}}_{(1-\varepsilon)} \text{ with }
		|x-y|\leq \mathfrak{d}\right) \geq 1-c_\mathfrak{d}.
	\end{displaymath}
\end{lemma}

\begin{proof}
	 The aim is to lower bound the $\log$-LQG length contributed by the
	 points on $\Gamma\lvert_{[0,e^t]}$ with a local H\"older continuous metric around them by the $\log$-LQG distance
	 contributed by those segments of $\Gamma$ which satisfy that all points
	 on the segment have the same property.
 Thus for any $t>0$, we have
	\begin{equation}
		\label{e:3.15.551}
	\mathbb{P}_t\left( |x-y|^{\chi'}\leq \mathbf{Metric}_t^\varepsilon(x,y)
		\leq |x-y|^\chi \text{ for all } x,y \in
		\overline{\mathbb{D}}_{(1-\varepsilon)} \text{ with }
		|x-y|\leq \mathfrak{d}\right)\geq \frac{\sum_{j=0}^{i(t)-1} J_{j}}{t},
	\end{equation}
	where we use the definition of $i(t)$ from \eqref{e:3.15.421}. It now remains to
	show that almost surely in the randomness induced by $h$, we have
	$\frac{\sum_{j=0}^{i(t)-1} J_{j}}{t}=1-c_\mathfrak{d}$
	as $t\rightarrow \infty$ for some constant $c_\mathfrak{d}$ satisfying
	$c_\mathfrak{d}\rightarrow 0$ as $\mathfrak{d}\rightarrow 0$. To see this, we first write
	\begin{equation}
		\label{e:3.15.553}
		\frac{\sum_{j=0}^{i(t)-1}
		J_{j}}{t}=\frac{\sum_{j=0}^{i(t)-1}
		J_{j}}{i(t)}\times \frac{i(t)}{t}
	\end{equation}
	imitating \eqref{e:3.15.51}. By using \eqref{e:3.15.47}, we know that
	$\frac{i(t)}{t}\rightarrow \frac{1}{\mathbb{E}G_0}$ almost surely
	as $t\rightarrow \infty$. Combining this with Lemma \ref{3.24} and using
	that $i(t)\rightarrow \infty$ almost surely as $t\rightarrow \infty$, 
	\begin{equation}
		\label{e:3.15.554}
			\frac{\sum_{j=0}^{i(t)-1}
			J_{j}}{t}\rightarrow \frac{\mathbb{E}J_{0}}{\mathbb{E}G_0}.
	\end{equation}
	Now, Lemma \ref{3.23} implies that $\mathcal{I}_{0}\stackrel{d}{\rightarrow}
	1$ as $\mathfrak{d}\rightarrow 0$, which implies
	that
	$G_0I_{0}\stackrel{d}{\rightarrow}
	G_0$ as
	$\mathfrak{d}\rightarrow 0$ which is the same as
	$J_{0}\stackrel{d}{\rightarrow}
	G_0$. This implies that
	$\mathbb{E}J_{0}\rightarrow\mathbb{E}G_0$
	as $\mathfrak{d}$ goes to $0$; this can be seen by using $|J_0|\leq G_0$
	along with $\mathbb{E}G_0^8<\infty$ to obtain that $\left\{J_0\right\}$ is uniformly
	integrable as a family in $\mathfrak{d}$. Thus we can choose the mapping $\mathfrak{d}\mapsto c_{\mathfrak{d}}$
satisfying
\begin{equation}
	\label{e:3.15.555}
	\frac{\mathbb{E}J_{0}}{\mathbb{E}G_0}
	= 1-c_\mathfrak{d}
\end{equation}
along with $c_\mathfrak{d}\rightarrow 0$ as $\mathfrak{d}\rightarrow 0$.
On combining this with \eqref{e:3.15.554} and \eqref{e:3.15.551}, the proof is
complete.
\end{proof}
For any
$d\in C( \overline{\mathbb{D}}_{(1-\varepsilon)}\times
\overline{\mathbb{D}}_{(1-\varepsilon)})$
	 and $\mathfrak{d}\in(0,1)$, define the modulus of
	 continuity $\omega(d,\mathfrak{d})$ by
	 \begin{equation}
		 \label{e:3.15.57}
		 \omega(d,\mathfrak{d})=\sup_{x_1,y_1,x_2,y_2\in
		 \overline{\mathbb{D}}_{(1-\varepsilon)};|x_1-x_2|\leq \mathfrak{d},|y_1-y_2|\leq
		 \mathfrak{d}}|d(x_1,y_1)-d(x_2,y_2)|.
	 \end{equation}
	 The above lemma yields the following corollary about the modulus of
continuity of $\mathbf{Metric}_t^\varepsilon$.
\begin{lemma}
	\label{3.25*}
		For every fixed $\varepsilon>0$ and any $\mathfrak{d}\in(0,1)$, there exists a constant
	$c_\mathfrak{d}$ satisfying $c_\mathfrak{d}\rightarrow 0$ as
	$\mathfrak{d}\rightarrow 0$ such that almost surely in the
	randomness of $h$,
	\begin{displaymath}
		\liminf_{t\rightarrow \infty} \mathbb{P}_t\left(
		\omega(\mathbf{Metric}_t^{\varepsilon},\mathfrak{d}) \leq
		2\mathfrak{d}^\chi \right)\geq 1-c_\mathfrak{d}
	\end{displaymath}
\end{lemma}
\begin{proof}
		First notice that for any metric $d\in C(\overline{\mathbb{D}}_{(1-\varepsilon)}\times
 \overline{\mathbb{D}}_{(1-\varepsilon)})$ and any points
 $(x_1,y_1),(x_2,y_2)\in  \overline{\mathbb{D}}_{(1-\varepsilon)}\times
 \overline{\mathbb{D}}_{(1-\varepsilon)}$, we have by the triangle inequality,
 \begin{equation}
	 \label{e:triangle}
	 |d(x_1,y_1)-d(x_2,y_2)|\leq d(x_1,x_2)+d(y_1,y_2)
 \end{equation}
 Thus if we know that $d(x_1,x_2)\leq \mathfrak{d}^\chi$ and $d(y_1,y_2)\leq
 \mathfrak{d}^\chi$, we can conclude that $|d(x_1,y_1)-d(x_2,y_2)|\leq
 2\mathfrak{d}^\chi$. The proof can now be completed by an application of
 Lemma \ref{3.25} using the definition \eqref{e:3.15.57} along with a union bound.
\end{proof}
We can now complete the proof of
Theorem \ref{main2} by using Proposition \ref{metric:3.21} and Lemma \ref{3.25}. 
 \begin{proof}[Proof of Theorem \ref{main2}] Recall our proof strategy from the discussion at the beginning of this subsection about 
considering the measures
	 $\mathbf{Metric}^\varepsilon_t$ for any $\varepsilon\in(0,1)$ (see \eqref{e:3.15.53}) and showing that almost surely in the randomness of
	 $h$,
	 \begin{equation}
		 \label{e:3.15.56}
		 \mathbf{Metric}^\varepsilon_t\stackrel{d}{\rightarrow} \mathbf{Metric}^{\varepsilon}
	 \end{equation}
We first quickly show why this suffices. For any $\varepsilon_1<\varepsilon_2$, we have the consistency property  $\mathbf{Metric}^{\varepsilon_1}\lvert_{\overline{\mathbb{D}}_{(1-\varepsilon_2)}\times
\overline{\mathbb{D}}_{(1-\varepsilon_2)}}\stackrel{d}{=}\mathbf{Metric}^{\varepsilon_2}$  since the same property holds for 
	 $\mathbf{Metric}^\varepsilon_t$ by definition.
	 With the above consistency, we can define $\mathbf{Metric}$ such that
	 $\mathbf{Metric}\lvert_{\overline{\mathbb{D}}_{(1-\varepsilon)}\times
	 \overline{\mathbb{D}}_{(1-\varepsilon)}}\stackrel{d}{=}\mathbf{Metric}^\varepsilon$.
	 A quick way to see this by using Kolmogorov's extension theorem for
	 $\mathbb{R}^\mathbb{N}$ is to use the above consistency to define the law of
	 $\mathbf{Metric}\lvert_{\mathbb{Q}^2\cap
	 (\mathbb{D}\times\mathbb{D})}$ as a measure on the space
	 $\mathbb{R}^{\mathbb{Q}^2\cap
	 (\mathbb{D}\times\mathbb{D})}$. Since $\mathbf{Metric}^\varepsilon$ is a.s.\
	 continuous for any fixed $\varepsilon\in (0,1)$, $\mathbf{Metric}\lvert_{\mathbb{Q}^2\cap
	 (\mathbb{D}\times\mathbb{D})}$ further restricted to
	 $\mathbb{Q}^2\cap(\overline{\mathbb{D}}_{(1-\varepsilon_i)}\times
	 \overline{\mathbb{D}}_{(1-\varepsilon_i)})$  must be almost surely
	 uniformly continuous simultaneously for all $i\in \mathbb{N}$, where
	 $\varepsilon_i=2^{-i}$. Thus $\mathbf{Metric}\lvert_{\mathbb{Q}^2\cap
	 (\mathbb{D}\times\mathbb{D})}$ a.s.\ has a continuous
	 extension to $C(\mathbb{D}\times \mathbb{D})$ and this
	 completes the definition of $\mathbf{Metric}$ in terms of the
	 $\mathbf{Metric}^\varepsilon$. 

	 Choosing $\varepsilon_i=\frac{1}{2^i}$ for $i\in \mathbb{N}$ and using that
	 \eqref{e:3.15.56} holds almost surely together for all $i$, we obtain that almost surely in the randomness of $h$, 
	 \begin{displaymath}
		 \mathbf{Metric}_t\stackrel{d}{\rightarrow}\mathbf{Metric}
	 \end{displaymath}
	 as $t\rightarrow \infty$ where the convergence now is with respect to the
	 slightly modified topology of uniform convergence on compact subsets of
	 $\mathbb{D}\times \mathbb{D}$ which completes the proof of the theorem.

	 Thus, it suffices to show \eqref{e:3.15.56}. That is, we need to argue the weak convergence and show that the
	 limit $\mathbf{Metric}^\varepsilon$ is almost surely a metric.
	 Given a realisation of $h$, considering
	the law of $\mathbf{Metric}^\varepsilon_t$ as a measure on the space
	 $C(\overline{\mathbb{D}}_{(1-\varepsilon)}\times \overline{\mathbb{D}}_{(1-\varepsilon)})$, we first show the convergence of finite dimensional marginals.  For $(x_1,y_1),\dots,(x_l,y_l) \in \overline{\mathbb{D}}_{(1-\varepsilon)}\times
	 \overline{\mathbb{D}}_{(1-\varepsilon)}$  almost surely in the randomness of
	 $h$, the convergence in distribution of the law of the
	random vector
	 \begin{displaymath}
		 \left(\mathbf{Metric}^\varepsilon_t(x_1,y_1),\dots,\mathbf{Metric}^\varepsilon_t(x_l,y_l)\right)
	 \end{displaymath}
 as $t\rightarrow \infty$ to a deterministic measure depending on the points $(x_1,y_1),\dots,(x_l,y_l)$ 
	 follows by considering for 
	 any  $d\in
	 C(\mathbb{D}\times \mathbb{D})$, 
	 $\Upsilon_j(d)=d(x_j,y_j)$ for all $1\leq j\leq l$
in Proposition \ref{metric:3.21}.
	 
	 It remains to show the tightness of
	 $\mathbf{Metric}^\varepsilon_t$. 
 Notice that
	 $\mathbf{Metric}^\varepsilon_t( 0,0)$ is supported on the singleton
	 $\left\{ 0 \right\}$ for all $t$. By using Arzela-Ascoli theorem to
	 characterise tightness (see e.g.\ \cite[Theorem
	 7.1, Theorem 7.2, Theorem 7.3]{Bil99}; the statements in this
	 source are for the unit interval $[0,1]$ but the same
	 proof works for $\overline{\mathbb{D}} \times \overline{\mathbb{D}}$), it suffices to show that almost surely
	 in the randomness of $h$, for every $\theta_1,\theta_2$ positive, there exist
	 $\mathfrak{d},t_0$ depending on $\theta_1,\theta_2$ and the instance of $h$ such that we have
\begin{equation}
	\label{e:3.15.58}
	\mathbb{P}_t\left( \omega(\mathbf{Metric}_t^{\varepsilon},\mathfrak{d})\geq
	\theta_1 \right)\leq \theta_2
\end{equation}
for all $t\geq t_0$. This is an immediate consequence of Lemma \ref{3.25*}.
Thus putting everything together we obtain \eqref{e:3.15.56} and thus have
$\mathbf{Metric}^\varepsilon$, a random element of
$C(\overline{\mathbb{D}}_{(1-\varepsilon)}\times
\overline{\mathbb{D}}_{(1-\varepsilon)})$ which is almost surely a
pseudo-metric (as indicated earlier, one way to see this is to upgrade
weak convergence to an almost sure convergence via a Skorokhod embedding). 

It now remains to argue that the latter is in fact almost surely a
metric. First note that Lemma
\ref{3.25} implies that almost surely
\begin{displaymath}
	\mathbb{P}_t\left( \mathbf{Metric}_t^\varepsilon(x,y)
\geq |x-y|^{\chi'} \text{ for all } x,y \in
		\overline{\mathbb{D}}_{(1-\varepsilon)} \text{ with }
		|x-y|\leq \mathfrak{d}\right) \geq 1-c_{\mathfrak{d}}
\end{displaymath}
for all large $t$ for some constant $c_\mathfrak{d}\rightarrow 0$ as
$\mathfrak{d}\rightarrow 0$. Since the relevant subset of
$C(	\overline{\mathbb{D}}_{(1-\varepsilon)}\times
\overline{\mathbb{D}}_{(1-\varepsilon)})$ is closed, on taking weak limits
 along with the Portmanteau theorem,
 we in particular obtain that
 \begin{equation}
	 \label{e:hold0}
	\mathbb{P}_\infty\left( \mathbf{Metric}^\varepsilon(u,v)
\geq |u-v|^{\chi'} \text{ for all } u,v \in
		\overline{\mathbb{D}}_{(1-\varepsilon)} \text{ with }
		|u-v|\leq \mathfrak{d}\right)\rightarrow 1
 \end{equation}
 as $\mathfrak{d}\rightarrow 0$. This allows us to say that for close-by
 points $u,v$, $\mathbf{Metric}^\varepsilon(u,v)>0$ almost surely. {To handle the
 case when $u,v$ are far from each other}, we now use properties of length spaces to note that for any $\mathfrak{d}\leq
 \varepsilon/2$, we have the inclusions
		\begin{align}
			\label{e:hold1}
			&\left\{  \mathbf{Metric}_t^{\varepsilon/2}(x,y)
\geq |x-y|^{\chi'} \text{ for all } x,y \in
		\overline{\mathbb{D}}_{(1-\varepsilon/2)} \text{ with }
		|x-y|\leq \mathfrak{d} \right\} \\
		&\subseteq \left\{  \mathbf{Metric}_t^{\varepsilon/2}(x,y)
\geq |x-y|^{\chi'} \text{ for all } x \in
\overline{\mathbb{D}}_{(1-\varepsilon)}, y\in
\overline{\mathbb{D}}_{\mathfrak{d}}(x) \right\} \nonumber\\
		&\subseteq \left\{
		\mathbf{Metric}^\varepsilon_t(u,v)\geq \mathfrak{d}^{\chi'} \text{ for all }
			u,v\in \overline{\mathbb{D}}_{(1-\varepsilon)} \text{
			satisfying } |u-v|\geq \mathfrak{d}\right\}\nonumber.
		\end{align}
	The second inclusion above is true because for any $u,v\in \overline{\mathbb{D}}_{(1-\varepsilon)}$ with
	$|u-v|\geq \mathfrak{d}$, we can write
	\begin{equation}
		\label{e:hold2}
		\mathbf{Metric}^\varepsilon_t(u,v)\geq \inf_{z\in
	\mathbb{T}_\mathfrak{d}(u)}\mathbf{Metric}_t^{\varepsilon/2}(u,z).
	\end{equation}
	This is true because $\mathbf{Metric}_t$ itself is almost surely a
	continuous
	length metric since it is defined as a random induced metric of $D_h$ (see Lemma
	\ref{lengthcts}). That is, we know
	that $\mathbf{Metric}^\varepsilon_t(u,v)=\inf_{\zeta:u\rightarrow v, \zeta
	\subseteq \mathbb{D}} \ell (\zeta,
	\mathbf{Metric}_t)$ and we can now note that if $|u-v|\geq
	\mathfrak{d}$, then any path $\zeta$ from $u$ to $v$ must intersect the
	circle $\mathbb{T}_\mathfrak{d}(u)$ and this shows \eqref{e:hold2}.
	Now, by using \eqref{e:hold1} along with using Lemma \ref{3.25} with
	$\varepsilon$ replaced by $\varepsilon/2$, we obtain that almost surely in the
	randomness of $h$, we have that for any $\mathfrak{d}\leq \varepsilon/2$,
	\begin{align}
		\label{e:hold3}
		&\mathbb{P}_t\left( \mathbf{Metric}^\varepsilon_t(u,v)\geq
		\mathfrak{d}^{\chi'} \text{ for all }
			u,v\in \overline{\mathbb{D}}_{(1-\varepsilon)} \text{
			satisfying } |u-v|\geq \mathfrak{d} \right)\nonumber\\
			&\geq \mathbb{P}_t\left( \mathbf{Metric}_t^{\varepsilon/2}(x,y)
\geq |x-y|^{\chi'} \text{ for all } x,y \in
		\overline{\mathbb{D}}_{(1-\varepsilon/2)} \text{ with }
		|x-y|\leq \mathfrak{d} \right)\nonumber\\
		&\geq 1-c_\mathfrak{d}
	\end{align}
	for some constant $c_\mathfrak{d}$ which goes to $0$ as
	$\mathfrak{d}\rightarrow 0$. On taking weak limits and using \eqref{e:3.15.56}
 along with the Portmanteau theorem, the above yields
 \begin{equation}
	 \label{e:hold4}
	 \mathbb{P}_\infty\left( \mathbf{Metric}^\varepsilon(u,v)\geq
		\mathfrak{d}^{\chi'} \text{ for all }
			u,v\in \overline{\mathbb{D}}_{(1-\varepsilon)} \text{
			satisfying } |u-v|\geq \mathfrak{d} \right)\rightarrow 1
 \end{equation}
 as $\mathfrak{d}\rightarrow 0$. Now by combining \eqref{e:hold0} and
 \eqref{e:hold4} along with a union bound, we can say that
 \begin{equation}
	 \label{e:hold5}
	 \mathbb{P}_\infty\left(\mathbf{Metric}^\varepsilon(u,v)>0  \text{ for all }
			u,v\in \overline{\mathbb{D}}_{(1-\varepsilon)} \text{
			satisfying }u\neq v \right)\geq 1-c_\mathfrak{d}
 \end{equation}
 for all $\mathfrak{d}\leq \varepsilon/2$. On sending $\mathfrak{d}$ to zero,
 we have that $c_\mathfrak{d}\rightarrow 0$ and we
 thus obtain that $\mathbf{Metric}^\varepsilon$ is almost surely a metric
 for all fixed $\varepsilon>0$ and hence by further taking $\varepsilon$ along the sequence $\left\{
 2^{-i} \right\}$ yields that $\mathbf{Metric}$ is almost surely a
 metric. This
 completes the proof.
\end{proof}
\begin{remark}
	Note that \eqref{e:3.15.56} in the above proof along with Lemma \ref{3.25}
	yields that for any fixed $\varepsilon>0$, any $\chi \in (0,\xi(Q-2))$ and
	any $\chi'>\xi(Q-2)$,
		\begin{displaymath}
			\mathbb{P}_\infty\left( |x-y|^{\chi'}\leq \mathbf{Metric}^\varepsilon(x,y)
		\leq |x-y|^\chi \text{ for all } x,y \in
		\overline{\mathbb{D}}_{(1-\varepsilon)} \text{ with }
		|x-y|\leq \mathfrak{d}\right) \rightarrow 1
		\end{displaymath}
		as $\mathfrak{d}\rightarrow 0$.
		In fact, from the use of Proposition \ref{imp3} in the upcoming
		proof of Lemma \ref{3.23}, it can be shown that the above
		convergence happens polynomially as $\mathfrak{d} \rightarrow 0$. This
		gives a local H\"older continuity statement for the empirical metric
		$\mathbf{Metric}$ on compact subsets of $\mathbb{D}$.
\end{remark}

We now provide the outstanding proof of Lemma \ref{3.23}.
\begin{proof}[Proof of Lemma \ref{3.23}]
	We will work a fixed choice of
	$\varepsilon$. 
	Note that the definition
	$A_{r_1,r_2}$ from \eqref{e:3.15.54} can be expanded as 
	 \begin{align}
	 \label{e:3.15.54*}
	 A_{r_1,r_2}=&\Big\{ \left| \frac{u-v}{\delta|w|}
	 \right|^{\chi'}\leq (
	 \delta |w|)^{-\xi Q} e^{-\xi
	 \mathbf{Av}(h,\mathbb{T}_{\delta |w|}(w))}
	 D_h(u,v;\mathbb{D}_{ \delta|w|}(w))\leq
	 \left| \frac{u-v}{\delta|w|} \right|^{\chi}\nonumber \\&\text{ for all }
	 w\in \mathbb{C}_{[r_1/(1-\delta),r_2]}\text{ and } u,v\in \overline{\mathbb{D}}_{(1-\varepsilon)\delta
	 |w|}(w)\text{ with } |u-v|\leq \mathfrak{d} \delta |w|
	 \Big\}.
 \end{align}
 	Observe that it
	suffices to show that
$\mathbb{P}\left( A_{1,S_{\iota(1)+1}}\right)\rightarrow 1.$
	 We will be relying on Propositions \ref{imp3} and
	 \ref{imp3*} but need to make certain alterations. Towards this,  we
	 first bound the tail of $\iota(1)$ to get rid of the random
 radius $S_{\iota(1)+1}$. Given $\theta>0$, choose a deterministic real number
	$\ell>1/(1-\delta)$ large enough such that
	$
		\mathbb{P}\left( S_{\iota(1)+1} \geq \ell \right)\leq \theta.$
Thus to complete the proof it simply suffices to show that for $\mathfrak{d}$ small enough,
	$
		\mathbb{P}\left( A_{1,\ell} \right)\geq
		1-\theta.
$

	Notice that \eqref{e:3.15.54*} involves the term $(
	 \delta |w|)^{-\xi Q} e^{-\xi
	 \mathbf{Av}(h,\mathbb{T}_{\delta |w|}(w))}$ which is not
	 present in the H\"older continuity estimates that we wish to use. To address this we show that the above term lies in the interval $[1/x_1,x_1]$ for
	 some small $x_1>0$ and all $w$ in the given set with high probability. That is, 
	\begin{equation}
		\label{e:3.15.593}
		\mathbb{P}\left( 
		\frac{1}{x_1}\leq (
	 \delta |w|)^{-\xi Q} e^{-\xi
	 \mathbf{Av}(h,\mathbb{T}_{\delta |w|}(w))} \leq x_1 \text{ for all } w\in
	 \mathbb{C}_{[1/(1-\delta),\ell]}\right)\geq
	 1-\theta/2
	\end{equation}
	and this uses the continuity of the circle average process as
	described in Section \ref{ss:GFF}.
	We now define the event $B_{1,\ell}$ by
	\begin{align}
	 \label{e:3.15.594}
	 B_{1,\ell}=&\Big\{ \frac{x_1\left| u-v
	 \right|^{\chi'}}{(\delta/(1-\delta))^{\chi'}}\leq 
	 D_h(u,v;{\mathbb{D}_{ \delta|w|}(w)})\leq
	 \frac{\left| u-v \right|^{\chi}}{x_1(\delta \ell)^\chi}\nonumber \\&\text{ for all }
	 w\in \mathbb{C}_{[1/(1-\delta),\ell]}\text{ and } u,v\in \overline{\mathbb{D}}_{(1-\varepsilon)\delta
	 |w|}(w)\text{ with } |u-v|\leq \mathfrak{d} \delta |w|
	 \Big\}
 \end{align}
 and substituting $|w|$ with the bounds $1/(1-\delta)\leq |w|\leq
 \ell$ in \eqref{e:3.15.593}, it is easy to see that we need only show that 
 \begin{equation}
	 \label{e:3.15.595}
	 \mathbb{P}\left( B_{1,\ell} \right)\geq 1-\theta/2
 \end{equation}
 for all $\mathfrak{d}$ small enough.

 Finally, we note that the description of the event in \eqref{e:3.15.594} stipulates the  points $u,v$ to lie close to a point $w$. In order to match the
 statements of the H\"older estimates, we wish to eliminate the use of $w$ and we define the event
 $C_{1,\ell}$ by
 \begin{align}
	 \label{e:3.15.596}
	 C_{1,\ell}=&\Big\{ \frac{x_1\left| u-v
	 \right|^{\chi'}}{(\delta/(1-\delta))^{\chi'}}\leq D_h(u,v) \text{ and
	 }
	 D_h(u,v;{\mathbb{D}_{ 2|u-v|}(u)})\leq
	 \frac{\left| u-v \right|^{\chi}}{x_1(\delta \ell)^\chi}\nonumber\\
	 &\qquad\text{ for all }
	   u,v\in \mathbb{C}_{[1,\ell(1+\delta)]}\text{ with } |u-v|\leq \mathfrak{d} \delta
	   \ell
	 \Big\}.
 \end{align}
We now show that	$C_{1,\ell}\subseteq B_{1,\ell}$
 as long as $\mathfrak{d}$ is small enough so that
 $(1-\varepsilon)\delta+2\mathfrak{d}\delta\leq \delta$, i.e. $\mathfrak{d}<\varepsilon/2$. To see this, first note that if $u\in
 \mathbb{C}_{[1/(1-\delta),\ell]}$, then 
 $\mathbb{D}_{\delta|u|}(u)\subseteq \mathbb{C}_{[1,\ell(1+\delta)]}$. Now
 $D_h(u,v;{\mathbb{D}_{ \delta|w|}(w)})\geq D_h(u,v)$ because the
 induced metric takes an infimum over a smaller collection of paths.
 Also, if  $(1-\varepsilon)\delta+2\mathfrak{d}\delta\leq
 \delta$ along with $u,v\in \overline{\mathbb{D}}_{(1-\varepsilon)\delta
	 |w|}(w)$ and $|u-v|\leq \mathfrak{d} \delta |w|$ for some $w\in
	 \mathbb{C}_{[1/(1-\delta),\ell]}$, then  $\mathbb{D}_{ 2|u-v|}(u)\subseteq
 \mathbb{D}_{\delta|w|}(w),$ which implies that $D_h(u,v;{\mathbb{D}_{
 2|u-v|}(u)})\geq D_h(u,v;{\mathbb{D}_{ \delta|w|}(w)})$. 
 
 Thus, it is enough to
 show that $
	\mathbb{P}\left( C_{1,\ell} \right)\geq 1-\theta/2
 $ for all $\mathfrak{d}$ small enough, and this is immediate by Propositions
 \ref{imp3} and \ref{imp3*} since $x_1,\delta,\ell$ are just constants.
 \end{proof}

 \subsection{Proof of Theorem \ref{main1}}
 \label{ss:tightnessfield}
 The structure of the argument to prove Theorem \ref{main1} is almost identical to the one used for
 Theorem \ref{main2}. To avoid repetition we will be brief.
 The convergence of finite dimensional
 distributions is obtained from Proposition \ref{3.21}. Regarding tightness, we
 use that the GFF measure is supported on $H_0^{-\varepsilon}(\mathbb{D})$ for any
 $\varepsilon>0$ and that the embedding $H_0^{-\varepsilon/2}(\mathbb{D})\subseteq
 H_0^{-\varepsilon}(\mathbb{D})$ is compact. Indeed,
 by embedding theorems for Sobolev spaces, $H_0^{-\varepsilon/2}(\mathbb{D})$ is
 compactly contained in $H_0^{-\varepsilon}(\mathbb{D})$; here, we have used the statement that
 $W_0^{a,2}(\mathbb{D})$ is compactly contained in $W_0^{b,2}(\mathbb{D})$ if $a-b>0$, where we
	use the standard notation $W_0^{k,p}$ for Sobolev spaces with
	Dirichlet boundary conditions. In our setting, we
	have $H_0^{-\varepsilon}(\mathbb{D})=W_0^{-\varepsilon,2}(U)$ and
	$H_0^{-\varepsilon/2}(\mathbb{D})=W_0^{-\varepsilon/2,2}(\mathbb{D})$. The
	above implies that bounded sets in $H_0^{-\varepsilon/2}(\mathbb{D})$ are pre-compact in
	$H_0^{-\varepsilon}(\mathbb{D})=W_0^{-\varepsilon,2}(\mathbb{D})$. The above statements are true
	for nice enough open sets $U$ instead of $\mathbb{D}$ and a detailed
	discussion is present in \cite[Theorem
 6.3]{AF03}.

 Since most of the argument is identical to the one for Theorem
 \ref{main2}, we only define the analogues of the quantities from
 Section \ref{ss:metrictight} and state the counterpart results and simply 
 elaborate on the part where the compactness of the embedding $H_0^{-\varepsilon/2}(\mathbb{D})\subseteq
 H_0^{-\varepsilon}(\mathbb{D})$ is used to complete the proof of Theorem
 \ref{main2}.

From the discussion after Proposition \ref{imp4}, we know that for any fixed
$\epsilon>0$, for almost
every instance of $h$, the laws of $\mathbf{Field}_t$ for all $t>0$ are all
supported on the space $H_0^{-\varepsilon}(\mathbb{D})$. Recalling the convergence of finite dimensional distributions from
Proposition \ref{3.21}, it remains to establish that the sequence $\left\{ \mathbf{Field}_t
\right\}_{t\geq 0}$ is almost surely tight in the space $H_0^{-\varepsilon}(\mathbb{D})$
for all $l>0$. Note that the space
$H_0^{-\varepsilon}(\mathbb{D})$ is a Polish space-- it is complete by
definition (see the discussion before \eqref{e:innprod}) and one way to argue
separability is consider elements of $H_0^{-\varepsilon}(\mathbb{D})$ with
rational coefficients in their expansion in terms of the Laplace
eigenvectors (see \cite[Section 1.5]{Ber07}). Thus Prokhorov's theorem is
valid for the space $H_0^{-\varepsilon}(\mathbb{D})$. 

For $N>0$, now define the event $A^{N}_{r_1,r_2}[\mathtt{h}]$ by
\begin{align}
	 \label{e:a1}
	 A^{N}_{r_1,r_2}[\mathtt{h}]=\left\{
	 \|\underline{\mathtt{h}}_w\|_{H_0^{-\varepsilon}(\mathbb{D})}\leq N \text{ for all }
	 w\in \mathbb{C}_{[r_1/(1-\delta),r_2]} \right\}.
 \end{align}
 and similarly define $A_{r_1,r_2}^N$ be replacing $\mathtt{h}$ by $h$.
 As we did in the previous subsection, we will suppress $N$ and simply use $
 A_{r_1,r_2}[\mathtt{h}]$ to denote $ A^{N}_{r_1,r_2}[\mathtt{h}]$. Continuing to suppress the $N$ dependence, define the
 random variables $I_{i}$ and $J_{i}$ for $i\in \mathbb{Z}$ by 
 \begin{gather}
	 \label{e:a3.15.55}
	 I_{i}=\mathbbm{1}\left(A_{S_i,S_{\iota(i)}}\right)\nonumber\\
	 J_{i}=G_i I_{i}
 \end{gather}
 
 We now state a sequence of lemmas which are analogous to the ones used in
 Section \ref{ss:metrictight}. Recall the notation $\iota(\cdot)$ and the
 variant $\iota(\cdot,\cdot)$ from
 \eqref{e:iota}.
 \begin{lemma}
	 \label{a3.211}
For any fixed $N>0$, the sequences $\mathcal{I}_{i}=I_{i}$ and $\mathcal{J}_{i}=J_{i}$
	 indexed by $\mathbb{Z}$ are  stationary sequences having exponential
	 decay of correlations.
	 \end{lemma}

 \begin{lemma}
	\label{a3.24}
	For each $N>0$, almost surely, $
		\lim_{n\rightarrow
		\infty}\frac{\sum_{i=0}^n J_{i}}{n}\rightarrow
		\mathbb{E}J_{0}
$.
\end{lemma}
The following lemma has a proof analogous to that of Lemma \ref{3.23} but
uses Proposition \ref{imp4} as the input instead of Propositions
\ref{imp3} and \ref{imp3*}.
\begin{lemma}
	\label{a3.23}
	$\mathbb{P}\left( I_0=1
		\right)\rightarrow 1$
	as $N\rightarrow \infty$. 
\end{lemma}
An argument analogous to the one used in Lemma \ref{3.25*} now yields the
following lemma.
\begin{lemma}
	\label{3.21.2}
	Fix any $\epsilon>0$. For every fixed $N>0$,  there exists a constant
	$c_N$
	satisfying $c_N\rightarrow 0$ as $N\rightarrow \infty$ such that almost
	surely in the randomness of $h$,
	\begin{displaymath}
		\liminf_{t\rightarrow \infty}\mathbb{P}_t\left(
		\|\mathbf{Field}_t\|_{H_0^{-\varepsilon}(\mathbb{D})}\leq N\right)\geq 1-c_N.
	\end{displaymath}
\end{lemma}
We can now complete the proof of tightness and hence Theorem \ref{main1} by using the above
lemmas.
\begin{proof}[Proof of Theorem \ref{main1}]

We wish to show that a.s., for any $\theta_2>0$, there exists a  compact set
$S_{\theta_2}\subseteq H_0^{-\varepsilon}(\mathbb{D})$ and $t_0>0$ (both possibly depending on the
instance of $h$) such that
\begin{equation}
	\label{e:a2}
	\mathbb{P}_t\left( \mathbf{Field}_t\in S_{\theta_2}\right)\geq
	1-\theta_2
\end{equation}
for all $t\geq t_0$. We now use $\mathfrak{B}_{(N,\epsilon/2)}$ to denote the closed
ball of radius $N$ around the $0$ in the Banach space
$H_0^{-\varepsilon/2}(\mathbb{D})$. Using Lemma \ref{3.21.2} with $\epsilon/2$
instead of $\epsilon$, we obtain that almost surely in the randomness of $h$, for any
$\theta_2>0$, there exists an $N_{\theta_2}$ large enough such that 
\begin{equation}
	\label{e:a3}
	\mathbb{P}_t\left( \mathbf{Field}_t\in
	\mathfrak{B}_{(N_{\theta_2},\epsilon/2)}\right)\geq
	1-\theta_2
\end{equation}
for all $t\geq t_0$. Now the fact that the embedding
$H_0^{-\varepsilon/2}(\mathbb{D})\subseteq H_0^{-\varepsilon}(\mathbb{D})$ is compact as described
at the beginning of the subsection and that $\mathfrak{B}_{N,\epsilon/2}$ is a bounded set in
$H_0^{-\varepsilon/2}(\mathbb{D})$, imply that the closure
$\overline{\mathfrak{B}}_{(N_{\theta_2},\epsilon/2)}$ of
$\mathfrak{B}_{(N_{\theta_2},\epsilon/2)}$ as a subset of
$H_0^{-\varepsilon}(\mathbb{D})$ is compact. Setting
$S_{\theta_2}=\overline{\mathfrak{B}}_{(N_{\theta_2},\epsilon/2)}$ completes the proof.
\end{proof}
We conclude this section with a remark on the convergence of the empirical LQG measure along $\Gamma$.

\begin{remark}\label{measure1}
	Using the measurable map $\mathtt{h}\mapsto \mu_\mathtt{h}$
 from Section \ref{ss:LQG}, one can define the measure
 $\underline{\mu}_{\mathtt{h},x}$ on the unit
 disk $\mathbb{D}$ by 
 \begin{equation}
	 \label{e:measure}
	 \underline{\mu}_{\mathtt{h},x}=|\delta
	 x|^{-\gamma Q}
	 e^{-\gamma
	 \mathbf{Av}(\mathtt{h},\mathbb{T}_x(\delta
	 |x|))}\mu_{\mathtt{h}}\circ
	 \Psi_{x,\delta}
 \end{equation}
whenever $\mathbb{D}_{\delta|x|}(x)\subset U$. Note the exponent $\xi$ in \eqref{e:metric} switches to $\gamma$ in \eqref{e:measure} since $\xi=\frac{\gamma}{d_\gamma}$ where $d_\gamma$ is the scaling between   the volume and the distance exponents.
 
One can  define $\emv_t$ as the random variable
 $\underline{\mu}_{h,\Gamma_{e^{\mathfrak{t}}}}$, where $\mathfrak{t}\sim
 \mathtt{Unif}(0,t)$ and prove a counterpart to Theorems \ref{main1}
 and \ref{main2} using similar arguments. We chose not to carry out the details in view of not lengthening this manuscript further.

\end{remark}

\section{Constructing $\mathbf{Field}$ and $\mathbf{Metric}$ via a
finite geodesic}
\label{ss:construct}
Recall the discussion from Section \ref{s:iop} and Figure \ref{f.proofsketch3}, of an explicit description of $\emf$ and $\emm$
as the size biased version of the 
empirical field and metric corresponding to a finite geodesic in
$\mathbb{C}_{> 1}$ for the metric $D_h$. Having such a representation, as opposed to just being the almost sure weak limits of
$\emf_t$ and $\emm_t$, is useful to deduce properties of the limiting objects without having to resort to, say, invoking the Portmanteau theorem. Indeed, the latter isn't particularly useful in cases where we will consider sets that are neither open or closed with respect to the relevant topology.

We also recall the definition of $G_i$ from
\eqref{e:3.17}. We now present the main statement of this section.

\begin{theorem}
	\label{ac:1}
	The measures $\nu_\mathbf{Field}$ on
	$\mathcal{D}'(\mathbb{D})$ and $\nu_{\mathbf{Metric}}$ on
	$C(\mathbb{D}\times \mathbb{D})$ from Theorems \ref{main1} and \ref{main2} have the following alternate
	descriptions. After fixing an instance of $h$, let $\mathfrak{t}$
	be a random variable distributed as  $\mathtt{Unif}[\log
	D_h(0,p_0),\log D_h(0,p_1)]$.	
Let $\nu$ denote the random measure on
	$\mathcal{D}'(\mathbb{D})\times C(\mathbb{D}\times \mathbb{D})$ defined as:
	\begin{equation}
	 \label{e:density*}
	 \nu=\frac{G_0}{\mathbb{E}G_0}\times
	 \left(\text{law of }
	 (\underline{h}_{\Gamma_{e^\mathfrak{t}}},\underline{D}_{h,\Gamma_{e^\mathfrak{t}}})\right)
	\end{equation}
In words, $\nu$ is the joint law of the empirical field and empirical metric around the point
$\Gamma_{e^\mathfrak{t}}$  if we sample a field obtained by size biasing the law of $h$ by $G_0$. Let $\nu_1$ and $\nu_2$ denote the projections of $\nu$ to 
$\mathcal{D}'(\mathbb{D})$ and  $C(\mathbb{D}\times \mathbb{D})$ respectively (note that they are random). Then
	\begin{gather*}
		\nu_\mathbf{Field}=\mathbb{E}\nu_1,\\
		\nu_\mathbf{Metric}=\mathbb{E}\nu_2.
	\end{gather*}
	In particular, $\mathbb{E}\nu$ yields a coupling between
	$\nu_\mathbf{Field}$ and $\nu_\mathbf{Metric}$.
\end{theorem}

\begin{proof}
	We will only prove the statement
	$\mathbb{E}\nu_1=\nu_{\mathbf{Field}}$ since the proof of the
	corresponding claim about $\nu_{\mathbf{Metric}}$ is analogous. Note that the geodesic $\widetilde{\Gamma}$ between $p_0$ and
	$p_1$ is unique because the geodesic $\Gamma$ passes through $p_0$ and
	$p_1$ and is unique.

	Since both $\mathbb{E}\nu_1$ and $\nu_\mathbf{Field}$ are deterministic measures on
	$\mathcal{D}'(\mathbb{D})$, it suffices to show that they induce the same finite
	dimensional distributions. 
	
	Note that by definition 
	\begin{equation}
		\label{e:nu1}	\nu_1=\frac{G_0}{\mathbb{E}G_0}\times (\text{law of }
		\underline{h}_{\Gamma_{e^\mathfrak{t}}}).
	\end{equation}
	Let $\phi_1,\dots,\phi_l\in \mathcal{D}(\mathbb{D})$ be fixed. We will show for any $\lambda_1,\dots,\lambda_l\in \mathbb{R}$, 
		\begin{equation}
		\label{eac:1}
		\int
		e^{\mathbf{i}\sum_{j=1}^l\lambda_j(\mathtt{h},\phi_j)}d\nu_{\mathbf{Field}}(\mathtt{h})=\int
		e^{\mathbf{i}\sum_{j=1}^l\lambda_j(\mathtt{h},\phi_j)}d(\mathbb{E}\nu_1)(\mathtt{h}).
	\end{equation}
	Also by definition, 
	\begin{equation}
		\label{eac:2}
			\int
		e^{\mathbf{i}\sum_{j=1}^l\lambda_j(\mathtt{h},\phi_j)}d\nu_{\mathbf{Field}}(\mathtt{h})=
		\lim_{t\rightarrow \infty}\mathbb{E}_t\left[ \exp\left( \sum_{j=1}^l \mathbf{i}\lambda_j
		(\emf_t,\phi_j) \right) \right]=\frac{\mathbb{E}
		  Z_0}{\mathbb{E}G_0},
	\end{equation}
	where we have used \eqref{e:3.15.52}; we recall from \eqref{upsilon1},
	\eqref{e:3.5.20} that the $Z_i$ variables are defined
	with the choice
	$\Upsilon_j(\cdot)=(\cdot,\phi_j)$ for all $j\in \left\{ 1,\dots,l
	\right\}$. 
	
	We now show that
	\begin{equation}
		\label{eac:4}
		\int
		e^{\mathbf{i}\sum_{j=1}^l\lambda_j(\mathtt{h},\phi_j)}d\nu_1(\mathtt{h})=\frac{Z_0}{\mathbb{E}G_0}.
	\end{equation}
	First note that on the event $\left\{ \mathcal{P}_0=0 \right\}$, both the
	left and right hand sides in the above equation are $0$ and this is
	because we have $G_0=Z_0=0$ on the above event. Assuming 
	$\left\{ \mathcal{P}_0=1 \right\}$ and thus $G_0>0$, and noting that $G_0=\log D_h(0,p_1)-\log
	D_h(0,p_0)$, we can write
	\begin{equation}
		\label{eac:4*}
		\int
		e^{\mathbf{i}\sum_{j=1}^l\lambda_j(\mathtt{h},\phi_j)}d\nu_1(\mathtt{h})=
		\frac{G_0}{\mathbb{E}G_0}\times \frac{1}{G_0}\int_{\log L_{0}}^{\log L_1}e^{\mathbf{i}\sum_{j=1}^l
	\lambda_j\Upsilon_j(\underline{h}_{\Gamma_{e^s}})}ds=\frac{G_0}{\mathbb{E}G_0}
	\times
	\frac{Z_0}{G_0}=\frac{Z_0}{\mathbb{E}G_0}.
	\end{equation}
	where we used definition \eqref{e:3.5.20}. Now standard measure
	theoretic arguments imply that the characteristic function of the expectation of a random measure
	is the expected value of the characteristic function of the random
	measure, and this shows
	\eqref{eac:1}, thereby completing the proof.	
\end{proof}

 Theorem \ref{ac:1} provides a coupling $(\emf,\emm),$ whose law we will
 henceforth be denoting by $\mathbb{P}_\infty$ which we had earlier
 used to only denote the marginals.
We will
now use Theorem \ref{ac:1} to capture null sets for $\mathbb{P}_\infty$,
and we begin by noting the following simple
consequence of Theorem \ref{ac:1}.

\begin{lemma}
	\label{null}
	A measurable set $A\subseteq \mathcal{D}'(\mathbb{D})\times
	C(\mathbb{D}\times \mathbb{D})$ satisfies
	$\mathbb{P}_\infty\left( (\mathbf{Field},\mathbf{Metric}) \in A \right)=0$ if and only if it satisfies $\nu(A)=0$
	almost surely in the randomness of $h$. 
\end{lemma}
\begin{proof}
	The proof follows by using that $\mathbb{P}_\infty\left( (\mathbf{Field},\mathbf{Metric}) \in A
	\right)=(\mathbb{E}\nu)(A)=\mathbb{E}[\nu(A)]$.
\end{proof}
The above lemma yields the following
consequence.
\begin{lemma}
	\label{null1}
	Suppose that a measurable set $A\subseteq
	\mathcal{D}'(\mathbb{D})\times
	C(\mathbb{D}\times \mathbb{D})$
	satisfies 
	\begin{displaymath}
	 \P\left(
	 (\underline{h}_{\Gamma_{e^\mathfrak{t}}},\underline{D}_{h,\Gamma_{e^\mathfrak{t}}})\in
	 A\big \vert h\right)=0.
	\end{displaymath}
	almost surely in the randomness of $h$. Then  $\nu(A)=0$ almost surely.
\end{lemma}
\begin{proof}
	The proof follows by using \eqref{e:density*}. 
Indeed, if a measure gives zero mass to a set, then the same is true for any constant multiple of the measure as well.
\end{proof}
The following consequence of the above lemmas will be used in Sections \ref{s:singularity} and 
\ref{s:ac} to obtain
information about null sets of $(\mathbf{Field},\mathbf{Metric})$ using the
corresponding information about the pre-limiting variables
$(\mathbf{Field}_t,\mathbf{Metric}_t)$ (and avoid the use of the Portmanteau theorem).
\begin{lemma}
	\label{null2*}
	Suppose that a measurable set $A\subseteq
	\mathcal{D}'(\mathbb{D})\times
	C(\mathbb{D}\times \mathbb{D})$
	satisfies that almost surely in the randomness of $h$, 
	$\mathbbm{1}\left(
	(\underline{h}_{\Gamma_t},\underline{D}_{h,\Gamma_t})\in
	A\right)=0$
	for almost every $t\in
	(0,\infty)$ with respect to
	the Lebesgue measure. Then, $\mathbb{P}_\infty\left( (\mathbf{Field},\mathbf{Metric}) \in A
	\right)=0$.
\end{lemma}
\begin{proof}
	To show that
  $\mathbb{P}_\infty\left( (\mathbf{Field},\mathbf{Metric}) \in A
	\right)=0$, by using Lemma \ref{null} and Lemma
	\ref{null1}, it suffices to show that almost surely in the randomness of
	$h$, we have
	$\P\left(
	 (\underline{h}_{\Gamma_{e^\mathfrak{t}}},\underline{D}_{h,\Gamma_{e^\mathfrak{t}}})\in
	 A\big\vert h\right)=0.$
This is an immediate consequence of the hypothesis after recalling that $\mathfrak{t}$ is distributed according to the
	uniform measure on $[\log
	D_h(0,p_0),\log D_h(0,p_1)]$.
\end{proof}
Lemma \ref {null2*} gives a handle on the null sets of
$(\mathbf{Field},\mathbf{Metric})$ in terms of the null sets for the
pre-limit. As a corollary, we immediately get the following lemma about its projections. 
\begin{lemma}
	\label{null2}
	Suppose that for a measurable set $A\subseteq
	\mathcal{D}'(\mathbb{D})$,
	 almost surely in the randomness of $h$,
	 	$\mathbbm{1}\left(\underline{h}_{\Gamma_t}\in
	A\right)=0$
	for (Lebesgue) almost every $t\in
	(0,\infty)$. Then $\nu_{\mathbf{Field}}(A)=0$.
An analogous result holds for $A\subseteq
	C(\mathbb{D}\times\mathbb{D})$
	and $\nu_{\mathbf{Metric}}(A)$.
\end{lemma}

The following gives an example of how the above lemma will typically be
applied.
\begin{lemma}
	\label{metric_length}
	$\mathbf{Metric}$ is almost surely a length metric.
\end{lemma}
\begin{proof}
	Consider the set $A=\left\{ d\in C(\mathbb{D}\times \mathbb{D}): d
	\text{ is a length metric} \right\}$. Note that $A$ is measurable by the
	discussion in Section \ref{ss:meas}. Since $D_h$ is a.s.\ a continuous
	length metric and induced metrics of length metrics are also continuous
	length metrics (see Lemma \ref{lengthcts}), we have that a.s.\ $\mathbbm{1}\left(
\underline{D}_{h,\Gamma_t}\in
	A\right)=0$
	for almost every $t\in
	(0,\infty)$. An application of Lemma \ref{null2} now completes the
	proof.
\end{proof}

\section{Singularity}
\label{s:singularity}
In this section, we prove Theorem \ref{mainsi1} (1). As described in
Section \ref{s:iop}, we will first prove Proposition \ref{singprop}
and then later use Proposition \ref{mainmeas1} to conclude Theorem \ref{mainsi1}
(1). For convenience we recall the statement of Proposition \ref{singprop}--
the random metrics $\emm$ and $D_{h\lvert_{\mathbb{D}}}(\cdot,\cdot)$ have
	mutually singular laws.
	
	To prove Proposition \ref{singprop}, it suffices to exhibit an event which has
probability $1$ for $D_{h\lvert_{\mathbb{D}}}$ and probability $0$ for
$\emm$. To do this, we will rely on Lemma \ref{null2} and instead
produce a set $A\subseteq C(\mathbb{D}\times \mathbb{D})$ such that
$\mathbbm{1}(D_{h\lvert_{\mathbb{D}}}\in A)=1$ almost surely while a.s.\ in the
randomness of $h$, we have $\mathbbm{1}\left(\underline{D}_{h,\Gamma_t}\in
	A\right)=0$
	for all $t\in
	(0,\infty)$.

Recalling the proof discussion in Section \ref{s:iop} and Figure
\ref{f.proofsketch2}, consider the measurable set
\begin{align}
	A_{\varepsilon}:=& \Big\{ d\in C(\mathbb{D}\times \mathbb{D}): d \text{ is a
	length metric, there exists some
	path }\zeta \text{ in } \mathbb{C}_{(3\varepsilon,4\varepsilon)}\nonumber\\
	&\text{ which disconnects }\mathbb{T}_{3\varepsilon}
	\text{ and
	$\mathbb{T}_{4\varepsilon}$ and satisfies }
	\ell(\zeta;d)<d(\mathbb{T}_{2\varepsilon},\mathbb{T}_{3\varepsilon}) \Big\}.
	\label{e:si1}
\end{align}
The measurability of $A_{\varepsilon}$ follows from the discussion in Section
\ref{ss:meas}.
	We will take the set $A$ to be 
	\begin{equation}
		\label{e:defnA}
		A=\limsup_{i\rightarrow \infty}A_{5^{-i}}=\bigcap_{n\geq
		1}\bigcup_{i\geq n}A_{5^{-i}}.
	\end{equation}
		To show that $\mathbbm{1}\left(\underline{D}_{h,\Gamma_t}\in
	A\right)=0$ almost surely, the idea is to notice that for a point $w\in
		\Gamma\setminus \left\{ 0 \right\}$ and the metric $\underline{D}_{h,w}$ on
		$\mathbb{D}$,  {$\Gamma \cap
		\mathbb{D}_{\delta|w|}(w)$ (which might not be connected)} when
		translated and scaled to align with
		the unit disk $\mathbb{D}$, contains a geodesic segment for $\underline{D}_{h,w}$ between two points on $\mathbb{T}$ which passes
		through $0$ and thus crosses
		the annuli $\mathbb{C}_{(2\varepsilon,3\varepsilon)}$ for all $\varepsilon$
		small (say $\varepsilon\leq 1/5$ so that
		$\mathbb{T}_{4\varepsilon}$ is comfortably inside $\mathbb{D}$). This implies that the event from
		\eqref{e:si1} can never occur for such an environment as its
		occurrence would imply that any geodesic between two points on
		$\mathbb{T}$ must short-cut by using the path
		$\zeta\subseteq \mathbb{C}_{(3\varepsilon,4\varepsilon)}$ instead of
		crossing the annulus $\mathbb{C}_{(2\varepsilon,3\varepsilon)}$. This
		would imply that almost surely in the
randomness of $h$, for each fixed $\varepsilon$,  $\mathbbm{1}\left(\underline{D}_{h,\Gamma_t}\in
	A_\varepsilon\right)=0$
	for every $t\in
	(0,\infty)$. Taking a countable union over $\varepsilon=5^{-i}$ would yield
	 that almost surely in the
randomness of $h$, we have $\mathbbm{1}\left(\underline{D}_{h,\Gamma_t}\in
	A\right)=0$ for every $t\in(0,\infty)$.
	
	On the other hand, to show that $D_{h\lvert_{\mathbb{D}}}\in
	A$ almost surely, we will first show that $\mathbb{P}\left(
	D_{h\lvert_{\mathbb{D}}}\in A_\varepsilon
	\right)$ is not dependent on $\varepsilon$ and is positive. Iterating this event across the scales $\varepsilon=5^{-i}$ yields that
	$\mathbb{P}\left( D_{h\lvert_{\mathbb{D}}}\in
	A\right)=1$ by using the independence structure of the GFF. This argument
	has been used multiple times in the literature already.

	We start with a series of lemmas towards proving that  $\mathbb{P}\left(
	D_{h\lvert_{\mathbb{D}}}\in A
	\right)=1$.

\begin{lemma}
	\label{si2}
The event $\left\{
	D_{h\lvert_{\mathbb{D}}}\in A_\varepsilon \right\}$ is measurable with
	respect to
	$\sigma(h\lvert_{\mathbb{C}_{(\varepsilon,4\varepsilon)}})$.

\end{lemma}
\begin{proof}
Observe that almost surely, $\ell(\zeta;D_{h\lvert_{\mathbb{D}}})=\ell(\zeta;D_{h\lvert_{\mathbb{D}}}(\cdot,\cdot;\mathbb{C}_{(\varepsilon,4\varepsilon)}))=\ell(\zeta;D_{h\lvert_{\mathbb{C}_{(\varepsilon,4\varepsilon)}}})$
 for all paths $\zeta\subseteq
\mathbb{C}_{(\varepsilon,4\varepsilon)}$ simultaneously, where we have used the
locality from Proposition \ref{b3} to obtain the last equality. Also,
$D_{h\lvert_{\mathbb{D}}}(\mathbb{T}_{2\varepsilon},\mathbb{T}_{3\varepsilon})=D_{h\lvert_{\mathbb{D}}}(\mathbb{T}_{2\varepsilon},\mathbb{T}_{3\varepsilon};\mathbb{C}_{(\varepsilon,4\varepsilon)})$
since $D_{h\lvert_{\mathbb{D}}}$ is a.s.\ a length metric.
Indeed, to obtain the above equality, notice that for any $z\in \mathbb{T}_{2\varepsilon}$ and $w\in
\mathbb{T}_{3\varepsilon}$ and any path $\zeta:[a,b]\rightarrow \mathbb{D}$  from
$z$ to $w$, there exist $a_1$, $b_1$ satisfying $a\leq a_1<b_1\leq b$
and $\zeta(a_1)\in \mathbb{T}_{2\varepsilon}, \zeta(b_1) \in
\mathbb{T}_{3\varepsilon}$ such that 
$\zeta\lvert_{[a_1,b_1]}\subseteq
\mathbb{C}_{[2\varepsilon,3\varepsilon]}\subseteq
\mathbb{C}_{(\varepsilon,4\varepsilon)} $. Since
$D_{h\lvert_{\mathbb{D}}}(\mathbb{T}_{2\varepsilon},\mathbb{T}_{3\varepsilon};\mathbb{C}_{(\varepsilon,4\varepsilon)})\leq
\ell(\zeta\lvert_{[a_1,b_1]};D_{h\lvert_{\mathbb{D}}})\leq \ell(\zeta;D_{h\lvert_{\mathbb{D}}})$, we can infimize the
right hand side over all $z,w,\zeta$ and use that $D_{h\lvert_{\mathbb{D}}}$
is a.s.\ a length metric to  obtain
that
$D_{h\lvert_{\mathbb{D}}}(\mathbb{T}_{2\varepsilon},\mathbb{T}_{3\varepsilon};\mathbb{C}_{(\varepsilon,4\varepsilon)})\leq
D_{h\lvert_{\mathbb{D}}}(\mathbb{T}_{2\varepsilon},\mathbb{T}_{3\varepsilon})$ and thus obtain
$D_{h\lvert_{\mathbb{D}}}(\mathbb{T}_{2\varepsilon},\mathbb{T}_{3\varepsilon})=D_{h\lvert_{\mathbb{D}}}(\mathbb{T}_{2\varepsilon},\mathbb{T}_{3\varepsilon};\mathbb{C}_{(\varepsilon,4\varepsilon)}),$ since
 induced distances can only increase. Again, by the locality of the LQG
 metric, we have that
 $D_{h\lvert_{\mathbb{D}}}(\mathbb{T}_{2\varepsilon},\mathbb{T}_{3\varepsilon};\mathbb{C}_{(\varepsilon,4\varepsilon)})=D_{h\lvert_{\mathbb{C}_{(\varepsilon,4\varepsilon)}}}(\mathbb{T}_{2\varepsilon},\mathbb{T}_{3\varepsilon})$
 and this completes the proof.
\end{proof}

\begin{lemma}
	\label{si3}
	$\mathbb{P}\left(D_{h\lvert_{\mathbb{D}}}\in A_{\varepsilon}
	\right)$ is independent of $\varepsilon$ and is positive.
\end{lemma}

\begin{proof}
	By using Lemma \ref{b4}, we obtain
	\begin{equation}
		\label{e:scale}
			D_{h(r\cdot)-\mathbf{Av}(h,\mathbb{T}_r)}(u,v)=r^{-\xi
			Q} e^{-\xi \mathbf{Av}(h,\mathbb{T}_r)}D_h(r
			u,r v)
	\end{equation}
	a.s.\ for all $u,v\in \mathbb{C}$ and any $r>0$. Recall from \eqref{e:b4.1} that
	$h(r\cdot)-\mathbf{Av}(h,\mathbb{T}_r)\stackrel{d}{=}h$.
	If we choose $r=\varepsilon$, then the above along with \eqref{e:scale} implies that for any $\varepsilon\leq 1/5$, the
	metric $D_h(\varepsilon\cdot,\varepsilon\cdot)$ on $\mathbb{C}$ is distributed as 
	a random multiple of $D_h$. Since the lengths of paths and distances
	between points scale by the same factor if the metric is multiplied by a constant, it follows that 	
	\begin{equation}
		\label{e:scale1}
		\mathbb{P}\left( D_h(\cdot,\cdot;\mathbb{D})\in A_\varepsilon
		\right)=\mathbb{P}\left( D_h(\cdot,\cdot;\mathbb{D})\in A_{1/5} \right)
	\end{equation}
	for all $\varepsilon \leq 1/5$. The locality of the LQG metric from
	Proposition \ref{b3} (2) implies that
	$D_h(\cdot,\cdot,\mathbb{D})=D_{h\lvert_{\mathbb{D}}}$ almost surely and
	thus $\mathbb{P}\left(D_{h\lvert_{\mathbb{D}}}\in A_{\varepsilon}
	\right)$ is independent of $\varepsilon$.
	
	The
	fact that 	$\mathbb{P}\left(D_{h\lvert_{\mathbb{D}}}\in A_\varepsilon
	\right)>0$
	 is a consequence of an absolute continuity argument and is proved in \cite[Lemma 6.1]{Gwy20}.
\end{proof}

\begin{lemma}
	\label{si4} $\mathbb{P}\left(
	D_{h\lvert_{\mathbb{D}}}\in A
	\right)=1$, i.e.,  almost surely, the events
	$\left\{D_{h\lvert_{\mathbb{D}}}\in A_{5^{-i}}\right\}$ occur infinitely often.
\end{lemma}
\begin{proof}
	By Lemma \ref{si2}, for any fixed $\varepsilon\leq 1/5$, the event $\left\{
	D_{h\lvert_{\mathbb{D}}}\in A_\varepsilon \right\}$ is measurable with
	respect to
	$\sigma(h\lvert_{\mathbb{C}_{(\varepsilon,4\varepsilon)}})\subseteq
	\sigma\left( h\lvert_{\mathbb{C}_{<4\varepsilon}} \right)$. We now let
	$\varepsilon$ take values along the sequence $\left\{ 5^{-i}
	\right\}_{i\in \mathbb{N}}$ and this implies that $\left\{
	D_{h\lvert_{\mathbb{D}}}\in A
	\right\}\in
	\cap_{r>0}\sigma(h\lvert_{\mathbb{D}_r})$. By the triviality of the
	above $\sigma$-algebra (see e.g.\ \cite[Lemma 2.2]{HS18}),
	$\mathbb{P}\left( 	D_{h\lvert_{\mathbb{D}}}\in A \right)$ is
	$0$ or $1$. By Lemma \ref{si3}, $\mathbb{P}\left(
	D_{h\lvert_{\mathbb{D}}}\in A \right)>0$ and hence $\mathbb{P}\left(
	D_{h\lvert_{\mathbb{D}}}\in A \right)=1$ which completes the
	proof.
\end{proof}
The rest of the proof is devoted to showing that $\mathbb{P}_{\infty}\left(
\emm\in A \right)=0$.
\begin{lemma}
	\label{si5}
	Fix an $\varepsilon\leq 1/5$. Almost surely in the
randomness of $h$, $\mathbbm{1}\left(\underline{D}_{h,\Gamma_t}\in
	A_\varepsilon\right)=0$
	for every $t\in
	(0,\infty)$.
\end{lemma}
\begin{proof}
	By looking at the definition of $A_\varepsilon$, it suffices to show that almost
	surely, simultaneously for all $z\in \Gamma\setminus \left\{ 0 \right\}$, if 
	$\zeta$ is a path in $\mathbb{C}_{(3\varepsilon,4\varepsilon)}$ which disconnects $\mathbb{T}_{3\varepsilon}$ and
	$\mathbb{T}_{4\varepsilon}$, then
	$\ell(\zeta;\underline{D}_{h,z})\geq
	\underline{D}_{h,z}(\mathbb{T}_{2\varepsilon},\mathbb{T}_{3\varepsilon})$. Recalling Figure \ref{f.proofsketch2} would be helpful. 
Let $s_1<s_2<s_3$ be such that $\Gamma_{s_1},\Gamma_{s_3}\in
	\mathbb{T}_{4\varepsilon\delta|z|}(z)$, $\Gamma_{s_2}=z$ and $\Gamma_s\in
	\mathbb{D}_{\delta|z|}(z)$ for all $s_1\leq s\leq s_3$; we can simply take
	$s_1=\sup_{[0,s_2]}\left\{
	s:\Gamma_s\in \mathbb{T}_{4\varepsilon\delta|z|}(z) \right\}$ and
	$s_3=\inf_{[s_2,\infty)}\left\{
	s:\Gamma_s\in \mathbb{T}_{4\varepsilon\delta|z|}(z) \right\}$.
	Now, let $\zeta^*\colon [a,b]\rightarrow \mathbb{C}_{(3\varepsilon
	\delta|z|,4\varepsilon \delta|z|)}(z)$ be some path which disconnects
	$\mathbb{T}_{3\varepsilon \delta|z|}(z)$ and $\mathbb{T}_{4\varepsilon
	\delta|z|}(z)$ . Thus by planarity, there exist $t_1, t_2$ satisfying
	$s_1<t_1<s_2<t_2<s_3$ along with $\Gamma_{t_1}= \zeta^*(a_1)$ and
	$\Gamma_{t_2}= \zeta^*(b_1)$ for some distinct $a_1,b_1\in [a,b]$. 
	Without loss of generality, let $a_1<b_1$. Then
	\begin{align}
		\ell(\zeta^*;D_h(\cdot,\cdot;\mathbb{D}_{\delta|z|}(z))&=\ell(\zeta^*;D_h) \geq
		\ell(\zeta^*\lvert_{[a_1,b_1]};D_h)
	\geq
		\ell(\Gamma\lvert_{[t_1,t_2]};D_h)\nonumber\\
		&\geq \ell(\Gamma\lvert_{[s_2,t_2]};D_h)\nonumber
		\geq D_h(\mathbb{T}_{2\varepsilon
		\delta|z|}(z),\mathbb{T}_{3\varepsilon
		\delta|z|}(z))\nonumber\\
		&={D_h(\mathbb{T}_{2\varepsilon
		\delta|z|}(z),\mathbb{T}_{3\varepsilon
		\delta|z|}(z);\mathbb{D}_{\delta|z|}(z))}
		\label{e:si2}
	\end{align}
	{The first equality is an application of Lemma \ref{intrin} while} the second inequality above uses the fact that
	$\Gamma\lvert_{[t_1,t_2]}$ is a geodesic from $\Gamma_{t_1}$ to
	$\Gamma_{t_2}$ and $\zeta\lvert_{[a_1,b_1]}$ is a path from
	$\Gamma_{t_1}$ to $\Gamma_{t_2}$. The fourth inequality uses that the
	path $\Gamma\lvert_{[s_2,t_2]}$ must cross the annulus
	$\mathbb{C}_{(2\varepsilon\delta|z|,3\varepsilon\delta|z|)}(z)$. To obtain the
	last equality, we note that
	\begin{align}
		\label{e:intrin}
			D_h(\mathbb{T}_{2\varepsilon
		\delta|z|}(z),\mathbb{T}_{3\varepsilon
		\delta|z|}(z))&=\inf\left\{ \ell(\zeta^*; D_h)\Big \vert x\in \mathbb{T}_{2\varepsilon\delta|z|(z)}, y
		\in \mathbb{T}_{3\varepsilon\delta|z|(z)}, \zeta^*:x\rightarrow
		y\right\} \nonumber\\
		&=\inf\left\{ \ell(\zeta^*; D_h)\Big \vert x\in \mathbb{T}_{2\varepsilon\delta|z|(z)}, y
		\in \mathbb{T}_{3\varepsilon\delta|z|(z)}, \zeta^*:x\rightarrow
		y, \zeta^* \subseteq \mathbb{D}_ {\delta |z|}(z)\right\}\nonumber\\
		&= D_h(\mathbb{T}_{2\varepsilon
		\delta|z|}(z),\mathbb{T}_{3\varepsilon
		\delta|z|}(z);\mathbb{D}_{\delta|z|}(z)).
	\end{align}
	To see the second equality above, note that if we have a path $\zeta^*:x\rightarrow y$
	for some $ x\in \mathbb{T}_{2\varepsilon\delta|z|(z)}, y
		\in \mathbb{T}_{3\varepsilon\delta|z|(z)}$ then by going to a
		sub-path of $\zeta^*$, we can obtain a path $\zeta_1^*:x_1\rightarrow y_1$
		such that $ x_1\in \mathbb{T}_{2\varepsilon\delta|z|(z)}, y_1
		\in \mathbb{T}_{3\varepsilon\delta|z|(z)}$ satisfying
		$\ell(\zeta_1^*,D_h)\leq \ell(\zeta^*,D_h)$. The third equality is
		obtained by using Lemma \ref{intrin} along with the definition of the
		induced metric.

	We now complete the proof of the lemma. Let $\zeta$ be a
	path in $\mathbb{C}_{(3\varepsilon,4\varepsilon)}$ which
	disconnects $\mathbb{T}_{3\varepsilon}$ and $\mathbb{T}_{4\varepsilon}$. Note
	that $\delta|z|\zeta+z$ is now a
	path in $\mathbb{C}_{(3\varepsilon
	\delta|z|,4\varepsilon \delta|z|)}(z)$ which disconnects
	$\mathbb{T}_{3\varepsilon \delta|z|}(z)$ and $\mathbb{T}_{4\varepsilon
	\delta|z|}(z)$. By the definition of $\underline{D}_{h,z}$ from
	\eqref{e:metric}, with $\mathfrak{c}=|\delta z|^{-\xi Q}
	 e^{-\xi
	 \mathbf{Av}({h},\mathbb{T}_{\delta|z|}(z))}$, we have
	\begin{displaymath}
		\ell(\zeta;\underline{D}_{h,z})=\mathfrak{c}
		\ell(\delta|z|\zeta+z;D_h(\cdot,\cdot;\mathbb{D}_{\delta|z|}(z)))\geq
		\mathfrak{c}D_h(\mathbb{T}_{2\varepsilon
		\delta|z|}(z),\mathbb{T}_{3\varepsilon
		\delta|z|}(z);\mathbb{D}_{\delta|z|}(z))=
	\underline{D}_{h,z}(\mathbb{T}_{2\varepsilon},\mathbb{T}_{3\varepsilon}).
	\end{displaymath}
	The second equality above uses \eqref{e:si2} with the path $\zeta^*$
	being $\delta|z|\zeta+z$. This completes the proof.
\end{proof}
\begin{lemma}
	\label{si6}
	 Almost surely in the
randomness of $h$,  $\mathbbm{1}\left(\underline{D}_{h,\Gamma_t}\in
	A\right)=0$
	for every $t\in
	(0,\infty)$.
\end{lemma}
\begin{proof}
	This is an immediate consequence of Lemma \ref{si5} and a union bound
	over $\varepsilon$ taking values in the sequence $\left\{ 5^{-i}
	\right\}_{i\in \mathbb{N}}$.
\end{proof}

We have the following immediate corollary.
\begin{lemma}
	\label{si8}
	We have that $\mathbb{P}_\infty\left( \mathbf{Metric}\in A \right)=0$. In
	other words, almost surely, $\left\{\mathbf{Metric}\in A_{5^{-i}}\right\}$ does not occur
	for any $i\in \mathbb{N}$.
\end{lemma}
\begin{proof}
The proof is an application of Lemmas \ref{si6} and \ref{null2}.
\end{proof}
Now we can complete the proof of Proposition \ref{singprop}.
\begin{proof}[Proof of Proposition \ref{singprop}]
An immediate consequence of Lemmas \ref{si4} and \ref{si8}.
\end{proof}
Following the discussion in Section \ref{s:iop}, we now complete the proof
of Theorem \ref{mainsi1} (1) by using Proposition \ref{singprop} and Proposition
\ref{mainmeas1}. The proof of Proposition
\ref{mainmeas1} will be provided in Section \ref{s:f-field}.

\begin{proof}[Proof of Theorem \ref{mainsi1} (1)]
	By Proposition \ref{mainmeas1}, we have a measurable function  $\Psi_\mathbb{D}: (C_c^{\infty}(\D))'\to C(\D\times \D)$ satisfying $\Psi_\mathbb{D}(h\lvert_{\mathbb{D}})=D_{h\lvert_{\mathbb{D}}}$ almost
	surely along with
	$\Psi_\mathbb{D}(\mathbf{Field})=\mathbf{Metric}$ almost
	surely. By Proposition \ref{singprop}, $D_{h\lvert_{\mathbb{D}}}$ is
	mutually singular with respect to $\mathbf{Metric}$	and thus there exists
	a measurable set $A\subseteq C(\mathbb{D}\times\mathbb{D})$ such that 
	 $\mathbb{P}\left( D_{h\lvert_{\mathbb{D}}}\in A \right)=1$ and
	$\mathbb{P}_\infty\left( \mathbf{Metric}\in A \right)=0$. Hence,
	$\mathbb{P}\left( h\lvert_{\mathbb{D}}\in \Psi_\mathbb{D}^{-1}(A) \right)=1$ and
	$\mathbb{P}_\infty\left( \mathbf{Field}\in \Psi_\mathbb{D}^{-1}(A) \right)=0$ which
 finishes the proof. 
\end{proof}

\section{Absolute Continuity: proof of Theorem \ref{mainsi1} (2).}
\label{s:ac}
We begin by recalling the
statement for convenience-- For any $\delta'\in (0,1)$, 
	$\emf\lvert_{\mathbb{C}_{(\delta',1)}}$ is absolutely continuous to 
	 $h\lvert_{\mathbb{C}_{(\delta',1)}}.$
We will in fact prove the following stronger statement which will be useful for us in Section \ref{s:f-field}. Recall the coupling of  $(\emf,\emm)$
from Lemma \ref{ac:1} with law $\mathbb{P}_\infty$.
\begin{proposition}
	\label{ac:2.1*}
	For any $\delta'\in (0,1)$, 
$$\left(\emf\lvert_{\mathbb{C}_{(\delta',1)}},\emm(\cdot,\cdot;\mathbb{C}_{(\delta',1)})\right)\text{
is absolutely continuous to }\left(h\lvert_{\mathbb{C}_{(\delta',1)}},D_h(\cdot,\cdot;\mathbb{C}_{(\delta',1)})\right).$$
\end{proposition}
Proposition \ref{ac:2.1*} will follow from the next lemma. 
\begin{lemma}
	\label{ac:3}
	Let $A\subseteq \mathcal{D}'(\mathbb{C}_{(\delta',1)})\times
	C(\mathbb{C}_{(\delta',1)}\times \mathbb{C}_{(\delta',1)})$ be a measurable set
	satisfying
	\begin{equation}
		\label{e:fm0}
		\mathbb{P}\left(
		\left(h\lvert_{\mathbb{C}_{(\delta',1)}},D_h(\cdot,\cdot;\mathbb{C}_{(\delta',1)})\right)\in
		A\right)=0.
	\end{equation}
	 Then almost surely
	in the randomness of $h$, 
	\begin{displaymath}	\mathbbm{1}\left(\left(\underline{h}_{\Gamma_t}\lvert_{\mathbb{C}_{(\delta',1)}},\underline{D}_{h,\Gamma_t}(\cdot,\cdot;\mathbb{C}_{(\delta',1)})\right)\in
	A\right)=0
	\end{displaymath}
	for almost every $t\in(0,\infty)$ with respect to the Lebesgue measure on
	$(0,\infty)$.
\end{lemma}

Postponing the proof of Lemma \ref{ac:3} for now, we first complete the proof of Proposition \ref{ac:2.1*}. 

\begin{proof}[Proof of Proposition \ref{ac:2.1*} assuming Lemma \ref{ac:3}]
	We need to show that for a measurable $A\subseteq
	\mathcal{D}'(\mathbb{C}_{(\delta',1)})\times
	C(\mathbb{C}_{(\delta',1)}\times \mathbb{C}_{(\delta',1)})$ 
	satisfying
	\eqref{e:fm0}, we have
	\begin{equation}
		\label{e:fm1}
		\mathbb{P}_\infty\left( \left(\emf \lvert_{\mathbb{C}_{(\delta',1)}},\emm(\cdot,\cdot;\mathbb{C}_{(\delta',1)})\right)\in A
	\right)=0.
	\end{equation}
	We seek to apply  Lemma \ref{null2*} and thus a set $A$ as above first
	needs to be pulled back to a measurable set in $\mathcal{D}'(\D)\times
	C(\D \times \D).$ Towards this, let $f$ denote the map from $
	 \mathcal{D}'(\mathbb{D})\times C(\mathbb{D}\times\mathbb{D})$ to
	 $\mathcal{D}'(\mathbb{C}_{(\delta',1)})\times
	 C(\mathbb{C}_{(\delta',1)}\times \mathbb{C}_{(\delta',1)})$ given by $(\mathtt{h},d)\mapsto
	 (\mathtt{h}\lvert_{\mathbb{C}_{(\delta',1)}},d(\cdot,\cdot;\mathbb{C}_{(\delta',1)}))$.
	 If we define the set $B$ by $B=\left\{ d\in C(\mathbb{D}\times
	 \mathbb{D}): d \text{ is a length metric} \right\}$, then by Lemma
	 \ref{metric_length}, 
	 $\mathbb{P}_\infty(\mathbf{Metric}\in B)=1$ and by the discussion
	 regarding induction of length metrics in Section
	 \ref{ss:meas}, the map
	 $f\lvert_{\mathcal{D}'(\mathbb{D})\times B}$ is measurable from
	 $\mathcal{D}'(\mathbb{D})\times B$ to
	 $\mathcal{D}'(\mathbb{C}_{(\delta',1)})\times
	 C(\mathbb{C}_{(\delta',1)}\times \mathbb{C}_{(\delta',1)})$.

We now consider the measurable set $\left(
\mathcal{D}'(\mathbb{D})\times
B\right)\cap f^{-1}(A^c)\subseteq
	\mathcal{D}'(\mathbb{D})\times C(\mathbb{D}\times\mathbb{D})$. 
	By the proof of Lemma \ref{metric_length}, almost surely, $\mathbbm{1}\left(
	(\underline{h}_{\Gamma_t},\underline{D}_{h,\Gamma_t})\in
	\mathcal{D}'(\mathbb{D})\times
B
	\right)=1$ for all $t\in (0,\infty)$.
	This along with Lemma \ref{ac:3} implies, almost surely in the randomness of $h$,
	$\mathbbm{1}\left(
	(\underline{h}_{\Gamma_t},\underline{D}_{h,\Gamma_t})\in
	f^{-1}(A^c)\right)=1$
	for Lebesgue almost every $t\in(0,\infty)$. Thus we have that almost surely, 
	$$\mathbbm{1}\left(
	(\underline{h}_{\Gamma_t},\underline{D}_{h,\Gamma_t})\in
	\left(	\mathcal{D}(\mathbb{D})'\times
B\right)\cap f^{-1}(A^c)\right)=1$$ for Lebesgue almost every
$t\in(0,\infty)$. Hence, Lemma \ref{null2*} implies
\begin{displaymath}
		\mathbb{P}_\infty\left( (\mathbf{Field},\mathbf{Metric})\in \left(\mathcal{D}(\mathbb{D})'\times
B\right)\cap f^{-1}(A^c)\right)=1
\end{displaymath}
which is the
	same as the statement $\mathbb{P}_\infty\left( \left(\emf
	\lvert_{\mathbb{C}_{(\delta',1)}},\emm(\cdot,\cdot;\mathbb{C}_{(\delta',1)})\right)\in A
	\right)=0$, thereby completing the proof.
\end{proof}

The next subsection is devoted to the proof of Lemma \ref{ac:3}.
\subsection{Proof of Lemma \ref{ac:3}}
The proof follows the strategy outlined in Section
\ref{s:iop}. It might be useful to also recall Figure \ref{f.proofsketch2}. In this subsection, we always work with a fixed choice of $\delta'\in
(0,1)$. We choose two constants $\varepsilon_1<\varepsilon_2$ satisfying
$0<\varepsilon_1<\varepsilon_2<1$. 
The constant $\varepsilon_2$ is chosen to additionally satisfy that for any $z\in \mathbb{C}\setminus
\left\{ 0 \right\}$, we have
\begin{equation}
	\label{e:calc}
	\mathbb{C}_{(\delta'\delta|u|,\delta|u|)}(u)\subseteq
\mathbb{C}_{>2\varepsilon_2\delta|z|}(z)
\end{equation}
for all
$u\in \mathbb{D}_{\varepsilon_2\delta|z|}(z)$.

By a small calculation, to ensure the above, it is sufficient to impose 
\begin{gather*}
	(1-\varepsilon_2\delta)+\delta'\delta(1-\varepsilon_2 \delta)-1\geq 2\varepsilon_2\delta.
\end{gather*}

The above constraint amounts to choosing $\varepsilon_2\leq  
\delta'(3+\delta\delta')^{-1}$.
We now fix such a choice of
$\varepsilon_2$ and a choice of $\varepsilon_1$ which will remain
unchanged throughout this subsection.

As explained in Section \ref{s:iop}, we will require a result saying that all segments $\Gamma\lvert_{[s,t]}$ of
$\Gamma$ can be simultaneously well approximated by geodesics between
rational points in $\mathbb{C}$. To this end we record the following result.
\begin{lemma}[{\cite[Lemma 3.10]{GPS20}}]
	\label{ac:3.2}
	Almost surely in the randomness of $h$, for any $0<s<t$ and any
	neighbourhoods $U,V$ of $\Gamma_s$ and $\Gamma_t$ respectively, we have
	that there exist points $u'\in \mathbb{Q}^2\cap U$ and $v'\in
	\mathbb{Q}^2\cap V$ such that the unique geodesic $\Gamma(u',v')$
	satisfies 
	\begin{displaymath}
		\Gamma(u',v')\setminus \Gamma\lvert_{[s,t]}\subseteq U\cup V.
	\end{displaymath}
\end{lemma}
In fact \cite{GPS20} proves a more general version of the statement asserting that the above is true simultaneously for all
geodesics starting at $0$ and not just the infinite geodesic $\Gamma$.

For any fixed $z\in \mathbb{Q}^2\cap \mathbb{C} \setminus\{{0}\}$, the following lemma describes how one can use the noise  
$h\lvert_{\mathbb{D}_{\varepsilon_2 \delta |z|}(z)}$ to create countably many candidate paths
$\zeta_i\subseteq {\mathbb{D}_{\varepsilon_2 \delta |z|}(z)}$ with the property that
any excursion of $\Gamma$ into $\mathbb{D}_{\varepsilon_1\delta|z|}(z)$ must be a
subset of one of the $\zeta_i$s.

\begin{lemma}
	\label{ac:4}
	Fix $z\in  \mathbb{Q}^2\cap \mathbb{C} \setminus\{{0}\}$. There exist countably many random
	$D_{h}(\cdot,\cdot;\mathbb{D}_{\varepsilon_2\delta|z|}(z))$
	geodesics $\zeta_1,\zeta_2,\dots\subseteq \mathbb{D}_{\varepsilon_2\delta
	|z|}(z)$
	measurable with respect to
	$h\lvert_{\mathbb{D}_{\varepsilon_2\delta|z|}(z)}$
	and having the property that almost surely in the randomness of $h$, if
	$\Gamma\lvert_{(s,t)}\subseteq \mathbb{D}_{\varepsilon_1\delta|z|}(z)$ for any
	$0<s<t$, then there exists an $i$ satisfying $\Gamma\lvert_{(s,t)}\subseteq
	\zeta_i$. Further, the geodesics $\zeta_i$ can be chosen to have rational
	endpoints.
\end{lemma}	
	
	\begin{proof}
		Let $(\mathfrak{p}_i,\mathfrak{q}_i)_{i\in \mathbb{N}}$
		be an enumeration of 
		$(\mathbb{D}_{\varepsilon_2\delta|z|}(z)\cap
		\mathbb{Q}^2)\times(\mathbb{D}_{\varepsilon_2\delta|z|}(z)\cap
		\mathbb{Q}^2)$. For each $i$, make a choice of one 
		$D_{h}(\cdot,\cdot;\mathbb{D}_{\varepsilon_2\delta|z|}(z))$
		geodesic (if it exists) between $\mathfrak{p}_i$ and
		$\mathfrak{q}_i$ (such a geodesic is a.s.\ unique if it exists; for
		comments related to this and measurability, refer to the discussion in Section \ref{ss:meas}).

		This yields a set of
$D_{h}(\cdot,\cdot;\mathbb{D}_{\varepsilon_2\delta|z|}(z))$
geodesics $\left\{ \zeta_1,\zeta_2,\dots\right\}$  measurable with
respect to 
$h\lvert_{\mathbb{D}_{\varepsilon_2\delta|z|}(z)}$ and
parametrized according to the LQG distance with respect to
$D_{h}(\cdot,\cdot;\mathbb{D}_{\varepsilon_2\delta|z|}(z))$. 

It remains to be shown
that $\zeta_1,\zeta_2,\dots$ satisfy the claimed property. Suppose that $\Gamma\lvert_{(s,t)}\subseteq \mathbb{D}_{\varepsilon_1\delta|z|}(z)$
for some $0<s<t$. Choose $s',t'$ such that $s'<s<t<t'$ and
$\Gamma_{s'},\Gamma_{t'}\in
\mathbb{C}_{(\varepsilon_1\delta|z|,\varepsilon_2\delta|z|)}(z)$ and
$\Gamma\lvert_{[s',t']}\subseteq \mathbb{D}_{\varepsilon_2\delta|z|}(z)$. Note that $s',t'$ exist because $\Gamma$ is a continuous
curve which is inside $\mathbb{D}_{\varepsilon_1\delta|z|}(z)$ for the length
interval $(s,t)$ but outside $\mathbb{D}_{\varepsilon_2\delta|z|}(z)$ at length
zero and for all large enough length values.

Let
$U,V\subseteq \mathbb{D}_{\varepsilon_2\delta|z|}(z)$ be neighbourhoods disjoint from
$\mathbb{D}_{\varepsilon_1\delta|z|}(z)$ satisfying $\Gamma_{s'}\in
U,\Gamma_{t'}\in V$. By an application of Lemma \ref{ac:3.2}, we obtain
rational points $u'\in U,v'\in V$ such that $\Gamma(u',v')\setminus
\Gamma\lvert_{[s',t']} \subseteq U\cup V$ which in particular implies that
$\Gamma\lvert_{(s,t)}\subseteq \Gamma(u',v')$. This also implies that
$\Gamma(u',v')\subseteq \mathbb{D}_{\varepsilon_2\delta|z|}(z)$ and by using
Lemma \ref{intrin}, we
infer that
that $\Gamma(u',v')$ is in fact also the unique
$D_{h}(\cdot,\cdot;\mathbb{D}_{\varepsilon_2\delta|z|}(z))$
geodesic from $u'$ to $v'$.
Thus we have that
$\Gamma(u',v')=\zeta_i$ for some $i$ and this completes the proof.
	\end{proof}

The paths $\zeta_i$ obtained from Lemma \ref{ac:4} will always be
parametrised to cover
$D_{h}(\cdot,\cdot;\mathbb{D}_{\varepsilon_2\delta|z|}(z))$ length at unit
rate. We use
the notation $\mathtt{in}(\zeta_i)$ and $\mathtt{fi}(\zeta_i)$ to
denote the initial and final points of $\zeta_i$ respectively. For a fixed instance of
$h$, consider random variables
$\mathfrak{t}_1,\mathfrak{t}_2,\dots$ satisfying
\begin{equation}
	\label{eac:6}
	\mathfrak{t}_i\sim\mathtt{Unif}(0,D_{h}(\mathtt{in}(\zeta_i),\mathtt{fi}(\zeta_i);\mathbb{D}_{\varepsilon_2\delta|z|}(z))).
\end{equation}

Note that the paths $\zeta_i$ have been
defined for a fixed $z\in \mathbb{Q}^2\cap \mathbb{C} \setminus\{{0}\}$, and thus can be defined simultaneously for all
$z\in \mathbb{Q}^2\cap \mathbb{C} \setminus\{{0}\}$. This will be done later in the
argument and we will use the notation $\zeta_{i,z}$ to make the dependence on $z$ explicit.

As described in Section \ref{ss:acidea},
{note that the paths $\zeta_{i,z}$ are determined by the noise
$h\lvert_{\mathbb{D}_{\varepsilon_2\delta|z|}(z)}$ while $\mathbb{C}_{(\delta'\delta|u|,\delta|u|)}(u)\subseteq 
\mathbb{C}_{>2\varepsilon_2\delta|z|}(z)$ for all
$u\in \mathbb{D}_{\varepsilon_2\delta|z|}(z)$ by \eqref{e:calc}. Since
$\mathbb{D}_{\varepsilon_2\delta|z|}(z)$ and
$\mathbb{C}_{>2\varepsilon_2\delta|z|}(z)$ are disjoint and are at
positive distance from each other, the Markov
property of the GFF implies that even if we condition on the noise in
{$\mathbb{D}_{\varepsilon_2\delta|z|}(z)$}, the field restricted to the region
$\mathbb{C}_{(\delta'\delta|u|,\delta|u|)}(u)$ for any fixed
$u\in\zeta_{i,z}\subseteq
\mathbb{D}_{\varepsilon_2\delta|z|}(z)$
is absolutely continuous with respect
to a GFF and thus the event
{$\left\{ \left(\underline{h}_u\lvert_{\mathbb{C}_{(\delta',1)}},\underline{D}_{h,u}(\cdot,\cdot;\mathbb{C}_{(\delta',1)})\right)\in
		A \right\}$} in the statement of Lemma \ref{ac:3}
should have $0$ probability. This is the content of the next lemma.

\begin{lemma}
	\label{ac:5}
	Fix  $z\in \mathbb{Q}^2\cap \mathbb{C} \setminus\{{0}\}$ 
	and a measurable set $A\subseteq \mathcal{D}'(\mathbb{C}_{(\delta',1)})\times
	C(\mathbb{C}_{(\delta',1)}\times \mathbb{C}_{(\delta',1)})$
	satisfying
	$\mathbb{P}\left(
		\left(h\lvert_{\mathbb{C}_{(\delta',1)}},D_h(\cdot,\cdot;\mathbb{C}_{(\delta',1)})\right)\in
		A\right)=0$. Then almost surely in the randomness of
	$h$, 
	\begin{displaymath}	\mathbbm{1}\left(
		\left(\underline{h}_{\zeta_i(\mathfrak{t}_i)}\lvert_{\mathbb{C}_{(\delta',1)}},\underline{D}_{h,\zeta_i(\mathfrak{t}_i)}(\cdot,\cdot;\mathbb{C}_{(\delta',1)})
		\right)\in
	A \right)=0	
	\end{displaymath} 
	with
	probability $1$ in the randomness of $\mathfrak{t}_i$ for all $i\in
	\mathbb{N}$.
	
In other words,
	\begin{displaymath}	\mathbbm{1}\left(
		\left(\underline{h}_{\zeta_i(t)}\lvert_{\mathbb{C}_{(\delta',1)}},\underline{D}_{h,\zeta_i(t)}(\cdot,\cdot;\mathbb{C}_{(\delta',1)})
		\right)\in
	A \right)=0	
	\end{displaymath} 
	for almost every
	$t\in
	(0,D_{h}(\mathtt{in}(\zeta_i),\mathtt{fi}(\zeta_i);\mathbb{D}_{\varepsilon_2\delta|z|}(z)))$
	with respect to the Lebesgue measure.
\end{lemma}

\begin{proof}
	First note that by the locality of the LQG metric from Proposition
	\ref{b3} (2), the defining condition $\mathbb{P}\left(
		\left(h\lvert_{\mathbb{C}_{(\delta',1)}},D_h(\cdot,\cdot;\mathbb{C}_{(\delta',1)})\right)\in
		A\right)=0$ is equivalent to the condition
		\begin{equation}
			\label{e:co1}
			\mathbb{P}\left(
			\left(h\lvert_{\mathbb{C}_{(\delta',1)}},D_{h\lvert_{\mathbb{C}_{(\delta',1)}}}\right)\in
		A\right)=0.
		\end{equation}

	We now come to the proof itself.  It suffices to work with a fixed $i\in
	\mathbb{N}$ and we denote the
	set $\mathbb{C}_{>2\varepsilon_2\delta|z|}(z)$ by $\mathfrak{C}$.  To prove the lemma, we will show that $$\mathbb{P}\left( \left(\underline{h}_{\zeta_i(\mathfrak{t}_i)}\lvert_{\mathbb{C}_{(\delta',1)}},\underline{D}_{h,\zeta_i(\mathfrak{t}_i)}(\cdot,\cdot;\mathbb{C}_{(\delta',1)})
		\right)\in
	A \right)=0.$$ 
	Since $\zeta_i\in
	\sigma(h\lvert_{\mathbb{D}_{\varepsilon_2\delta|z|}(z))})\subseteq
	\sigma(h\lvert_{\mathfrak{C}^c}),$ it follows that 
 $\mathfrak{t}_i$ is conditionally independent of
	$h\lvert_{ \mathfrak{C}}$ given
	$	\sigma(h\lvert_{\mathfrak{C}^c})$.	
	Thus
	$\zeta_i(\mathfrak{t}_i)$ is conditionally independent of
	$h\lvert_{ \mathfrak{C}}$ given
	$	\sigma(h\lvert_{\mathfrak{C}^c})$. 
	Now using 
	\begin{equation}
		\label{eac:7}
		\mathbb{P}\left( \left(\underline{h}_{\zeta_i(\mathfrak{t}_i)}\lvert_{\mathbb{C}_{(\delta',1)}},\underline{D}_{h,\zeta_i(\mathfrak{t}_i)}(\cdot,\cdot;\mathbb{C}_{(\delta',1)})
		\right)\in
	A \right)=\mathbb{E}\left[ \mathbb{P}\left( \left(\underline{h}_{\zeta_i(\mathfrak{t}_i)}\lvert_{\mathbb{C}_{(\delta',1)}},\underline{D}_{h,\zeta_i(\mathfrak{t}_i)}(\cdot,\cdot;\mathbb{C}_{(\delta',1)})
		\right)\in
	A \Big\vert
		\sigma(h\lvert_{\mathfrak{C}^c})\right) \right],
	\end{equation}
 to complete the proof, it suffices to show that almost surely
	\begin{equation}
		\label{eac:8}
		\mathbb{P}\left( \left(\underline{h}_{\zeta_i(\mathfrak{t}_i)}\lvert_{\mathbb{C}_{(\delta',1)}},\underline{D}_{h,\zeta_i(\mathfrak{t}_i)}(\cdot,\cdot;\mathbb{C}_{(\delta',1)})
		\right)\in
	A \Big\vert
		\sigma(h\lvert_{\mathfrak{C}^c})\right)=0.
	\end{equation}
	
Note
	that by the Markov property from Lemma \ref{b2} along with a
	reasoning similar to the proof of Lemma \ref{markov*}, we have the
	decomposition
	\begin{equation}
		\label{e:decomp}
		h\lvert_{
	\mathfrak{C}}-\mathbf{Av}(h,\mathbb{T}_{\varepsilon_2\delta|z|}(z))=\mathtt{h}+f,
	\end{equation}
	where $\mathtt{h}$ is the restriction of a zero boundary GFF on $\mathfrak{C}$ and $f$ is an independent
	random harmonic function which is measurable with respect to
	$\sigma(h\lvert_{\mathfrak{C}^c})$. We now
	observe that for any $x\in \mathbb{C}$, the quantities
	$\underline{h}_x$ and $\underline{D}_{h,x}$ depend only on $h$ viewed
	modulo an additive constant; indeed, this is a consequence of the
	removal of the circle averages in the definitions \eqref{e:field} and
	\eqref{e:metric}. In particular, this implies that the global random
	constant $\mathbf{Av}(h,\mathbb{T}_{\varepsilon_2\delta|z|}(z))$ in \eqref{e:decomp} does not play a role in
	the local quantities.

	Since $z$ is a fixed point, we can use the locality
	of the LQG metric as in Proposition \ref{b3} (2) to obtain that
	$D_h(\cdot,\cdot;\mathfrak{C})=D_{\mathtt{h}+f+\mathbf{Av}(h,\mathbb{T}_{\varepsilon_2\delta|z|}(z))}$ almost surely. Recall that $f$ is measurable with
	respect to $\sigma(h\lvert_{\mathfrak{C}^c})$ while $\mathtt{h}$ is independent of the same. In view
	of the above facts and the discussion from the preceding paragraph, it suffices to prove the following in order to show
	\eqref{eac:8}. Let
	$\mathtt{h}$ be a zero boundary GFF on $\mathfrak{C}$, let $\mathfrak{p}$ be an
	independent random point in $\mathbb{D}_{\varepsilon_2\delta|z|}(z)$ and let
	$f$ now be a fixed
	harmonic function on $\mathfrak{C}$; then denoting $\mathtt{h}+f$ by
	$\mathtt{h}^f$, 
	\begin{equation}
		\label{eac:9}
	 \mathbb{P}\left(
	 \left((\underline{\mathtt{h}^f}_{\mathfrak{p}})\lvert_{\mathbb{C}_{(\delta',1)}},\underline{D}_{\mathtt{h}^f,\mathfrak{p}}(\cdot,\cdot;\mathbb{C}_{(\delta',1)})\right)\in
		A\right)=0.
	\end{equation}
	Recalling \eqref{e:fieldg}, note the slight abuse of notation in $\underline{\mathtt{h}^f}_{\mathfrak{p}}$ since the latter  is strictly not
	defined as $\mathtt{h}$ and $f$ are defined only on $\mathfrak{C}$,
	but nevertheless
	$(\underline{\mathtt{h}^f}_{\mathfrak{p}})\lvert_{\mathbb{C}_{(\delta',1)}}$
	is well defined because of \eqref{e:calc}. A similar comment holds for
	$\underline{D}_{\mathtt{h}^f,\mathfrak{p}}(\cdot,\cdot;\mathbb{C}_{(\delta',1)})$.
	
	We now show \eqref{eac:9}. By conditioning on $\mathfrak{p}$ and using the independence of
	$\mathfrak{p}$ and $\mathtt{h}$, we further reduce to showing
	\eqref{eac:9} {for a fixed point $\mathfrak{p}\in
	\mathbb{D}_{\varepsilon_2\delta|z|}(z)$}. 	Since the point $\mathfrak{p}$ is fixed, {we can use locality and}
	Weyl scaling from Proposition \ref{b3} to obtain that almost surely, 
\begin{equation}
	\label{e:ew}
	\underline{D}_{\mathtt{h}^f,\mathfrak{p}}(\cdot,\cdot;\mathbb{C}_{(\delta',1)})=
	D_{(\underline{\mathtt{h}^f}_{\mathfrak{p}})\lvert_{\mathbb{C}_{(\delta',1)}}}
\end{equation}
and hence
\begin{equation}
	\label{e:ew1}
	\left((\underline{\mathtt{h}^f}_{\mathfrak{p}})\lvert_{\mathbb{C}_{(\delta',1)}},\underline{D}_{\mathtt{h}^f,\mathfrak{p}}(\cdot,\cdot;\mathbb{C}_{(\delta',1)})\right)=\left((\underline{\mathtt{h}^f}_{\mathfrak{p}})\lvert_{\mathbb{C}_{(\delta',1)}},	D_{(\underline{\mathtt{h}^f}_{\mathfrak{p}})\lvert_{\mathbb{C}_{(\delta',1)}}}\right).
\end{equation}
In view of the above and \eqref{e:co1}, to complete the proof of
\eqref{eac:9}, it suffices to show that the fields
$h\lvert_{\mathbb{C}_{(\delta',1)}}$ and $\underline{\mathtt{h}^f}_p$, both
defined on $\mathbb{C}_{(\delta',1)}$, are mutually absolutely continuous
with respect to each other. This follows by first noting that the above two
fields are mutually absolutely
	continuous when viewed modulo additive constants (see
	\cite[Proposition 2.11]{MS17}) and then fixing the normalization by
	imposing both the fields to have average $0$ on the unit circle $\mathbb{T}$.
	This completes the proof.
\end{proof}

As mentioned earlier, we now go from considering a fixed choice of $z\in
\mathbb{Q}^2\cap \mathbb{C} \setminus\{{0}\}$ to considering all $z\in \mathbb{Q}^2\cap \mathbb{C} \setminus\{{0}\}$
simultaneously. The conclusions of Lemmas \ref{ac:4} and \ref{ac:5}
remain true almost surely simultaneously for all rational $z$, and we use
$\mathcal{E}$ to denote the probability $1$ event on which this is true. We
now use Lemmas \ref{ac:4} and \ref{ac:5} to complete the proof of
Lemma \ref{ac:3}.
\begin{proof}[Proof of Lemma \ref{ac:3}]
	Recall that $\Gamma$ is the almost surely unique infinite
	geodesic starting from ${0}$ and is parametrised
	according to the LQG length with respect to $h$. It
	suffices to show that {on the event $\mathcal{E}$} defined above, we have
	$\mathbbm{1}\left(\left(\underline{h}_{\Gamma_t}\lvert_{\mathbb{C}_{(\delta',1)}},\underline{D}_{h,\Gamma_t}(\cdot,\cdot;\mathbb{C}_{(\delta',1)})\right)\in
	A\right)=0$
	for almost every $t\in
	(0,\infty)$ with respect to
	the Lebesgue measure.

	Fix a realization $\omega \in \mathcal{E}$. Regard $\Gamma$ as a
	continuous map
	from $(0,\infty)$ to $\mathbb{C}$. For a fixed $z\in
	\mathbb{Q}^2\cap \mathbb{C} \setminus\{{0}\}$, consider the inverse image of the
	set $\mathbb{D}_{\varepsilon_1 \delta |z|}(z)$ under this map. This is a
	bounded open subset of $(0,\infty)$ and hence must be at most a countable union of
	{disjoint} bounded open intervals. Let $\left\{
	(t_i,t_i')\right\}_{i=1}^{N_1}$ denote the intervals such
	that $\Gamma_t\in \mathbb{D}_{\varepsilon_1 \delta |z|}(z)$ for some $t\neq 0$ if and
	only if $t$ lies in one of these intervals. are
	Note here
	that $N_1=N_1(z)$ is random
 and takes values in $\mathbb{N}\cup \left\{ 0,\infty
	\right\}$ (the set of the above intervals can be empty 	which occurs if the geodesic $\Gamma$ does not intersect
	$\mathbb{D}_{\varepsilon_1 \delta |z|}(z)$ on the sample point $\omega$).

By Lemma \ref{ac:4}, we have that $\Gamma\lvert_{(t_i,t_i')}\subseteq
\zeta_{k_i,z}$ for some $k_i\in \mathbb{N}$ for all $i\in \left\{ 1,\dots,N_1
\right\}$. Now note that since $\Gamma$ is a $D_h$ geodesic,
$\Gamma\lvert_{(t_i,t_i')}$ is itself also a $D_h$ geodesic and thus in
particular a $D_h(\cdot,\cdot; \mathbb{D}_{\varepsilon_2\delta|z|}(z))$
geodesic by Lemma \ref{intrin}. {In particular, both $D_h$ and $D_h(\cdot,\cdot;
\mathbb{D}_{\varepsilon_2\delta|z|}(z))$ yield the same parametrization} for
$\Gamma\lvert_{(t_i,t_i')}$. Lemma \ref{ac:5} now implies that almost surely in the randomness of
$h$,
$\mathbbm{1}\left(\left(\underline{h}_{\Gamma_t}\lvert_{\mathbb{C}_{(\delta',1)}},\underline{D}_{h,\Gamma_t}(\cdot,\cdot;\mathbb{C}_{(\delta',1)})\right)\in
	A\right)=0$ for
 almost every $t\in (t_i,t_i')$ with respect to the Lebesgue measure on
 $(t_i,t_i')$. This
 further implies that almost surely in the randomness of $h$,
 the above expression is zero for
 almost every $t\in \bigcup_{i=1}^{N_1}(t_i,t_i')$ with respect to the
 Lebesgue measure. Since the countable union
 $\bigcup_{z\in \mathbb{Q}^2\cap
 \mathbb{C}\setminus\{0\}}\mathbb{D}_{\varepsilon_1 \delta |z|}(z)=\mathbb{C}\setminus\left\{ 0
 \right\}\supseteq \Gamma\lvert_{(0,\infty)}$, the countable union of $\bigcup_{i=1}^{N_1}(t_i,t_i')$ over all $z\in \mathbb{Q}^2\cap
 \mathbb{C}$ equals $(0,\infty)$. Thus almost surely in
 the randomness of $h$,
$\mathbbm{1}\left(\left(\underline{h}_{\Gamma_t}\lvert_{\mathbb{C}_{(\delta',1)}},\underline{D}_{h,\Gamma_t}(\cdot,\cdot;\mathbb{C}_{(\delta',1)})\right)\in
	A\right)=0$
 for Lebesgue almost every $t\in(0,\infty)$.
 This completes the proof.
\end{proof}
\section{$\mathbf{Metric}$ is a function of $\mathbf{Field}$}
\label{s:f-field}
In this section, we provide the proof of Proposition \ref{mainmeas1}.
Recall from Section \ref{ss:LQG} that the map $\mathtt{h}\rightarrow
D_\mathtt{h}$ is defined only almost surely if $\mathtt{h}$ is a GFF plus a
continuous function on some
domain $U$. In order to make the distinction between functions defined only
almost surely and genuine functions, we use the notation $\Psi_{U}$ to
denote a measurable function $\Psi_U:\mathcal{D}'(U)\rightarrow C(U\times U)$ which
encodes the field-metric correspondence i.e.,
\begin{equation}
	\label{e:ff1}
	D_\mathtt{h}=\Psi_U(\mathtt{h})
\end{equation}
almost surely for any GFF plus a continuous function $\mathtt{h}$ on a
domain $U$. Thus an arbitrary choice is involved in choosing a
specific $\Psi_U(\cdot)$ and we interpret the latter as a
version of the field-metric correspondence.

We will now fix some arbitrary choices of $\Psi_U$ for some specific domains
$U$ that we will encounter in our arguments. First, we fix  
$\Psi_{\mathbb{C}_{(2^{-i},1)}}$ for all $i\in \N$  and then use these to 
define a specific version $\Psi_{\mathbb{C}_{(0,1)}}$ which acts on a given
$\mathtt{h}\in \mathcal{D}'(\mathbb{C}_{(0,1)})$ by
\begin{equation}
	\label{e:ff2}
	\Psi_{\mathbb{C}_{(0,1)}}(\mathtt{h})=\lim_{i\rightarrow
	\infty}\Psi_{\mathbb{C}_{(2^{-i},1)}}(\mathtt{h}\lvert_{\mathbb{C}_{(2^{-i},1)}}).
\end{equation}
Note that, a priori, this defines $\Psi_{\mathbb{C}_{(0,1)}}$ on the set
where the {above limit exists point-wise}. However to ensure that 
$\Psi_{\mathbb{C}_{(0,1)}}$
is defined as a measurable function on
$\mathcal{D}'(\mathbb{C}_{(0,1)})$, we instead define the latter using \eqref{e:ff2}
 only on the set which we shall define next and shall finally
denote as $A$. 

Towards this, for $i\ge 1,$
let $\Omega_i:= \mathcal{D}'(\mathbb{C}_{(2^{-i},1)}),$ and let $\Psi_i:=\Psi_{\mathbb{C}_{(2^{-i},1)}}$. For any two rational points $x,y\in \mathbb{C}_{(0,1)},$ let $i_0=i_0(x,y)$ be
the smallest integer such that $x,y \in \C_{(2^{-i_0},1)}.$ Further, for any
$j>i,$ we denote by $\pi_{j,i}$ the natural projection map
$\Omega_{j}\to \Omega_i$. Analogously, we define $\pi_{\infty,i}$ to be the
projection map from $\mathcal{D}'(\mathbb{C}_{(0,1)})$ to $\Omega_i$.
Now given $x,y$ and  $j\ge i$,  define 
$$B_{j,i}(x,y)=\{\mathtt{h}\in \Omega_j: \Psi_{j}(\mathtt{h})(x,y)\le
\Psi_{i}(\pi_{j,i}(\mathtt{h}))(x,y)\}$$ 
if  $i\ge i_0(x,y)$
and $B_{j,i}(x,y)=\Omega_j$ otherwise.
Since the functions $\Psi_j, \Psi_{i}, \pi_{j,i}$ are all measurable, it follows that given $x,y,$ for any $j\ge i,$
$B_{j,i}(x,y)$ is a measurable set. 

Now we define $\widetilde{B}_j=\bigcap_{x,y \in \Q^2\cap
\mathbb{C}_{(0,1)}}\bigcap_{i\le
j}B_{j,i}(x,y),$ and 
finally $A_j$ inductively as $A_1=\widetilde{B}_1$ and for $j> 1$ as
$A_j=\pi_{j,j-1}^{-1}(A_{j-1})\cap \widetilde{B}_j.$
It is easy to check that $A_j$s have the following properties: 
\begin{itemize}
\item  $\P\left(\left\{\left(h \vert_{\C_{(2^{-j},1)}}+f\right) \in A_j \text{ for all continuous }
	f:\mathbb{C}_{(2^{-j},1)}\rightarrow \mathbb{R}\right\}\right)=1$ for all $j$, 
\item
	For any $j>i,$ $\pi_{j,i}(A_j)\subseteq A_i$.
\item For \emph{all} $x,y,$ for all $j>i\geq i_0(x,y),$ and $\mathtt{h}\in
	\Omega_j$, $\Psi_{j}(\mathtt{h})(x,y)\le \Psi_{i}(\pi_{j,i}(\mathtt{h}))(x,y)$. (The
result for rational $x,y$ follows by definition, which extends to all $x,y$
since $\Psi_{j}(\mathtt{h})$ and $\Psi_{i}(\pi_{j,i}(\mathtt{h}))$ are continuous functions.)
\end{itemize}
Note that the first property above is true because the Weyl scaling
from \eqref{e:b11} is true a.s.\ for all
continuous functions simultaneously. 

With the above definitions in hand, define
$A\subset \mathcal{D}'(\mathbb{C}_{(0,1)})$ by 
\begin{equation}
	\label{e:lastA}
	A=\bigcap_{i\ge 1}
\pi_{\infty,i}^{-1}(A_i)\cap \left\{ \mathtt{h}\in
\mathcal{D}'(\mathbb{C}_{(0,1)}):
\Psi_{i}(\pi_{\infty,i}(\mathtt{h})) \text{ is a continuous metric for all
} i\geq 1\right\}.
\end{equation}
To see the measurability of $A$, note that the maps $\pi_{\infty,i},\Psi_i$ and the
sets $A_i$ are measurable; also note that the set of continuous metrics
$\mathbb{C}_{(0,1)}$ is a
measurable subset of $C(\mathbb{C}_{(0,1)}\times \mathbb{C}_{(0,1)})$.
 Now by
 Theorem \ref{ac:2} (2), $\mathbb{P}_\infty\left( \mathbf{Field}\in A \right)=1$ and further by
monotonicity (the
third item in the above list), for all $\mathtt{h} \in A$ and all $x,y \in
\C_{(0,1)},$ the limit in \eqref{e:ff2} exists. To see that
$\Psi_{\C_{(0,1)}}(\mathtt{h})$ as defined through
\eqref{e:ff2} is continuous for every element of $A$, first note that for $\mathtt{h}\in A$,
$\Psi_{\mathbb{C}_{(0,1)}}(\mathtt{h})$ itself
satisfies the triangle inequality and is non-negative because of \eqref{e:ff2} along with the
fact that on $A$, $\Psi_{i}(\pi_{\infty,i}(\mathtt{h}))$ are metrics for all
$i\geq 1$. Thus for any fixed $x,y\in \mathbb{C}_{(0,1)}$, if there is a
sequence $(x_n,y_n)\rightarrow (x,y)$, then for all large enough $n$,
\begin{align*}
	|\Psi_{\mathbb{C}_{(0,1)}}(\mathtt{h})(x,y)-\Psi_{\mathbb{C}_{(0,1)}}(\mathtt{h})(x_n,y_n)|&\leq
\Psi_{\mathbb{C}_{(0,1)}}(\mathtt{h})(x,x_n)+\Psi_{\mathbb{C}_{(0,1)}}(\mathtt{h})(y,y_n)\\
&\leq
\Psi_{i_0}(\pi_{\infty,i_0}(\mathtt{h}))(x,x_n)+\Psi_{i_0}(\pi_{\infty,i_0}(\mathtt{h}))(y,y_n),
\end{align*}
where $i_0=i_0(x,y)$. The right hand side above now converges to $0$ as
$n\rightarrow \infty$ by using the continuity of
$\Psi_{i_0}(\pi_{\infty,i_0}(\mathtt{h}))$ and this shows that
$\Psi_{\C_{(0,1)}}(\mathtt{h})$ is continuous for $\mathtt{h}\in A$.
By using the measurability of $A$ and the above continuity, we can compute
$\Psi_{\C_{(0,1)}}(\mathtt{h})$ for $\mathtt{h}\in A$ by first computing
$\Psi_{\C_{(0,1)}}(\mathtt{h})(x,y)$ via \eqref{e:ff2} for all rational
$x,y$ and then taking the continuous extension. This shows that $\Psi_{\C_{(0,1)}}$
restricted to $A$ is measurable.

Having defined
$\Psi_{\mathbb{C}_{(0,1)}}(\mathtt{h})$, we simply define
$\Psi_\mathbb{D}$ which acts on a given $\mathtt{h}\in
\mathcal{D}'(\mathbb{D})$ by
\begin{equation}
	\label{e:ff3}
	\Psi_\mathbb{D}(\mathtt{h})=\text{ continuous
	extension of }
	\Psi_{\mathbb{C}_{(0,1)}}(\mathtt{h}\lvert_{\mathbb{C}_{(0,1)}})
	\text{ to } \mathbb{D}\times \mathbb{D}.
\end{equation}
The above defines
$\Psi_\mathbb{D}(\mathtt{h})$ on the set where the continuous extension
exists. Note that a similar measurability issue as after \eqref{e:ff2} arises
at this point. This can be addressed by replacing the set $A$ above by a set
$\widetilde{A}$ to
ensure that both \eqref{e:ff2} and \eqref{e:ff3} hold on $A$. To ensure the
latter, we simply define $\widetilde{A}$ by
\begin{align}
	\widetilde{A}=A\cap \Big\{&\mathtt{h}\in C_c(\mathbb{D})': \text{ if } x_n \in
\mathbb{Q}^2\cap \mathbb{C}_{(0,1)} \text{ is such that }
x_n\rightarrow 0, \text{ then for all } y\in \mathbb{Q}^2\cap
\mathbb{C}_{(0,1)},\nonumber\\
&\qquad \qquad \qquad \Psi_{\mathbb{C}_{(0,1)}}(\mathtt{h}\lvert_{\mathbb{C}_{(0,1)}})(x_n,y)
\text{ converges to a finite limit}  \Big\}.
\label{Ainter}
\end{align}

By using that $A$ is measurable and that the other set in the intersection is
defined in terms of countable operations, we obtain that $\widetilde{A}$ is
measurable. Thus $\Psi_\mathbb{D}(\mathtt{h})$ is defined by
\eqref{e:ff3} for $\mathtt{h}\in \widetilde{A}$ and is then extended to
$\mathcal{D}'(\mathbb{D})$ measurably.

The following simple lemma shows that the above definitions of
$\Psi_{\mathbb{C}_{(0,1)}}$ and $\Psi_\mathbb{D}$ are legitimate versions
of the field-metric correspondence.

\begin{lemma}
	\label{ff1}
	Let $\mathtt{h}_1$ be a GFF plus a continuous function on $\mathbb{C}_{(0,1)}$. Then,
	$D_\mathtt{h_1}=\Psi_{\mathbb{C}_{(0,1)}}(\mathtt{h}_1)$ almost surely. Similarly, if
	$\mathtt{h}_2$ is a GFF plus a continuous function 
	on $\mathbb{D}$, then $D_\mathtt{h_2}=\Psi_{\mathbb{D}}(\mathtt{h}_2)$ almost surely.
\end{lemma}
\begin{proof}
	First note that $\mathtt{h}_1\lvert_{\mathbb{C}_{(2^{-i},1)}}$ is a GFF
	plus
	a continuous function	
	on $\mathbb{C}_{(2^{-i},1)}$ for any $i\in
	\mathbb{N}$. Since $\Psi_{\mathbb{C}_{(2^{-i},1)}}$, by definition, is a 
	version of the field-metric correspondence it follows that
	$D_{\mathtt{h}_1\lvert_{\mathbb{C}_{(2^{-i},1)}}}=\Psi_{\mathbb{C}_{(2^{-i},1)}}(\mathtt{h}_1\lvert_{\mathbb{C}_{(2^{-i},1)}})$
	almost surely. Further, by the locality of the LQG metric from Proposition 
	\ref{b3} (2), almost surely
	\begin{equation}
		\label{e:ff3.1}
		D_{\mathtt{h}_1\lvert_{\mathbb{C}_{(2^{-i},1)}}}=D_{\mathtt{h}_1}(\cdot,\cdot;\mathbb{C}_{(2^{-i},1)}).
	\end{equation} We now observe that		$D_{\mathtt{h}_1}=\lim_{i\rightarrow
		\infty}	D_{\mathtt{h}_1}(\cdot,\cdot;\mathbb{C}_{(2^{-i},1)})$
	almost surely. To see this, note that $D_{\mathtt{h}_1}$ is
	almost surely a length metric and any path in $\mathbb{C}_{(0,1)}$ lies
	in $\mathbb{C}_{(2^{-i},1)}$ for some $i$. Using this and the definition \eqref{e:ff2}, we have 	$$D_{\mathtt{h}_1}=\lim_{i\rightarrow \infty}
	\Psi_{\mathbb{C}_{(2^{-i},1)}}(\mathtt{h}_1\lvert_{\mathbb{C}_{(2^{-i},1)}})=\Psi_{\mathbb{C}_{(0,1)}}(\mathtt{h}_1)$$
	almost surely. To obtain the last a.s.\ equality, we have used that
	$\mathbb{P}\left( \mathtt{h}_1\in A \right)=1$ along with the fact that in \eqref{e:ff2}, we had defined
	$\Psi_{\mathbb{C}_{(0,1)}}(\mathtt{h}) =\lim_{i\rightarrow
	\infty}\Psi_{\mathbb{C}_{(2^{-i},1)}}(\mathtt{h}\lvert_{\mathbb{C}_{(2^{-i},1)}})$
	for all $\mathtt{h}\in A$.
	
	In view of the above, to show that
$D_\mathtt{h_2}=\Psi_{\mathbb{D}}(\mathtt{h}_2)$ almost surely, it suffices
to show 
that almost surely, $D_{\mathtt{h}_2}(x,y)=D_{\mathtt{h}_2}(x,y;\mathbb{C}_{(0,1)})$ for all
$x,y\in \mathbb{C}_{(0,1)}$ simultaneously. To see why this is
sufficient, first note that by locality, we would obtain
$D_{\mathtt{h}_2}(x,y)=D_{\mathtt{h}_2\lvert_{\mathbb{C}_{(0,1)}}}(x,y)=\Psi_{\mathbb{C}_{(0,1)}}(\mathtt{h}_2\lvert_{\mathbb{C}_{(0,1)}})(x,y)$
a.s.\ for all $x,y\in \mathbb{C}_{(0,1)}$ and since $D_{\mathtt{h}_2}$ is a.s.\ a
continuous function on $\mathbb{D}$, the above implies that
$\mathtt{h}_2$ is a.s.\ in the set from \eqref{Ainter} and thus
$\mathbb{P}\left( \mathtt{h}_2\in \widetilde{A} \right)=1$. Now, the continuous extension of
$\Psi_{\mathbb{C}_{(0,1)}}(\mathtt{h}_2\lvert_{\mathbb{C}_{(0,1)}})(x,y)$ to
$\mathbb{D}$ must be a.s.\ equal to $D_{\mathtt{h}_2}$ since the latter is
known to be a.s.\ continuous. This would complete the proof because in
\eqref{e:ff3}, we had
defined $\Psi_\mathbb{D}(\mathtt{h})$ for all $\mathtt{h}\in \widetilde{A}$ as the continuous extension of 
$\Psi_{\mathbb{C}_{(0,1)}}(\mathtt{h}\lvert_{\mathbb{C}_{(0,1)}})$ to
$\mathbb{D}\times \mathbb{D}$.

We will now show that almost surely, $D_{\mathtt{h}_2}(x,y)=D_{\mathtt{h}_2}(x,y;\mathbb{C}_{(0,1)})$ for all
$x,y\in \mathbb{C}_{(0,1)}$ simultaneously. The basic reasoning for this is
the same as in as in Lemma \ref{si5} -- the
presence of shortcuts in annuli around $0$ makes it inefficient 
for a path
$\zeta_1:x\rightarrow y$ to touch $0$ and this holds for all $x,y\in
\mathbb{C}_{(0,1)}$. The slight difference in this setting is that we are
working with a GFF plus a continuous function $\mathtt{h}_2$ instead of
$h\lvert_{\mathbb{D}}$. We define the set $A^*_{\varepsilon,n}$ as
follows.
\begin{align}
	\label{A**}
	&A^*_{\varepsilon,n}= \Big\{d\in
	C(\mathbb{D}\times\mathbb{D}): d \text{ is a length metric},\text{ there exists some
	path }\zeta \text{ in } \mathbb{C}_{(3\varepsilon,4\varepsilon)}
	\text{ which disconnects }\mathbb{T}_{3\varepsilon}\nonumber\\
	&\qquad \text{and
	$\mathbb{T}_{4\varepsilon}$ and satisfies }
	\ell(\zeta;d)<n^{-1}d(\mathbb{T}_{2\varepsilon},\mathbb{T}_{3\varepsilon}) 
	\Big\}.
\end{align}

For the special case $\mathtt{h}_2=h\lvert_{\mathbb{D}}$, the argument
from Lemma \ref{si4}
assures the a.s.\ occurrence of the event $\left\{
D_{h\lvert_{\mathbb{D}}}\in
\cap_{n\in \mathbb{N}}\limsup_{i\rightarrow
\infty}A^*_{5^{-i},n} \right\}$. 
We now come to the case when $\mathtt{h}_2$ is any GFF plus a continuous
function. Though a GFF plus a continuous function on
 $\mathbb{D}$ was defined to be a zero boundary GFF on $\mathbb{D}$ plus a
 random continuous function, we note that it also be
written as $h\vert_{\mathbb{D}}+f$ for a random continuous function $f$ on
$\mathbb{D}$
coupled with $h$; to see this, note that by the Markov property for $h$
(see \cite[Lemma A.1]{Gwy20}), $h\lvert_{\mathbb{D}}$ can be
decomposed as a sum of a zero boundary GFF and a random harmonic
function. Since $f$ is continuous, it is bounded on compact sets and we
locally use
$\|\widetilde{f}\|_{\infty}$ to denote the finite quantity $\sup_{z\in
\overline{\mathbb{D}}_{4/5}}|f(z)|$.
By the above along with locality and Weyl scaling from
Proposition \ref{b3},
$D_{h\vert_{\mathbb{D}}}(\cdot,\cdot;\mathbb{D}_{4/5})
\stackrel{\text{a.s.}}{=}D_{h\vert_{\mathbb{D}_{4/5}}}$ and $D_{h\vert_{\mathbb{D}}+f}(\cdot,\cdot;\mathbb{D}_{4/5})
\stackrel{\text{a.s.}}{=}D_{ (h+f)\vert_{\mathbb{D}_{4/5}}}$ are bi-Lipschitz equivalent with
$e^{\xi\|\widetilde{f}\|_\infty}$ and its reciprocal as the Lipschitz constants. 

Now note that the condition
$\ell(\zeta;d)<n^{-1}d(\mathbb{T}_{2\varepsilon},\mathbb{T}_{3\varepsilon})$ in
\eqref{A**} is identical to the condition
$\ell(\zeta;d(\cdot,\cdot;\mathbb{D}_{4/5}))<n^{-1}d(\cdot,\cdot;\mathbb{D}_{4/5})(\mathbb{T}_{2\varepsilon},\mathbb{T}_{3\varepsilon})$
as long as $4\varepsilon \leq 4/5$ (for further detail, look at the proof of
Lemma \ref{si2}). Using the above along with the bi-Lipschitz equivalence
from the previous paragraph, the a.s.\ occurrence of the event $\left\{D_{h\lvert_{\mathbb{D}}}\in \cap_{n\in \mathbb{N}}\limsup_{i\rightarrow
\infty}A^*_{5^{-i},n}\right\}$ implies the a.s.\ occurrence of
the event $\left\{D_{h\lvert_{\mathbb{D}}+f}\in \limsup_{i\rightarrow
\infty}A^*_{5^{-i},1}\right\}$.
Thus
$\mathbb{P}\left( D_{\mathtt{h}_2}\in  \limsup_{i\rightarrow
\infty}A^*_{5^{-i},1} \right)=1$. By using that $D_{\mathtt{h}_2}$ is a.s.\ a
length metric along with the `shortcut' argument from Lemma \ref{si5}, it
can be obtained that on the event $\left\{ D_{\mathtt{h}_2}\in  \limsup_{i\rightarrow
\infty}A^*_{5^{-i},1} \right\}$, we have $D_{\mathtt{h}_2}(x,y)=D_{\mathtt{h}_2}(x,y;\mathbb{C}_{(0,1)})$ for all
$x,y\in \mathbb{C}_{(0,1)}$ simultaneously. This completes the proof.

\end{proof}

Having described the versions of the field-metric correspondences that we will
use, we move towards the proof of Proposition \ref{mainmeas1} by first showing that
induced metric $\mathbf{Metric}(\cdot,\cdot;\mathbb{C}_{(2^{-i},1)})$ can be
a.s.\ obtained measurably from the restricted field
$\mathbf{Field}\lvert_{\mathbb{C}_{(2^{-i},1)}}$ for any $2^{-i}\in (0,1)$.
\begin{lemma}
	\label{f-field:1}
	For each $i\in \mathbb{N}$, almost surely,
	$\mathbf{Metric}(\cdot,\cdot;\mathbb{C}_{(2^{-i},1)})=\Psi_{\mathbb{C}_{(2^{-i},1)}}(\mathbf{Field}\lvert_{\mathbb{C}_{(2^{-i},1)}}).$

\end{lemma}
\begin{proof}
	We will use Proposition \ref{ac:2.1*} with $\delta'=2^{-i}$. Define the measurable set
	$A_{2^{-i}}\subseteq \mathcal{D}'(\mathbb{C}_{(2^{-i},1)})\times
	C(\mathbb{C}_{(2^{-i},1)}\times \mathbb{C}_{(2^{-i},1)})$ by
\begin{equation}
	\label{e:acevent}
	A_{2^{-i}}=\left\{ (\mathtt{h},d)\in \mathcal{D}'(\mathbb{C}_{(2^{-i},1)})\times
	C(\mathbb{C}_{(2^{-i},1)}\times \mathbb{C}_{(2^{-i},1)}): d\neq
	\Psi_{\mathbb{C}_{(2^{-i},1)}}(\mathtt{h}) \right\}.
\end{equation}
The measurability of the above set follows by checking if $d(x,y)\neq
	\Psi_{\mathbb{C}_{(2^{-i},1)}}(\mathtt{h})(x,y)$ for any $x,y$ in the
	countable set $\mathbb{Q}^2\cap \left( \mathbb{C}_{(2^{-i},1)}\times
	\mathbb{C}_{(2^{-i},1)} \right)$ and then using the continuity of
	$d,\Psi_{\mathbb{C}_{(2^{-i},1)}}(\mathtt{h})$ along with the
	measurability of $\Psi_{\mathbb{C}_{(2^{-i},1)}}$. As a consequence of the locality of the LQG metric (Proposition
	\ref{b3} (2)) along with the a.s.\ equality
	$D_{h\lvert_{\mathbb{C}_{(2^{-i},1)}}}=\Psi_{\mathbb{C}_{(2^{-i},1)}}(h\lvert_{\mathbb{C}_{(2^{-i},1)}})$, we obtain 
\begin{equation}
		\label{e:acevent1}
		\mathbb{P}\left(
		\left(h\lvert_{\mathbb{C}_{(2^{-i},1)}},D_h(\cdot,\cdot;\mathbb{C}_{(2^{-i},1)})\right)\in
		A_{2^{-i}}\right)=0.
	\end{equation}
	By an application of Proposition
	\ref{ac:2.1*}, we obtain that
	\begin{equation}
		\label{e:acevent2}
		\mathbb{P}_\infty\left(
		\left(\mathbf{Field}\lvert_{\mathbb{C}_{(2^{-i},1)}},\mathbf{Metric}(\cdot,\cdot;\mathbb{C}_{(2^{-i},1)})\right)\in
		A_{2^{-i}}\right)=0
	\end{equation}
	and this completes the proof.
\end{proof}
We now use the above along with the version
$\Psi_{\mathbb{C}_{(0,1)}}$ from \eqref{e:ff2} to conclude that the induced
metric $\mathbf{Metric}(\cdot,\cdot;\mathbb{C}_{(0,1)})$ can be obtained as a
function of the restricted field $\mathbf{Field}\lvert_{\mathbb{C}_{(0,1)}}$.
\begin{lemma}
	\label{f-field:2}
	Almost surely,
	$\mathbf{Metric}(\cdot,\cdot;\mathbb{C}_{(0,1)})=\Psi_{\mathbb{C}_{(0,1)}}(\mathbf{Field}\lvert_{\mathbb{C}_{(0,1)}}).$	
\end{lemma}
\begin{proof}
By Lemma \ref{f-field:1},
	\begin{equation}
		\label{e:acevent2.1}
		\mathbf{Metric}(\cdot,\cdot;\mathbb{C}_{(2^{-i},1)})=\Psi_{\mathbb{C}_{(2^{-i},1)}}(\mathbf{Field}\lvert_{\mathbb{C}_{(2^{-i},1)}})
	\end{equation}
	almost surely for all $i\in \mathbb{N}$ simultaneously. Let $x,y\in
	\mathbb{C}_{(0,1)}$. Note that any path $\zeta:x \rightarrow y$
	satisfying $\zeta\subseteq \mathbb{C}_{(0,1)}$ must satisfy
	$\zeta\subseteq \mathbb{C}_{(2^{-i},1)}$ for some $i$ large enough. Thus
	by using the definition of the induced metric as the infimum over
	lengths of paths contained in the domain, we obtain that
	\begin{equation}
		\label{e:acevent3}
		\mathbf{Metric}(x,y;\mathbb{C}_{(0,1)})=\lim_{i\rightarrow
		\infty}\mathbf{Metric}(x,y;\mathbb{C}_{(2^{-i},1)})
	\end{equation}
	for all $x,y\in \mathbb{C}_{(0,1)}$. Note that the above also shows that
	the limit on the right hand side almost surely makes sense simultaneously
	for all $x,y\in \mathbb{C}_{(0,1)}$. Along with \eqref{e:acevent2.1} and
	the choice of the version
$\Psi_{\mathbb{C}_{(0,1)}}$ from \eqref{e:ff2},
	this implies that we almost surely have
	\begin{equation}
		\label{e:acevent4}
		\mathbf{Metric}(\cdot,\cdot;\mathbb{C}_{(0,1)})=\lim_{i\rightarrow
		\infty}
		\Psi_{\mathbb{C}_{(2^{-i},1)}}(\mathbf{Field}\lvert_{\mathbb{C}_{(2^{-i},1)}})=\Psi_{\mathbb{C}_{(0,1)}}(\mathbf{Field}\lvert_{\mathbb{C}_{(0,1)}}).
	\end{equation}
	To obtain the last a.s.\ equality, we have used that
	$\mathbb{P}_\infty\left( \mathbf{Field}\in A \right)=1$ and along with the fact that in \eqref{e:ff2}, we had defined
	$\Psi_{\mathbb{C}_{(0,1)}}(\mathtt{h}) =\lim_{i\rightarrow
	\infty}\Psi_{\mathbb{C}_{(2^{-i},1)}}(\mathtt{h}\lvert_{\mathbb{C}_{(2^{-i},1)}})$
	for all $\mathtt{h}\in A$.
\end{proof}
Note that Lemma \ref{f-field:2} shows that
$\mathbf{Metric}(\cdot,\cdot;\mathbb{C}_{(0,1)})$ can be constructed from
$\mathbf{Field}$, but we want to show a corresponding statement for
$\mathbf{Metric}$ itself. To do so, we will show that the metrics
$\mathbf{Metric}$ and $\mathbf{Metric}(\cdot,\cdot;\mathbb{C}_{(0,1)})$ are
almost surely equal when restricted to points in $\mathbb{C}_{(0,1)}$.
Towards this, let $U\subseteq 
	C(\mathbb{D}\times \mathbb{D})$ be defined by 

\begin{equation}
	\label{e:acevent4.1}
	U=\left\{ d\in 
	C(\mathbb{D}\times \mathbb{D}): d \text{ is a length metric} ,
	d\lvert_{\mathbb{C}_{(0,1)}\times
	\mathbb{C}_{(0,1)}}\neq
	d(\cdot,\cdot;\mathbb{C}_{(0,1)})\right\}.
\end{equation}

The measurability of the above set follows by checking the equality of
$d\lvert_{\mathbb{C}_{(0,1)}\times
	\mathbb{C}_{(0,1)}},
	d(\cdot,\cdot;\mathbb{C}_{(0,1)})$ (refer to Section \ref{ss:meas} for the
	measurability of the map $d\mapsto 	d(\cdot,\cdot;\mathbb{C}_{(0,1)})$) for
rational points in $\mathbb{D}\times \mathbb{D}$ and then using that both
the objects are continuous functions. The rest of the subsection will be devoted to proving the following lemma.
\begin{lemma}
	\label{f-field:3.1}
	Almost surely
	in the randomness of $h$, we have
\begin{displaymath}	
	\mathbbm{1}\left(
	\underline{D}_{h,\Gamma_t}\in U\right)=0
	\end{displaymath}
	for almost every $t\in(0,\infty)$.
\end{lemma}
Before giving the proof of Lemma \ref{f-field:3.1}, we first finish the proof
of Proposition \ref{mainmeas1} assuming Lemma \ref{f-field:3.1}.

\begin{proof}[Proof of Proposition \ref{mainmeas1} using Lemma \ref{f-field:3.1}]
	By an application of Lemma \ref{null2} along with Lemma \ref{f-field:3.1},
	$\mathbf{Metric}\lvert_{\mathbb{C}_{(0,1)}\times\mathbb{C}_{(0,1)}}=\mathbf{Metric}(\cdot,\cdot;\mathbb{C}_{(0,1)})$
	almost surely. On further applying Lemma \ref{f-field:2}, we obtain
	$\mathbf{Metric}\lvert_{\mathbb{C}_{(0,1)}\times\mathbb{C}_{(0,1)}}=\Psi_{\mathbb{C}_{(0,1)}}(\mathbf{Field}\lvert_{\mathbb{C}_{(0,1)}})$
	almost surely. 
	By definition $\mathbf{Metric}$ is a random continuous metric and thus
	$\Psi_{\mathbb{C}_{(0,1)}}(\mathbf{Field}\lvert_{\mathbb{C}_{(0,1)}})$ a.s.\
	has a continuous extension to $\mathbb{D}\times \mathbb{D}$ which is equal
	to $\mathbf{Metric}$.

	The above along with \eqref{f-field:2} implies that
	$\mathbf{Field}$ is
	a.s.\ inside the set in \eqref{Ainter} and thus $\mathbb{P}_\infty\left(
	\mathbf{Field}\in \widetilde{A}
	\right)=1$.
	In \eqref{e:ff3}, $\Psi_\mathbb{D}(\mathtt{h})$ for $\mathtt{h}\in
	\widetilde{A}$ was
	defined to be the continuous extension of
	$\Psi_{\mathbb{C}_{(0,1)}}(\mathtt{h}\lvert_{\mathbb{C}_{(0,1)}})$ to
	$\mathbb{D}\times \mathbb{D}$ and thus the above yields
	\begin{equation}
		\label{e:acevent4.3}
		\mathbf{Metric}=\Psi_{\mathbb{D}}(\mathbf{Field})
	\end{equation}
	almost surely. This completes the proof.
\end{proof}

\subsection{Proof of Lemma \ref{f-field:3.1}}
\label{ss:lem}
We will roughly show that almost surely in
the randomness of $h$, for every $t\in (0,\infty)$, the point $\Gamma_t$ has
paths of arbitrary small lengths around it (referred to  as ``short--loops'') which disconnect $\Gamma_t$ from $\infty$. Thus
any path $\zeta\subseteq \mathbb{D}_{\delta|\Gamma_t|}$ going via $\Gamma_t$ can be
modified to some path $\zeta'$ which does not pass through $\Gamma_t$. 
Since arbitrarily small short--loops exist, the LQG length of $\zeta'$ minus the LQG length of $\zeta$ can be made
arbitrarily small and this will be enough to complete the proof of Lemma
\ref{f-field:3.1}.

In fact, to avoid path
lengths and work with distances instead, we avoid the use of short--loops and
show that $	\sup_{u,v\in
		\mathbb{T}_{r}(z)}D_h(u,v;\mathbb{C}_{(r/2,2r)}(\Gamma_t))$ goes to $0$ as
$r\rightarrow 0$. We note that the effect of the occurrence of the above event for arbitrary small values of
$r$ is intuitively the same as the presence of arbitrarily small short--loops around the
point $\Gamma_t$.
The following lemma can be obtained by applying the H\"older
continuity estimate Lemma \ref{imp3} with $r=1$ and will be useful to us.
\begin{lemma}
	\label{gridholder}
	For each $\chi\in (0,\xi(Q-2))$ and each compact set $S\subseteq
	\mathbb{C}$, we almost surely have
	$D_h(u,v;\mathbb{D}_{2|v-u|}(u))\leq |u-v|^\chi$ for all $u,v\in
		S$ with $|u-v|\leq \varepsilon$, where $\varepsilon$ is sufficiently small
		depending on the instance of $h$.
\end{lemma}

We now write the complex plane as a countable union
$\mathbb{C}=\cup_{n\in \mathbb{N}}\overline{\mathbb{D}}_n$ of increasing closed balls
around the origin. Choose $\chi\in (0,\xi(Q-2))$ which will remain fixed
throughout this section. We denote the event obtained by using the above lemma
with $\overline{\mathbb{D}}_n$ by $\mathcal{E}_n$ and define
$\mathcal{E}=\cap_n\mathcal{E}_n$. Since each $\mathcal{E}_n$ occurs a.s., we
have $\mathbb{P}\left( \mathcal{E} \right)=1$.

On the event $\mathcal{E}$, all points in
$\mathbb{C}$ have arbitrarily small short--loops around them and we record this in the following
lemma.
\begin{lemma}
	\label{f-field:6}
	On $\mathcal{E}$, the following is true simultaneously for all $z\in
	\mathbb{C}$. 
		\begin{displaymath}
		\sup_{u,v\in
		\mathbb{T}_{r}(z)}D_h(u,v;\mathbb{C}_{(r/2,2r)}(z))\rightarrow
		0 \text{ as } r\rightarrow 0.
	\end{displaymath}

\end{lemma}
\begin{proof}
	It suffices to show that the above holds simultaneously for all
	$z\in\mathbb{D}_{n}$ for any fixed $n$. Let
	$\varepsilon$ be the random quantity obtained by applying Lemma
	\ref{gridholder} to the compact set $\overline{\mathbb{D}}_{n}$. For
	any $r<\min(\mathtt{dist}(z,\mathbb{T}_n)/2,\varepsilon)$, we choose $M$
	points ($M$ is a universal constant)
	$q_1,q_2,\dots,q_{M-1},q_M=q_1$ such that $|q_i-q_{i+1}|<r/32$ (for all
	$i$) with the property that for any $u\in \mathbb{T}_r(z)$,
	there exists a $q^u$ among the $q_i$'s such that $|u-q^u|<r/32$. Note that this implies that $\mathbb{D}_{2|q_i-q_{i+1}|}(q_i)\subseteq\mathbb{D}_{r/16}(q_i)\subseteq
	\mathbb{C}_{(r/2,2r)}(z)$ for all $i$. Now for any $u,v\in \mathbb{T}_r(z)$, we have $D_h(u,v;\mathbb{C}_{(r/2,2r)}(z))\leq
	D_h(u,q^u;\mathbb{D}_{2|u-q^u|}(z))+D_h(v,q^v;\mathbb{D}_{2|v-q^v|}(z))+\sum_{i=1}^{M-1}D_h(q_i,q_{i+1};\mathbb{D}_{2|q_i-q_{i+1}|})\leq
	(M+1)(r/32)^\chi\rightarrow 0$ as $r\rightarrow 0$. This completes the
	proof.
\end{proof}
We are now ready to complete the proof of Lemma \ref{f-field:3.1}.
\begin{proof}[Proof of Lemma \ref{f-field:3.1}]
	By the definition of the set $U$ from \eqref{e:acevent4.1}, it
	suffices to show that almost surely in the randomness of $h$,
	\begin{equation}
		\label{e:acevent7}
		\underline{D}_{h,\Gamma_t}\lvert_{\mathbb{C}_{(0,1)}\times
	\mathbb{C}_{(0,1)}}= \underline{D}_{h,\Gamma_t}(\cdot,\cdot;\mathbb{C}_{(0,1)})
	\end{equation}
	for all $t\in (0,\infty)$. Using the definition of
	$\underline{D}_{h,\Gamma_t}$ from \eqref{e:metric}, we need to
	equivalently show that a.s.\ we have
	\begin{equation}
		\label{e:acevent8}
	D_h(x,y;\mathbb{D}_{\delta|\Gamma_t|}(\Gamma_t))=	D_h(x,y;\mathbb{C}_{(0,\delta|\Gamma_t|)}(\Gamma_t))
	\end{equation}
	for all $t\in (0,\infty)$ and $x,y\in
	\mathbb{C}_{(0,\delta|\Gamma_t|)}(\Gamma_t)$. We will show
	the above on the event $\mathcal{E}$.	
	For the rest of the proof, we assume that the sample
	point $\omega\in \mathcal{E}$. Given a triple $t,x,y$ such that $t\in (0,\infty)$ and $x,y\in
	\mathbb{C}_{(0,\delta|\Gamma_t|)}(\Gamma_t)$,
	first note that
	$D_h(\cdot,\cdot;\mathbb{D}_{\delta|\Gamma_t|}(\Gamma_t))$ is a.s.\ a
	continuous length metric since $D_h$ itself is a continuous length metric
	(see Lemma \ref{lengthcts}). Thus given an $\varepsilon>0$, there exists a path
	$\zeta:[a,b]\rightarrow \mathbb{C}$ from $x$ to $y$
	satisfying $\zeta\subseteq \mathbb{D}_{\delta|\Gamma_t|}(\Gamma_t)$
	such that
	\begin{equation}
		\label{e:acevent9}
	\ell(\zeta;D_h)\leq
	D_h(x,y;\mathbb{D}_{\delta|\Gamma_t|}(\Gamma_t))+\varepsilon.	
	\end{equation}
	 In case $\Gamma_t\notin \zeta$, this immediately implies that 
	\begin{equation}
		\label{e:acevent10}
		D_h(x,y;\mathbb{C}_{(0,\delta|\Gamma_t|)}(\Gamma_t))\leq
	D_h(x,y;\mathbb{D}_{\delta|\Gamma_t|}(\Gamma_t))+\varepsilon.
	\end{equation}
	We now treat the case when $\zeta$ does hit $\Gamma_t$. First note that we can
	assume that $\zeta(s)=\Gamma_t$ for a unique value of $s$. This is because otherwise we can
	define $a_1=\inf\left\{ s:\zeta(s)=\Gamma_t \right\}$ and $b_1=\sup\left\{
	s:\zeta(s)=\Gamma_t
	\right\}$ and simply replace $\zeta$ by the
	concatenation of $\zeta\lvert_{[a,a_1]}$ and $\zeta\lvert_{[b_1,b]}$.
	
	We now use Lemma \ref{f-field:6} to obtain $r$ such
	that $5r/2< \min\left\{
	|\Gamma_t-x|,|\Gamma_t-y|
	\right\}$ and
	\begin{equation}
		\label{e:acevent11}
		\sup_{u,v\in
		\mathbb{T}_{r}(\Gamma_t)}D_h(u,v;\mathbb{C}_{(r/2,2r)}(\Gamma_t))<\varepsilon.
	\end{equation}
	Let $s_2\in [a,b]$ be the unique value such that
	$\zeta(s_2)=\Gamma_t$. 
	Note that $x,y\notin \mathbb{C}_{<2r}(\Gamma_t)$ and this is a
	consequence of the condition $5r/2< \min\left\{
	|\Gamma_t-x|,|\Gamma_t-y|
	\right\}$. Thus by continuity,  the path $\zeta$ must hit the
	circle $\mathbb{T}_{r}(\Gamma_t)$ at least once before and once after the time $s_2$. We thus define $s_1,s_3\in [a,b]$ such that
	$a<s_1<s_2<s_3<b$ and $\zeta(s_1),\zeta(s_3)\in \mathbb{T}_{r}(\Gamma_t)$.

	Recall from Lemma \ref{lengthcts} that a.s.\,\emph{all possible} induced metrics of $D_h$ are
	simultaneously continuous length metrics. Using this along with \eqref{e:acevent11}, we can obtain a path $\zeta_1$
	from $\zeta(s_1)$ to $\zeta(s_3)$ satisfying $\zeta_1\subseteq
	\mathbb{C}_{(r/2,2r)}(\Gamma_t) \subseteq
	\mathbb{C}_{(0,\delta|\Gamma_t|)}(\Gamma_t)$ along with
	$\ell(\zeta_1;D_h)\leq 2\varepsilon$. Now consider the path $\zeta_2$ from
	$x$ to $y$ obtained by composing the paths $\zeta\lvert_{[a,s_1]}$,
	$\zeta_1$ and $\zeta\lvert_{[s_3,b]}$. Since $\ell(\zeta_1;D_h)\leq
	2\varepsilon$ and
	$\ell(\zeta\lvert_{[a,s_1]};D_h)+\ell(\zeta\lvert_{[s_3,b]};D_h)\leq
	\ell(\zeta;D_h)$, we have, by using \eqref{e:acevent10}, 
	\begin{equation}
		\label{e:acevent12}
		\ell(\zeta_2;D_h)\leq
		D_h(x,y;\mathbb{D}_{\delta|\Gamma_t|}(\Gamma_t))+3\varepsilon
	\end{equation}
	and note that $\zeta_2$ is a path from $x$ to $y$ satisfying
	$\zeta_2\subseteq 	\mathbb{C}_{(0,\delta|\Gamma_t|)}(\Gamma_t)$. This now
	implies that even in the case when $\zeta$ does touch $\Gamma_t$, 
	\begin{equation}
		\label{e:acevent13}
		D_h(x,y;\mathbb{C}_{(0,\delta|\Gamma_t|)}(\Gamma_t))\leq
	D_h(x,y;\mathbb{D}_{\delta|\Gamma_t|}(\Gamma_t))+3\varepsilon.
	\end{equation}
	Since $\varepsilon$ was  arbitrary, 
 we obtain $D_h(x,y;\mathbb{C}_{(0,\delta|\Gamma_t|)}(\Gamma_t))\leq
	D_h(x,y;\mathbb{D}_{\delta|\Gamma_t|}(\Gamma_t)),$
and hence \eqref{e:acevent8}. This concludes the proof.
\end{proof}

\bibliography{refs}
\bibliographystyle{plain}

\end{document}